%% file: MF_final_Arxiv.tex
\documentclass[10pt]{amsart}

\usepackage{bmpsize}
\usepackage{feynmp}
\usepackage{latexsym}
\usepackage{amsfonts}
\usepackage{amssymb}
\usepackage{amscd}
\usepackage{amsbsy}
\usepackage{amsmath,tabu}
\usepackage[dvipdfmx]{graphicx}
\usepackage[all]{xy}
\usepackage{hyperref}
\usepackage{latexsym}
\usepackage{amsfonts}
\usepackage{amssymb}
\usepackage{amscd}
\usepackage{amsbsy}
\usepackage{amsmath}
\usepackage{bbm}
\usepackage{euscript}
\usepackage{mathrsfs}
\usepackage{yhmath}

\usepackage{tikz}

\usepackage{hyperref}

\usetikzlibrary{shapes,fit,intersections}
\usetikzlibrary{decorations.markings}
\usetikzlibrary{decorations.pathreplacing}
\usetikzlibrary{patterns}
\usetikzlibrary{matrix,arrows}
\usepgflibrary{decorations.pathmorphing}

\usepackage{tikz-cd}

\usepackage{mathrsfs}
\DeclareMathAlphabet{\mathscrbf}{OMS}{mdugm}{b}{n}

\usepackage{stackengine}
\stackMath

\usepackage{imakeidx}

\newrobustcmd{\SymbolPrint}[1]{\ensuremath{#1}\index[notationlist]{\ensuremath{#1}}}%

\makeindex[intoc=true,name=notationlist,columns=3,title={List of Symbols}]

\input gendef.tex
\input basdef.tex
\input letdef.tex

\input semlocdef.tex
\def\IR{ \mathbb{R} }
\def\IC{ \mathbb{C} }

\def\a{ \alpha }

\def\k{ \kappa }

\def\tho{ \theta_1 }

\def\xa{ a }
\def\xb{ b }

\def\rmT{ \mathrm{T} }

\def\ZZ{\mathbb{Z} }

\def\dgTsc{ \deg_{\Tsc}}
\def\shdg#1{ \mathsf{#1}}
\def\shq{ \shdg{q}}
\def\sht{ \shdg{t}}
\def\sha{}

\usepackage{todonotes}

\begin{document}
\today  

\newlength{\gap}
\setlength{\gap}{0.4cm}

\setlength{\unitlength}{3947sp}
\setlength{\unitlength}{1mm}

\def\xG{ G }
\def\xGC{ \xG_{\IC} }
\def\xal{\alpha}
\def\xbet{\beta}
\def\xgm{\gamma}
\def\xLv{ \mathbb{L} }
\def\xft{ \mathfrak{h} }
\def\xftC{ \xft_{\IC} }
\def\xfg{ \mathfrak{g}}
\def\xfgC{ \xfg_{\IC} }

\def\xTr{ \mathbb{T} }

\def\xalc{ \xal_{\mathrm{c}} }
\def\lY{ Y }
\def\cY{ \mathcal{\lY} }
\def\oO{ O }
\def\oOcv#1{ \oO_{#1} }
\def\oOca{ \oOcv{\xalc} }
\def\oOs{ \oO^{\ast} }
\def\oOscv#1{ \oO_{\xLv;#1} }
\def\oOsca{ \oOscv{\xalc} }

\def\xa{ a }
\def\xb{ b }
\def\xc{ c }
\def\xaz{ \xa_0 }
\def\xbz{ \xb_0 }
\def\xcz{ \xc_0 }

\def\xain{ \xa_\infty }
\def\xbin{ \xb_\infty }
\def\xcin{ \xc_\infty }

\def\rT{ \mathrm{T} }
\def\rTs{ \rT^\ast }

\def\xg{ g }
\def\xgc{ \xg_{\mathrm{c}} }

\def\xV{ V }
\def\xcV{ \mathcal{V} }
\def\lYv#1{\lY_{\xLv;#1} }
\def\lYalc{ \lYv{\xalc} }
\def\cYv#1{\cY_{\xLv;#1} }
\def\cYalc{ \cYv{\xalc} }

\def\mnN{ N }
\def\mnNv#1{ N_{#1} }
\def\mnNo{ \mnNv{1}}
\def\mnNt{ \mnNv{2}}
\def\mnNopt{ \mnNo+\mnNt }

\def\cM{ \mathcal{M} }
\def\cMv#1{ \cM_{\xLv;#1} }
\def\cMalc{ \cMv{\xlac} }

\def\SUv#1{ \mathrm{SU}(#1) }
\def\SUN{ \SUv{\mnN} }

\def\SLv#1{ \mathrm{SL}(#1) }
\def\SLN{ \SLv{\mnN} }

\def\GLv#1{ \mathrm{GL}(#1) }
\def\GLN{ \GLv{\mnN} }

\def\glv#1{ \mathfrak{gl}(#1) }
\def\glN{ \glv{\mnN}}

\def\slv#1{ \mathfrak{sl}(#1) }
\def\slt{ \slv{2}}

\def\rmT{ \mathrm{T}}
\def\rmTs{ \rmT^*}
\def\rmTsB{ \rmTs\cB }
\def\mfgs{ \mfg^* }

\def\bmtr#1{ \begin{bmatrix} #1 \end{bmatrix} }
\def\spob{ O }

\def\flF{ F }

\def\cM{ \mathcal{M} }
\def\cMS{ \cM_{\mathrm{S}} }
\def\cF{ \mathcal{F} }
\def\cB{ \mathcal{B}}
\def\cN{ \mathcal{N}}
\def\xF{ F }
\def\SLNC{ \mathrm{SL}(N,\IC)}
\def\SLtC{ \mathrm{SL}(2,\IC)}
\def\slNC{ \mathfrak{sl}(N,\IC)}
\def\xB{ B }
\def\mfb{\mathfrak{b}}
\def\mfn{\mathfrak{n}}
\def\mfg{\mathfrak{g}}

\def\Wz{ W_0 }
\def\spW{ W }

\def\pidv#1{ \pi^*_{#1}}
\def\pido{ \pidv{1}}
\def\pidt{ \pidv{2}}
\def\pidot{ \pidv{12}}
\def\pidth{ \pidv{23}}
\def\pidoh{ \pidv{13}}

\def\piuv#1{ \pi_{*,#1}}
\def\piuoh{ \piuv{13}}

\def\piv#1{ \pi_{#1}}
\def\pio{ \piv{1}}
\def\pit{ \piv{2}}
\def\piot{ \piv{12}}
\def\pith{ \piv{23}}
\def\pioh{ \piv{13}}
\def\piij{ \piv{ij}}

\def\mfG{ \xG }


\def\MF{ \mathrm{MF}^{sc} }
\def\MFcrit{ \mathrm{MF}^{\mathrm{sc,crit}}}
\def\MFs{\mathrm{MF}}
\def\MFG{ \MF_{\mfG} }
\def\MFBpt{ \MF_{\mBp\times\mBp}}
\def\MFGvv#1#2{ \MFG(#1;#2) }
\def\MFBptvv#1#2{ \MFBpt(#1;#2) }
\def\MFGXgW{ \MFGvv{\Xg}{\mfW} }
\def\MFGtW{ \MFGvv{\mfg\times\rmTsB\times\rmTsB}{\spW}}
\def\MFBptW{ \MFBptvv{\xmvr}{\spW} }


\def\xmvr{ \mfg\times\mfnp\times\mfnp \times\xG }

\def\MFvvv#1#2#3{ \MF_{#1} (#2;#3) }

\def\mfcE{ E }
\def\mfcF{ F }

\def\brgr{ \mathfrak{Br}}
\def\brgrv#1{ \brgr_{#1}}
\def\brgrN{ \brgrv{\mnN} }
\def\brgrNo{ \brgrv{\mnNo} }
\def\brgrNt{ \brgrv{\mnNt} }
\def\brgrNopt{ \brgrv{\mnNopt}}

\def\Ob{ \mathrm{Ob}}

\def\ctb#1{ [\![#1]\!] }
\def\dmm{ \cdot }

\def\brb{ \beta }
\def\brbv#1{ \brb_{#1}}
\def\brbo{ \brbv{1}}
\def\brbt{ \brbv{2}}

\def\xId{ \mathbbm{1}}
\def\xIdv#1{ \xId_{#1}}

\def\mB{ B }
\def\mBp{ \mB_+ }
\def\mfnp{ \mfn_+ }
\def\mfbp{ \mfb_+ }

\def\xX{ X }
\def\xY{ Y }
\def\xYv#1{ \xY_{#1} }
\def\xYz{ \xYv{0} }
\def\xYo{ \xYv{1} }
\def\xYt{ \xYv{2} }
\def\xYh{ \xYv{3} }
\def\xYi{ \xYv{i} }

\def\gxv#1{ g_{#1}}
\def\gxo{ \gxv{1}}
\def\gxt{ \gxv{2}}
\def\gxh{ \gxv{3}}
\def\gxot{ \gxv{12}}
\def\gxth{ \gxv{23}}
\def\gxoh{ \gxv{13}}

\def\xmctt{ \End(\spob) }

\def\xmtr#1{ \begin{bmatrix} #1 \end{bmatrix} }
\def\kmtr#1{ \begin{bmatrix} #1 \end{bmatrix} }

\def\dltv#1{ \delta_{#1} }
\def\dlto{ \dltv{1} }
\def\dltt{ \dltv{2} }

\def\xv#1{ x_{#1} }
\def\xo{ \xv{1} }
\def\xz{ \xv{0} }
\def\xmo{ \xv{-1}}

\def\yv#1{ y_{#1} }
\def\yo{ \yv{1} }
\def\yt{ \yv{2} }
\def\yi{ \yv{i} }

\def\exa{ a }
\def\exav#1{ \exa_{#1} }
\def\exaoo{ \exav{11} }
\def\exaot{ \exav{12} }
\def\exato{ \exav{21} }
\def\exatt{ \exav{22} }

\def\txz{ \tilde{x}_0 }
\def\ttxz{ \tilde{\tilde{x}}_0 }

\def\xxA{ A }
\def\Zt{ \ZZ_2 }
\def\xD{ D }

\def\mBpdb{\mBp\!\times\!\mBp}
\def\xxM{ M }
\def\xxV{ V }
\def\xdl{ d_{\mathrm{l}} }
\def\xdr{ d_{\mathrm{r}} }

\def\Ksz{ \mathrm{K} }
\def\Kszvvv#1#2#3{ \Ksz(#1;#2,#3) }
\def\ktht{\theta}


\def\crX{\mathcal{C}}
\def\crXv#1{ \crX_{#1}}
\def\crXpr{\crXv{\parallel}}
\def\crXbl{\crXv{\bullet}}
\def\crXp{\crXv{+}}
\def\crXn{\crXv{-}}

\def\brwsh#1#2{ { \langle #1,#2 \rangle } }
\def\brhsh#1#2#3{{ \langle #1,#2,#3 \rangle }}
\def\ztsh#1{ [#1] }
\def\ztsho{ \ztsh{1} }

\def\rgR{ R }
\def\rgRshv#1#2{ \rgR_{\scriptscriptstyle \brwsh{#1}{#2 } } }
\def\rgRmomo{ \rgRshv{-1}{-1}}

\def\mcnv{ * }

\def\xchv#1{ \chi_{#1} }
\def\xchp{ \xchv{+} }
\def\xchm{ \xchv{-} }

\def\exav#1{a_{#1}}
\def\exaoo{\exav{11}}
\def\exaot{\exav{12}}
\def\exato{\exav{21}}
\def\exatt{\exav{22}}
\def\exbv#1{b_{#1}}
\def\exboo{\exbv{11}}
\def\exbot{\exbv{12}}
\def\exbto{\exbv{21}}
\def\exbtt{\exbv{22}}
\def\excv#1{c_{#1}}
\def\excoo{\excv{11}}
\def\excot{\excv{12}}
\def\excto{\excv{21}}
\def\exctt{\excv{22}}

\def\xXtl{ \tilde{\xX} }
\def\dgBs{ \deg_{\mflg^2}}
\def\dltv#1{ \delta_{#1}}
\def\dlto{ \dltv{1} }
\def\dltt{ \dltv{2}}
\def\dlth{ \dltv{3}}


\def\yh{y_3}


\def\xYo{ \xY_1 }
\def\xYt{ \xY_2 }
\def\xYpo{ \xYp_1}
\def\xYpt{ \xYp_2}
\def\yo{ y_1 }
\def\yt{ y_2 }
\def\xgo{ g_1 }
\def\xgt{ g_2 }
\def\xg{ g }
\def\xgot{ g_{12} }
\def\xgth{ g_{23} }
\def\xgoh{ g_{13} }
\def\xgof{ g_{14} }
\def\xghf{ g_{34} }
\def\xigot{ g^{-1}_{12} }
\def\xigth{ g^{-1}_{23} }
\def\xigoh{ g^{-1}_{13} }
\def\xigof{ g^{-1}_{14} }
\def\xighf{g^{-1}_{34} }
\def\xh{ h }
\def\xho{ h_1 }
\def\xht{ h_2 }
\def\xXtl{ \tilde{\xX} }
\def\dgBs{ \deg_{\mflg^2}}
\def\dltv#1{ \delta_{#1}}
\def\dlto{ \dltv{1} }
\def\dltt{ \dltv{2}}

\def\xbmtr#1{\begin{bmatrix} #1 \end{bmatrix}}



\def\exav#1{a_{#1}}
\def\exaoo{\exav{11}}
\def\exaot{\exav{12}}
\def\exato{\exav{21}}
\def\exatt{\exav{22}}
\def\exaoh{\exav{13}}
\def\exath{\exav{23}}
\def\exaho{\exav{31}}
\def\exaht{\exav{32}}
\def\exahh{\exav{33}}
\def\exbv#1{b_{#1}}
\def\exboo{\exbv{11}}
\def\exbot{\exbv{12}}
\def\exbto{\exbv{21}}
\def\exbtt{\exbv{22}}
\def\exboh{ \exbv{13}}
\def\exbth{ \exbv{23} }
\def\exbhh{ \exbv{33}}
\def\exbho{ \exbv{31}}
\def\exbht{ \exbv{32} }
\def\excv#1{c_{#1}}
\def\excoo{\excv{11}}
\def\excot{\excv{12}}
\def\excto{\excv{21}}
\def\exctt{\excv{22}}
\def\excoh{ \excv{13}}
\def\excth{ \excv{23} }
\def\exchh{ \excv{33}}
\def\excho{ \excv{31}}
\def\excht{ \excv{32} }


\def\xoo{x_{11}}
\def\xot{x_{12}}
\def\xoh{x_{13}}
\def\xto{x_{21}}
\def\xtt{x_{22}}
\def\xth{x_{23}}
\def\xho{x_{31}}
\def\xht{x_{32}}
\def\xhh{x_{33}}
\def\txoo{\tx_{11}}
\def\txot{\tx_{12}}
\def\txoh{\tx_{13}}
\def\txto{\tx_{21}}
\def\txtt{\tx_{22}}
\def\txth{\tx_{23}}
\def\txho{\tx_{31}}
\def\txht{\tx_{32}}
\def\txhh{\tx_{33}}


\def\brh{ h }
\def\brht{ \brh_2 }
\def\briht{ \brh^{-1}_2 }
\def\ybet{ \beta }
\def\ybetth{ \ybet_{23} }
\def\ybettt{ \ybet_{22} }
\def\ybethh{ \ybet_{33} }


\newcommand{\bpi}{\bar{\pi}}
\newcommand{\bC}{\bar{C}}
\newcommand{\bR}{\bar{R}}
\newcommand{\bcalX}{\bar{\calX}}
\newcommand{\brpi}{\bar{\pi}}
\newcommand{\bcalC}{\bar{\calC}}


\newcommand{\scX}{\mathscr{X}}


\newcommand{\bfG}{\mathbf{G}}
\newcommand{\bfB}{\mathbf{B}}
\newcommand{\bfQ}{\mathbf{Q}}


\newcommand{\calB}{\mathcal{B}}
\newcommand{\calC}{\mathcal{C} }
\newcommand{\calF}{\mathcal{F}}
\newcommand{\calG}{\mathcal{G}}
\newcommand{\calE}{\mathcal{E}}
\newcommand{\calI}{\mathcal{I}}
\newcommand{\calH}{\mathcal{H}}
\newcommand{\calO}{\mathcal{O}}
\newcommand{\calP}{\mathcal{P}}
\newcommand{\calS}{\mathcal{S}}
\newcommand{\calK}{\mathcal{K}}
\newcommand{\calT}{\mathcal{T}}
\newcommand{\calZ}{\mathcal{Z}}


\newcommand{\bbS}{\mathbb{S}}

\newcommand{\cont}{\mathrm{Cont}}
\newcommand{\forg}{\mathrm{fgt}}
\newcommand{\ir}{\bar{i}}
\newcommand{\cl}{\mathrm{cl}}

\newcommand{\cohog}[2]{\mathrm{H}^{#1}(#2)}
\newcommand{\CE}{\mathrm{CE}}
\newcommand{\bCE}{\overline{\mathrm{CE}}}
\newcommand{\Com}{\mathrm{Com}}
\newcommand{\Mod}{\mathrm{Mod}}
\newcommand{\odel}{\stackon{\otimes}{\scriptstyle\Delta}}

\newcommand{\CC}{\mathbb{C}}

\newcommand{\Hilb}{\mathrm{Hilb}}
\newcommand{\Fl}{\mathrm{Fl}}
\newcommand{\calX}{\scX}
\newcommand{\calXr}{\overline{\scX}}
\newcommand{\calXsm}{\overline{\mathcal{X}}}
\newcommand{\calXsmu}{\underline{\calXsm}}
\newcommand{\calXu}{\underline{\calX}}
\newcommand{\St}{\mathrm{St}}
\newcommand{\scXr}{\bar{\mathscrbf{X}}}
\newcommand{\calCV}{\mathcal{CV}}

\newcommand{\supp}{\mathrm{supp}}
\newcommand{\Ker}{\mathrm{Ker}}
\newcommand{\adj}{\mathrm{adj}}
\newcommand{\ind}{\mathrm{ind}}
\newcommand{\indb}{\overline{\mathrm{ind}}}
\newcommand{\Ind}{\mathrm{Ind}}
\newcommand{\Indb}{\overline{\Ind}}
\newcommand{\Lie}{\mathrm{Lie}}
\newcommand{\xwr}{\mathrm{wr}}
\newcommand{\Trg}{\mathcal{T}r}

\newcommand{\frb}{\mathfrak{b}}
\newcommand{\frh}{\mathfrak{h}}
\newcommand{\frn}{\mathfrak{n}}
\newcommand{\frt}{\mathfrak{t}}
\newcommand{\gl}{\mathfrak{gl}}
\newcommand{\frg}{\mathfrak{g}}
\newcommand{\frq}{\mathfrak{q}}

\newcommand{\Brgr}{\mathfrak{Br}}

\newcommand{\Icr}{\mathcal{I}^{crit}}
\newcommand{\Spec}{\mathrm{Spec}}
\newcommand{\Ann}{\mathrm{Ann}}
\newcommand{\Brgrn}{\Brgr_n}



\newcommand{\Wr}{\overline{W}}


\newcommand{\Tsc}{T_{sc}}
\newcommand{\GL}{\mathrm{GL}}


\thanks{The work of A.O. was supported in part by the Sloan Foundation the NSF CAREER grant DMS-1352398}
\thanks{The work of L.R. was supported in part by  the NSF grant DMS-1108727}

\title{Knot Homology and sheaves on the Hilbert scheme of points on the plane}
\author[A.~Oblomkov]{Alexei Oblomkov}
\address{
A.~Oblomkov\\
Department of Mathematics and Statistics\\
University of Massachusetts at Amherst\\
Lederle Graduate Research Tower\\
710 N. Pleasant Street\\
Amherst, MA 01003 USA
}
\email{oblomkov@math.umass.edu}

\author[L.~Rozansky]{Lev Rozansky}
\address{
L.~Rozansky\\
Department of Mathematics\\
University of North Carolina at Chapel Hill\\
CB \# 3250, Phillips Hall\\
Chapel Hill, NC 27599 USA
}
\email{rozansky@math.unc.edu}

\begin{abstract} For each braid $\beta\in \Brgrn$ we construct a $2$-periodic complex $\bbS_\beta$ of  quasi-coherent $\CC^*\times \CC^*$-equivariant sheaves  on the non-commutative nested Hilbert
scheme $\Hilb_{1,n}^{free}$.
We show that
the triply graded vector space of the hypercohomology $ \mathbb{H}( \bbS_{\beta}\otimes \wedge^\bullet (\calB))$ with $\calB$ being tautological vector bundle, is an isotopy invariant of the  knot obtained by the
closure of $\beta$. We also show that the support of cohomology of the complex $\bbS_\beta$ is supported on the ordinary nested Hilbert
scheme $\Hilb_{1,n}\subset\Hilb_{1,n}^{free}$, that allows us to relate the triply graded knot homology to the sheaves on $\Hilb_{1,n}$ \end{abstract}


\maketitle
\tableofcontents

\section{Introduction}\label{sec:intro}

\subsection{Motivation}
In recent years there was a lot of interest in establishing a connection between  knot homology, Macdonald polynomials and Hilbert schemes of points on a plane
\cite{AS},\cite{Ch},\cite{GN},\cite{GORS},\cite{ORS}.
In particular, there are many conjectural
formulas expressing the knot homology of torus knots in terms of Macdonald polynomials or in terms of spaces of sections of special sheaves on the Hilbert schemes of points \cite{ORS},\cite{GORS},\cite{GN}. An agreement between these formulas is an interesting result at the intersection of algebra and geometry.
%
%

When this paper was posted and its results were announced, the  attempts to attack these conjectures directly, without a new computational method, were not succeeding. More recently, a version of a categorical localization was developed and applied to some of these conjectures in  \cite{Ho,EHo,Me}.


The current mathematical construction of triply graded homology based on Soergel bimodules is not related to Hilbert schemes of points. Hence the approach of  \cite{Ho,EHo,Me} is to compute the Soergel-based knot homology explicitly
and compare the answer with the homology of sheaves on the Hilbert scheme where the answer is known. One would still like to find a more direct link between the knot homology and sheaves on the Hilbert scheme of points on the plane.
Around the time of posting of this paper the authors of
\cite{GNR} proposed a conjectural connection between the Soergel bimodules and the category of coherent sheaves on the Hilbert scheme of points.

Our approach is different from that of \cite{GNR}. Roughly speaking, to a braid $\beta\in\Brgrn$ we associate a sheaf $\calS_{\beta}$ on the Hilbert scheme $\SymbolPrint{\Hilb_n}$ of $n$ points on $\CC^2$. The space of global sections of $\calS_\beta$ is the homology of the link $L(\beta)$ which is the closure of $\beta$.

The construction of $\calS_\beta$ is quite involved and relies on a homomorphism from the braid group $\Brgrn$ to the convolution algebra of a category of matrix factorizations on an auxiliary space. We expect this homomorphism to be injective, but we neither prove nor use this fact in the paper.
%

Compared to Soergel bimodule construction, our approach is more complicated conceptually, but it is more geometric and, surprisingly, in some important cases it is less computationally challenging.
For example, in a subsequent paper \cite{OR} we  show that adding a full twist $\SymbolPrint{Tw}$ to the braid results in the twisting of $\calS_{\beta}$ by a special line bundle:
$\calS_{\beta\cdot Tw}=\calS_{\beta}\otimes L$.  Existence of a knot theory with such property was predicted in \cite{ORS}.


\subsection{Main result}
 We consider a `non-commutative analog' of the nested Hilbert scheme defined as a quotient
  $\SymbolPrint{\Hilb_{1,n}^{free}}=\SymbolPrint{\widetilde{\Hilb}_{1,n}^{free}}/\SymbolPrint{B}$, where
 \[
 \widetilde{\Hilb}_{1,n}^{free} = \{
 (X,Y,v)\in \frb\times\frn\times V\,|\,\CC\langle X,Y\rangle v=\SymbolPrint{V}
 \}
 \]
 is an open subset of the product $\frn\times\frb\times V$ ($\frn,\frb\subset\mathfrak{gl}_n(\SymbolPrint{V})$ being the nilpotent radial and Borel subalgebra), the Borel subgroup $B\subset\mathrm{GL}_n$ acts on the first two factors by conjugation and  $V$ is its fundamental representation.

Denote a flag $ V_\bullet=(V_1\subset\cdots \subset V_{n-1}\subset V_n=V)\in \mathrm{Fl}$ and denote a point in the cotangent bundle $T^*\mathrm{Fl}$ as a pair \( (Y,V_\bullet )\), where $Y$ is a nilpotent matrix preserving the flag: $Y V_i\subset V_{i-1}$.

In order to  avoid a quotient by a non-reductive group \(B\) we can define \(\Hilb^{free}_{1,n}\) equivalently as a GIT \(\mathrm{GL}_n\)-quotient of
%
%
 the space  of quadruples \((X,Y,V_\bullet,v)\in \frg\times T^*\mathrm{Fl}\times V\) such that \(\CC\langle X,Y\rangle v=V\) and \(X V_\bullet\subset V_\bullet\).
 The fact that these definitions match follows from covering
  \(\widetilde{\Hilb}_{1,n}\) by affine charts with free \(B\)-action, for details see proposition~\ref{prop: aff cover}. The affine cover also provides a "hands on" description of the algebraic structure of the quotient manifold.

Let us denote by  $\SymbolPrint{\calB^\vee}$  the trivial vector bundle over $\widetilde{\Hilb}_{1,n}^{free}$ with the fiber $V$, we use the same notation for
 its   descent to the $B$-quotient. Respectively, $\SymbolPrint{\calB}$ is the vector bundle dual to $\calB^\vee$.


The two-dimensional torus
$\SymbolPrint{\Tsc}=\CC^*\times\CC^*$ acts on  $\Hilb_{1,n}^{free}$ by scaling $X,Y$: an element $(\lambda,\mu)\in T_{sc}$ turns a pair of matrices $(X,Y)$ into $(\lambda^{-2}X,\lambda^2 \mu^2Y)$.
We refer to the weights of this action as $q$-degree and $t$-degree.
The torus \(T_{sc}\) acts on the space \(V\) by \((\lambda,\mu)\cdot v=\lambda v\), that is the vectors of \(V\) have degree \(t\).
We use notation \(\SymbolPrint{\shq^k\sht^l}\) for the shift of the \(T_{sc}\)-action: \[(\lambda,\mu)(\shq^k\sht^l\cdot x)=\lambda^k\mu^l (\lambda,\mu)\cdot x.\]

We denote by $\SymbolPrint{D^{per}_{T_{sc}}(\Hilb_{1,n}^{free})}$ the derived category of two-periodic
 complexes on $\Hilb_{1,n}^{free}$ which are $T_{sc}$-equivariant. Let $\Brgrn$ denote a braid group with $n$ strands and let $\SymbolPrint{\xwr(\beta)}$ denote the writhe of a braid $\beta$
defined as the difference between the positive and negative crossings.
Our main result is the construction of the object
$$ \SymbolPrint{\bbS_\beta}\in D^{per}_{T_{sc}}(\Hilb_{1,n}^{free}),$$
such that the hypercohomology of the complex
$$ \SymbolPrint{\mathbb{H}^k(\beta)}:=\mathbb{H}(\bbS_\beta\otimes \wedge^k\calB)$$
defines an isotopy invariant of $L(\beta)$. To simplify notations we assume  \(\mathbb{H}^k(\dots)=0\) for \(k\notin \{0,1,\dots,n\}\).

\begin{theorem}\label{thm: intro main} For any $\beta\in \Brgr_n$ the triply graded space:
\[\SymbolPrint{\mathcal{H}(\beta)}:=\oplus_{k\in\ZZ}\SymbolPrint{\mathcal{H}^k(\beta)},\]
with \(\mathcal{H}^k(\beta)\)  doubly graded  defined by
$$\calH^k(\beta):=\shq^{-\xwr(\beta)+n}\cdot\mathbb{H}^{(k-\xwr(\beta)+n-1)/2}(\beta),$$
is an isotopy invariant of the braid closure $L(\beta)$.
\end{theorem}

Let us also introduce a notation for the graded pieces of $\calH^k$: $\SymbolPrint{\calH^{i,j,k}}\subset \calH^k$ is the subspace of vectors that are of weight
$q^it^j$.
We show that our invariant categorifies HOMFLY-PT polynomial.

\def\Hvv#1{ \calH^{#1}}
\def\Hijk{ \Hvv{i,j,k}}

\begin{theorem}\label{thm:HOMFLY}
The bi-graded Euler characteristic of $\calH^k(\beta)$ is equal to the HOMFLY-PT polynomial of the link $L(\beta)$:
\def\Phmf{P}
\[
\Phmf\bigl(L(\beta)\bigr) =\sum_{i,j,k\in\ZZ} (-1)^j q^i a^k\dim\Hijk(\beta).
\]
\end{theorem}


The first categorification of HOMFLY-PT polynomial was discovered by Khovanov and Rozansky \cite{KhR} and it is natural to expect that the homology
theory from this paper coincides with one from \cite{KhR}. We do not have a proof for a match of these theories.

\subsection{Sheaves on Hilbert schemes}
The usual nested Hilbert scheme $\SymbolPrint{\Hilb^L_{1,n}}$
is a subvariety of $\Hilb_{1,n}^{free}$ defined by the commutativity constraint on $X,Y$:
\[ [X,Y]=0.\]
It turns out that the support of the homology of the complex $\bbS_\beta$ is contained in $\Hilb_{1,n}^L$. Hence the sheaf homology of the complex is the  sheaf
\[\SymbolPrint{\calS_\beta}=\SymbolPrint{\calS_\beta^{odd}}\oplus\SymbolPrint{\calS_\beta^{even}}:=\mathcal{H}^*(\Hilb_{1,n}^{free},\mathbb{S}_\beta)\] on $\Hilb_{1,n}$ and we immediately have the following:
\begin{theorem}\label{thm: spec seq}
There is a spectral sequence with $E_2$ term being
$$ (\textup{H}^*(\Hilb^L_{1,n},\calS_\beta\otimes\wedge^k\calB),d)$$
$$
d: \textup{H}^k(\Hilb^L_{1,n},\calS^{odd/even}_\beta\otimes\wedge^k\calB)\to
\textup{H}^{k-1}(\Hilb^L_{1,n},\calS^{even/odd}_\beta\otimes\wedge^k\calB),$$
that converges to $\mathbb{H}^k(\beta)$.
\end{theorem}

We conjecture that one can extract the link invariant for $L(\beta)$ directly from the sheaf $\calS_\beta$ if $\beta$ is sufficiently positive

\begin{conjecture} For any \(\beta\in\Brgr_n\) and sufficiently positive \(k\) we have  \[\textup{H}^{>0}(\Hilb^L_{1,n},\calS_{\beta\cdot Tw^k})=0.\]
\end{conjecture}

There is a natural projection from the nested Hilbert scheme to the usual Hilbert scheme $p: \Hilb_{1,n}\to\Hilb_n$ and  it turns out that the push-forward $\SymbolPrint{p_*}(\mathcal{S}_\beta)\in
D_{T_{sc}}^{per}(\Hilb_n)$  has slightly better properties than
$\calS_\beta$:

\begin{theorem} The element of \(D_{T_{sc}}^{per}(\Hilb_n)\)
$$\SymbolPrint{\calS_{[\beta]}}:=p_*(\calS_\beta)$$
depends only on the conjugacy class of $\beta$.
\end{theorem}

In the sequel to this paper we  compute the sheaf $\bbS_\beta$ for many elements of the braid group. That allows us to give a  geometric description of the homology of a torus knot $T_{n,nk+1}$. More precisely,
we study the positive braid $\beta_k\in \Brgrn$ such that  $L(\beta_k)=T_{n,nk+1}$. We compute explicitly the sheaf corresponding to this braid. It is supported on the punctual Hilbert scheme:
$\SymbolPrint{\Hilb_{1,n}^{0}}=\widetilde{\Hilb}^{0}_{1,n}/B$,
$\widetilde{\Hilb}_{1,n}^{0}:=\widetilde{\Hilb}_{1,n}\cap \frn^2\times V$ and we show

\begin{theorem}\cite{OR}\label{thm:Tnk} For any \(k\in \ZZ\) we have
  $$\calS_{\beta_k}=[\calO_{\Hilb_{1,n}^{0}}]^{vir}\otimes \det(\mathcal{B})^k,$$
  where \([\calO_{\Hilb_{1,n}^0}]^{vir}\) is the Koszul complex defining the punctual Hilbert scheme
  \(\widetilde{\Hilb}^0_{1,n}\subset \widetilde{\Hilb}_{1,n}^{free}\).
\end{theorem}

The sheaf $\calO_{\Hilb_n^0(\CC^2)}\otimes \det(\mathcal{B})^k$
attracted a lot of attention in connection to combinatorics \cite{Hai}, representation theory \cite{GS} and knot theory
\cite{ORS,GORS}. In particular, in \cite{ORS} it was conjectured that the global sections of this sheaf is a particular double-graded subspace  of the
space of Khovanov-Rozansky homology of the torus knot $T_{n,1+nk}$. We expect that the push forward of the complex \(p_*([\calO_{\Hilb_{1,n}^0}]^{vir})\) is
\(\calO_{\Hilb_n^0(\CC^2)}\) and thus
our theorem is consistent with this conjecture and its generalization in \cite{GORS}.
In a forthcoming paper we explore this regular behavior to gain more insight into the behavior of our knot invariant
on large classes of knots.

\subsection{Braids and matrix factorizations}

Our construction  of the complex of sheaves $\bbS_\beta\in D^{per}_{T_{sc}}(\Hilb_{1,n}^{free})$ is based on a homomorphism from the braid group
$\Brgr_n$ to the convolution algebra of \(T_{sc}\)-equivariant and (weakly)  $B^2$-equivariant matrix  factorizations \(\MF_{B^2}(\calXr_2(G_n),\Wr)\) on the space $\SymbolPrint{\calXr_2(G_n)}:=\frb\times G_n\times \frn$ with the potential
$$\SymbolPrint{\Wr}=\Tr(X\Ad_g(Y)),$$
where $X\in \frb$ in upper triangular, $Y\in \frn$ is strictly upper triangular and \(\SymbolPrint{G_n}=GL_n\)\footnote{The space \(\calXr_2(G_n)\) appears naturally in the
  context on the Kapustin-Saulina-Rozansky topological field theory, we describe this physical theory in the forthcoming paper.}. The potential $\Wr$ has degree $t^2$ with respect to $T_{sc}$ so we  also assume  that the differentials of the matrix factorizations from
$\MF_{B^2}(\calXr_2( G_n),\Wr)$ have degree $t$.

  In the section~\ref{sec: inc fun gens} we prove the following

\begin{theorem}\label{thm: braid realization} The homotopic category $\MF_{B^2}(\calXr_2(G_n),\Wr)$ has a natural monoidal structure and there is a homomorphism
$$  \Brgr_n\to  \MF_{B^2}(\calXr_2(G_n),\Wr),\quad \beta\mapsto \SymbolPrint{\bcalC_\beta}$$
defined by an explicit choice of images of braid generators.
\end{theorem}

The construction of the homomorphism from the theorem is similar in spirit to the construction of Bezrukavnikov and Riche \cite{BR},
though we do not rely on the results of \cite{BR}
and
our methods differ from their approach, mainly because we try to make our construction as explicit as it possible. Let us also remark that recently Arkhipov and Kanstrup 
\cite{AK}
provided
a construction of braid group   action on  categories matrix factorizations. We discovered our construction  independently and  our methods differ from theirs.

There is a similar homomorphism $\Brgr_n\to \MF_{B^2}(\calXr_2(G_n),-\Wr)$, $ \beta\mapsto \bcalC'_\beta$  and, in particular, $\bcalC'_\parallel$ is the complex of the identity braid.
Informally, the complex $\bbS_\beta$ should be thought as sheaf homology of the $2$-periodic complex $\bcalC_\beta\otimes \bcalC'_\parallel\in \MF_{B^2}(\calXr_2(G_n),0)$ followed by the
restriction to the stable part of $\calXr_2(G_n)$ and extracting $B$ invariant part of the product ($B$ is embedded diagonally inside $B^2$).
To turn this informal definition into a rigorous mathematical construction, in section~\ref{sec: link inv} we introduce an auxiliary space $\calX_\ell(G_n)$ and
work with matrix factorizations on this space.

\subsection{Structure of the paper}
In the section~\ref{sec:notations} we collect most of notations used in the paper, we also state and prove some key
properties of the matrix factorizations. In section~\ref{sec: eq push} we  state
and prove the results about matrix factorizations that are needed for the main body of the paper.
In section~\ref{sec:convolution} we introduce our main tool: the convolution algebra on the space of
matrix factorizations with the potential. In this section we show that the convolution is associative and discuss generalizations of the construction.
The section \ref{sec: Knor} is devoted to the study of Knorrer functor which intertwines the convolution algebras on the smaller space $\calXr_2(G_n)$ with
the convolution algebra on the bigger space $\calX_2(G_n)$.

The sections \ref{sec: inc} and  \ref{sec: inc fun gens} contain the technical results about the convolution algebra that are used in the  rest of the paper, to be more precise we discuss various induction functors between the categories of matrix factorizations associated to the braid groups of different rank and
their relation with the convolution product.
 At the
end of section \ref{sec: inc} we introduce the matrix factorizations that generate the braid group. In section~\ref{sec: aux sec} we prove several auxiliary results that
we use later, in particular we show how the convolution with the matrix factorization for the elementary braid could be computed with the use of rank $1$ Chevalley-Eilenberg complex.

Our proof of the braid relations for the matrix factorizations introduced in the section~\ref{sec: inc} is spread between sections~\ref{sec: two str},\ref{sec: three str 2} and~\ref{sec: three str 3}.
In the section~\ref{sec: link inv} we construct the sheaf $\calS_{[\beta]}$ and show that it only depends on the conjugacy class of $\beta$.
Finally, in the section~\ref{sec: link inv 2} we prove our main theorem about the link invariant. Finally in the very last section we prove the theorem~\ref{thm:HOMFLY} and define the differential
$d_{m|n}$ that allows us to state a conjecture relating our cohomology to the $gl(m|n)$ type quantum invariants.

Finally, at the very end of the paper we compose the list of notations used in the paper. We tabulate the first appearance of a symbol and the place of more in depth
discussion of the corresponding object.

{\bf Acknowledgments:} Authors would like to thank Dima Arinkin, Roman Bezrukavnikov, Tudor Dimofte, Ben Elias,  Tina Kanstrup, Davesh Maulik, Peter Samuleson, Jenia Tevelev for useful discussions. L.R. is especially thankful to Dima Arinkin for illuminating explanations and to Eugene Gorsky for stimulating discussions and sharing the results of his ongoing research.  A.O. is especially thankful to Eugene Gorsky, Matt Hogancamp, Andrei Negu\c{t} and
Jake Rasmussen for extensive discussions about  the subject and for sharing their insight and intuition.
Authors are very thankful to an anonymous referee for his/her Herculean effort to improve the paper.
A.O. would like to thank Simons Center for Geometry and Physics for support during month-long stay in May 2015, the large portion of this work
was written during this stay and the conference at the end of the program presented us with opportunity to report on this work.
Work of A.O. was partially supported by Sloan Foundation and NSF CAREER grant DMS-1352398. The work of L.R. is supported by the NSF grant DMS-1108727.

\section{Basic space and Matrix factorizations}\label{sec:notations}
\subsection{Notations, conventions}
\subsubsection{}Let us fix the following notations for the groups that appear in our work:
$$ \SymbolPrint{G_n}=\mathrm{GL}_n,\quad \SymbolPrint{T_n}\subset B_n\subset G_n.$$

 We use gothic script for the Lie algebras of the respective groups. Whenever the rank of the group is obvious from the context of the paper
we omit the subscripts.
Let \(U_n=[B_n,B_n]\) and \(\frn_n=\textup{Lie}(U_n)\). We also denote the Lie algebra of
$T_n$ as $\SymbolPrint{\mathfrak{h}_n}$.

Let us also fix the notations for the projections on the upper and the lower triangular parts of the matrix: for a given square matrix $X$ we denote by $\SymbolPrint{X_+}$  and
$\SymbolPrint{X_{++}}$ the
upper-triangular and the strictly upper-triangular part of $X$:
\[(X_{++})_{ij}=\begin{cases} X_{ij},& i<j\\
                            0,&i \ge j
                          \end{cases},\quad (X_{+})_{ij}=\begin{cases} X_{ij},&i\le j\\
                            0,&i>j
                          \end{cases}
                        \]
                        respectively $X_-:=X-X_{++}$ and $X_{--}:=X-X_+$ are the lower-triangular and strictly lower-triangular parts of $X$.

The main geometric object in our study is the ($G_n\times B_n^\ell$)-space $\calX_\ell=\frg_n\times (G_n\times \frn_n)^\ell$ with the following action of the group:
$$ (b_1,\dots,b_\ell)\cdot (X,g_1,Y_1,\dots,g_\ell, Y_\ell)=(X,g_1\cdot b_1^{-1},\Ad_{b_1}(Y_1),\dots,g_\ell\cdot b_\ell^{-1},\Ad_{b_\ell}(Y_\ell)),$$
$$ h\cdot (X,h\cdot g_1,Y_1,\dots,g_\ell, Y_\ell)=(\Ad_h(X),h\cdot g_1,Y_1,\dots,h\cdot g_\ell,Y_\ell).$$

The space \(\calX_2\) is the enlarged version of the space \(\calXr_2(G_n)\), that was discussed in the introduction. Sometimes we use notation \(\calX_2(G_n)\) for
\(\calX_2\).

\def\yHom{ \Hom }

\subsection{Matrix factorizations: three lemmas} In this subsection we fix notations for matrix factorizations and remind some basic facts about them.
In particular we prove three lemmas that we will use in the main body of the paper. These lemmas essentially contain all facts about matrix factorizations
that are used here. For more thorough review of the theory of matrix factorization the reader could consult the beautiful original paper \cite{E} or later surveys in \cite{D}.

Given an affine variety $\mathcal{Z}$ and a function $F$ on it we define \cite{E} the homotopy category $\SymbolPrint{\MFs}(\mathcal{Z},F)$ of matrix factorizations whose objects are complexes of $\CC[\calZ]$-modules
$M=M_{odd}\oplus M_{even}$ equipped with the differential \[D\in \yHom_{\CC[\calZ]}(M_{odd},M_{even})\oplus  \yHom_{\CC[\calZ]}(M_{even},M_{odd})\] such that $D^2=F$.
Thus $\MFs(\mathcal{Z},F)$ is a triangulated category as explained in subsection 3.1 of \cite{Or}.

If $\calZ$ is a variety with $G$-action  and $F$ is a $G$-invariant function then
we define the category of strongly  equivariant matrix factorizations $\SymbolPrint{\MFs^{str}_G}(\mathcal{Z},F)$ as the category of matrix factorizations with extra condition  that $D$ is a $G$-equivariant matrix.
Unfortunately, this definition is too restrictive and we use a slightly more sophisticated definition for the category of  equivariant matrix factorization. We postpone the details of our definition till the next section
where we define matrix factorizations that are equivariant up to homotopy.
However, the naive equivariant matrix factorization as above are objects in our category.

Let us discuss some specific types of matrix factorizations that will be used in our paper. Suppose $C_\bullet=[C_0\to C_1\to\dots \to C_n]$ is the complex of $\CC[\calZ]$-modules.
Then we construct a matrix factorization $\SymbolPrint{[C_\bullet]_{per}}\in \MFs(\calZ,0)$ by setting $M=C_{odd}\oplus C_{even}$,
$C_{even}:=\oplus_{i\in 2\ZZ} C_i, C_{odd}= \oplus_{i\in 1+2\ZZ} C_i$ and the differential is induced from the complex $C_\bullet$.

Another type of matrix factorization appearing in our work is obtained by "doubling" the differential $d^+$ of an ordinary complex $(C_\bullet,d^+)$. Indeed, let us denote by $\MFs^\ell(\calZ,F)$ the set of the following
complexes:
$$
C_\bullet=\oplus_{i=0}^\ell C_i,\quad
d^+_i: C_i\to C_{i-1},\quad
d^-_i:C_i\to C_{>i},\quad (\sum_i d_i^-+d_i^+)^2=F\cdot \mathrm{Id}.$$

\def\xHom{ \Hom }
\def\xExt{ \Hom }
\def\xExtdp{ \xExt_{d_+}}
\def\xExtdpv#1{ \xExtdp^{#1}}
\def\xExtdpk{ \xExtdpv{k}}
The complex $C_\bullet$ with only  the positive differentials $d^+_i$ is called the positive part of the complex.
There is a folding map $\SymbolPrint{\MFs^{\ell}}(\calZ,F)\to \MFs(\calZ,F)$ that associates to $(C_\bullet,d^\pm_\bullet)$ a two-periodic complex $[C_\bullet,d^\pm_\bullet]_{per}$.
The matrix factorizations of the above type occur very naturally in a geometric setting as explained in the lemma below. The notations for the morphisms are as follows:
$\xHom^{k}(C_\bullet,C_\bullet)=\oplus_{m} \yHom(C_m,C_{m-k})$ is notations for the maps of $\CC[\calZ]$-modules and
$\xExtdpk(C_\bullet,C_\bullet)$ is the subspace of maps from $\yHom^k(C_\bullet,C_\bullet)$ that commute with the differential $d^+$.

\begin{lemma}\label{lem: ext MF} Let $(C_\bullet,d_\bullet^+)$ be the complex
$$C_0\xleftarrow{d_1^+}C_1\xleftarrow{d_2^+}\dots\xleftarrow{d_\ell^+} C_\ell$$
of $\CC[\calZ]$ modules and $F\in \CC[\calZ]$ such that
\begin{enumerate}
\item Elements of $\xExtdpv{<0}(C_\bullet,C_\bullet)$ are homotopic to $0$
\item The element $F\in \xExtdpv{0}(C_\bullet,C_\bullet)$ is homotopic to $0$.
\end{enumerate}
Then the complex $(C_\bullet,d_\bullet^+)$ can be extended to $(C_\bullet,d^\pm_\bullet)\in \MFs^{\ell}(\calZ,F)$.
\end{lemma}
\begin{proof}
By assumption of the lemma, the  endomorphism $F\cdot \mathrm{Id}$  of $(C_\bullet,d_\bullet^+)$ is homotopic to  $0$. This means that there exist morphisms of $\CC[\calZ]$ modules
$h^{(-1)}_i: C_i\to C_{i+1}$ such that $F=h^{(-1)}_i\circ d_{i+1}^++d_i^+\circ h^{(-1)}_{i-1}$. Then a new differential $D^{(-1)}:=d^++h^{(-1)}$ satisfies the relation
$$ \left(D^{(-1)}\right)^2-F =(h^{(-1)})^2\in \yHom^{\leq-2}(C_\bullet,C_\bullet).$$
Thus $D^{(-1)}$ provides a base
of induction for the proof of existence $h^{(-2i-1)}\in \yHom^{-2i-1}(C_\bullet,C_\bullet)$, $i\ge 0$ such that
\begin{enumerate}
\item $[h^{(-1)},d^+]_+=F$
\item $[h^{(-2i-1)},d^+]_+=\sum_{j=0}^{i-1} h^{(-2j-1)} h^{(-2i+2j+1)}$, $i>0$
\end{enumerate}
were we used the notation $[a,b]_+=ab+ba$.

Indeed, suppose we constructed $h^{(-2i-1)}$ for all $i<j$. Define $D^{(-2j+1)}=d^++\sum_{k=0}^{j-1} h^{(-1-2k)}$, then the inductive
assumption implies:
\begin{equation}\label{eq: D2}
\left( D^{(-2j+1)}\right)^2-F-R^{(-2j)}\in \yHom^{<-2j}(C_\bullet,C_\bullet),\quad R^{(-2j)}=\sum_{k=0}^{j-1} h^{(-2k-1)} h^{(-2j+1+2k)}
\end{equation}
On the other hand, a direct computation using the inductive assumptions and the fact that $F$ commutes with all $h^{(2i-1)}$ and $d^+$ implies
\begin{equation}\label{eq: com d+}
 [d^+,R^{(-2j)}]=0.
\end{equation}
In more details, let us introduce notation $R^{(<-2j)}:=(D^{(-2j-1)})^2-F-R^{(-2j)}$ then
\begin{multline*}
0=[D^{(-2j-1)}, (D^{(2j-1)})^2]=[d^+ +\sum_i h^{(-2i-1)}, F+R^{(-2j)}+R^{(<-2j)}]\\=[d^+,R^{(-2j)}]+[d_++\sum_i h^{(-2i-1)},R^{(<-2j)}]+[\sum_i h^{(-2j-1)},R^{(-2j)}+R^{(-2j-1)}].
\end{multline*}
The first term in the last sum is from $\xHom^{-2j+1}(C_\bullet,C_\bullet)$, on the other hand the last two   are from
$\xHom^{<-2j+1}(C_\bullet,C_\bullet)$ thus (\ref{eq: com d+}) follows.

Thus $R^{(-2j)}$ is the morphism of the complex $(C_\bullet,d^+)$ of degree $-2j$, that is, an element of $\xExtdpv{-2j}(C_\bullet,C_\bullet)$. The assumption of our lemma implies that such
morphism is homotopic to zero and the homotopy morphism $h^{(-2j-1)}\in \xHom^{-2j-1}(C_\bullet,C_\bullet)$ provides the induction step
since $[h^{(-2j-1)},d^+]_+=R^{(-2j)}$.

Thus by setting $d_i^-=\sum_i h^{(-2i-1)}$ we obtain the desired extension because of (\ref{eq: D2}).
\end{proof}

\def\xHomdpm{\xHom_{d^{\pm}}}

The extension described in Lemma~\ref{lem: ext MF} is unique up to a natural equivalence relation. First, let us fix notations: the space of morphisms
 $\xHomdpm(A,B)$, $A,B\in \MFs^{\ell}(\calZ,F)$ consists of $\CC[\calZ]$-module homomorphism that (super-)commute with the differentials $d_\bullet^\pm$.
 This space is $\ZZ$-graded with $\xHomdpm^i(A,B)$ consisting of morphisms $h_k: A_k\to A_{k+i}$.

\begin{lemma}\label{lem: ext uniq} Let $A=(C_\bullet,d^\pm_\bullet), B=(C_\bullet,\tilde{d}^\pm_\bullet)\in \MFs^{\ell}(\calZ,F)$ such that
\begin{enumerate}
\item $(\tilde{C},\tilde{d}^+_\bullet)=(C_\bullet,d^+_\bullet)$
\item Elements of $\xExtdpv{<0}(C_\bullet,C_\bullet)$ are homotopic to $0$
\item The element $F\in \xExtdpv{0}(C_\bullet,C_\bullet)$ is homotopic to $0$.
\end{enumerate}
Then there is $\Psi=\sum_{i\le 0}\Psi_i$, $\Psi_i\in \Hom^{i}(A,B)$ such that $\Psi_0=\mathrm{Id}$ and
$$ \Psi\circ (d^++d^-)\circ \Psi^{-1}=(\tilde{d}^++\tilde{d}^-).$$
\end{lemma}
\begin{proof}

For brevity let us denote by $C^+$ the complex $(C_\bullet,d^+_\bullet)$, respectively we use notation
$\xExtdp(C^+,C^+)$ for the morphisms that respect the differential $d^+$ and $\Hom(C,C)$ for the morphisms
of complexes simply as modules. In particular in these notations we have
$$ D=d^++d^-,\quad  \tilde{D}=d^++\tilde{d}^-\in \Hom^{1}(C,C)\oplus \Hom^{\le -1}(C,C).$$

We claim that there is sequence of morphisms $\Psi^{(k)}\in \Hom^{\le 0}(C,C)$, $k=0,1,2,\dots$ such that
\begin{enumerate}
\item $\Psi^{(k)}_0=\mathrm{Id}$,
\item $D^{(k)}-\tilde{D}\in \Hom^{-k-1}(C,C)$,
\end{enumerate}
where $D^{(k)}=\Psi^{(k)}\circ D\circ (\Psi^{(k)})^{-1}$.

Indeed, for $k=0$ we can set $\Psi^{(0)}=\mathrm{Id}$. For the step of induction we assume the existence of $\Psi^{(k-1)}$, then
the leading term $\Delta=(D^{(k-1)}-\tilde{D})_{-k}\in \xHom^{-k}(C,C)$ is actually an element of $\xExtdp^{-k}(C^+,C^+)$ as could be seen from
degree $-k+1$ part of the equation
$$ (D^{(k-1)})^2-D^2=0.$$

On the other hand the second assumption of the Lemma implies that all elements of $\xExtdp^{-k}(C^+,C^+)$ are homotopic to $0$ and
hence there is $h\in Hom^{-k-1}(C,C)$ such that $[h,d^+]=\Delta$. Hence the element $\Psi^{(k)}:=(1+h)\circ \Psi^{(k-1)}$ satisfies the condition of
our induction statement.

Finally, since the complex $C$ is bounded see that the morphism $\Psi^{(\ell)}$ has the desired properties and our proof is complete.
\end{proof}

The last lemma is very elementary but it is responsible for appearance of the Hilbert scheme of points on the plane in our construction of the knot homology.

\begin{lemma}\label{lem: crit} Let \(Z\) be smooth and affine. The elements of the ideal $\calI_{crit}$ generated by the partial derivatives of the potential
$F$ act by zero homotopies on the elements of $\MFs(\calZ,F)$
\end{lemma}
\begin{proof}
Let $z$ be a local coordinate on $\calZ$ and $D$ is the differential from some matrix factorization with
the potential $F$. Then we have
$$ \frac{\partial F}{\partial z}=\frac{\partial D}{\partial z}\circ D+ D\circ \frac{\partial D}{\partial z}.$$
Hence $\frac{\partial D}{\partial z}$ is the map that provides a homotopy of $\frac{\partial F}{\partial z}$ to $0$.
\end{proof}

\subsection{Koszul matrix factorizations}\label{ssec: str eq Kosz} In this subsection we set notations for a particular type of matrix factorization that is ubiquitous in our work and in the previous work that relates
matrix factorizations to knot invariants \cite{KR}. We need slightly more general setting than the one from \cite{KR} because we work with equivariant matrix factorizations.
We state some simple properties of these factorizations that are used later in the paper.

Suppose $\calZ$ is a variety with the action of a group $G$ and $F$ is a $G$-invariant potential.
An object of the category $\MFs_{B^2}^{str}(\calZ,F)$ is a free $B^2$-equivariant $\Zt$-graded $\CC[\calZ]$-module $\xxM$ with the odd $G$-invariant differential $\xD$ such that $\xD^2 = F\xIdv{M}$. In particular, a free $G$-equivariant $\CC[\calZ]$-module $\xxV$ with two elements $\xdl\in\xxV$, $\xdr\in\xxV^*$ such that $(d_l,d_r)=F$, determines a Koszul matrix factorization $\SymbolPrint{\Kszvvv{\xxV}{\xdl}{\xdr}} = \bigwedge^{\bullet}\xxV$ with the differential $Dv = \xdl\wedge v + \xdr\cdot v$ for $v\in \bigwedge^{\bullet}\xxV$. We use a more detailed notation by choosing a basis $\ktht_1,\ldots,\ktht_n\in \xxV$ and presenting $\xdl$ and $\xdr$ in terms of components: $\xdl=a_1\ktht_1 +\cdots a_n\ktht_n$, $\xdr=b_1\ktht^*_1+\cdots + b_n\ktht^*_n$:
\begin{equation}
\label{eq:kszm}
\Kszvvv{\xxV}{\xdl}{\xdr} = \kmtr{a_1 & b_1 &\ktht_1\\ \vdots & \vdots & \vdots \\
a_n & b_n & \ktht_n}
\end{equation}
The structure of $G$-module is described by specifying the action of $G$ on the basis $\ktht_1,\ldots,\ktht_n$.
In some cases when $G$-equivariant structure of the module $\xxM$ is clear from the context we omit the last columns from the notations. We call a matrix presenting Koszul matrix factorization Koszul matrix.
For example, %
%
%
if  we change the basis $\theta_1,\dots, \theta_n$ to the basis $\theta_1,\dots,\theta_i+c\theta_j,\dots, \theta_j,\dots,\theta_n$ the $i$-th and
$j$-th rows of the Koszul matrix will change:
$$ \begin{bmatrix}
a_i& b_i&\theta_i\\
a_j& b_j&\theta_j
\end{bmatrix}
\mapsto
\begin{bmatrix}
a_i+ca_j& b_i&\theta_i+c\theta_j\\
a_j& b_j-cb_i&\theta_j
\end{bmatrix}
$$

Suppose $a_1,\dots,a_n\in \CC[\calZ]$ is a regular sequence and $F\in (a_1,\dots,a_n)$. We can choose $b_i$ such that $F=\sum_i a_i b_i$ and
$\xdl$ and $\xdr$ are as above: $\xdl=a_1\ktht_1 +\cdots a_n\ktht_n$, $\xdr=b_1\ktht^*_1+\cdots + b_n\ktht^*_n$.
In general, there is no unique choice for $b_i$ but all choices lead to homotopy equivalent
Koszul matrix factorizations (in the non-equivariant case they would be simply isomorphic).
In other words, if $b'_i$ is a another collection of elements such that $F=\sum a_i b'_i$  and $\xdr'=b'_1\ktht^*_1+\cdots + b'_n\ktht^*_n$
then Lemma~\ref{lem: ext uniq} implies that the complexes $\Kszvvv{\xxV}{\xdl}{\xdr}$ and $\Kszvvv{\xxV}{\xdl}{\xdr'}$ are homotopy equivalent.

Indeed, we need to apply Lemma~\ref{lem: ext uniq} for \((C_\bullet,d^+_\bullet)=(\tilde{C}_\bullet,\tilde{d}_\bullet^+)=(\Lambda^\bullet V,\xdl)\).
The elements \(a_i\) form a regular sequence hence \(\Lambda^\bullet V,\xdl\) is homotopic to \(\mathcal{O}=\CC[Z]/I\), \(I=(a_1,\dots,a_n)\) and
we conclude that \(\Hom_{d_+}^{<0}(C_\bullet,C_\bullet)=\Hom_{d_+}^{<0}(\mathcal{O},C_\bullet)=0\). Moreover, the differentials \(\xdr\) and \(\xdr'\)
are homotopies between \(F\in Hom_{d_+}^0(C_\bullet,C_\bullet)\) and \(0\).
Thus there is a homotopy between $\Kszvvv{\xxV}{\xdl}{\xdr}$ and $\Kszvvv{\xxV}{\xdl}{\xdr'}$ and  from now on we use notation $\SymbolPrint{\mathrm{K}^F}(a_1,\dots,a_n)$ for such matrix factorization.

\input{part2_gl.tex}

\input{part3_gl.tex}

\input{part4_gl.tex}

\input{part5_gl.tex}

\input{part6_gl.tex}

\section{Link invariant}\label{sec: link inv 2}
The nested Hilbert scheme $\Hilb_{1,n}$ is a very singular scheme and it has very complicated geometry, the main source of complexity is the commutativity relation between the matrices.
On the other hands, we can define  the free Hilbert scheme $\Hilb^{free}_{1,n}$ by omitting commutativity condition, this turn out to be relatively simple variety which is an iterated fibration of
projective spaces (see proposition~\ref{prop: aff cover}). In this section we explain a construction for $\bbS_\beta \in D^{per}_{\Tsc}(\Hilb^{free}_{1,n})$ for $\beta\in \Brgr_n$ with the property that
the total cohomology of $\bbS_\beta$ is an isotopy  invariant of the closure $L(\beta)$ (see theorem~\ref{thm: main}).

\subsection{Geometry of $\Hilb^{free}_{1,n}$} Let us define by the following open subset of $\frb\times\frn\times V$:
$$\widetilde{\Hilb}^{free}_{1,n}:=\{ (X,Y,v)| \CC\langle X,Y\rangle v=V\}.$$
There is an action of $B\times \CC^*$ on $\widetilde{\Hilb}^{free}_{1,n}$ and we use notation $\Hilb^{free}_{1,n}$ for the quotient.

Define a map $p: \Hilb^{free}_{1,n}\to \Hilb^{free}_{1,n-1}$ as the map induced by the projection map on the vector space of upper-triangular matrices: the map
forgets the first row of the matrix. Let us also introduce the map $\chi: \Hilb^{free}_{1,n}\to \frh_n=\CC^{n}$, $\chi(X,Y,v)=(x_{11},\dots,x_{nn})$.  Given $z\in \frh_n$ we denote by $\Hilb^{free}_{1,n}(z)$ the preimage
$\chi^{-1}(z)$. By restricting the map $p$ to the space of diagonal matrices we get the map $p:\frh_n\to \frh_{n-1}$.

\begin{proposition}\label{prop: aff cover} The fibers of the map $p: \Hilb^{free}_{1,n}(z)\to \Hilb^{free}_{1,n-1}(p(z))$ are projective spaces $\mathbb{P}^{n-1}$.
\end{proposition}
\begin{proof}
Below we construct an affine cover of the space $\Hilb^{free}_{1,n}(z)$, these are the affine charts for the space $\Hilb^{free}_{1,n}(z)$. The charts a labeled by the following combinatorial data: $(\vec{S}_1,\vec{S}_2)$
where $\vec{S}_1=(S_1^1,\dots,S_1^{n-1})$, $\vec{S}_2=(S_2^1,\dots, S_2^{n-1})$ where $S_i^j$ are the subsets of integers with the following properties:
$$ S_1^i, S_2^i\subset \{1,\dots,i\},\quad S_i^j\subset S_i^{j+1},\quad |S_1^i|+|S_2^i|=i.$$

For a given $\SymbolPrint{(\vec{S}_1,\vec{S}_2)}$ we define the affine subspace $\SymbolPrint{\mathbb{A}_{\vec{S}_1,\vec{S}_2}}$ inside $\frb\times \frn\times V$ by the following equations:
\begin{gather}
x_{n-i, n+1-k}=0 \mbox{ if } k\in S_1^i,\quad y_{n-i,n+1-k}=0,\mbox{ if } k\in S_2^i, \label{eq: slice 1}\\
x_{n-i,n+1-k}=1 \mbox{ if } k \in S_1^{i+1}\setminus S_1^i\quad y_{n-i,n+1-k}=1 \mbox{ if } k \in S_2^{i+1}\setminus S_2^i, \label{eq: slice 2}
\end{gather}
where \( i=1,\dots, n-1,\) and
\begin{equation}
 v_i=\delta_{in}\label{eq: slice 3}.
\end{equation}

Below we use induction by $n$ to show that
\begin{enumerate}
\item The isotropy group of an element of $\mathbb{A}_{\vec{S}_1,\vec{S}_2}$ inside $B$ is trivial
\item $\widetilde{\Hilb}^{free}_{1,n}(z)=\cup_{(\vec{S}_1,\vec{S}_2)} B\cdot \mathbb{A}_{\vec{S}_1,\vec{S}_2}.$
\end{enumerate}

Indeed, the seed of induction is clear. Let us assume that statement true for $\Hilb^{free}_{1,n-1}(p(z))$. Hence a $B$-orbit $\calO\subset\widetilde{\Hilb}_{1,n}(z)$ intersect  the variety
$p^{-1}(a)$ for some $a\in \mathbb{A}_{\vec{S}_1,\vec{S}_2}$ for some $\vec{S}_1, \vec{S}_2$.   By the induction, the subgroup of $B$ preserving the intersection
$\calO\cap p^{-1}(a)$ has Lie group spanned by the matrix units $E_{1i}$, $i=2,\dots,n$.

 Let us pick an element $u\in \calO\cap p^{-1}(a)$.  We can use group generated by the
$e^{tE_{1i}}$, $i=2,\dots,n$ to translate an element $u$ to an element $w=(X,Y,v)$ that satisfies equation (\ref{eq: slice 1}) and (\ref{eq: slice 3}). The element $w$ is unique up to the
action of the group $\CC^*=exp(E_{11}t)$. Finally, since $w=(X,Y,v)\in \widetilde{\Hilb}^{free}_{1,n}(z)$ we have $\CC\langle X,Y\rangle v=V$ and
hence either $(x_{1,n+1-k})_{k\notin S^n_1}$ or $(y_{1,n+1-k})_{k\notin S^n_2}$ is non zero. If the first vector is non-zero, we set $S_1^{n+1}\setminus S_1^n$ to the index of the non-zero entry.
The case $(y_{1,n+1-k})_{k\notin S^n_2}\ne 0$ is analogous. As a last step we use the group $exp(E_{11}t)$ to scale this non-zero entry to $1$ and thus
obtain a unique element of the orbit $\calO$ that satisfies equations (\ref{eq: slice 1},\ref{eq: slice 2},\ref{eq: slice 3}).
\end{proof}

Let us denote by $\tilde{\calB}_n$ the  vector bundle over $\widetilde{\Hilb}^{free}_{1,n}$  with the fiber over the point $(X,g,Y,v)$ equal to the space $V$.
This vector bundle descends to the vector bundle $\calB_n$ on the quotient $\Hilb^{free}_{1,n}$.

Let us denote by $\chi_i\in \Hom(T,\CC^*)$, $T\subset B$  such that $\chi_i(\exp(\sum_{i=1}^n t_i E_{ii}))=\exp(t_i)$ and we use the same notation for
the induced characters of $\mathbf{B}$ and for the restriction to the subgroup $B$. The trivial line bundle on $\widetilde{\Hilb}_{1,n}^{free}$ twisted by the character $\chi_i$
descends to line bundle on $\Hilb_{1,n}^{free}$ which we denote by $\calO_i(-1)$. The line bundle $\calO_n(-1)$ admits more geometric description
$$ \calO_n(-1)=\calB_n /p^*(\calB_{n-1}),$$
which we use in the proof of the following

\begin{proposition}\label{prop: push forward} Let $p$ be the map $\Hilb^{free}_{1,n}(z)\to\Hilb^{free}_{1,n-1}(p(z))$ then we have
\begin{enumerate}
\item $p_*(\wedge^k\calB_n)=\wedge^k\calB_{n-1}.$
\item $p_*(\calO_n(m)\otimes \wedge^k\calB_n)=0$ if $m\in [-1,-n+2].$
\item $p_*(\calO_n(-n+1)\otimes \wedge^k\calB_n)=\wedge^{k-1}\calB_{n-1}[n]$
\end{enumerate}
\end{proposition}
\begin{proof}
Let $U\subset \Hilb_{1,n-1}^{free}$ is an open subset then we can apply the global section functor $\Gamma(p^{-1}(U),-)$ to the short exact sequence of sheaves on :
\begin{multline*}
  0\to \calO_n(m)\otimes\wedge^k p^*(\calB_{n-1})\to \calO_n(m)\otimes\wedge^k\calB_n\\
  \to \calO_n(m)\otimes\calB_n/p^*(\calB_{n-1})\otimes \wedge^{k-1}(p^*(\calB_{n-1}))\to 0.
  \end{multline*}

The fibers of the map $p$ are projective spaces $\mathbb{P}^n$ and the line bundle $\calO_n(-1)$ restricted to a fiber is the standard tautological bundle $\calO(-1)$.
Since $\calB_n/p^*(\calB_{n-1})=\calO_n(-1)$ then by the previous observation and well known formulas for the homology of $\calO(k)$ on the projective space combined with the short exact sequence from above imply
 $$\textup{H}^i(p^{-1}(U),\calO_n(m)\otimes\wedge^k\calB_n)=0,\quad i\ne 0, n $$
 $$ \textup{H}^0(p^{-1}(U),\calO_n(m)\otimes \wedge^k\calB_n)=\begin{cases} &\textup{H}^0(U,\wedge^k\calB_{n-1}),\mbox{ if } m=0\\
                                                                                                                                             & 0,\mbox{ if } m<0.
                                                                                                                    \end{cases}
                                                                                                                    $$

 $$ \textup{H}^n(p^{-1}(U),\calO_n(m)\otimes \wedge^k\calB_n)=\begin{cases} &\textup{H}^0(U,\wedge^{k-1}\calB_{n-1}),\mbox{ if } m=-n\\
                                                                                                                                             & 0,\mbox{ if } m>-n.
                                                                                                                    \end{cases}
                                                                                                                    $$
 Hence the statement of the proposition follows.
\end{proof}

\subsection{Braid complex on $\Hilb^{free}_{1,n}$}
In this subsection we explain how one could construct for every $\beta\in \Brgr_n$ a periodic complex $\mathbb{S}_\beta\in D^{per}_{T_{sc}}(\Hilb^{free}_{1,n})$ such that
the hypercohomology of the complex is an isotopy invariant of the link $L(\beta)$.  First let us explain a construction of the sheaf.

First let us introduce the analogs of the previous constructions for the Kn\"orrer reduced spaces.
We define $\calXr_{2,fr}$ to be a subvariety of $\calXr_2\times V^0$ consisting of $(X,g,Y,u)$ such that
\[\CC\langle X,\Ad_{g}^{-1}(Y)\rangle u=V\]

The space $\calXr_{2,fr}$ is an open subvariety of $\calXr_{2,fr}\times V$ and we can restrict the  Kn\"orrer $\Phi_n\times \mathrm{id}_V$
on this open set to obtain:
$$\Phi_n: \MF_{B^2}(\calXr_{2,fr},\Wr)\to \MF_{B^2}(\calX_{2,fr},W).$$

There is forgetful map $\forg:\calXr_{2,fr}\rightarrow \calXr_2$ and the argument almost identical to the proof of
Lemma~\ref{lem: for homom} implies that $\forg^*$ is the homomorphism of the convolution algebras and
\begin{equation}\label{eq: Phi for}
  \Phi_n\circ \forg^*=\forg^*\circ\Phi_n.
 \end{equation}

It is natural to define $\SymbolPrint{\calXr_1}:=\frb\times\frn$ and
respectively we have the analog of the closure map $j_e:\calXr_1\rightarrow \calXr_2$ by $j_e(X,Y)=(X,e,Y)$.
Similarly, we define $\calXr_{1,fc}\subset \frb\times\frn\times V$ to be $\widetilde{\Hilb}^{free}_{1,n}$ and we have a natural
map $\SymbolPrint{j_e}:\calXr_{1,fc}\rightarrow \calX_{2,fr}$.

For any $\beta\in \Brgr_n$ we define the element of $D^{per}_{T_{sc}}(\Hilb_{1,n}^{free})$:
$$\SymbolPrint{\mathbb{S}_\beta}:=j_e^*(\bar{\calC}_\beta).$$



For a element $\bbS\in D^{per}_{\Tsc}(\Hilb^{free}_{1,n})$ we denote by $\mathbb{H}(\bbS)$ the doubly graded space of the hypercohomology of the complex and let us introduce:
$$\mathbb{H}^k(\beta):=\mathbb{H}(\bbS_\beta\otimes \wedge^k\calB)$$
 The main result of the paper is the following

\begin{theorem}\label{thm: main} The doubly graded space
$$\calH^k(\beta):=\mathbb{H}^{(k+\xwr(\beta)-n-1)/2}(\beta)$$
is the isotopy invariant of $L(\beta)$.
\end{theorem}

Let us denote by $\calH^*(\bbS_\beta\otimes \wedge^m\calB)=\calH^{odd}(\bbS_\beta\otimes\wedge^k \calB)\oplus\calH^{even}(\bbS_\beta\otimes\wedge^k\calB)$ the sheaf on $\Hilb^{free}_{1,n}$ that is the sheaf of homology of the complex $\bbS_\beta$. Since
the matrix factorization $\bbS'_\beta$ is supported on the critical locus of $\Wr$, the sheaf $\calH^*(\bbS_\beta)$ is supported on subvariety $\Hilb_{1,n}$. Moreover, the following statement implies
the theorem~\ref{thm: spec seq} from the introduction.

\begin{corollary} There is a spectral sequence with $E_2$ term equal
$$ (\textup{H}^*(\Hilb^L_{1,n},\calH^*(\bbS_\beta\otimes\wedge^k\calB)),d)$$
$$
d: \textup{H}^k(\Hilb^L_{1,n},\calH^{odd/even}(\bbS_\beta\otimes\wedge^k\calB))\to
\textup{H}^{k-1}(\Hilb^L_{1,n},\calH^{even/odd}(\bbS_\beta\otimes\wedge^k\calB)),$$
that converges to $\mathbb{H}^k(\beta)$.
\end{corollary}


\subsection{Proof of theorem~\ref{thm: main}}
By the construction the $\mathbb{H}^k(\beta)$ does not depends on the presentation of $\beta$ as a product of the elementary braids.
Let us explain why $\mathbb{H}^k(\beta)$ actually only depends only conjugacy class of $\beta$.
First let us notice that because of the equality (\ref{eq: Phi for}) we have a construction for $\mathbb{H}^k(\beta)$ as hyper-cohomology of the complex on
$\calX_{1,fc}$:
$$\mathbb{H}^k(\beta)=\mathbb{H}(\Lambda^k\mathcal{B}\otimes \calC_\beta),$$
here $\calB$ is the vector bundle over $\calX_{1,fc}$ with fiber over $(X,g,Y,v)$ equal $V$.

Thus given a presentation $\beta=\sigma_{i_1}^{\epsilon_1}\cdot\dots\cdot\sigma_{i_\ell}^{\epsilon_\ell}$ we have
$$\mathbb{H}^k(\beta)=\mathbb{H}(\Lambda^k\pi_1^*(\mathcal{B})\otimes \CE_{\frn^\ell}(\pi_{12}^*(\calC_{\sigma_{i_1}^{\epsilon_1}})\otimes
  \dots \pi_{\ell,1}^*(\calC_{\sigma_{i_\ell}^{\epsilon_\ell}}))).$$
  where $\pi_{ij}:\calX_{\ell,fc}\to \calX_{2,fr}$ are the natural projections. The  LHS of the last formula  only depends on
  the cyclic order of the elements $\sigma_{i_k}^{\epsilon_k}$, because the space \(\calX_{\ell,fc}\) has \(\ZZ_\ell\) cyclic symmetry and
  the group element \(g_{i,i+1}\) provides an isomorphism of the  vector bundles \(\pi_i^*(\calB)\) and \(\pi_{i+1}^*(\calB)\) for all \(i\).
  Hence we showed that $\mathbb{H}^k(\beta)$ only depends on the conjugacy
  class of $\beta$.

To complete our proof we need to show that
\begin{equation}\label{eq: Markov}
\mathbb{H}^k(\beta\cdot \sigma_1)=\mathbb{H}^k(\beta'),\quad \mathbb{H}^k(\beta\cdot \sigma_1^{-1})=\mathbb{H}^{k-1}(\beta')
\end{equation}
for $\beta'\in \Brgr_{n-1}$ and $\beta=1\boxtimes\beta'\in \Brgr_n$ is the inclusion of $\beta'$ in the braid with $n$ strings in a such way it tangles only the last $n-1$ strings.

Next we can use proposition~\ref{prop: small prod 1} to conclude:

\begin{multline*}\mathbb{H}^k(\beta\cdot \sigma_1^{\pm 1})=
\mathbb{H}(\forg^*(\CE_\frn(\cl^*(\bar{\pi}_{13*}(\CE_{\frn_n}(\bar{\pi}_{12}\otimes_{B_n}\bar{\pi}_{23})^*(\bar{\calC}_\beta\boxtimes\bar{\calC}_{\sigma_1^{\pm 1}}))^{T_n}))^{T_n}\otimes \Lambda^k\calB)\\
=\mathbb{H}(\forg^*(\CE_\frn(\cl^*(\bar{\pi}_{13*}(\CE_{\frn_2}(\bar{\pi}_{12}\otimes_{B_2}\bar{\pi}_{23})^*(\bar{\calC}_\beta\boxtimes\bar{\calC}_{\sigma^{\pm 1}_1}))^{T_2}))^{T_n}\otimes\Lambda^k\calB)
\end{multline*}

Thus we have shown that
$$ \mathbb{H}^k(\beta\cdot \sigma_{n-1}^{\pm 1})= \mathbb{H}(\forg^*(\CE_{\frn\oplus \frn_2}(j^*_{e}(\bar{\pi}_{12}\otimes_{B_2}\bar{\pi}_{23})^*(\bcalC_{\beta}\boxtimes\bcalC_{\sigma_{1}^{\pm 1}}))^{T_n\times T_2}\otimes \wedge^k\calB_n),$$
where $j_e:\frb\times G_2\times \frn\to \calXr_3(G_n,G_{n,n-1})$ is the embedding that identifies $\frb\times G_2\times \frn$ with the subvariety of $\calXr_3(G_n,G_{n,n-1})$ defined by the
equation $j_{e}(X,g,Y)=(X,g^{-1},g,Y)$.

Let us simplify the complex $(j^*_{e}(\bar{\pi}_{12}\otimes_{B_2}\bar{\pi}_{23})^*(\bcalC_{\beta}\boxtimes\bcalC_{\sigma_{1}^{\pm 1}})\in \MF_{B_n\times B_2}(\frb\times G_2\times\frn,0)$.
Since $\bcalC_\beta=\indb_{n-1}(1\otimes\bcalC_{\beta'})\in \MF_{B^2}(\calXr_2,\Wr)$ it has the following description
$$\bcalC_\beta=(M\otimes\Lambda^* V,D+d^++d^-,\partial_l,\partial_r) \in \MF_{B^2}(\calXr_2,\Wr),$$
where $V=\CC^{n-1}=\langle \theta_2,\dots,\theta_n\rangle$; $(M,D)\in \Mod_{per}(\calXr_2)$ such that $D^2=\Wr $ modulo ideal $I_{n-1}=(g_{12},\dots, g_{1n})\subset \CC[\calXr_2]$;
$d^+=\sum_{i=2}^n g_{1i}\theta_i$ and $d^-=\sum_{\vec{i}} d^-_{\vec{i}}\frac{\partial}{\partial \theta_{\vec{i}}}$; $\partial_l,\partial_r$ define equivariant structure such that
$(M,D,\partial_l,\partial_r)$ modulo the ideal $I_{n-1}$ is an equivariant matrix factorization $p_1^*(\bcalC_{\beta'})$.

As we explained previously in the subsection~\ref{ssec: row ops} conjugation by $\exp(\phi\frac{\partial}{\partial\theta_2})$ of the differential of $(\bar{\pi}_{12}\otimes_{B_2}\bar{\pi}_{23})^*(\bcalC_{\beta}\boxtimes\bcalC_{\sigma_{n-1}^{\pm 1}})$
 results in the row transformation of
 of  the first type. Using these transformations we can obtain a complex $\calC'$ that is homotopic to $(\bar{\pi}_{12}\otimes_{B_2}\bar{\pi}_{23})^*(\bcalC_{\beta}
 \boxtimes\bcalC_{\sigma_{n-1}^{\pm 1}}) $ and the only dependence of it differentials on  $g_{12}$
comes from the term $g_{12}\theta_2$ of $d^+$. Then we can contract $\calC'$ along the differential $g_{12}\theta_2$ to obtain a homotopy equivalence
$$(\bar{\pi}_{12}\otimes_{B_2}\bar{\pi}_{23})^*(\bcalC_{\beta}
 \boxtimes\bcalC_{\sigma_{n-1}^{\pm 1}}) \sim\calC''\sim (\bcalC_{\beta,12}\otimes \bar{\pi}_{23}^*(\bcalC_\pm)),$$
where the last complex is the complex of $\CC[\frb\times B_2\times G_n\times \frn]$-modules and
$\bcalC_{\beta,12}\in \MF_{B\times B_2}(\frb\times B_2\times \frn,0)$ is the matrix factorization defined by the data
$$\bcalC_{\beta,12}=(j_{12}^*(M)\otimes\Lambda^* V_{12},d^+_{12}+j^*_{12}(D+d^++d^-),j^*_{12}(\partial_l),j^*_{12}(\partial_r)),$$
where $j_{12}:\frb\times B_2\times G_n\times\frn\to \frb \times G_2\times G_n\times \frn$ is the embedding; $V_{12}=\langle \theta_3,\dots, \theta_n\rangle$ and
$d^+_{12}=\sum_{i=3}^n g_{1i}\theta_i$.

Next note that since $\CE_{\frn_2}(\CC[B_2])^{T_2}=\CC$ there is a homotopy equivalence
\begin{multline*} \CE_{\frn_2}( j^*_{e}(\bar{\pi}_{12}\otimes_{B_2}\bar{\pi}_{23})^*(\bcalC_{\beta}\boxtimes\bcalC_{\sigma_{n-1}^{\pm 1}}))^{T_2}\sim
  j_{e,e}^*(\bcalC_{12,\beta}\otimes \bar{\pi}^*_{23}(\bcalC_{\pm}))\\=j_{e,e}^*(\bcalC_{\beta,12})\otimes j_{e,e}^*(\bar{\pi}^*_{23}(\bcalC_{\pm})),
  \end{multline*}
where $\SymbolPrint{j_{e,e}}: \frb\times 1\times 1\times \frn\to \frn\times B_2\times G_n\times \frn$.
Below we analyze two terms in the last tensor product.

Let us notice that $\bcalC_{\pm}|_{g=1}=[x_{11}-x_{22},0]^{\pm}$ where where $[x_{11}-x_{22},0]^+=[x_{11}-x_{22},0]$ and $[x_{11}-x_{22},0]\langle -1\rangle$.Thus we obtain
\[ j_{e,e}^*(\bar{\pi}^*_{23}(\bcalC_{+}))=[x_{11}-x_{22},0]^+:=[x_{11}-x_{22},0],\]
\[j_{e,e}^*(\bar{\pi}^*_{23}(\bcalC_{-}))=[x_{11}-x_{22},0]^-:=[x_{11}-x_{22},0]\otimes \chi_n^{-1}.\]

Let us denote by $p$ the natural projection map $\frb_n\times \frn_n\to \frb_{n-1}\times \frn_{n-1}$ that gives the map $p:\widetilde{\Hilb}^{free}_{1,n}\to \widetilde{\Hilb}^{free}_{1,n-1}$ when
restricted to the stable locus.
The key observation about the other term in the last tensor product of the matrix factorization is that it contains a sub complex $p^*(j_{e}^*(\bcalC_{\beta'}))=(M',D',\partial'_l+\partial_r')$,
$\partial'_l+\partial_r':M'\otimes \Lambda^\bullet(\frn_{n-1})\to M'\otimes \Lambda^{<\bullet}(\frn_{n-1})$
$$j_{e,e}^*(\bcalC_{12,\beta})=(M'\otimes \Lambda^\bullet V_{12},D'+d'_-,\partial'_l+\partial'_r+\partial^f)$$
where $d'_-=j_{e,e}^*(j_{12}^*(d^-))$ and $$\partial^f+\partial'_l+\partial'_r: M'\otimes \Lambda^\bullet(\frn_{n})\otimes \Lambda^* V_{12}\to M'\otimes \Lambda^{<\bullet}(\frn_{n})\otimes \Lambda^* V_{12},$$
with $(M'\otimes \Lambda^\bullet (\frn_{n-1}),\partial'_l+\partial'_r)\otimes \Lambda^* V_{12}$ being a sub complex of $( M'\otimes \Lambda^\bullet(\frn_{n})\otimes \Lambda^* V_{12},\partial^f+\partial'_l+\partial'_r)$.
Moreover $\partial^f$ has a non-trivial degree along with respect to the exterior power $\Lambda^* V_{12}$:
$$\partial^f: M'\otimes \Lambda^*(\frn_{n})\otimes \Lambda^\bullet V_{12}\to M'\otimes \Lambda^{*}(\frn_{n})\otimes \Lambda^{<\bullet} V_{12},$$

We complete our argument with an analysis of spectral sequence for the differential $D_{tot}=(D'+\partial'_l+\partial'_r)+d_{ce}+d_c+\partial^f+d'_-$ by first
computing homology with respect to $d_{ce}$ then with respect to $d_c$, we give details below.

Let us summarize the previous discussion, we have shown that
\[\mathbb{H}^k(\beta\cdot\sigma_{1}^{\pm 1})=
\mathbb{H}(
M'\otimes V_{12}\otimes [x_{11}-x_{22},0]^{\pm}\otimes \Lambda^k\calB_n,  D_{tot}),\]
\[\quad D_{tot}=(D'+\partial'_l+\partial'_r)+(d_{ce}+d_c+\partial^f+d'_-).
\]

In the total differential $D_{tot}$ the term $d_c$ is the Cech differential for some Cech cover
of $\widetilde{\Hilb}_{1,n}$.
In particular, if we choose Cech cover of $\widetilde{\Hilb}_{1,n}^{free}$ by the affine varieties $U_{\vec{S}_1,\vec{S}_2}=B\cdot \mathbb{A}_{\vec{S}_1,\vec{S}_2}$ from the proof of proposition~\ref{prop: aff cover}, we
can derive from fact that $B$-action on $U_{\vec{S}_1,\vec{S}_2}$ is free that
$$\textup{H}^k(\CC[U_{\vec{S}_1,\vec{S}_2}],d_{ce})=\begin{cases} &\CC[T\cdot \mathbb{A}_{\vec{S}_1,\vec{S}_2}],\mbox{ if } k=0\\ &0,\mbox{ if } k>0.\end{cases}.$$

Thus we can conclude that \(\mathbb{H}(\beta\cdot\sigma_{1}^{\pm 1})\) is equal to
$$ \mathbb{H}(\Hilb_{1,n}^{free}, (M'\otimes V_{12}\otimes [x_{11}-x_{22},0]^{\pm})^{B}\otimes \Lambda^k\calB_n,
(D'+\partial'_l+\partial'_r)+d_c+\partial^f+d'_-).$$
where superscript $B$ indicate the descend of the complex to the quotient $\widetilde{\Hilb}_{1,n}^{free}/B=\Hilb_{1,n}^{free}$.

Let $\Delta_{1,2}\subset\frh$ be the root hyperplane for  the root $\epsilon_1-\epsilon_{2}$ and $\Hilb^{free}_{1,n}(\Delta_{1,2}):=\chi^{-1}(\Delta_{1,2})$. By contracting the differential
of the complex $[x_{11}-x_{22},0]$ we obtain the last simplification for \(\mathbb{H}^k(\beta\cdot\sigma_{1}^{\pm 1})\):
$$ \mathbb{H}(\Hilb_{1,n}^{free}(\Delta_{1,2}), (M'\otimes\Lambda^* V_{12}\otimes\chi_n^{-1/2\pm 1/2})^{B}\otimes \Lambda^k\calB_n,
(D'+\partial'_l+\partial'_r)+d_c+\partial^f+d'_-).$$


Let us present the torus $T_n$ of $B_n$ as a product $T_1\times T_{n-1}$ where $T_1$ is the one-dimensional torus that preserves $B_{n-1}$ and  $T_{n-1}$ is the torus of $B_{n-1}$.
Hence the torus $T_1$  acts non-trivially on the fibers of $p:\widetilde{\Hilb}_{1,n}^{free}\to \widetilde{\Hilb}_{1,n-1}^{free}$ and $T_{n-1}$ fixes these fibers.

As the final step we push forward the last complex along the morphism $p$.
 For that let us observe that the elements $\theta_3,\dots,\theta_n$ have weight $-1$ with respect to $T_1$ and $0$ with respect to $T_{n-1}$ thus
 $$ (\Lambda^* V_{12})^B=\Lambda^* \calO_n(-1)^{\oplus n-2}.$$

Since $\chi_n^{-1}$ descends to $\calO_n(-1)$ we can use proposition~\ref{prop: push forward} to conclude that for the map $p:\Hilb^{free}_{1,n}(\Delta_{n,n-1})\to \Hilb^{free}_{1,n}$ we have
$$ p_*(\wedge^k\calB_n\otimes\Lambda^*(V_{12})^B)=\wedge^k\calB_{n-1},\quad  p_*(\wedge^k\calB_n\otimes \calO_n(-1)\otimes \Lambda^*(V_{12})^B)=\wedge^{k-1}\calB_{n-1}[n].$$

Since $(M',D',\partial'_l+\partial'_r)=p^*(j_{e}(\bcalC_{\beta'}))$ and both $\partial^f$ and $d'_-$ have non-trivial degree along the exterior power
$\Lambda^* V_{12}$  we can apply the projection formula to conclude
\[ \mathbb{H}^k(\beta\cdot\sigma_{1})=\mathbb{H}(\Hilb^{free}_{1,n-1}, j^*_e(\bcalC_{\beta'})\otimes \Lambda^k\calB_{n-1}),\]\
\[\mathbb{H}^k(\beta\cdot\sigma_{1}^{-1})=\mathbb{H}(\Hilb^{free}_{1,n-1}, j^*_e(\bcalC_{\beta'})\otimes \Lambda^{k-1}\calB_{n-1}[n]).\]
 The last formula  implies the desired equation (\ref{eq: Markov}).

 \section{Connections with other knot invariants}

\subsection{Unknot}
\label{sec:unknot}

Let us evaluate our knot invariant on the unknot \(L(1)\), \(1\in \Brgr_1\).

\begin{proposition} The Poincare polynomial of \(L(1)\) is
  \[\sum_{i,j,k}\mathcal{H}^{i,j,k}q^it^j a^i=\frac{a^{-1}+at}{q^{-1}-q}.\]
\end{proposition}
\begin{proof}
  We have \(X=\calXr_2(G_1)=\frg_1\times G_1\times 0\). Respectively \(\bar{\mathcal{C}}_1=\mathcal{O}_X\) and since \(\calX_1(G_1)=\frg_1\times 0\), we have
  \(\mathbb{S}_1=j^*_e(\bar{\mathcal{C}}_1)=\mathcal{O}_{\frg_1}\). Since weight of \(T_{sc}\)-action along \(\frg_1=\CC\)  is \(q^2\) and \(\mathcal{B}=\CC\) is
  the trivial vector bundle of weight \(t\), we conclude the formula from the statement.
\end{proof}
\subsection{HOMFLY-PT polynomial specialization}
The space $\calH^k(\beta)$ is  doubly-graded with $q,t$ gradings induced by the action of $T_{sc}$ on $\Hilb^{free}_{1,n}$. We denote by $\calH^{i,j,k}(\beta)$ the subspace of
$\calH^k(\beta)$ of degree $t^iq^j$. Respectively we introduce the super-polynomial:
$$\SymbolPrint{\calP}(L(\beta)):= \sum_{i,j,k} t^iq^j a^k \calH^{i,j,k}(\beta),$$
where $\beta\in \Brgr_n$.
In this section we explain why our super polynomial $\calP(L)$ specializes to the HOMFLY-PT polynomial $P(L)$.

First of all let us notice that
$\calP(L(\beta))$ has  a well defined specialization at $t=-1$. The space \(\mathbb{H}^k(\beta)\) is infinite dimensional doubly-graded space with each graded piece is
finite-dimensional. However if we forget about one of the gradings the last property might stop being true but it does not happen under the specialization \(t=-1\).

Indeed,  only coordinates $Y$ along $\frn$ have non-trivial $t$-weights in coordinate ring of $\widetilde{\Hilb}_{1,n}^{free}$. On the other hand since
$Y$ is strictly upper triangular all these variables have non-trivial positive weights with respect to the maximal torus \(T\subset B\).
The other coordinates on $\widetilde{\Hilb}_{1,n}^{free}$ have nonnegative weights with respect to the maximal torus and positive \(q\)-weight.
Thus only finitely many powers of $Y_{ij}$ could possibly contribute to $\mathbb{H}^k(\beta)$ which is constructed out of \(T\)-isotypical components
of the rings of functions on the open Cech charts of \(\widetilde{\Hilb}_{1,n}^{free}\).

Let us fix  conventions about the HOMFLY-PT polynomial $P(L)$:
$$ aP(L_+)-a^{-1} P(L_-)=(q-q^{-1})P(L_0), \quad P(\mathrm{unknot})=1,$$
where $L_+$ and $L_-$ are positive and negative crossings.

\begin{theorem}\label{thm:HOMFLYPT} For any  $\beta\in \Brgr_n$ we have
$$P(L(\beta))= \calP(L(\beta))|_{t=-1}$$
\end{theorem}
\begin{proof}
First, let us notice that the theorem~\ref{thm: main} implies that the map \[\Tr_n: \Brgr_n\to \calP(L(\beta))|_{t=-1}\] is a trace map: $\Tr_n(ab)=\Tr_n(ba)$.
Also, the  construction of the elements $\calC_{\pm}^{(i)}$ and the description of $\calC_+$ and $\calC_-$ as cones in proposition~\ref{prop: cones} implies  that $\Tr_n$ satisfies the skein relations:
$$ a\Tr_n(\tau\sigma_i)-a^{-1}\Tr_n(\tau\sigma_i^{-1})=(q-q^{-1})\Tr_{n}(\tau),$$
$\tau\in \Brgr_{n-1}$.

Let us spell out the details for the skein relations.
Since we imposed condition that all differentials in our matrix factorizations have degree \(t\),  setting \(t=-1\) corresponds to passing to the K-theory of the category matrix factorizations.
We denote by \([\calC]\) the corresponding class of the matrix factorization in the K-theory. Thus the statement in proposition~\ref{prop: cones} implies the linear relations in K-theory:
\[[\calC_+]=-\mathrm{q}^{-1}\left([\calC_\parallel]-[\calC_\bullet\langle-1,-1\rangle]\right),\]
\[[\calC_-]=-\mathrm{q}^{-1}\left(-\mathrm{q}^2[\calC_\parallel]+[\calC_\bullet\langle-1,-1\rangle]\right).\]
Taking into account the shift \(k-\mathrm{wr}(\beta)+n-1\) in the definition of \(\mathcal{H}^k(\beta)\) we the skein relation by canceling the
terms related to class \([\calC_\bullet\langle-1,-1\rangle]\).

Since $L(\tau\sigma)=L(\tau\sigma^{-1})$ the inductive argument shows that the skein relations imply that $\Tr_{n}(1)=A^{n}$ where $A=\Tr_1(1)=(a^{-1}-a)/(q^{-1}-q)$. Indeed, since \(L(\sigma_n^{\pm 1}1_{n+1})=L(1_n)\), \(1_n\in\Brgr_n\) and \(Tr_n(\beta)\) is an isotopy invariant of \(L(\beta)\), we can use the skein relations:
\[a\Tr_{n+1}(\sigma_n)-a^{-1}\Tr_{n+1}(\sigma_n^{-1})=(a-a^{-1})\Tr_n(1)=(q-q^{-1})\Tr_{n+1}(1).\]

Let us defined renormalized trace
\begin{equation}\label{eq: trace OJ}
 \textup{tr}'_n(\beta):=\left(\frac{a}{q}\right)^{\xwr(\beta)}A^{-n}\Tr_n(\beta),
\end{equation}
where \(\xwr(\sigma_{i_1}^{\epsilon_1}\dots \sigma_{i_\ell}^{\epsilon_\ell})=\sum_{i=1}^{\ell} \epsilon_i.\)

The Hecke algebra \(H_n\) is generated by \(g_i\), \(i=1,\dots,n-1\) modulo relations:
\[g_ig_{i+1}g_{i}=g_{i+1}g_ig_{i+1},
  \quad i=1,\dots,n-2,\]
\[g_i-g_i^{-1}=q-q^{-1},\quad i=1,\dots, n-1.\]

There is a natural algebra homomorphism $\pi:\Brgr_n\to H_n$, \(\sigma_i\mapsto g_i\).
The skein relations imply that the map $\textup{tr}'_n$ factors through the projection $\pi$:
$\textup{tr}_n=\textup{tr}_n\circ \pi$ where $\textup{tr}'_n: H_n\to \CC(a,q)$ is the trace on $H_n$. Since the trace $\Tr_n$ satisfies the Markov move $\Tr_n(\tau\sigma_{n-1}^{\pm1 })=\Tr_{n-1}(\tau)$ we see that
the sequence of traces $\textup{tr}_n$ satisfies the same restrictions as Ocneanu-Jones trace (see section 5 of \cite{Jones}):
\begin{enumerate}
\item $\textup{tr}_n(1)=1$
\item $\textup{tr}_n(\alpha\beta)=\textup{tr}(\beta\alpha)$
\item $\textup{tr}_n(\tau g_{n-1})=z\textup{tr}_{n-1}(\tau),$  $g\in H_{n-1}$
\end{enumerate}
where $z=aq^{-1}A^{-1}$.
The factor \(z\) appears in the last formula because
\[\textup{tr}_n(\tau g_{n-1})=\left(\frac{a}{q}\right)^{\mathrm{wr}(\tau g_{n-1})}A^{-n}\Tr_n( \tau g_{n-1})=
\left(\frac{a}{q}\right)^{\xwr(\tau)}A^{-n+1} \Tr_{n-1}(\tau) aq^{-1}A^{-1}.\]

According to the section 5 of \cite{Jones} the trace on $\cup H_n$ with such properties is unique and according to the section 6 of \cite{Jones} the renormalized trace $\Tr_n$ as in equation
\ref{eq: trace OJ} defines the HOMFLY-PT polynomial.
\end{proof}

\subsection{$gl(m|n)$-homology differential}
Recall that the vector bundle $\calB_n^\vee$ over $\widetilde{\Hilb}^{free}_{1,n}$ is  a vector bundle with the fiber $V $ over the point $(X,Y,v)$ where $V=\CC^n$ and $v$ is
the cyclic vector. Define a section $\varphi_{m|n}$ of $\calB_n^\vee$ with the value $X^mY^n v\in V$ at the fiber at the point $(X,Y,v)$. Contraction with this
section
$$i_{\varphi_{m|n}}: \Lambda^\bullet \calB_n\to \Lambda^{\bullet-1}\calB_n.$$
induces the differential  $d_{m|n}: \bbS_\beta\otimes \Lambda^\bullet\calB_n\to \bbS_\beta\otimes \Lambda^{\bullet-1}\calB_n$, which
commutes with the \v{C}ech differential
$d_c$, the Chevalley-Eilenberg differential $d_{ce}$, the equivariant correction differentials $\partial_l+\partial_r$ and the matrix factorization differential $D$. Thus  we can define following homology
$$ \mathbb{H}_{m|n}(\beta)=(\oplus_{i=0}^n \bbS_\beta\otimes \Lambda^\bullet\calB_n, D_{tot}^{m|n}),\quad D_{tot}^{m|n}:=D+\partial_l+\partial_r+d_{ce}+d_{c}+d_{m|n}.$$
The homology $\mathbb{H}_{m|n}(\beta)$ is only doubly graded. In forthcoming paper we prove the following

\begin{theorem}\cite{OR} $\mathbb{H}_{m|n}(\beta)$ is an isotopy invariant of $L(\beta)$.
\end{theorem}

The differential $d_{m|n}$ is very much in spirit of the conjectures in the papers \cite{ORS, GGS} and thus we propose the conjecture:

\begin{conjecture} The homology $\mathbb{H}_{m|n}$ categorifies the quantum invariant of type $gl(m|n)$.
\end{conjecture}

\printindex[notationlist]

\end{document}

\end{document}

%% file: gendef.tex
\def\rhs{{\it r.h.s.\ }}

\def\End{\mathop{{\rm End}}\nolimits}

\def\Hom{\mathop{{\rm Hom}}\nolimits}

\def\Ob{\mathop{{\rm Ob}}\nolimits}

\def\im{\mathop{{\rm im}}\nolimits}
\def\Tr{\mathop{{\rm Tr}}\nolimits}

\def\deg{ \mathop{{\rm deg}}\nolimits }

\def\End{\mathop{{\rm End}}\nolimits}

\def\Ad{ {\rm Ad} }
\def\Adv#1{ \Ad_{#1} }

\def\xG{ \mfG }



%% file: basdef.tex
\parskip=\medskipamount                 
\thicklines                     
\def\inbar{\vrule height1.5ex width.4pt depth0pt}
\def\IC{\relax\,\hbox{$\inbar\kern-.3em{\rm C}$}}
\def\IN{\relax{\rm I\kern-.18em N}}
\def\IQ{\relax\,\hbox{$\inbar\kern-.3em{\rm Q}$}}
\def\IR{\relax{\rm I\kern-.18em R}}
\def\ZZ{\relax{\sf Z\kern-.4em Z}}
\def\a{\alpha}

 \def\cB{{\cal B}}  

\def\cF{{\cal F}}   
   \def\cM{{\cal M}}
\def\cN{{\cal N}}   
   
\newtheorem{theorem}{Theorem}[section]
\newtheorem{proposition}[theorem]{Proposition}
\newtheorem{corollary}[theorem]{Corollary}
\newtheorem{conjecture}[theorem]{Conjecture}
\newtheorem{lemma}[theorem]{Lemma}

\newtheorem{remark}[theorem]{Remark}

\catcode`\@=11

\newif\if@fewtab\@fewtabtrue

\catcode`\@=11

\newif\if@fewtab\@fewtabtrue

{\count255=\time\divide\count255 by 60
\xdef\hourmin{\number\count255} \multiply\count255
by-60\advance\count255 by\time
\xdef\hourmin{\hourmin:\ifnum\count255<10 0\fi\the\count255}}
\def\ps@draft{\let\@mkboth\@gobbletwo
    \def\@oddhead{}
    \def\@oddfoot
      {\hbox to 7 cm{\footnotesize {\em Draft of \jobname:} \draftdate
       \hfil}\hskip -7cm\hfil\rm\thepage \hfil}
    \def\@evenhead{}\let\@evenfoot\@oddfoot}

\def\ceqno{\global\@fewtabfalse
    \ifcase\@eqcnt \def\@tempa{& & &}\or \def\@tempa{& &}
      \or \def\@tempa{&}
      \or\def\@tempa{}\fi\@tempa
{\rm(\theequation)}}

\def\aeqno#1{\global\@fewtabfalse
    \ifcase\@eqcnt \def\@tempa{& & &}\or \def\@tempa{& &}
      \or \def\@tempa{&}
      \or\def\@tempa{}\fi\@tempa
{\rm(\theequation,#1)}}

\def\label#1{\ifnum\draftcontrol=1
 \global\def\draftnote{$\scriptstyle #1$}\fi
 \@bsphack\if@filesw {\let\thepage\relax
   \def\protect{\noexpand\noexpand\noexpand}%
\xdef\@gtempa{\write\@auxout{\string
      \newlabel{#1}{{\@currentlabel}{\thepage}}}}}\@gtempa
   \if@nobreak \ifvmode\nobreak\fi\fi\fi
  \@esphack}

\def\alabel#1#2{\label{#1}\global\@fewtabfalse
    \ifcase\@eqcnt \def\@tempa{& & &}\or \def\@tempa{& &}
      \or \def\@tempa{&}
      \or\def\@tempa{}\fi\@tempa
{\hbox to 3cm{\phantom{\rm(\theequation,#2)} \draftnote
\hfil}\hskip -3cm {\rm(\theequation,#2)}}}

\def\clabel#1{\label{#1}\global\@fewtabfalse
    \ifcase\@eqcnt \def\@tempa{& & &}\or \def\@tempa{& &}
      \or \def\@tempa{&}
      \or\def\@tempa{}\fi\@tempa
{\hbox to 3cm{\phantom{\rm(\theequation)} \draftnote \hfil}\hskip
-3cm{\rm(\theequation)}}}

\def\eqnarray{\def\draftnote{{}}\global\@fewtabtrue
\stepcounter{equation}\let\@currentlabel=\theequation
\global\@eqnswtrue
\global\@eqcnt\z@\tabskip\@centering\let\\=\@eqncr
$$\halign to \displaywidth\bgroup\@eqnsel\hskip\@centering\@eqcnt\z@
  $\displaystyle\tabskip\z@{##}$&\global\@eqcnt\@ne
  \hskip 1\arraycolsep \hfil$\displaystyle{##}$\hfil
  &\global\@eqcnt\tw@ \hskip 1\arraycolsep
$\displaystyle\tabskip\z@{##}$ \hfil
\tabskip\@centering&\global\@eqcnt\thr@@\llap{##}\tabskip\z@ \cr}

\def\endeqnarray{\@@eqncr\egroup
      \global\advance\c@equation\m@ne$$\global\@ignoretrue}

\def\@eqnnum{\hbox to 3cm{\phantom{\rm(\theequation)} \draftnote
                         \hfil}\hskip -3cm {\rm(\theequation)}}

\def\@@eqncr{\let\@tempa\relax
    \ifcase\@eqcnt \def\@tempa{& & &}\or \def\@tempa{& &}
      \or \def\@tempa{&}
      \or\def\@tempa{}
\fi\@tempa \if@eqnsw \if@fewtab\@eqnnum\fi
\stepcounter{equation}\fi\global
\@eqnswtrue\global\@eqcnt\z@\global\@fewtabtrue\cr}

\def\draftcite#1{\ifnum\draftcontrol=1#1\else{}\fi}

\def\@lbibitem[#1]#2{\item{}\hskip -3cm \hbox to 2cm
{\hfil$\scriptstyle\draftcite{#2}$}\hskip
1cm[\@biblabel{#1}]\if@filesw
     {\def\protect##1{\string ##1\space}\immediate
      \write\@auxout{\string\bibcite{#2}{#1}}}\fi\ignorespaces}

\def\@bibitem#1{\item\hskip -3cm \hbox to 2cm
{\hfil $\scriptstyle\draftcite{#1}$}\hskip 1cm \if@filesw
\immediate\write\@auxout
       {\string\bibcite{#1}{\the\value{\@listctr}}}\fi\ignorespaces}

\def\draftdate{\number\month/\number\day/\number\year\ \ \ \hourmin }
 \global\def\draftcontrol{0}
\catcode`\@=12

\def\theequation{{\thesection.\arabic{equation}}}

\def\qq{\begin{eqnarray}}
\def\qqq{\end{eqnarray}}
\def\ee{\begin{eqnarray}}
\def\eee{\end{eqnarray}}

\def\xlee#1{ \begin{eqnarray} \label{#1} }
\def\xeee{ \end{eqnarray} }
\def\ylee#1{ \begin{eqnarray}\nonumber }
\def\yeee{ \end{eqnarray} }
\def\zlee#1{ \begin{displaymath} }
\def\zleee{ \end{displaymath} }
\def\wlee#1{ $ }

\newlength{\shiftwidth}
\def\shift#1{&&\hbox to \shiftwidth{\hfill $\displaystyle#1$}}
\newlength{\sshiftwidth}
\def\sshift#1{\lefteqn{\hbox to
\sshiftwidth{\hfill$\displaystyle#1$}}}

\def\qbezier{\bezier{120}}
\def\DottedCircle{
\bezier{4}(0.966,-0.259)(1.04,0)(0.966,0.259)
\bezier{4}(0.966,0.259)(0.897,0.518)(0.707,0.707)
\bezier{4}(0.707,0.707)(0.518,0.897)(0.259,0.966)
\bezier{4}(0.259,0.966)(0,1.04)(-0.259,0.966)
\bezier{4}(-0.259,0.966)(-0.518,0.897)(-0.707,0.707)
\bezier{4}(-0.707,0.707)(-0.897,0.518)(-0.966,0.259)
\bezier{4}(-0.966,0.259)(-1.04,0)(-0.966,-0.259)
\bezier{4}(-0.966,-0.259)(-0.897,-0.518)(-0.707,-0.707)
\bezier{4}(-0.707,-0.707)(-0.518,-0.897)(-0.259,-0.966)
\bezier{4}(-0.259,-0.966)(0,-1.04)(0.259,-0.966)
\bezier{4}(0.259,-0.966)(0.518,-0.897)(0.707,-0.707)
\bezier{4}(0.707,-0.707)(0.897,-0.518)(0.966,-0.259) }
\def\Endpoint[#1]{
\ifcase#1 \put(1,0){\circle*{0.15}}
\or\put(0.866,0.5){\circle*{0.15}}
\or\put(0.5,0.866){\circle*{0.15}} \or\put(0,1){\circle*{0.15}}
\or\put(-0.5,0.866){\circle*{0.15}}
\or\put(-0.866,0.5){\circle*{0.15}} \or\put(-1,0){\circle*{0.15}}
\or\put(-0.866,-0.5){\circle*{0.15}}
\or\put(-0.5,-0.866){\circle*{0.15}} \or\put(0,-1){\circle*{0.15}}
\or\put(0.5,-0.866){\circle*{0.15}}
\or\put(0.866,-0.5){\circle*{0.15}} \fi}
\def\Arc[#1]{
\thicklines         
\ifcase#1 \bezier{25}(0.966,-0.259)(1.04,0)(0.966,0.259) \or
\bezier{25}(0.966,0.259)(0.897,0.518)(0.707,0.707) \or
\bezier{25}(0.707,0.707)(0.518,0.897)(0.259,0.966) \or
\bezier{25}(0.259,0.966)(0,1.04)(-0.259,0.966) \or
\bezier{25}(-0.259,0.966)(-0.518,0.897)(-0.707,0.707) \or
\bezier{25}(-0.707,0.707)(-0.897,0.518)(-0.966,0.259) \or
\bezier{25}(-0.966,0.259)(-1.04,0)(-0.966,-0.259) \or
\bezier{25}(-0.966,-0.259)(-0.897,-0.518)(-0.707,-0.707) \or
\bezier{25}(-0.707,-0.707)(-0.518,-0.897)(-0.259,-0.966) \or
\bezier{25}(-0.259,-0.966)(0,-1.04)(0.259,-0.966) \or
\bezier{25}(0.259,-0.966)(0.518,-0.897)(0.707,-0.707) \or
\bezier{25}(0.707,-0.707)(0.897,-0.518)(0.966,-0.259) \fi}
\def\DottedArc[#1]{
\ifcase#1 \bezier{4}(0.966,-0.259)(1.04,0)(0.966,0.259) \or
\bezier{4}(0.966,0.259)(0.897,0.518)(0.707,0.707) \or
\bezier{4}(0.707,0.707)(0.518,0.897)(0.259,0.966) \or
\bezier{4}(0.259,0.966)(0,1.04)(-0.259,0.966) \or
\bezier{4}(-0.259,0.966)(-0.518,0.897)(-0.707,0.707) \or
\bezier{4}(-0.707,0.707)(-0.897,0.518)(-0.966,0.259) \or
\bezier{4}(-0.966,0.259)(-1.04,0)(-0.966,-0.259) \or
\bezier{4}(-0.966,-0.259)(-0.897,-0.518)(-0.707,-0.707) \or
\bezier{4}(-0.707,-0.707)(-0.518,-0.897)(-0.259,-0.966) \or
\bezier{4}(-0.259,-0.966)(0,-1.04)(0.259,-0.966) \or
\bezier{4}(0.259,-0.966)(0.518,-0.897)(0.707,-0.707) \or
\bezier{4}(0.707,-0.707)(0.897,-0.518)(0.966,-0.259) \fi}
\def\Chord[#1,#2]{
\thinlines \ifnum#1>#2\Chord[#2,#1] \else\ifnum#1<#2 \ifcase#1
\ifcase#2 \or\qbezier(1,0)(0.516,0.138)(0.866,0.5)
\or\qbezier(1,0)(0.45,0.26)(0.5,0.866)
\or\qbezier(1,0)(0.327,0.327)(0,1)
\or\qbezier(1,0)(0.179,0.311)(-0.5,0.866)
\or\qbezier(1,0)(0.0536,0.2)(-0.866,0.5) \or\put(1, 0){\line(-2,
0){2}} \or\qbezier(1,0)(0.0536,-0.2)(-0.866,-0.5)
\or\qbezier(1,0)(0.179,-0.311)(-0.5,-0.866)
\or\qbezier(1,0)(0.327,-0.327)(0,-1)
\or\qbezier(1,0)(0.45,-0.26)(0.5,-0.866)
\or\qbezier(1,0)(0.516,-0.138)(0.866,-0.5) \fi \or\ifcase#2\or
\or\qbezier(0.866,0.5)(0.378,0.378)(0.5,0.866)
\or\qbezier(0.866,0.5)(0.26,0.45)(0,1)
\or\qbezier(0.866,0.5)(0.12,0.446)(-0.5,0.866)
\or\qbezier(0.866,0.5)(0,0.359)(-0.866,0.5)
\or\qbezier(0.866,0.5)(-0.0536,0.2)(-1,0) \or\put(0.866,
0.5){\line(-5, -3){1.73}}
\or\qbezier(0.866,0.5)(0.146,-0.146)(-0.5,-0.866)
\or\qbezier(0.866,0.5)(0.311,-0.179)(0,-1)
\or\qbezier(0.866,0.5)(0.446,-0.12)(0.5,-0.866)
\or\qbezier(0.866,0.5)(0.52,0)(0.866,-0.5) \fi \or\ifcase#2\or\or
\or\qbezier(0.5,0.866)(0.138,0.516)(0,1)
\or\qbezier(0.5,0.866)(0,0.52)(-0.5,0.866)
\or\qbezier(0.5,0.866)(-0.12,0.446)(-0.866,0.5)
\or\qbezier(0.5,0.866)(-0.179,0.311)(-1,0)
\or\qbezier(0.5,0.866)(-0.146,0.146)(-0.866,-0.5) \or\put(0.5,
0.866){\line(-3, -5){1}} \or\qbezier(0.5,0.866)(0.2,-0.0536)(0,-1)
\or\qbezier(0.5,0.866)(0.359,0)(0.5,-0.866)
\or\qbezier(0.5,0.866)(0.446,0.12)(0.866,-0.5) \fi
\or\ifcase#2\or\or\or \or\qbezier(0,1.)(-0.138,0.516)(-0.5,0.866)
\or\qbezier(0,1.)(-0.26,0.45)(-0.866,0.5)
\or\qbezier(0,1.)(-0.327,0.327)(-1,0)
\or\qbezier(0,1.)(-0.311,0.179)(-0.866,-0.5)
\or\qbezier(0,1.)(-0.2,0.0536)(-0.5,-0.866) \or\put(0, 1){\line(0,
-2){2}} \or\qbezier(0,1.)(0.2,0.0536)(0.5,-0.866)
\or\qbezier(0,1.)(0.311,0.179)(0.866,-0.5) \fi
\or\ifcase#2\or\or\or\or
\or\qbezier(-0.5,0.866)(-0.378,0.378)(-0.866,0.5)
\or\qbezier(-0.5,0.866)(-0.45,0.26)(-1,0)
\or\qbezier(-0.5,0.866)(-0.446,0.12)(-0.866,-0.5)
\or\qbezier(-0.5,0.866)(-0.359,0)(-0.5,-0.866)
\or\qbezier(-0.5,0.866)(-0.2,-0.0536)(0,-1) \or\put(-0.5,
0.866){\line(3, -5){1}}
\or\qbezier(-0.5,0.866)(0.146,0.146)(0.866,-0.5) \fi
\or\ifcase#2\or\or\or\or\or
\or\qbezier(-0.866,0.5)(-0.516,0.138)(-1,0)
\or\qbezier(-0.866,0.5)(-0.52,0)(-0.866,-0.5)
\or\qbezier(-0.866,0.5)(-0.446,-0.12)(-0.5,-0.866)
\or\qbezier(-0.866,0.5)(-0.311,-0.179)(0,-1)
\or\qbezier(-0.866,0.5)(-0.146,-0.146)(0.5,-0.866) \or\put(-0.866,
0.5){\line(5, -3){1.73}} \fi \or\ifcase#2\or\or\or\or\or\or
\or\qbezier(-1,0)(-0.516,-0.138)(-0.866,-0.5)
\or\qbezier(-1,0)(-0.45,-0.26)(-0.5,-0.866)
\or\qbezier(-1,0)(-0.327,-0.327)(0,-1)
\or\qbezier(-1,0)(-0.179,-0.311)(0.5,-0.866)
\or\qbezier(-1,0)(-0.0536,-0.2)(0.866,-0.5) \fi
\or\ifcase#2\or\or\or\or\or\or\or
\or\qbezier(-0.866,-0.5)(-0.378,-0.378)(-0.5,-0.866)
\or\qbezier(-0.866,-0.5)(-0.26,-0.45)(0,-1)
\or\qbezier(-0.866,-0.5)(-0.12,-0.446)(0.5,-0.866)
\or\qbezier(-0.866,-0.5)(0,-0.359)(0.866,-0.5) \fi
\or\ifcase#2\or\or\or\or\or\or\or\or
\or\qbezier(-0.5,-0.866)(-0.138,-0.516)(0,-1)
\or\qbezier(-0.5,-0.866)(0,-0.52)(0.5,-0.866)
\or\qbezier(-0.5,-0.866)(0.12,-0.446)(0.866,-0.5) \fi
\or\ifcase#2\or\or\or\or\or\or\or\or\or
\or\qbezier(0,-1.)(0.138,-0.516)(0.5,-0.866)
\or\qbezier(0,-1.)(0.26,-0.45)(0.866,-0.5) \fi
\or\ifcase#2\or\or\or\or\or\or\or\or\or\or
\or\qbezier(0.5,-0.866)(0.378,-0.378)(0.866,-0.5) \fi\fi\fi\fi}
\def\FullChord[#1,#2]{
\Endpoint[#1] \Endpoint[#2] \Arc[#1] \Arc[#2] \Chord[#1,#2] }
\def\EndChord[#1,#2]{
\Endpoint[#1] \Endpoint[#2] \Chord[#1,#2] }
\def\Picture#1{
\begin{picture}(2,1)(-1,-0.167)
#1
\end{picture}
}
\def\DottedChordDiagram[#1,#2]{
\Picture{\DottedCircle \FullChord[#1,#2]} }


%% file: letdef.tex
\def\ZZ{ \mathbb{Z} }
\def\IQ{ \mathbb{Q} }
\def\IC{ \mathbb{C} }
\def\IR{ \mathbb{R} }

\def\xo{ o }

\def\tx{ \tilde{x} }

\def\mfg{ \mathfrak{g} }

\def\mfG{ \mathfrak{G} }

\def\cB{ {\cal B} }

\def\cB{ \mathcal{B} }

\def\cF{ \mathcal{F} }

%% file: part2_gl.tex
\section{Equivariant structure of matrix factorizations and their push-forwards}\label{sec: eq push}

In this section we discuss  the properties of equivariant matrix factorizations and construct the push forward of matrix factorizations
along the embedding of affine varieties $j: Spec(R)\to Spec(S)$. There is a more general, more sophisticated
construction for push forward of matrix factorization that could be found in \cite{DM} but the functor from \cite{DM} is of different nature then ours.
Our construction of the push-forward relies on the generalizations of the lemmas from the previous section and we discuss these extensions first.

\subsection{Technical lemmas: non-equivariant version}

We keep notations consistent with the previous section. Let $\SymbolPrint{\Mod_{per}}(\calZ)$ be a category  of $\ZZ_2$-graded free $\CC[\calZ]$ modules with  endomorphisms.
An object $\calF\in \Mod_{per}(\calZ)$ is a pair $(M,D)$, $M=M^0\oplus M^1$, $M^i=\CC[\calZ]^{d_i}$ and $D\in \Hom_{\CC[\calZ]}(M^1,M^0)\oplus \Hom_{\CC[\calZ]}(M^0,M^1)$.
The space of morphisms between two elements of $\Mod_{per}(\calZ)$ consists of  morphisms of the corresponding   $\CC[\calZ]$-modules that intertwine the endomorphisms.

We denote by $\SymbolPrint{\Com_{per}}(\calZ)$ be the category of complexes of $\ZZ_2$-graded modules with endomorphisms.
We denote by \(\Mod_{per}(\calZ)\) the abelian category with objects \((M,D)\) where \(M=M^0\oplus M^1\) is a \(\ZZ_2\)-graded \(\CC[\calZ]\)-module
and \(D\in \Hom_{\CC[\calZ]}(M^i,M^{i+1})\). The morphisms in \(\Mod_{per}(\calZ)\) preserve \(\ZZ_2\)-grading and intertwine the endomorphisms.
We would like to emphasize that the elements of \(\Mod_{per}(\calZ)\) are not matrix factorizations in general since we do not
impose constraint \(D^2=W\) and do not require the modules \(M_i\) to be free.

An object of
$\Com_{per}(\calZ)$ is a collection of objects $\calF_i\in \Mod_{per}(\calZ)$, $\calF_i=(M_i,D_i)$, $M_i=\CC[\calZ]^{r_i}$ and morphisms
$d^+_i: M_i\to M_{i-1}$ which are morphisms of elements of $\Mod_{per}(\calZ)$ and $(d^+)^2=0$ as a  linear map $M_i\to M_{i-2}$.
Since \(d^+_i\) intertwines the endomorphisms \(D_{i-1}\) and \(D_i\) the kernel \(\ker d^+_i\subset M_i \) and image \( \im d^+_i\subset M_{i-1}\) are
elements of \(\Mod_{per}(\calZ)\) and we use notation $H^i(\calF)$ for
the quotient \(\ker d^+_i/\im d^+_{i+1}\). For subset of $\Com_{per}(\calZ)$ consisting of complexes $\calF$ such that $M_i=0$ for $i\notin [0,\ell]$
we use notation $\SymbolPrint{\Com^\ell_{per}}(\calZ)$.

Analogously to the previous case we introduce notation $\Hom^k(\calF_\bullet,\calF_\bullet)=\oplus_i \Hom(\calF_i,\calF_{i-k})$ respectively we use notation
$\xExt^k_{d^+}(\calF_\bullet,\calF_\bullet)$ for the elements of $\Hom^k(\calF_\bullet,\calF_\bullet)$ that commute with the differential $d^+$. The homotopy relations
between the elements of $\xExt^k_{d^+}(\calF_\bullet,\calF_\bullet)$ is defined in the natural way:
$$ f\sim g,\mbox{ iff there is } h\in \Hom^{k-1}(\calF_\bullet,\calF_\bullet), \quad f-g=d^+\circ h-h\circ d^+.$$

The technical lemma that we need later is a generalization of lemmas~\ref{lem: ext MF} and \ref{lem: ext uniq}, the proof is essentially identical to the proof of the above mentioned lemmas.
\begin{lemma}
  \label{lem: ext MF+} Let $(\calF_\bullet,d_\bullet^+)$, $\calF_i=(M_i,D_i)$ be an element of $\Com^\ell_{per}(\calZ)$ and $F\in \CC[\calZ]$ such that \( M_i\) are free \(\CC[\calZ]\)-modules
  and
\begin{enumerate}
\item Elements of $\xExtdpv{<0}(\calF_\bullet,\calF_\bullet)$ are homotopic to $0$
\item Element $F-D^2\in \xExtdpv{0}(\calF_\bullet,\calF_\bullet)$ is homotopic to $0$
\end{enumerate}
Then there are maps  $d_{ij}^-\in \Hom( \calF_i, \calF_j)$, $j>i$ such that
$$ [(\calF_\bullet, D+d^++d^-)]_{per}\in \MFs(\calZ,F).$$
\end{lemma}
\begin{proof}
By assumption of the lemma there exist  morphisms of $\CC[\calZ]$ modules
$h^{(-1)}_i\in\Hom( \calF_i, \calF_{i+1})$ such that $F-D^2=h^{(-1)}_i\circ d_{i+1}^++d_i^+\circ h^{(-1)}_{i-1}$. Then a new differential $D^{(-1)}:=D+d^++h^{(-1)}$ satisfies the relation
$$ \left(D^{(-1)}\right)^2-F =(h^{(-1)})^2\in \yHom^{\leq-2}(C_\bullet,C_\bullet).$$
Thus $D^{(-1)}$ provides a base 
of induction for the proof of existence $h^{(-2i-1)}\in \yHom^{-2i-1}(\calF_\bullet,\calF_\bullet)$, $i\ge 0$ such that
\begin{enumerate}
\item $[h^{(-1)},d^+]_+=F-D^2$
\item $[h^{(-2i-1)},d^+]_+=\sum_{j=0}^{i-1} h^{(-2j-1)} h^{(-2i+2j+1)}$, $i>0$
\end{enumerate}
were we used the notation $[a,b]_+=ab+ba$.

Indeed, suppose we constructed $h^{(-2i-1)}$ for all $i<j$. Define $D^{(-2j+1)}=D+d^++\sum_{k=0}^{j-1} h^{(-1-2k)}$, then the inductive
assumption implies:
\begin{equation}\label{eq: D2 ext}
  \begin{aligned}
    \left( D^{(-2j+1)}\right)^2-F+D^2-R^{(-2j)}\in \yHom^{<-2j}(\calF_\bullet,\calF_\bullet),\\
    R^{(-2j)}=\sum_{k=0}^{j-1} h^{(-2k-1)} h^{(-2j+1+2k)}
    \end{aligned}
\end{equation}
On the other hand, a direct computation using the inductive assumptions and the fact that $F$ commutes with all $h^{(2i-1)}$ and $d^+$ implies
\begin{equation}\label{eq: com d+ ext}
 [d^+,R^{(-2j)}]=0.
\end{equation}
In more details, let us introduce notation $R^{(<-2j)}:=(D^{(-2j-1)})^2-F-R^{(-2j)}$ then
\begin{multline*}
0=[D^{(-2j-1)}, (D^{(2j-1)})^2]=[D+d^+ +\sum_i h^{(-2i-1)}, F-D^2+R^{(-2j)}+R^{(<-2j)}]\\=[d^+,R^{(-2j)}]+[d_++\sum_i h^{(-2i-1)},R^{(<-2j)}]+[\sum_i h^{(-2j-1)},R^{(-2j)}+R^{(-2j-1)}].
\end{multline*}
The first term in the last sum is from $\xHom^{-2j+1}(\calF_\bullet,\calF_\bullet)$, on the other hand the last two   are from
$\xHom^{<-2j+1}(\calF_\bullet,\calF_\bullet)$ thus (\ref{eq: com d+ ext}) follows.

Thus $R^{(-2j)}$ is the morphism of the complex $(\calF_\bullet,d^+)$ of degree $-2j$, that is, an element of $\Hom^{-2j}(\calF_\bullet,\calF_\bullet)$. The assumption of our lemma implies that such
morphism is homotopic to zero and the homotopy morphism $h^{(-2j-1)}\in \xHom^{-2j-1}(\calF_\bullet,\calF_\bullet)$ provides the induction step
since $[h^{(-2j-1)},d^+]_+=R^{(-2j)}$.

Thus by setting $d_i^-=\sum_i h^{(-2i-1)}$ we obtain the desired extension because of (\ref{eq: D2 ext}).
\end{proof}

\begin{lemma}\label{lem: uniq MF+} Let $(\calF_\bullet,d_\bullet^+)$ be an element of $\Com^\ell(\calZ)$ and $F\in \CC[\calZ]$ that satisfy conditions
of the previous lemma. Suppose there two sets of differentials  $d^-_{ij},\hat{d}^-_{ij}: \calF_i\to \calF_j$, $j>i$ such that
$$ [(\calF_\bullet, D+d^++d^-)]_{per}\in \MFs(\calZ,F),\quad [(\calF_\bullet, D+d^++\hat{d}^-)]_{per}\in \MFs(\calZ,F)$$
then there is $\Psi=\sum_{i\le 0}\Psi_i$, $\Psi_i\in \Hom^{-i}(\calF_\bullet,\calF_\bullet)$, $\Psi=\mathrm{Id}$ such
$$ \Psi\circ (D+d^++d^-)\circ \Psi^{-1}= D+d^++\hat{d}^-.$$
\end{lemma}

Proof of the this lemma is identical to the proof of lemma~\ref{lem: ext uniq} and we omit it.

\subsection{Push forward}\label{ssec: noneq push f}
Let us assume that $R=S/I$ where $I=(f_1,\dots,f_n)$ and functions $f_i$ form a regular sequence in $S$. In other words, the Koszul complex $K(I)=(\Lambda^\bullet W\otimes \CC[\calZ],d_K)$, $W=\CC^n$ has homologies only in degree $0$.
For $W\in S$ let $w$ be its image under the quotient map $j\colon S\rightarrow R$.
For a matrix factorization $\calF=(M,D)\in \MFs(R,w)$, $M=R^k$, we can construct
a matrix factorization
$j_*(\calF)\in \MFs(S,W)$
in two steps.

{\it Step 1:} Let us choose a lift $\tilde{D}\in \Hom (\tilde{M},\tilde{M})$, $\tilde{M}:=
S^k$ for the matrix coefficients of $D$.

{\it Step 2 :} The difference $D^2- (\tilde{D})^2$ has entries in the ideal $I$. The data $((\tilde{M}\otimes\Lambda^\bullet W,\tilde{D}),d_K)$ defines an element of $\Com_{per}^n(\calZ)$.
By our first remark this element satisfies the second condition of lemma~\ref{lem:  ext MF} and the regularity of the sequence $f_i$ implies the first condition.Thus we can apply the lemma to
obtain differentials $d^-_{ij}: \tilde{M}\otimes \Lambda^i W\to \tilde{M}\otimes \Lambda^j W$ such that
$$(\tilde{M}\otimes \Lambda^\bullet W, d_K+D+d^-)_{per}\in \MFs(\calZ,F),$$
we denote this element $j_*(\calF)$.

The first step of our construction involves a choice of $\tilde{D}$,  but the homotopy class of  $((\tilde{M}\otimes\Lambda^\bullet W,\tilde{D}),d_K)\in \Com_{per}^n(\calZ)$ is independent of this choice since
any two choices of lift $D$ differ by the  endomorphism that is homotopic to $0$. The second step produces an element that is unique up to automorphism according to lemma~\ref{lem: uniq MF+}.

To show that  the element $j_*(\calF)\in \MFs(\calZ,F)$ is well-defined,  we also need to show that we can extend $j_*$ to the space of morphisms between the objects and that
homotopy equivalent morphisms get mapped into the homotopy equivalent maps. We prove this later for the equivariant version of the push-forward and since there is a natural forgetful functor from the
equivariant to non-equivariant category the well defines of $j_*(\calF)$ follows.

\begin{remark} \label{rem: section} In the case when $j$ is a section of a  map \(\pi\) the first step of our construction could be made canonical and  we encounter this situation in our work. In more details: suppose there is map
 $\pi:\calZ\to \calZ_0$ such that $\pi\circ j=\mathrm{id}$ and $\calF\in \MFs(\calZ_0,w)$. Then the pull back $\pi^*(\calF)\in \MFs(\calZ,\pi^*(w))$ provides an extension of $\calF$ on $\calZ$ required by the
 first step of our construction.
\end{remark}

 \subsection{Chevalley-Eilenberg complex} Let us remind the general setting of Che\-val\-ley-Eilenberg homology. Suppose that $\frh$ is a Lie algebra. Chevalley-Eilenberg complex
 $\SymbolPrint{\CE_\frh}$ is the complex $(V_\bullet(\frh),d)$ with $V_p(\frh)=U(\frh)\otimes_\CC\Lambda^p \frh$ and differential $\SymbolPrint{d_{ce}}=d_1+d_2$ where:
 \def\dtheta{d}
 $$ d_1(u\otimes x_1\wedge\dots \wedge x_p)=\sum_{i=1}^p (-1)^{i+1} ux_i\otimes x_1\wedge\dots \wedge \hat{x}_i\wedge\dots\wedge x_p,$$
 $$ d_2(u\otimes x_1\wedge\dots \wedge x_p)=\sum_{i<j} (-1)^{i+j} u\otimes [x_i,x_j]\wedge x_1\wedge\dots \wedge \hat{x}_i\wedge\dots\wedge \hat{x}_j\wedge\dots \wedge x_p,$$

 For $\frh$ module $V$ we define $\SymbolPrint{\CE_\frh(V)}:=\Hom_{U(\frh)}(\CE_\frh,V)$. The homology of this complex is called Chevalley-Eilenberg cohomology of $V$: $\mathrm{H}^*_{\Lie}(\frh,V)$.
  Consider a group $H$ such that $\mathrm{Lie}(H)=\frh$, $\calZ$ is an affine variety with $H$-action and $H$-invariant function $F$.
 Suppose we  are given $\calF\in \MFs_\frh(\calZ,F)$, that is, we have a $\ZZ_2$-graded $\CC[\calZ]$ module $M$ with differential $D$ such that $D^2=F\cdot \mathrm{Id}$ and
 $D(h\cdot m)=h\cdot D(m)$ for any $h\in H$. The complex $\Hom_{U(\frh)}(\CE_\frh,M)$ is a $\ZZ\times \ZZ_2$-graded module and we collapse the first
 $\ZZ$-grading to $\ZZ_2$-grading.

 The map $D+d_{ce}$ respects $\ZZ_2$-grading and the direct computation shows that $D$ and $d_{ce}$ anti-commute.
 Hence $(D+d_{ce})^2=F$ and we define $\CE_\frh(\calF):=(\Hom_{U\frh}(\CE_\frh,M),D+d_{ce})$. The differential $d_{ce}$ commutes with the elements of the ring invariants
 $\CC[\calZ]^\frh$, thus the complex $\CE_\frh(\calF)\in \MFs(\calZ/H,F)$ where $\calZ/H=Spec(\CC[\calZ]^\frh)$.

 Later on we will need an analog for the lemma~\ref{lem: ext MF} in the case of equivariant matrix factorizations. Our proof the lemma~\ref{lem: ext MF} relies on the fact that
 every map from $\Hom^{<0}(C^+,C^+)$ is homotopically trivial and hence we can  construct a homotopy that connects this map to $0$. It appears to us that the analog of this
 statement does not hold in the equivariant setting: the homotopy connecting the map to $0$ in general is equivariant up to homotopy.

 Thus to make our proof
 of lemma~\ref{lem: ext MF} work in equivariant setting one would have to enlarge our category of equivariant complexes to the category of $L_\infty$-complexes.
 That would make our theory extremely technical and we postpone  this development for subsequent publications. Below we explain more elementary way to
 get around this problem: we incorporate Eilenberg-Chevalley  complex in our definition of equivariant matrix factorizations. We give details below and we start with some
 preliminary material.

Let us define by $\Delta$ the standard map $\frh\to \frh\otimes \frh$ defined by $x\mapsto x\otimes 1+1\otimes x$.
 Suppose $V$ and $W$ are modules over Lie algebra $\frh$ then we use notation $\SymbolPrint{V\stackon{\otimes}{\scriptstyle\Delta} W}$ for $\frh$-module which is isomorphic to $V\otimes W$ as vector space, the $\frh$-module structure being defined by  $\Delta$. Respectively, for a given $\frh$-equivariant matrix factorization $\calF=(M,D)$ we denote by $\SymbolPrint{\CE_{\frh}\stackon{\otimes}{\scriptstyle\Delta} \calF}$
 the $\frh$-equivariant matrix factorization $(CE_\frh \stackon{\otimes}{\scriptstyle\Delta}\calF, D+d_{ce})$. The $\frh$-equivariant structure on $\CE_{\frh}\stackon{\otimes}{\scriptstyle\Delta} \calF$ originates from the
 left action of $U(\frh)$ that commutes with right action on $U(\frh)$ used in the construction of $\CE_\frh$.

 In more details, for a given $\frh$-module $M$ the complex $\CE_\frh\stackon{\otimes}{\scriptstyle\Delta} M$ has terms $U(\frh)\otimes \Lambda^i(\frh)\otimes M$ with $\frh$-module structure
 $$ x\cdot(u\otimes \omega\otimes m)=x\cdot u\otimes \omega\otimes m,$$
 and the differential of the complex is   $d_{ce}=\dtheta_1+\dtheta_2$ where:
 \begin{multline*}
  \dtheta_1(u\otimes x_1\wedge\dots \wedge x_p\otimes m)=\sum_{i=1}^p (-1)^{i+1} \left(ux_i\otimes m x_1\wedge\dots \wedge \hat{x}_i\wedge\dots\wedge x_p\otimes m+\right.\\
\left. x_1\wedge\dots \wedge \hat{x}_i\wedge\dots\wedge x_p\otimes x_i\cdot m\right) ,
\end{multline*}
 $$ \dtheta_2(u\otimes x_1\wedge\dots \wedge x_p)=\sum_{i<j} (-1)^{i+j} u\otimes [x_i,x_j]\wedge x_1\wedge\dots \wedge \hat{x}_i\wedge\dots\wedge \hat{x}_j\wedge\dots \wedge x_p.$$

A slight modification of the standard fact that $\CE_\frh$ is the resolution of the trivial module implies that $\CE_\frh\stackon{\otimes}{\scriptstyle\Delta} M$ is a free resolution of the
$\frh$-module $M$.

\subsection{Equivariant matrix factorizations}
We keep previous notations: $\calZ$ is an affine variety with $H$-action and $F\in \CC[\calZ]^\frh$. Define a category $\Mod_{\frh}$ whose objects are finite-dimensional $H$-representations
$V$.
Now define the category $\SymbolPrint{\MFs_\frh}(\calZ,F)$ whose objects are triples $(M,D,\partial)$ where $M=M^0\oplus M^1$ and $M^i=\CC[\calZ]\otimes V^i$, $V^i \in \Mod_{H}$,
$\SymbolPrint{\partial}\in \oplus_{i>j} \Hom_{\CC[\calZ]}(\Lambda^i\frh\otimes M, \Lambda^j\frh\otimes M)$ and $D$ is an odd endomorphism
$D\in \Hom_{\CC[\calZ]}(M,M)$ such that
$$D^2=F,\quad \SymbolPrint{ D_{tot}}^2=F,\quad D_{tot}=D+d_{ce}+\partial,$$
where the total differential $D_{tot}$ is an endomorphism of $\CE_\frh\odel M$, that commutes with the $U(\frh)$-action.


Note that we do not impose the equivariance condition on the differential $D$ in our definition of matrix factorizations. On the other hand, if $\calF=(M,D)\in \MFs(\calZ,F)$ is a matrix factorization with
$D$ that commutes with $\frh$-action on $M$ then $(M,D,0)\in \MFs_\frh(\calZ,F)$.

There is a natural forgetful functor $\MFs_\frh(\calZ,F)\to \MFs(\calZ,F)$ that forgets about the correction differentials:
$$\calF=(M,D,\partial)\mapsto \SymbolPrint{\calF^\sharp}:=(M,D).$$

Given two $\frh$-equivariant matrix factorizations $\calF=(M,D,\partial)$, $\tilde{\calF}=(\tilde{M},\tilde{D},\tilde{\partial})$ the space of morphisms $\Hom(\calF,\tilde{\calF})$ consists of
homotopy equivalence classes of elements $\Psi\in \Hom_{\CC[\calZ]}(\CE_\frh\odel M, \CE_\frh\odel \tilde{M})$ such that $\Psi\circ D_{tot}=\tilde{D}_{tot}\circ \Psi$ and $\Psi$ commutes with
$U(\frh)$-action on $\CE_\frh\odel M$. Two map $\Psi,\Psi'\in \Hom(\calF,\tilde{\calF})$ are homotopy equivalent if
there is $h\in  \Hom_{\CC[\calZ]}(\CE_\frh\odel M,\CE\frh\odel\tilde{M})$ such that $\Psi-\Psi'=\tilde{D}_{tot}\circ h+ h\circ D_{tot}$ and $h$ commutes with $U(h)$-action on  $\CE_\frh\odel M$.

 Given two $\frh$-equivariant matrix factorizations $\calF=(M,D,\partial)\in \MFs_\frh(\calZ,F)$ and $\tilde{\calF}=(\tilde{M},\tilde{D},\tilde{\partial})\in \MFs_\frh(\calZ,\tilde{F})$
 $\calF\otimes\tilde{\calF}\in \MFs_\frh(\calZ,F+\tilde{F})$ as the equivariant matrix factorization $(M\otimes \tilde{M},D+\tilde{D},\partial+\tilde{\partial})$.


 \subsection{Extension lemmas} We have the following analog of Lemma~\ref{lem: ext MF} and respective uniqueness result.

\begin{lemma}\label{lem: ext MF eq} Let $(C_\bullet,d_\bullet^+)$ be a $\frh$-equivariant  complex
$$C_0\xleftarrow{d_1^+}C_1\xleftarrow{d_2^+}\dots\xleftarrow{d_\ell^+} C_\ell$$
of $\CC[\calZ]$ modules and $F\in \CC[\calZ]^\frh$ such that
\begin{enumerate}
\item Elements of $\xExtdpv{<0}(C_\bullet,C_\bullet)$ are homotopic to $0$
\item The element $F\in \xExtdpv{0}(C_\bullet,C_\bullet)$ is homotopic to $0$.
\end{enumerate}
Then there are homomophisms $d^-_{ij}: C_i\to C_j$, $i>j$ of $\CC[\calZ]$-modules such that
$$[(C_\bullet,d^++d^-)]_{per}\in \MFs(\calZ,F)$$
and there are homomorphisms of $\CC[\calZ]$, $\partial_{ij}^-: \Lambda^\bullet\frh\otimes C_i\to \Lambda^{<\bullet}\frh\otimes C_j$, $j>i$ such that
$$ [(C_\bullet, d^-+d^+,\partial^-)]_{per}\in \MFs_\frh(\calZ,F).$$

\end{lemma}

 We do not provide a proof of this lemma since it is a particular case of a more general statement that is proven below. The more general statement is an equivariant version of lemmas~\ref{lem: ext MF+} and \ref{lem: uniq MF+}.

Let $\Mod^{\frh}_{per}(\calZ)$ be a category  of $\ZZ_2$-graded free  $\CC[\calZ]$ modules with endomorphisms equipped with $\frh$-equivariant structure.
That is an object $\calF\in \Mod^\frh_{per}(\calZ)$ is the triple $(M,D,\partial)$, $M=M^0\oplus M^1$, $M^i=\CC[\calZ]\otimes W^i$
where $W^i\in \Mod_\frh$, $D\in \Hom_{\CC[\calZ]}(M^1,M^0)\oplus \Hom_{\CC[\calZ]}(M^0,M^1)$, $\partial=\sum_{i> j} \partial_{ij}$,
$\partial_{ij}\in \Hom_{\CC[\calZ]}(\Lambda^i\frh\otimes M, \Lambda^j\frh\otimes M)$.

The $\frh$-action on $M_i$ is defined by the Hopf formula and using this action one define differential $d_{ce}$ on $\CE_\frh\odel M$.
The space of morphism between two objects $(M,D,\partial),  (\tilde{M},\tilde{D},\tilde{\partial})\in \Mod^{\frh}_{per}(\calZ)$ consists of morphisms $\Psi\in \Hom_{\CC[\calZ]}(M,\tilde{M})$  that intertwine the endomorphisms:
$$ \Psi\circ(D+d_{ce}+\partial)=(\tilde{D}+\tilde{d}_{ce}+\tilde{\partial})\circ \Psi.$$
We do not require $\frh$-equivariance for the morphisms in the category and we do not require $\frh$-equivariance of $D$.


Similarly we define $\Com^{\frh}_{per}(\calZ)$ to be the category of (not necessarily $\frh$-equivariant) complexes of objects of $\Mod_{per}^\frh(\calZ)$, we do not impose $\frh$-equivariance conditions
on the morphisms and homotopies in $\Com^{\frh}_{per}(\calZ)$.
In this setting we have the following equivariant version lemma~\ref{lem: ext MF+},\ref{lem: uniq MF+}:

\begin{lemma}
\label{lem: ext MF+eq} Let $(\calF_\bullet,d_\bullet^+)$, $\calF_i=(M_i,D_i,\partial_i)$ be an element of $\Com^{\frh,\ell}_{per}(\calZ)$ and $F\in \CC[\calZ]^\frh$ such that
\begin{enumerate}
\item Elements of $\xExtdpv{<0}(\calF_\bullet,\calF_\bullet)$ are homotopic to $0$
\item Element $F-D^2\in \xExtdpv{0}(\calF_\bullet,\calF_\bullet)$ is homotopic to $0$
\item Element $F-(D+\partial+d_{ce})^2\in \xExtdpv{0}(\CE_\frh\odel \calF_\bullet,\CE_\frh\odel \calF_\bullet)$ is homotopic to $0$.
\end{enumerate}
Then there are maps $d^-_{ij}: \calF_i\to \calF_j$, $i>j$ of elements of $\Mod_{per}(\calZ)$ such that
$$[(\calF_\bullet,d^++D+d^-)]_{per}\in \MFs(\calZ,F)$$
and there are maps   of elements of $\Mod_{per}(\calZ)$, $\partial_{ij}^-: \Lambda^\bullet\frh\otimes\calF_i\to \Lambda^{<\bullet}\frh\otimes \calF_j$, $j>i$ such that
$$ [(\calF_\bullet, D+d^-+d^+,\partial+\partial^-)]_{per}\in \MFs_\frh(\calZ,F).$$
\end{lemma}
\begin{proof}
In the proof we would use short hand notations $\Hom^{i}$, $\xExtdpv{i}$ for space morphisms $\xHom^i(\calF_\bullet,\calF_\bullet)$ and
$\xExtdpv{i}(\calF_\bullet,\calF_\bullet)$.
We also assume that we have chosen the signs in our construction of $d_+$ in such that the equivariance condition is equivalent to
$[d_++D,d_{ce}+\partial]_+=0$.
Given linear map $A\in \xHom(\CE_\frh\odel \calF_\bullet,\CE_\frh\odel \calF_\bullet)$ we denote by $[A]_{k}$ the part of $A$ that is in
$\xHom(\CE_\frh,\CE_\frh)\otimes \xHom^i$.

As first step of our construction we apply the lemma~\ref{lem: ext MF+} to construct differentials $d^-_{ij}: \calF_i\to \calF_j$ such that
$$[(\calF_\bullet,D+d^++d^-)]_{per}\in \MFs(\calZ,F).$$

By the first step and the assumption of the lemma we have
\begin{multline*} (d^++d^-+D+\partial+d_{ce})^2-F=[d^++d^-+D,\partial +d_{ce}]_+\\+(\partial+d_{ce})^2\in \Hom(\Lambda^*\frh,\Lambda^*\frh)\otimes \Hom^{\le 0}.\end{multline*}
Since $\partial+d_{ce}\in \Hom(\Lambda^\bullet \frh,\Lambda^{<\bullet}\frh)\otimes \Hom^0$, we conclude that there is
$h^{(-1)}\in \Hom(\Lambda^\bullet \frh,\Lambda^{<\bullet}\frh)\otimes \Hom^{-1}$ such that $[h^{(-1)},d_+]_+=(\partial+d_{ce})^2$
and hence
\begin{gather*}
  d^{(-1)}:= (d^++d^-+D+\partial+d_{ce})^2-F-[h^{(-1)},d_+]_+\\
  =[d^-,\partial+d_{ce}]\in \Hom(\Lambda^*\frh,\Lambda^*\frh)\otimes \Hom^{<0}.
  \end{gather*}

 Let us define $D^{(0)}:=d^++D+\partial+d_{ce}+d^-+h^{(-1)}$ then we see that
 $$ (D^{(0)})^2-F=d^{(-1)}+[h^{(-1)},d^-+D+\partial+d_{ce}]_++(h^{(-1)})^2\in \Hom(\Lambda^\bullet\frh,\Lambda^{<\bullet})\otimes \Hom^{<0}.$$


We have proved the seed of induction for the following statement. For any $k$ there is $h^{(k)}\in \Hom(\Lambda^{\bullet}\frh,\Lambda^{<\bullet}\frh)\otimes \Hom^{-k}$ such that
$$ (D^{(k)})^2-F\in \Hom( \Lambda^\bullet\frh,\Lambda^{<\bullet}\frh) \otimes \Hom^{<-k},\quad  D^{(k)}=D^{(k-1)}+h^{(-k-1)}.$$

Suppose we have constructed $D^{(k)}$ and $h^{(-i)}$, $i\le k+1$. By induction assumption we have $ (D^{(k)})^2=
F+[(D^{(k)})^2]_{<-k}$ and we see that

\begin{multline*}
0=[D^{(k)}, (D^{(k)})^2]=[D^{(k)},[(D^{(k)})^2]_{<-k}]=[d_+, [(D^{(k)})^2)]_{-k-1}]\\+[d_+,[(D^{(k)})^2]_{<-k-1}]+[D+\partial+d_{ce}+d^-+\sum_i h^{(-i)},[(D^{(k)})^2]_{<-k}].
\end{multline*}
the first summand of the last expression is from $\Hom(\Lambda^\bullet\frh,\Lambda^{<\bullet}\frh)\otimes \Hom^{-k}$ and the last two summands are from
$\Hom(\Lambda^\bullet\frh,\Lambda^{<\bullet}\frh)\otimes \Hom^{<-k}$. Hence $[(D^{(k)})^2]_{-k-1}\in \Hom_{d_+}^{-k-1}$ and hence it is homotopic to $0$ and thus there
is \[h^{(-k-2)}\in \Hom(\Lambda^\bullet\frh,\Lambda^{<\bullet}\frh)\otimes \Hom^{-k-2}\] such that $[(D^{(k)})^2]_{-k-1}=[d_+,h^{(-k-2)}]_+$.
The differential $D^{(k+1)}=D^{(k)}+h^{(-k-2)}$ provides the induction step.

By setting $\partial^-=\sum_{i} h^{(-i)}$ we complete our proof.
\end{proof}

\begin{lemma}\label{lem: uniq MF+eq} Let $(\calF_\bullet,d_\bullet^+)$ be an element of $\Com^{\frh,\ell}(\calZ)$ and $F\in \CC[\calZ]^H$ that satisfy conditions
of the previous lemma. Suppose there two sets of differentials
  $d^-,\hat{d}^-: \calF_i\to \calF_j$,
  $\partial_{ij}^-,\hat{\partial}_{ij}: \Lambda^\bullet\frh\otimes\calF_i\to \Lambda^{<\bullet}\frh\otimes \calF_j$, $j>i$ such that
$$ [(\calF_\bullet, D+\partial+d^++d^-+d_{ce}+\partial^-)]_{per}, [(\calF_\bullet, D+\partial+d^++d_{ce}+\hat{d}^-+\hat{\partial}^-)]_{per}\in \MFs(\calZ,F)$$
then there is $\Psi=\sum_{i\le 0}\Psi_i$, $\Psi_{-i}\in \Hom^{-i}(\CE_\frh\odel \calF_\bullet,\CE_\frh\odel\calF_\bullet)$, $\Psi_0=\mathrm{Id}$, $\Psi_{-1}=0$,
such
$$ \Psi\circ (D+\partial+d^++d^-+d_{ce}+\partial^-)\circ \Psi^{-1}= D+\partial+d^++\hat{d}^-+d_{ce}+\hat{\partial}^-.$$
\end{lemma}
\begin{proof} Let $d_0=D+\partial+d_{ce}$, $d_-=d^-+\partial^-$ and $D_{tot}=d^++d_0+d_-$, $\hat{d}_-=\hat{\partial}^-+\hat{d}^-$, $\hat{D}_{tot}=d^++d_0+\hat{d}_-$.
Suppose we constructed  $\Psi_{-i}\in \Hom^{-i}(\CE_\frh\odel \calF_\bullet,\CE_\frh\odel\calF_\bullet)$, $k>i\ge 0$ such that  $\Psi^{(k)}:=\sum_{0\le i<k}\Psi_{-i}$ satisfies
$$ \hat{D}_{tot}-D^{(k)}_{tot}\in \Hom^{\ge -k+1},\quad  D^{(k)}_{tot}=\Psi^{(k)}\circ D_{tot}\circ (\Psi^{(k)})^{-1}.$$

Let us use notation $\Delta=-\hat{D}_{tot}+D^{(k)}_{tot}$ then we have
$$ 0=(\hat{D}_{tot}+\Delta)^2-(\hat{D}_{tot})^2=[[\Delta]_{-k+1},d_+]_+  + [[\Delta]_{\le -k},d_+]_++[\Delta,d_0+\hat{d}_-]_++\Delta^2.$$
Since the first summand is from $\Hom^{-k+2}$ and the rest of summands are from $\Hom^{\le -k+1}$ we get that $[[\Delta]_{-k+1},d_+]_+=0$.
Thus there is $\Psi_{-k}\in \Hom^{-k}$ such that $[\Delta]_{-k+1}=[d^+,\Psi_{-k}]$. The element $\Psi^{(k+1)}:=\Psi^{(k)}+\Psi_{-k}$ prove our inductive statement.

\end{proof}

\subsection{Lifting lemma for morphisms} To define equivariant push-forward below we need to state few extra facts about our lifting procedure.

\begin{lemma}\label{lem: ext mor} Suppose we are
are given $(\calF_\bullet,d_\bullet^+)$, $\calF_i=(M_i,D_i)$  and $(\tilde{\calF}_\bullet,\tilde{d}_\bullet^+)$, $\tilde{\calF}_i=(\tilde{M}_i,\tilde{D}_i)$ be an element of $\Com^{\frh,\ell}_{per}(\calZ)$ and $F\in \CC[\calZ]^\frh$ such that
they satisfy conditions of lemma~\ref{lem: ext MF+eq} and $\partial,d^-$, $\tilde{\partial},\tilde{d}^-$ are as in conclusion of lemma~\ref{lem: ext MF+eq}. Let us also assume that
 elements of $\mathrm{\xExt}^{<0}_{d^+,\tilde{d}^+}(\calF_\bullet,\tilde{\calF}_\bullet)$ are homotopic to $0$

Suppose also that there is a 
 map $\Psi\in \Hom^{k}(\CE_\frh\odel \calF_\bullet,\CE_\frh\odel\tilde{\calF}_\bullet)$, $k\le  0$ such that
 $\tilde{d}^+\circ \Psi=-\Psi\circ d^+$ and
$(\tilde{D}+\tilde{d}_{ce}+\partial)\circ \Psi-\Psi\circ (D+d_{ce}+\tilde{\partial})$ is $d^+,\tilde{d}^+$ homotopic to $0$.
Then there are $\frh$-equivariant elements $\Psi_i  \in \Hom^{k-i}(\CE_\frh\odel \calF_\bullet,\CE_\frh\odel\tilde{\calF}_\bullet)$, $i>0$ such that
$$\tilde{D}_{tot}\circ(\Psi+\sum_i \Psi_i)=(\Psi+\sum_i \Psi_i)\circ D_{tot},$$ $$ \tilde{D}_{tot}:=\tilde{d}^++\tilde{D}+\partial+\tilde{d}^-+\tilde{d}_{ce}+\tilde{\partial}^-,\quad
D_{tot}:= d^++D+\partial+d_{ce}+\partial^-,$$

The elements $\Psi_i$ are unique up-to the homotopy: if \[\Psi'_i \in \Hom^{k-i}(\CE_\frh\odel \calF_\bullet,\CE_\frh\odel\tilde{\calF}_\bullet)\] is another collection of elements
satisfying previous equation then there is $h$ such that
$$\sum_i \Psi_i-\Psi'_i=\tilde{D}_{tot}\circ h- h\circ D_{tot}$$
\end{lemma}
\begin{proof}
Let us introduce short-hand notations $d_0=D+\partial+d_{ce}$ and $\tilde{d}_0=\tilde{D}+\tilde{\partial}+\tilde{d}_{ce}$, $d_-=d^-+\partial^-$, $\tilde{d}^-=\tilde{d}^-+\tilde{\partial}^-$. We also use notations
$\Hom^i$, $\xExtdpv{i}$ as in the previous lemmas.

First let us notice that
$$ \tilde{D}_{tot}\circ \Psi-\Psi\circ D_{tot}=\left( \tilde{d}_0\circ \Psi-\Psi\circ d_0\right)+\left(\tilde{d}_-\circ \Psi-\Psi\circ d_-\right).$$
The first summand is homotopic to $0$ by the theorem assumption and the second term is from $\Hom^{<k}$ and
$$ [d^+,\tilde{d}_-\circ \Psi-\Psi\circ d_-]= [d^+,\tilde{d}_-]_+\circ \Psi-\Psi\circ [d^+,d_-]_+=(F-\tilde{d}_0^2)\circ\Psi-\Psi\circ (F-d_0^2)=0.$$

Thus there are maps $\Psi_1\in \Hom^{k-1}, \Psi_2\in \Hom^{k-2}$ such that $[d^+,\Psi_2]_+=[\tilde{d}_-\circ \Psi-\Psi\circ d_-]_{k-1}$ and
$[d^+,\Psi_1]_+= \tilde{d}_0\circ \Psi-\Psi\circ d_0$. Hence $\Psi^{(2)}=\Psi+\Psi_1+\Psi_2$ such that
$$ \tilde{D}_{tot}\circ \Psi^{(2)}-\Psi^{(2)}\circ D_{tot}\in \Hom^{\le k-2}.$$
That is the seed to the inductive construction of $\Psi^{(m)}=\Psi+\sum_{i=1}^m \Psi_m$ such that
$$ \tilde{D}_{tot}\circ \Psi^{(m)}-\Psi^{(m)}\circ D_{tot}\in \Hom^{\le k-m}.$$
To prove the step of induction we observe that
$$\tilde{D}_{tot}\circ\left(  \tilde{D}_{tot}\circ \Psi^{(m)}-\Psi^{(m)}\circ D_{tot}\right)-\left(  \tilde{D}_{tot}\circ \Psi^{(m)}-\Psi^{(m)}\circ D_{tot}\right)\circ D_{tot}=0 $$
because $D_{tot}^2=F$ and $\tilde{D}_{tot}=F$. If we denote by $R$ the expression $\tilde{D}_{tot}\circ \Psi^{(m)}-\Psi^{(m)}\circ D_{tot}$ then we get:
\begin{multline*}
  0=\tilde{D}_{tot}\circ R-R\circ D_{tot}=[d^+, [R]_{k-m}]+\left((\tilde{d}_0+\tilde{d}_-)\circ [R]_{k-m}\right.\\
  \left.-
[R]_{k-m}\circ(d_0+d_-)\right)+\left(\tilde{D}_{tot}\circ [R]_{<k-m}-[R]_{<k-m}\circ D_{tot}\right).
\end{multline*}
Since the first summand of the last expression is from $\Hom^{k-m+1}$ and the last two terms are from $\Hom^{\le k+m}$ we conclude that
$[d^+,[R]_{k-m}]=0$. Thus $[R]_{k-m}\in \xExtdpv{k-m}$ and there is $\Psi_{m+1}\in \Hom^{k-m-1}$ such that $[R]_{k-m}=[\Psi_{m+1},d^+]_+$.
Hence $\Psi^{(m+1)}=\Psi^{(m)}+\Psi_{m+1}$ proves the induction step.

The last part of the lemma follows from the following statement. Suppose $\Psi\in \Hom^{<k}$, $k\le 0$ is such that $\tilde{D}_{tot}\circ \Psi=\Psi\circ D_{tot}$ then there
is $h\in \Hom^{<k+1}$ such that $\Psi'=\Psi-(\tilde{D}_{tot}\circ h+ h\circ D_{tot})\in \Hom^{<k-1} $. Indeed, since
$$\tilde{D}_{tot}\Psi'-\Psi'\circ D_{tot}=(\tilde{D}_{tot})^2\circ h-h\circ (D_{tot})^2=0,$$
we can run inductive argument and construct $g$ such that $\Psi=\tilde{D}_{tot}\circ g+ g\circ D_{tot}$.

The statement itself follows from the computation below:
\begin{multline*}
  0=\tilde{D}_{tot}\circ \Psi-\Psi\circ D_{tot}=[d^+,[\Psi]_{k-1}]+\left( (\tilde{d}_0+\tilde{d}_-)\circ[\Psi]_{k-1}]\right.
  \\\left.-[\Psi]_{<k-1}\circ (d_0+d_-)\right)
+\left(\tilde{D}_{tot}\circ [\Psi]_{<k-1}-[\Psi]_{<k-1}\circ D_{tot}\right)
\end{multline*}
since the first summand is from $\Hom^k$ and the other summands are from $\Hom^{<k}$ and thus $[d^+,[\Psi]_{k-1}]=0$ and
$[\Psi]_{k-1}$ is homotopically trivial with respect to $d_+$. It is easy to check that $h$, such that $[d^+,h]_-=[\Psi]_{k-1}$ provides a proof of the statement.
\end{proof}

\subsection{Equivariant push-forward} \label{ssec: eq push f} Let $S=\CC[\calZ]$ and Lie algebra $\frh$ acts on $\calZ$.
Let us assume that $R=S/I$ where $I=(f_1,\dots,f_n)$ and functions $f_i$ form a regular sequence in $S$ and $I$ is $\frh$-invariant. Let \(\calZ_0=\mathrm{Spec}(R)\).
The Koszul complex $K(I)=(\Lambda^\bullet W\otimes \CC[\calZ],d_K)$, $W=\CC^n$ is $\frh$-equivariant and has homologies only in degree $0$.
Let $w\in R^\frh$ and $W\in S^\frh$ such that $W-w\in I $. Let us choose $\calF=(M,D,\partial)\in \MFs_\frh(\calZ_0,w)$, $M=S^k$ then we can construct
a matrix factorization
${}^\frh j_*(\calF)\in \MFs_\frh(\calZ,W)$
in two steps.

{\it Step 1:}  We apply method of subsection~\ref{ssec: noneq push f} to construct $j_*(\calF)=(\tilde{M},\tilde{D}+d_K+d^-)\in \MFs(\calZ,W)$.

{\it Step 2 :} Use lemma~\ref{lem: ext MF+eq} to construct $\partial^-$ such that $\tilde{\calF}:=(\tilde{M}, \tilde{D}+d_K+d^-,\partial^-+\partial)\in \MFs_\frh(\calZ,W)$.
 The element $\calF$ is defined up to isomorphism according to lemma~\ref{lem: uniq MF+eq} and we set $ j^*(\calF):=\tilde{\calF}$. To complete our construction
 of the functor:
 \begin{equation}\label{eq: push f eq}
   j_*: \MFs_\frh(\calZ_0,w)\to \MFs_\frh(\calZ,W).
 \end{equation}
 we still need to discuss the action of the functor on the space of morphisms and homotopies between the morphisms and that is done below.

 Let $\calF=(M,D,\partial)$, $\tilde{\calF}=(\tilde{M},\tilde{D},\partial)\in \MFs_\frh (\calZ_0,w)$ and $\psi$ is the morphism between $\calF$ and $\tilde{\calF}$ in the category $\MFs_\frh(\calZ_0,w)$.
 Let us constrict $\calF'=(M',D',\partial')=j_*(\calF)$ and $\tilde{\calF}'=(\tilde{M}',\tilde{M}',\tilde{\partial}')=  j_*(\tilde{\calF})$ as above.
 The lemma~\ref{lem: ext mor} implies that we extend it to the morphism $\Psi$ in $\MFs_\frh(\calZ,W)$. Indeed if we choose a extension $\tilde{\psi}\in \Hom_{\CC[\calZ]}(M,\tilde{M})$,
  of $\psi$ then $\tilde{\psi}$ satisfies the assumptions of the lemma~\ref{lem: ext mor} and the lift $\Psi=\tilde{\psi}+\sum_{i>0}\tilde{\psi}_i$ exists. Moreover the second part of the lemma implies that
  the lift is unique up to the  homotopy thus we have the map:
  $$  j_*:  \Hom_{\MFs_\frh(\calZ_0,w)}(\calF,\tilde{\calF})\to    \Hom_{\MFs_\frh(\calZ_0,w)}(\calF',\tilde{\calF}').$$

  Thus completed our construction of the functor of homotopic categories (\ref{eq: push f eq}).




\section{Convolution algebra: definition}\label{sec:convolution}
In this section we define convolution product on  categories $\MFs_{B^2}(\calX_2,W)$ and $\MFs_{B^2}(\calXr_2,\Wr)$ and explain how
these two convolutions are related. Thus we construct the functor $$\Phi: \MFs_{B^2}(\calXr_2,\Wr)\to \MFs_{B^2}(\calX_2,W)$$ which is an equivariant
version of Knorrer periodicity functor \cite{Kn}.

\subsection{Category of $B^2$-equivariant matrix factorizations}
The Borel subgroup $B$ is the semi-direct product $B=T\ltimes U$ where $U$ is the unipotent group with $\Lie(U)=\frn$.
Suppose we are given a variety $\calZ$ with $B^2$-action and $F\in \CC[\calZ]^{B^2}$.
Then we define $\MFs_{B^2}(\calZ,F)$ as a subcategory of $\MFs_{\frn^2}(\calZ,F)$ that has objects and morphisms
which are strongly $T^2$-equivariant. In more details, an object $\calF\in \SymbolPrint{\MFs_{B^2}}(\calZ,F)$ is the following collection of data:
$$ \calF=(M,D,\partial_l,\partial_r),$$
where $M=\CC[\calZ]\otimes W$, $W\in \Mod_{\frn^2}^{filt}$ is a $\ZZ_2$-graded $B^2$-module with the strict filtration; $D$ is $T^2$-equivariant odd $\CC[\calZ]$-equivariant endomorphism of $M$;
$\partial_l,\partial_r\in \Hom(\Lambda^\bullet (\frn),\Lambda^{<\bullet}\frn)\otimes \Hom_{\CC[\calZ]}(M_\bullet,M_{>\bullet})$ are $T^2$-equivariant maps such that
$$D^2=F,\quad (D_{tot})^2=F,  \quad D_{tot}=D+d_{ce}+\partial_l+\partial_r,$$ where $D_{tot}$ is $\frn^2$-equivariant endomorphism of $\CE_{\frn^2}\odel M$.
Here we assume that
$$ (\partial_l+\partial_r)(\omega_l\otimes \omega_r\otimes m)=\mathrm{sw}(\omega_r\otimes\partial_l(\omega_l\otimes m))+\omega_l\otimes\partial_r(\omega_r\otimes m),$$
where $\frn^2=\frn_l\oplus \frn_r$ and $\omega_l\in \frn_l$, $\omega_r\in \frn_r$ and $\mathrm{sw}: \Lambda^\bullet\frn\otimes\Lambda^\bullet\frn\otimes M\to \Lambda^\bullet\frn \otimes \Lambda^\bullet \frn \otimes M$ is
the linear map that switches the first two factors.

For a matrix factorization $\tilde{\calF}=(\tilde{M},\tilde{D},\partial)$, the  morphism space $\Hom_{\MFs_{B^2}}(\calF,\tilde{F})$ consists of homotopy equivalence classes of
$T^2$-equivariant maps \[\Psi\in \Hom_{\CC[\calZ]}(\CE_{\frh^2}\odel M,\CE_{\frh^2}\odel\tilde{M})\] such that  $\Psi\circ D_{tot}=\tilde{D}_{tot}\circ \Psi$ and $\Psi$-commutes with
$\frh^2$ action on $\CE_{\frh^2}\odel M$ and $\CE_{\frh^2}\odel \tilde{M}$

More generally, for a variety $\calZ$ with $B^\ell$-action and $F\in \CC[\calZ]^{B^\ell}$ there is analogously defined category $\MFs_{B^\ell}(\calZ,F)$ with morphism and objects being $T^\ell$-equivariant in the
strong sense. Details of the definition of $\MFs_{B^\ell}(\calZ,F)$ are basically identical to the previously given definition.

\def\MFGBv#1{ \MF_{G\times B^{#1}}}
\def\MFGBt{ \MFGBv{2}}
\def\MFGBh{ \MFGBv{3}}
\def\MFBv#1{ \MF_{B^{#1}}}
\def\MFBt{ \MFBv{2}}
\def\MFBh{ \MFBv{3}}
\def\cXv#1{ \calX_{#1}}
\def\cXt{ \cXv{2}}
\def\cXh{ \cXv{3}}

\subsection{Potential and main category of matrix factorizations} In this section we introduce an associative convolution operation in the category of matrix factorizations on the space $\calX_2$. In subsequent sections we construct a homomorphism from the braid group into this
convolution algebra. This allows us to associate a particular matrix factorization to a braid and in the section~\ref{sec: link inv 2} we explain how one can extract knot homology out of this matrix factorization.

We choose coordinates on  $\calX_2=\frg\times (G\times \frb)^2$ as $(X,g_1,Y_1,g_2,Y_2)$ and consider a $G\times B^2$-invariant potential on $\calX_2$:
\begin{equation}\label{eq: potential}
 \SymbolPrint{W}(X,g_1,Y_1,g_2,Y_2)=\Tr(X(\Ad_{g_1}(Y_1)-\Ad_{g_2}(Y_2))).
\end{equation}

The torus $T_{sc}$ acts on $\calX_2$ and $\dgTsc W = t^2$.
Consider the category of 
$\SymbolPrint{\MF_{B^2}(\calX_2,W)}$  
of $\Tsc$-graded $ B^2$-equivariant matrix factorizations whose differentials have degree $\dgTsc D = t$.
The morphisms  and homotopies in the
category are assumed to be $T_{sc}$-invariant. 

Our main category is the subcategory $\MF_{ B^2}(\calX_2,W)^G$ of
$\MF_{B^2}(\calX_2,W)$ that consists of the complexes that are $G$-equivariant in the strong sense. The $G$-action on $\calX_2$ is
free and as we show in subsection \ref{ssec: Gslice} the category $\MF_{ B^2}(\calX_2,W)^G$ is equivalent to 
the category of $B^2$-equivariant matrix factorizations on $\calX_2/G$. To simplify notations, from now on we use notation 
$\MF_{B^2}(\calX_2,W)$ for $\MF_{B^2}(\calX_2,W)^G$.


We also use notation $\SymbolPrint{\mathbf{q}^k\mathbf{t}^m}$ for shift functor for the weights of the action of $T_{sc}$:  if $V$ is a vector space with $\CC^*_t\times\CC^*_q$ action then $\mathbf{q}^k\mathbf{t}^m\cdot V$
is the same vector space with $q$- and $t$-degrees shifted by $k$ and $m$: for $v\in V$ and $(\lambda,\mu)\in\CC^*_t\times\CC^*_q$ we define $(\lambda,\mu)\cdot(\mathbf{q}^k\mathbf{t}^m v) = \lambda^k\mu^m \mathbf{q}^k\mathbf{t}^m((\lambda,\mu)\cdot v)$.
Respectively we use notation $\mathbf{q}^k\mathbf{t}^m\cdot \calF\in \MF_{B^2}(\calX_2,W)$ for the matrix factorization with the same differentials as $\calF\in \MF_{B^2}(\calX_2,W)$ but the the twisted $T_{sc}$-equivariant structure of the underlying module.

Let us make a small clarification about the  $t$-grading in our categories. Let $F$ be 
a potential on the space $\calZ$ and $\calZ$ has $T_{sc}=\CC^*_t\times \CC^*_q$ action such that $F$ has  the weight $t^2$. Then we define $\MF(\calZ,F)$ to the 
homotopic category whose objects are  $2$-quasi-periodic complexes:
$$\dots \xrightarrow{D_{-1}}M_{0}\xrightarrow{D_0}M_1\xrightarrow{D_1}M_2\xrightarrow{D_2}\dots,$$
such that $D_{i+2}=D_i$ and $M_{i+2}=\mathbf{t}^2 \cdot M_i$ and $D_i$ are of 
degree $t$. In particular, if $F=0$ then for 
$\calF=(M_\bullet,D_\bullet)\in \MF(\calZ,F)$ we have well defined $T_{sc}$-equivariant modules:
$$ \SymbolPrint{\textup{H}^{even}}(\calF):=\textup{H}^0(M_\bullet,D_\bullet),\quad 
\SymbolPrint{\textup{H}^{odd}}(\calF):=\textup{H}^1(M_\bullet,D_\bullet).$$


 \subsection{Convolution}
There are three natural projection maps $\SymbolPrint{\pi_{12},\pi_{23},\pi_{13}}:\calX_3\to\calX_2$ and $\pi_{12}^*(W)+\pi_{23}^*(W)=\pi_{13}^*(W)$. If we extend the action of $G\times B^2$ on $\calX_2$ to the action
of $G\times B^3$ by natural projectors $\pi_{ij}:G\times B^3\to G\times B^2$ then
the maps $\pi_{ij}$ become $G\times B^3$-equivariant and hence
 for any $\calF,\calG\in \MF_{B^2}(\calX_2,W)$ we can define an object
 $$\pi_{12}^*(\calF)\otimes\pi_{23}^*(\calG)\in \MF_{B^3}(\calX_3,\pi_{13}^*(W)).$$

 Naively we could define the convolution between $\calF$ and $\calG$ as a quotient of the tensor product by the action of the second copy $B^{(2)}$ of $B$ inside $G\times B^3$.
 We define a derived quotient with the help of the Chevalley-Eilenberg complex, we also explain how to define the convolution in a bigger category $\MF_{B^2}(\calZ,W)$.


 Let $\calF,\calG\in \MF_{B^2}(\calX_2,W)$, $\calF=(M,D,\partial_l,\partial_r)$, $\calG=(\tilde{M},\tilde{D},\tilde{\partial}_l,\tilde{\partial}_r)$.
 The tensor product $M\otimes\tilde{M}$ is naturally a module over $B^3$ with the action $(b_1,b_2,b_3)(m\otimes\tilde{m})=(b_1,b_2)\cdot m\otimes (b_2,b_3)\cdot \tilde{m}$, hence
 we can define a tensor product of pull backs $\pi_{12}^*(\calF)$ and $\pi_{23}^*(\calG)$ as an object of $\MF_{B^3}(\calX_3,\pi_{13}^*(W))$:
 $$ \pi_{12}^*(\calF)\otimes \pi_{23}^*(\calG)=(\pi_{12}^*(M)\otimes\pi_{23}^*(\tilde{M}),D+\tilde{D},\partial_l,\partial_r+\tilde{\partial}_l,\tilde{\partial}_r).$$

  From here on  we use notation $\SymbolPrint{\frn^{(1)},\frn^{(2)},\frn^{(3)}}$ for the first, second and third copy of $\frn$ on $\frn^3$; respectively for given
  $\frn^3$ module $M$ we denote by $d^{(i)}_{ce}$ the  Chevalley-Eilenberg differential on $\CE_{\frn^3}\odel M$ corresponding $i$-th copy of $\frn$.

  The complex $\CE_\frn$ is
 $T$-equivariant  hence $T^{(2)}$-invariant part  $\CE_{\frn^{(2)}}(\pi_{12}^*(\calF)\otimes\pi_{23}^*(\calG))^{T^{(2)}}$ is the matrix factorization
 from $\MF_{B^2}(\calX_3/B_{(2)},\pi_{13}^*(W))$. In more details:
 $\CE_{\frn^{(2)}}(\pi_{12}^*(\calF)\otimes_B\pi_{23}^*(\calG))^{T^{(2)}}$ is the matrix factorization with underlying module
 $$\CE_{\frn^{(2)}}(\pi_{12}^*(M)\otimes\pi_{23}^*(\tilde{M}))^{T^{(2)}}:=\Hom_{\frn^{(2)}}(\CE_{\frn^{(2)}},\CE_{\frn^{(2)}}\odel \left( \pi_{12}^*(M)\otimes\pi_{23}^*(\tilde{M}) \right))^{T^{(2)}},$$
 the matrix factorization differential $D^{(2)}:=d^{(2)}_{ce}+\partial_r+\tilde{\partial}_l+D+\tilde{D}$ and $\frn^{(1)}\times\frn^{(2)}$ equivariant structure differentials are
 $\partial_l,\tilde{\partial}_r.$

  In particular   $(\CE_{\frn^{(2)}}(\pi_{12}^*(M)\otimes\pi_{23}^*(\tilde{M}))^{T^{(2)}}, D^{(2)})$ is
 a complex over the ring $\pi_{13}^*(\CC[\calX_2])\subset\CC[\calX_3]^{\frb_{(2)}}$ and as such it is an object of
 $\MF_{B^2}(\calX,W)$ which we denote $\pi_{13*}(\CE_{\frn_{(2)}}(\pi_{12}^*(\calF)\otimes\pi_{23}^*(\calG))^T_{(2)}$. The last complex
 provides us a bilinear operation on $\MF_{B^2}(\calX_2,W)$:
 \begin{multline}\label{eq: convolution}
\calF\SymbolPrint{\star}\calG:=\pi_{13*}(\CE_{\frn^{(2)}}(\pi_{12}^*(\calF)\otimes\pi_{23}^*(\calG))^{T^{(2)}})=\\(\CE_{\frn^{(2)}}(\pi_{12}^*(M)\otimes\pi_{23}^*(\tilde{M}))^{T^{(2)}},D^{(2)},\partial_l,\tilde{\partial}_r).
 \end{multline}
 The proof of the following proposition is standard: it literally repeats the proof of the associativity of the convolution from the book \cite{CG}:
 \begin{proposition} The binary operation $\star$ defines a strictures of associative algebra on $\MF_{B^2}(\calX_2,W)$.
 \end{proposition}

\subsection{Convolution of the $G$-slice}\label{ssec: Gslice}
As we mentioned earlier,  $\MF_{B^2}(\calX_2,W)^G$ signifies $ B^2$-equivariant matrix factorizations from $\MFs_{B^2}(\calX_2,W)$ that \(G\)-invariant.
This description becomes simpler when
we introduce a smaller space $\SymbolPrint{\calX_2^\circ}:=\frg\times G\times \frb^2$ and the map
$\SymbolPrint{p_\circ}: \calX_2\to \calX_2^\circ$ defined by \[p_\circ(X,g_1,Y_1,g_2,Y_2)=(\Ad_{g_1}^{-1}(X),g_{12},Y_1,\Ad_{g_{12}}(Y_2)_{++}).\] This map is $G$-equivariant with the
trivial $G$-action on the target. The space $\calX_2^\circ$ is a slice to $G$-action on $\calX_2$.

If we define the function $W$ on $\calX^\circ_2$ by
$$ W(X,g,Y_1,Y_2)=\Tr(X(Y_1-\Ad_g(Y_2))) $$
Then the pull back $(p_\circ)^*:\MF_{B^2}(\calX_2^\circ,W)\to \MFs_{B^2}(\calX_2,W)^G$ provides an equivalence of categories.

Indeed, we have an  \(G\times B^2\)-equivariant isomorphism \(\SymbolPrint{\Psi_G}:\calX_2\to G\times \calX_2^\circ\):
\[\Psi_G(X,g_1,Y_1,g_2,Y_2)=(g_1,p_\circ(X,g_1,Y_1,g_2,Y_2)).\]
The group \(G\)-acts on the first factor of \(G\times\calX_2^\circ\) by left multiplication and acts trivially on the
second factor hence the pull-back functor \(pr_2^*: \MF_{B^2}(G\times\calX_2^\circ,W)^G\to \MF_{B^2}(\calX_2^\circ,W)\) is
an isomorphism. Thus the equivalence follows because \(p_\circ=\Psi_G\circ pr_2\).
Thus we can work with the category
$\MF_{B^2}(\calX_2^\circ,W)$ instead of $\MF_{B^2}(\calX_2,W)^{G}$.


Define $\calX_3^\circ:=\frg\times G^2\times \frn^3$.
The convolution for the objects of $\MFBt(\calX^\circ_2,W)$ is defined with modified maps $\SymbolPrint{\pi_{ij}^\circ}:\calX_3^\circ\to\calX_2^\circ$:
$$\pi^\circ_{12}(X,g_{12},g_{23},Y_1,Y_2,Y_3)=(X,g_{12},Y_1,Y_2),$$
$$\pi^\circ_{23}(X,g_{12},g_{23},Y_1,Y_2,Y_3)=(\Ad_{g_{12}}^{-1}(X),g_{23},Y_2,Y_3),$$
$$\pi^\circ_{13}(X,g_{12},g_{23},Y_1,Y_2,Y_3)=(\Ad_{g_{13}}^{-1}(X),g_{13},Y_1,Y_3),$$
 where $g_{13}=g_{12} g_{23}$.
 The convolution is then defined by
 $$ \calG\star \calF=\pi^\circ_{13*}(\CE_{\frn^{(2)}}( \pi^{\circ,*}_{12}(\calF)\otimes_B\pi^{\circ,*}_{23}(\calG)))^{T^{(2)}}.$$
Obviously, the pull-back map $p_\circ^*$ is an isomorphism of the convolution algebras.

%% file: part3_gl.tex
\section{Knorrer reduction}\label{sec: Knor}
\subsection{Construction of the Knorrer functor} In this section we introduce a slightly smaller space $\calXr_2$ which could be used to model our convolution algebra. The advantage of this perspective is that the generators of the
braid group have a simpler description.

Let $\calXr_2:=\frb\times G\times \frn$ and let us define the following potential on it
\begin{equation}\label{eq: red potential}
\SymbolPrint{\Wr}(X,g,Y)=\Tr(X\Ad_g(Y)).
\end{equation}

This potential is $B\times B$ equivariant with respect to the following action
$$(b_1,b_2)\cdot (X,g,Y)=(\Ad_{b_1}(X) , b_1 g b_2^{-1},\Ad_{b_2}(Y)).$$
The space $\calX_2$ projects  to the reduced space $\calXr_2$ by the projection $\bar{\pi}$:
$$\bar{\pi}(X,g_1,Y_1,g_2,Y_2)=(\Ad_{g_1}^{-1}(X)_+,g_1^{-1}g_2,Y_2).$$
The map is $B$-equivariant with respect to the second copy of $B$ in $B^2$.

The category of matrix factorizations $\MF_{B^2}(\calXr_2,\Wr)$ is equivalent to our main category
\begin{proposition}\label{prop: Knorrer} There is an embedding of categories
$$\SymbolPrint{\Phi}: \MF_{B^2}(\calXr_2,\Wr)\to \MF_{B^2}(\calX_2,W).$$
\end{proposition}
\begin{proof} Using  the elementary properties of the trace we obtain:
$$\Tr(\Ad^{-1}_{g_1}(X)Y_1-\Ad^{-1}_{g_{2}}(X)Y_{2})=\Tr(\tilde{X}_{1+}+\Ad_{g_{12}}(Y_2))-\Tr(\tilde{X}_{1--}(Y_{1}-\Ad_{g_{12}}(Y_2)_{++})),$$
where $\SymbolPrint{ g_{12}}=g_1^{-1}g_2,\quad \SymbolPrint{\tilde{X}_{i}}=\Ad^{-1}_{g_{i}}(X).$ Let us denote the first summand in the last formula as $\Wr$.

The last summand in the last formula is quadratic hence the non-equivariant version of the statement is exactly theorem 3.1 of \cite{Kn}.
The above mentioned theorem from \cite{Kn} provides the functor \(H: \textup{MF}(\CC^n,f)\to \textup{MF}(\CC^{n}\times\CC^2,f+vu)\). It is shown in \cite{Or} that
\(\CC^n\) in the theorem can be replaced by any affine variety and we use the result of \cite{Or} here since
\begin{equation}\label{eq:prod X}
  \calX_2=\calXr_2\times \CC^{2N}
  \end{equation}
with \(\tilde{X}_{1--}\) and \(Y_1-\Ad_{g_{12}}(Y_2)_{++}\) providing coordinates along \(\CC^{2N}\).
In particular, rewriting the functor
\(H\) in our notations  we obtain the functor $\Phi$:
\begin{equation}\label{eq: Knorrer functor}
 \Phi'(\mathcal{F})=\bar{\pi}^*(\mathcal{F})\otimes \begin{bmatrix} (\tilde{X}_{1--} )& (Y_{1}-\Ad_{g_{12}}(Y_2)_{++})\end{bmatrix},
 \end{equation}
where the last term is the Koszul matrix factorization and $(M)$ stands for the column composed of the non-trivial entries of $M$.

The equation (\ref{eq: Knorrer functor}) defines a functor $\Phi': \MF(\calXr_2,\Wr)\to \MF(\calX_2,W)$ which gives a non-equivariant version of our embedding.
Unfortunately, the map $\bar{\pi}$ is not $B^2$-equivariant hence we can not use the pull back along this map to define
an equivariant version of the functor $\Phi$. To get around this problem we give another interpretation of the functor $\Phi$ that extends to the equivariant setting.

In our construction we use the space $\SymbolPrint{\widetilde{\calX}_2}=\frb\times G\times \frn\times G\times \frn$ that projects on $\calXr_2$:
$$\SymbolPrint{\pi_y}: \widetilde{\calX}_2\to \calXr_2,\quad \pi_y(\Ad_{g_1}^{-1}(X)_+,g_1,Y_1,g_2,Y_2)=(X,g_1^{-1}g_2,Y_2),$$
and is isomorphic to the sub variety of  $\calX_2$  defined by the ideal $I_{--}$ generated by the entries of the matrix $\tilde{X}_{1--}$:
$$ \SymbolPrint{j^x}: \widetilde{\calX}_2\to \calX_2.$$

Since we have the product decomposition (\ref{eq:prod X}) and the entries of \(\tilde{X}_{1--}\)  are coordinates along \(\CC^{2N}\),
the generators of the ideal $I_{--}$ form a regular sequence and $\pi_y^*(\Wr)|_{\widetilde{\calX}_2}=W|_{\widetilde{\calX_2}}$.
Hence, according to the subsection~\ref{ssec: eq push f}, the functor \[j^x_*: \MF_{B^2}(\widetilde{\calX}_2,\pi_y^*(\Wr))\to \MF_{B^2}(\calX_2,W)\] is
well defined.
Now let us observe that $\Phi'(\calF)=j^x_*\circ \pi_y^*(\calF)^\sharp$. Indeed, $\bar{\pi}^*(\calF)\otimes [(\tilde{X}_{1,--}),0]$ is an element of 
$\textup{Com}_{per}^\ell(\calX_2)$ and the elements $\tilde{X}_{1,--}$ form a regular sequence, hence $\Phi'(\calF)$ is a unique extension of
$\bar{\pi}^*(\calF)\otimes [(\tilde{X}_{1,--}),0]$ to an element of $\MF(\calX_2,W)$ from the lemma~\ref{lem: ext MF+eq}.
On the other hand since $\bar{\pi}|_{\widetilde{\calX_2}}=\pi_y$, thus $\bar{\pi}^*(\calF)$ is an extension of $\pi_y^*(\calF)$ on $\calX_2$ (as required by 
the first step of our construction of $j_*^x$ from the  section~\ref{ssec: noneq push f}). Thus $j_*^x\circ\pi_y^*(\calF)$ is also an extension of
$\bar{\pi}^*(\calF)\otimes [\tilde{X}_{1,--},0]$ to an element of $\MF(\calX_2,W)$. Hence by the uniqueness result we have $\Phi'(\calF)=j_*^x\circ\pi_y^*(\calF)^\sharp$.



There is a unique $B^2$ equivariant structure on $\widetilde{\calX}_2$ that makes  the map $\pi_y$  $B^2$-equivariant. Let's fix this $B^2$-equivariant structure on $\widetilde{\calX}_2$.
Moreover, the ideal $I_{--}$ is also $B^2$-invariant. Let $\hat{Y}:=Y_{1}-\Ad_{g_{12}}(Y_2)$.
Then the variables $\hat{Y}_{++},  Y_2, \tilde{X}_1, g_{12}, g_1$ form a coordinate system on $\calX_2$ and are $B^{(2)}$-equivariant for trivial reasons.
On the other hand, the action of the Lie algebra of the unipotent part of $B^{(1)}$ is given
by
\[ E_{ij}\cdot \tilde{X}_{1,kl}=\delta_{jl} \tilde{X}_{1,ki}-\delta_{ki} \tilde{X}_{1,jl},\quad  E_{ij}\cdot \hat{Y}_{kl}=\delta_{jl} \Ad_{g_{12}}(Y_2)_{ki}-\delta_{ki} \Ad_{g_{12}}(Y_2)_{jl},\]
\[E_{ij}\cdot (g_{12})_{kl}=
\delta_{ki}(g_{12})_{jl},\quad E_{ij}\cdot (g_1)_{kl}=\delta_{ki}(g_1)_{kl}
\]
 and the action on the rest of coordinates is trivial. Thus $I_{--}$ is $\frn^2$ invariant.
  Moreover, the above mentioned $B^2$-action on $\widetilde{\calX}_2$ coincides with the standard $B^2$-action on $\calX_2$ restricted to
 on $B^2$-invariant subvariety $\widetilde{\calX}_2$.

 Thus we defined a functor:
 $$\SymbolPrint{\Phi}:\MF_{B^2}(\calXr_2,\Wr)\to \MF_{B^2}(\calX_2,W),\quad \Phi(\calF)=j^x_*\circ \pi_y^*(\calF).$$


To show that the functor $\Phi$ is an embedding we define the functor $\Psi : \MF_{B^2}(\calX_2,W)\to \MF_{B^2}(\calXr_2,\Wr)$ such that $\Psi\circ\Phi= 1$.
Since the ideal $I_{--}$ generated by the matrix elements of $\tilde{X}_{1,--}$ is $B^2$ invariant we have well defined functor
$\SymbolPrint{Rest}: \MF_{B^2}(\calX_2,W)\to \MF_{B^2}(\widetilde{\calX_2}, \pi_y^*(\Wr))$ which is given by taking a quotient by the ideal $I_{--}$.
The push-forward along the projection $\pi_y: \calX_2\times \frn\to \calX_2$ provides the last step of the construction for
$\Psi:=\pi_{y*}\circ Rest$.
\end{proof}

\subsection{Properties of the Knorrer functor} As in the previous proof, $CC$ denotes the correction complex $$CC:= \begin{bmatrix} (\tilde{X}_{1--} )& (Y_{1}-\Ad_{g_{12}}(Y_2)_{++})\end{bmatrix}$$
Let us explain the equivariant structure of this matrix factorization. Since $B_{(2)}$-action does not change entries of the complex $\SymbolPrint{CC}$, the $B_{(2)}$ equivariant structure of $B_{(2)}$ is trivial.

To describe $B_{(1)}$ equivariant structure, let us write the differential of $CC$ as
\[ \SymbolPrint{D^{kn}}=D^{kn}_++D^{kn}_-, \quad D^{kn}_+=\sum_{k>l}(\tilde{X}_{1,--})_{kl}\theta_{kl},\]
\[D^{kn}_-=\sum_{k>l}
 (Y_{1}-\Ad_{g_{12}}(Y_2)_{++})_{lk}
\frac{\partial}{\partial \theta_{kl}},\]
where  $\theta_{ij}$ are odd variables. The equivariant structure is defined by the action on $\theta_{kl}$:
$$ E_{ij}\cdot \theta_{kl}=\sum_{k>i}\theta_{ki}-\sum_{j>l}\theta_{jl}$$

The differential $D^{kn}$ is not $\frn_{(2)}$ invariant:
$$ E_{ij}\cdot D^{kn}_+=0,\quad E_{ij}\cdot D^{kn}_-=\sum_{j>l} [\Ad_{g_{12}}(Y_2)_-]_{il}\frac{\partial}{\partial\theta_{jl}}-\sum_{k>i} [\Ad_{g_{12}}(Y_2)_-]_{ik}\frac{\partial}{\partial\theta_{ki}}.$$
Thus the correction differentials are non-trivial in $\Phi(\calF)$ but the correction differentials are not arbitrary and below we emphasize properties of the correction
differentials that are needed for the proofs that follow.

\def\mltV{V_{\mathrm{m}}}
\begin{proposition} Let $\calF=(M,D,\partial_l,\partial_r)\in \MF_{B^2}(\calXr_2,\Wr)$, $M=\mltV\otimes \CC[\calXr_2]$, $\mltV\in \Mod^{filt}_{\frn^2}$
The matrix factorization $\Phi(\calF)\in \MF_{B^2}(\calX_2,W)$ is the equivariant matrix factorization
$$ (\mltV\otimes \CC[\calX_2], D+D^{kn}, \tilde{\partial}_l,\tilde{\partial}_r)$$
where $\tilde{\partial}_r=\partial_r$.
\end{proposition}
\begin{proof}
The only part that requires a discussion is the fact that the differential $\Phi(\calF)^{\sharp}$ has differential $D+D^{kn}$: a priori our construction of the push forward
$j^x_*(\pi_y^*(\calF))$ relies on the extension lemma~\ref{lem: ext MF+} and involves choice of an extension $\tilde{D}$ of $D$ from variety $\widetilde{\calX}_2$ to $\calX_2$
and completion of $\tilde{D}+D^{kn}_+$  to the differential $\tilde{D}+D^{kn}_++d^-$ such that $(\tilde{D}+D^{kn}_++d^-)^2=W$.

However, in our case map $j^x$ is actually a section
of the (non-equivariant) map: $$\pi_x:\calX_2\to \widetilde{\calX}_2,\quad \pi_x(X,g_1,Y_1,g_2,Y_2)=(\Ad_{g_1}^{-1}(X)_+,g_1^{-1}g_2,Y_2).$$
 Thus according to the remark~\ref{rem: section}
we can choose $\tilde{D}$ to be $D=\bar{\pi}^*(D)=\pi_x^*(\pi_y^*(D))$. Moreover, $d^-=D_-^{kn}$ because $(D+D^{kn})^2=W$ and the completion of  the differential is unique up to an isomorphism.
\end{proof}



\subsection{Convolution on the reduced space}
In order to define the convolution on the smaller space $\calXr_2$, consider
the space $\SymbolPrint{\calXr_3}:=\frb\times G^2\times\frn$ with the following $B^3$-action
$$(b_1,b_2,b_3)\cdot (X,g_{12},g_{23},Y)=(\Ad_{b_1}(X),b_1 g_{12} b_2^{-1}, b_2 g_{23} b_3^{-1}, \Ad_{b_3}(Y)).$$
 There are the following
maps $\SymbolPrint{\bar{\pi}_{ij}}:\calXr_3\to\calXr_2$:
\[ \bar{\pi}_{12}(X,g_{12},g_{13},Y)=(X,g_{12},\Ad_{g_{23}}(Y)_{++}),\] \[\bar{\pi}_{23}(X,g_{12},g_{13},Y)=(\Ad_{g_{12}}^{-1}(X)_+,g_{23},Y),\]
\[\bar{\pi}_{13}(X,g_{12},g_{13},Y)=(X,g_{12}g_{23},Y).\]

\begin{proposition}
$\bar{\pi}_{12}^*(\Wr)+\bar{\pi}_{23}^*(\Wr)=\bar{\pi}_{13}^*(\Wr)$.
\end{proposition}
\begin{proof}
Indeed we have
\begin{multline*}
\bar{\pi}_{12}^*(\Wr)+\bar{\pi}_{23}^*(\Wr)=\Tr(X\Ad_{g_{12}}(\Ad_{g_{23}}(Y)_{++}))+\Tr(\Ad_{g_{12}}^{-1}(X)_+\Ad_{g_{23}}(Y))=\\
\Tr(\Ad^{-1}_{g_{12}}(X)\Ad_{g_{23}}(Y)_{++})+\Tr(\Ad^{-1}_{g_{12}}(X)_+\Ad_{g_{23}}(Y)).
\end{multline*}
Thus the statement follows is implied by
\begin{multline*} \Tr(A B_{++})+\Tr(A_+B)=\Tr(AB)-\Tr(AB_-)+\Tr(A_+B_-)\\
  =\Tr(AB)+\Tr(A_-B_-)=\Tr(AB).
\end{multline*}

\end{proof}

For given $\calF,\calG\in \MF_{B\times B}(\calXr_2,W)$ it is tempting  to define
\begin{equation}\label{eq: red convolution prelim}
\calF\SymbolPrint{\bar{\star}}\calG:=\bar{\pi}_{13*}(\CE_{\frn^{(2)}}( \bar{\pi}_{12}^*(\calF)\otimes\bar{\pi}_{23}^*(\calG))))^{T^{(2)}},
\end{equation}
where $\frn^{(2)}$, $T^{(2)}$ are the subalgebras of the middle factor in $\frb^3$ acting on $\calXr_3$. The only issue with this definition is that
the maps $\bar{\pi}_{12}$ and $\bar{\pi}_{23}$ are not $B^3$ equivariant. Thus we have to define the $B^3$-equivariant structure on
the pull back $\bar{\pi}_{12}^*(\calF)\otimes\bar{\pi}_{23}^*(\calG)$ and we do it below.

\subsection{Reduced vs. non-reduced convolution}



Let $\bar{\pi}_{con}:\calX_3\to \calXr_3$ be the map defined by
$$\SymbolPrint{\bar{\pi}_{con}}(X,g_1,Y_1,g_2,Y_2,g_3,Y_3)=(\Ad_{g_1}^{-1}(X)_+,g_1^{-1}g_2,g_2^{-1}g_3,Y_3).$$

\begin{proposition} \label{prop: Knorrer factor} Suppose $\calF,\calG\in \MF_{B^2}(\calXr_2,\Wr)$ then there is a sequence of  row transformations that identifies
$\pi_{12}^*(\Phi(\calF))\otimes_B\pi_{23}^*(\Phi(\calG))$ with the $B^3$-equivariant matrix factorization $\calH=(M,D,\partial_l,\partial_m,\partial_r)$,
\begin{equation}\label{eq: cor diffs}
\partial_l,\partial_m,\partial_r:
\Lambda^\bullet \frn \otimes M_{\star}\to \Lambda^{<\bullet}\otimes M_{\ge \star},
\end{equation}
$$\calH^\sharp=(M,D)=\bar{\pi}^*_{con}( \bar{\pi}_{12}^*(\calF))\otimes \bar{\pi}^*_{con}(\bar{\pi}_{23}^*(\calG))\otimes  [Y_2 -\Ad_{g_{23}}(Y_3)_{++},0]\otimes\pi_{13}^*(CC)$$
\end{proposition}

This proposition allows us to construct $B^3$-equivariant structure on $\bar{\pi}_{12}^*(\calF)\otimes\bar{\pi}_{23}^*(\calG)$.
First, we introduce few auxiliary spaces and map: $\SymbolPrint{\widehat{\calX}_3}:=\frg\times \frn\times G^3\times \frn$ and $\hat{\pi}:\calX_3\to\widehat{\calX}_3$ is given by
$$\SymbolPrint{\hat{\pi}}(X,g_1,Y_2,g_2,Y_2,g_3,Y_3)=(X,Y_1,g_1,g_2,g_3,Y_3).$$
We can embed $\widehat{\calX}_3$ into $\calX_3$ by
$$ \SymbolPrint{\hat{j}}:\widehat{\calX}_3\to \calX_3,\quad\hat{ j}(X,Y_1,g_1,g_2,g_{3},Y_3)\mapsto (X,g_{1},Y_1,g_2,\Ad_{g_{23}}(Y_3)_{++},g_{3},Y_3).$$
The map $\hat{j}$ the section of $\hat{\pi}$.

On the other hand, the matrix factorization $\calH^\sharp$ has Koszul complex $ [Y_2 -\Ad_{g_{23}}(Y_3)_{++},0]$ as one of the factors.
In other words, $\calH^\sharp=(M,D''+D')$, where $D'=\sum_{ij} (Y_2 -\Ad_{g_{23}}(Y_3)_{++})_{ij}\frac{\partial}{\partial\theta_{ij}}$.
Since the differential has the property (\ref{eq: cor diffs}) we contract $\calH$ along the differential $D'$ to obtain a homotopy  equivalence of
the matrix factorizations over $\CC[\widehat{\calX}_3]$
$$\calH\sim \widehat{\calH}=(\tilde{M},\tilde{D},\tilde{\partial}_l,\tilde{\partial}_m,\tilde{\partial}_r)\in \MF_{B^2}(\widehat{\calX}_3,\hat{j}^*\circ\pi_{13}^*(W))$$
where $\tilde{D}=\hat{j}^*(D'')$, $\hat{\partial}_l=\hat{j}(\partial_l)$, $\hat{\partial}_m=\hat{j}^*(\partial_m)$, $\hat{\partial}_r=\hat{j}^*(\partial_r)$.

Next, we introduce an analog of the Knorrer functor:
\[
  \SymbolPrint{\Phi^{(13)}}: \MF_{B^2}(\calXr_3,\bar{\pi}_{13}^*(\Wr))\to \MF_{B^2}(\widehat{\calX}_3,\hat{j}^*\circ\pi_{13}^*(W)).
\]
For that we introduce the space $$\SymbolPrint{\widetilde{\calX}_3}=\frb\times \frn \times G^3\times \frn$$ and fix a $B^3$-equivariant embedding
$\SymbolPrint{\tilde{j^x}}:\widetilde{\calX}_3\to \widehat{\calX}_3$, by identifying $\widetilde{\calX}_3$ with the subvariety of $\widehat{\calX}_3$ defined by the ideal
$\tilde{I}_{--}$ generated by the entries of $\Ad_{g_1}^{-1}(X)_{--}$.

The projection $\SymbolPrint{\tilde{\pi}_y}: \widetilde{\calX}_3\to \calXr_3$ defined by $$\tilde{\pi}_y(X,Y_1,g_1,g_2,g_3,Y_3)=(X,g_1^{-1}g_2,g_2^{-1}g_3,Y_3).$$
is $B^2$-equivariant. Thus since $\tilde{I}_{--}$ is $B^3$-equivariant we can define the functor:
$$ \Phi^{(13)}: \MF_{B^3}(\calXr_3,\Wr)\to  \MF_{B^3}(\widehat{\calX}_3,\hat{j}^*\circ\pi^*_{13}(W)),\quad \Phi(\calF)=\tilde{j}^x_*\circ\tilde{\pi}_y^*(\calF).$$

The following statement is analogous to the proposition~\ref{prop: Knorrer} and omit its proof:

\begin{proposition} The functor $\Phi^{(13)}$ provides an embedding of categories and
$$\Phi^{(13)}(\calF)^\sharp=\tilde{\pi}_{con}^*(\calF)\otimes [(\tilde{X}_{1,--}) (Y_1-\Ad_{g_{13}}(Y_3)_{++})],$$
where $\tilde{\pi}_{con}(X,Y_1,g_1,g_2,g_3,Y_3)=(\Ad_{g_1}^{-1}(X)_+,g_1^{-1}g_2,g_2^{-1}g_3,Y_3).$
\end{proposition}

Let us define a matrix factorization $\overline{\calH}\in \MF_{B^2}(\calXr_3,\Wr)$ in two steps. First, let us observe that there is
a well-defined functor $Rest: \MF_{B^2}(\widehat{\calX}_3,\hat{j}^*\circ\pi_{13}^*(W))\to \MF(\widetilde{\calX}_3,\tilde{\pi}_y^*(\Wr))$
which is the functor of taking quotient by the $B^3$-equivariant  ideal $\tilde{I}_{--}$. Second as a second step we apply
push forward along the projection map:
$$ \overline{\calH}:=\tilde{\pi}_{y*}\circ Rest(\widehat{\calH}).$$

The main result of this subsection is the following
\begin{proposition}
$$\widehat{\calH}=\Phi^{(13)}(\overline{\calH}).$$
\end{proposition}
\begin{proof}
Indeed, proposition~\ref{prop: Knorrer factor} implies that
$$\widehat{\calH}^\sharp=\tilde{\pi}^*_{con}( \bar{\pi}_{12}^*(\calF)\otimes \bar{\pi}_{23}^*(\calG))\otimes [(\tilde{X}_{1,--}) (Y_1-\Ad_{g_{13}}(Y_3)_{++})].$$
Hence $\overline{\calH}=\bar{\pi}_{12}^*(\calF)\otimes \bar{\pi}_{23}^*(\calG)$ and both
$\widehat{\calH}$ and $\Phi^{(13)}(\overline{\calH})$ provide an $B^2$-equivariant extension of $\tilde{\pi}_y^*(\overline{\calH})$ from $\widetilde{\calX}_3$ to
$\widehat{\calX}_3$. Hence the result follows from the uniqueness of the extension.
\end{proof}

The proof of the previous proposition explains that $\overline{\calH}$ is $B^3$-equivariant matrix factorization such that
$\overline{\calH}^\sharp=\bar{\pi}_{12}^*(\calF)\otimes\bar{\pi}_{23}^*(\calG).$ Thus we equipped the pull back
$\bar{\pi}_{12}^*(\calF)\otimes\bar{\pi}_{23}^*(\calG)$ with $B^3$-equivariant structure. We use notation
$$\SymbolPrint{(\bar{\pi}_{12}\otimes_B\bar{\pi}_{23})^*}(\calF\boxtimes\calG)\in \MF_{B^3}(\calX_3,\bar{\pi}_{13}^*(\Wr)).$$

We define binary operation on $\MF_{B^2}(\calXr_2,\Wr)$ by
$$\calF\bar{\star}\calG:=\bar{\pi}_{13 *}\CE_{\frn^{(2)}}\left((\bar{\pi}_{12}\otimes_B\bar{\pi}_{23})^*(\calF\boxtimes\calG)\right)^{T^{(2)}}.$$

The associativity of this binary operation follows from
\begin{corollary} For any $\calF,\calG\in\MF_{B^2}(\calXr_2,\Wr)$ we have
$$\Phi(\calF)\star\Phi(\calG)=\Phi(\calF\bar{\star}\calG).$$
\end{corollary}
\begin{proof}
There is a natural projection map $\hat{\pi}_{13}: \widehat{\calX}_3\to \calX_2$, $$\hat{\pi}_{13}(X,Y_1,g_1,g_2,g_3,Y_3)=(X,Y_1,g_1,g_3,Y_3).$$
Since we have homotopy $\hat{\calH}\sim \calH$, we conclude that
$$\Phi(\calF)\star\Phi(\calG)=\pi_{13*}(\calH)=\hat{\pi}_{13*}(\widehat{\calH})=\hat{\pi}_{13*}(\Phi^{(13)}(\overline{\calH})).$$
Thus to prove the statement we need to show that $\hat{\pi}_{13*}\circ \Phi^{(13)}=\Phi\circ\bar{\pi}_{13*}$ or equivalently
$\hat{\pi}_{13*}\circ \tilde{j}^x_*\circ \tilde{\pi}^*_y=j_*^x\circ\tilde{\pi}^*_y\circ\bar{\pi}_{13*}.$
The later identity holds because the map $\hat{\pi}_{13}\times\bar{\pi}_{13}: \widehat{\calX}_3\times\calXr_3\to \calX_2\times \widetilde{\calX}_2$ sends
the kernel of $\Phi^{(13)}$, which is $\tilde{j}^x\times \tilde{\pi}_y(\widetilde{\calX}_3)\subset  \widehat{\calX}_3\times\calXr_3$ to
the kernel of the map $\Phi$, which is $j^x\times \tilde{\pi}_y(\widetilde{\calX}_2)\subset  \calX_2\times\widetilde{\calX}_2$. More formally, we can proceed with the sequence of base isomorphisms as follows. We have the following diagram of maps:
\[
\begin{tikzcd}
\widehat{\calX_3}\arrow[d,"\hat{\pi}_{13}"]&\widetilde{\calX}_3\arrow[r,"\tilde{\pi}_y"]\arrow[d,dashed,"\tilde{\pi}_{13}"]\arrow[l,"\tilde{j}^x"]
&\overline{\calX}_3\arrow[d,"\bar{\pi}_{13}"]\\
\calX_2&\widetilde{\calX}_2\arrow[r,"\tilde{\pi}_y"]\arrow[l,"j^x"]&\overline{\calX}_2
\end{tikzcd}
\]
where the dashed arrow map $\tilde{\pi}_{13}$ is a unique smooth map such that the diagram commutes.
Then we can use change of base isomorphism to show:
\[
j_*^x\circ\tilde{\pi}_y^*\circ \bar{\pi}_{13*}=j_*^x\circ \tilde{\pi}_{13*}\circ\tilde{\pi}^*_y=\hat{\pi}_{13*}\circ\tilde{j}^x_*\circ\tilde{\pi}_y^*
\]

\end{proof}

Before explaining the proof of the proposition~\ref{prop: Knorrer factor} let us discuss of the elementary transformations in the context of our extension lemma~\ref{lem: ext MF+eq}.

\subsection{Row operations for generalized equivariant Koszul matrix factorizations}\label{ssec: row ops}
We use notations of the lemma~\ref{lem: ext MF+eq}
here and we also assume that the complex $(\calF_\bullet,d^+)$ is the Koszul complex tensored with matrix factorization $(M,D,\partial)$:
$$\calF_\bullet=\tilde{M} \otimes \Lambda^\bullet V, \quad d^+=\sum_{j=1}^n f_i\theta_i,$$
where $\theta_i$ form a basis of $V$ and $f_1,\dots,f_n$ form a regular sequence in $\CC[\calZ]$.
We call such matrix factorizations generalized equivariant Koszul matrix factorizations.

In these notations the corrective differentials that are computed by the iterative procedure from lemma~\ref{lem: ext MF+eq} are of the form
$$ d^-=\sum_{\vec{i}} d^-_{\vec{i}}\frac{\partial}{\partial \theta_{\vec{i}}}, \quad \partial^-=\sum_{\vec{i}}\partial^-_{\vec{i}} \frac{\partial}{\partial \theta_{\vec{i}}},$$
where $\theta_{\vec{i}}=\theta_{i_1}\dots \theta_{i_k}$ and $d^-_{\vec{i}}\in Hom_{\CC[\calZ]}(\tilde{M},\tilde{M})$,
$\partial_{\vec{i}}^-\in Hom(\Lambda^\bullet \frn,\Lambda^{<\bullet}\frn)\otimes Hom_{\CC[\calZ]}(\tilde{M},\tilde{M})$.

Let us use notation $D_{tot}:=d^++D+\partial+d^-+\partial^-+d_{ce}$.
There are three types of (generalized) row transformations for the complex $(\calF_\bullet,d^++D+d^-,\partial+\partial^-)$. The first type is given by automorphism
 $\exp(\phi\frac{\partial}{\partial\theta_k})$, $\phi\in \Hom_{\CC[\calZ]}(\calF,\calF)$ is an odd automorphism and the affect of this automorphism on the differential is:
 \begin{multline}
  \exp(-\phi\frac{\partial}{\partial \theta_k})\circ D_{tot}\circ \exp(\phi\frac{\partial}{\partial \theta_k})= \left(D+d_-+d_++[\phi,D]_+\frac{\partial}{\partial \theta_k}+f_k\phi\right)\\
 +d_{ce}+\partial+\partial^-+[\phi,d_{ce}+\partial+\partial_-]_+\frac{\partial}{\partial\theta_k}.
\end{multline}
That is the result of the row transformation of the first type is the equivariant matrix factorization
$$ (\tilde{M},D+d_-+d_++[\phi,D]_+\frac{\partial}{\partial \theta_k}+f_k\phi,\partial+\partial^-+[\phi,d_{ce}+\partial+\partial_-]_+\frac{\partial}{\partial\theta_k}).$$

The second type elementary row transformation is given by conjugation by the automorphism $\exp(\phi\theta)$. The result of conjugation of $D_{tot}$ by this automorphims is the following expression
$$ \exp(-\phi\theta_k)\circ D_{tot}\circ\exp(\phi\theta_k)=D_{tot}+D'+\partial'.$$
\begin{equation}\label{eq: sec row 1}
D'=[\phi,D]_+\theta_k+\sum_{k\in \vec{i} }d^-_{\vec{i}}\phi\frac{\partial}{\partial\widehat{\theta}_{\vec{i}}}+\theta_k\sum_{\vec{i}}[\phi,d^-_{\vec{i}}]\frac{\partial}{\partial\theta_{\vec{i}}},
\end{equation}
\begin{equation}\label{eq: sec row 2}
\partial'=\sum_{k\in \vec{i} }\partial^-_{\vec{i}}\phi\frac{\partial}{\partial\widehat{\theta}_{\vec{i}}}+\phi d_{ce}^k+
\theta_k(\sum_i[\phi,d_{ce}^i]\frac{\partial}{\partial\theta_i}+[\phi,\partial]_+ +\sum_{\vec{i}}[\phi,\partial^-_{\vec{i}}]\frac{\partial}{\partial\theta_{\vec{i}}}),
\end{equation}
where $\partial\widehat{\theta}_{\vec{i}}=\partial \theta_{i_1}\dots \widehat{\partial\theta_k}\dots\partial_{i_m}$, $\vec{i}=(i_1,\dots,i_m)$ and
$d_{ce}=\sum_i d_{ce}^i\frac{\partial}{\partial \theta_i}.$

The computation above indicates that if last summands in (\ref{eq: sec row 1}) and (\ref{eq: sec row 2}) vanish
\begin{equation}\label{eq: adm row tr}
\sum_{\vec{i}}[\phi,d^-_{\vec{i}}]\frac{\partial}{\partial\theta_{\vec{i}}}=0,\quad \sum_i[\phi,d_{ce}^i]\frac{\partial}{\partial\theta_i}+[\phi,\partial]_+ +\sum_{\vec{i}}[\phi,\partial^-_{\vec{i}}]\frac{\partial}{\partial\theta_{\vec{i}}}=0
\end{equation}
 then result of the row transformation is the equivariant matrix factorization:
 $$(\tilde{M},d^++D+d^-+[\phi,D]_+\theta_k+\sum_{k\in \vec{i} }d^-_{\vec{i}}\phi\frac{\partial}{\partial\widehat{\theta}_{\vec{i}}},\partial+\partial^-+\sum_{k\in \vec{i} }\partial^-_{\vec{i}}\phi\frac{\partial}{\partial\widehat{\theta}_{\vec{i}}}+\phi d_{ce}^k)$$

Finally the third type of row operation is given by automorphism $\exp(c\theta_i\frac{\partial}{\partial\theta_j})$ where $c\in \Hom_{\CC[\calZ]}(\calF,\calF)$ is an even morphism. This type of row transformation
was discussed in the subsection~\ref{ssec: str eq Kosz} for strongly equivariant matrix factorizations.
 We will use these type row operations only for strongly equivariant matrix factorization and omit a discussion
about complications that arise for weakly  equivariant generalized Koszul complexes.

The key feature of the row transformations discussed here is that differentials that encode the equivariant structure of the matrix factorization do not
contribute to the change of the non-equivariant shadow of the matrix factorization. Other words, if $T$ is one the transformations discussed above then
$$ T(\calF)^\sharp=T(\calF^\sharp).$$

\subsection{Proof of proposition~\ref{prop: Knorrer factor}}


Let us expand the definition of convolution to obtain the following expression for $\pi_{12}^*(\Phi(\calF)^\sharp)\otimes\pi_{23}^*(\Phi(\calG)^\sharp)$:
\begin{multline}\label{eq: long MF}
\pi_{12}^*(\bar{\pi}^*(\calF^\sharp)\otimes CC)\otimes\pi_{23}^*(\bar{\pi}^*(\calG^\sharp)\otimes CC)\\
=(\pi_{12}\circ\bar{\pi})^*(\calF^\sharp)\otimes (\pi_{23}\circ\bar{\pi})^*(\calG^\sharp)\otimes \pi_{12}^*(CC)\otimes\pi^*_{23}(CC).
\end{multline}
The last factor is the Koszul complex $\begin{bmatrix} (\tilde{X}^2_{--} )& (Y_{2}-\Ad_{g_{23}}(Y_3)_{++})\end{bmatrix}$ hence we can use
row transformations of second type to eliminate variables $Y_2$  from the matrix factorization (\ref{eq: long MF}).
More precisely, use the row transformations of the second type to eliminate $Y_2$ from
the first factor $(\pi_{12}\circ \bar{\pi})^*(\calF)$  and from the second to last factors $\pi_{12}^*(CC)$.

 We can apply the second type of row transformation because
 $(\pi_{12}\circ \bar{\pi})^*(\calF^\sharp)$ and $\pi_{12}^*(CC)$ are the factors of $\pi^*_{12}(\Phi(\calF)^\sharp)$, thus differentials in these matrix factorizations
 commute with the differentials from the factor $\pi^*_{23}(\Phi(\calG)^\sharp)$ and hence the admissibility condition (\ref{eq: adm row tr}) is satisfied.
 On the other hand,  the second factor $(\pi_{23}\circ\bar{\pi})^*(\calG^\sharp)$ has no dependence on $Y_2$.

The result of elimination of $Y_2$ from the factor before the last  is $$CC_{123}=[\tilde{X}_{1,--}, Y_1-\Ad_{g_{12}}(\Ad_{g_{23}}(Y_3)_{++})_{++}]$$ and the first two factors are
\begin{gather*}
\pi^*_{12}(\bar{\pi}^*(\calF))|_{Y_2=\Ad_{g_{23}}(Y_3)_{++}}=\calF|_{Y=\Ad_{g_{23}}(Y_3)_{++},X_+=\tilde{X}_{1,+}}=\bar{\pi}^*_{con}(\bar{\pi}_{12}^*(\calF))\\
\pi^*_{23}(\bar{\pi}^*(\calG))|_{Y_2=\Ad_{g_{23}}(Y_3)_{++}}=\calG|_{Y=Y_3,X_+=\Ad_{g_{12}}^{-1}(\tilde{X}_{1})_+},
\end{gather*}
where $(X_+,g_{12},g_{23},Y)$ are coordinates on $\calXr_3$.
After the elimination the only term that has dependence on $Y_2$ is the Koszul complex  $\begin{bmatrix} G& (Y_{2}-\Ad_{g_{23}}(Y_3)_{++})\end{bmatrix}$ and since $\pi_{13}^*(W)$ does not depend
on $Y_2$ and the correction correction differentials $\partial$, which could depend on $Y_2$ does not contribute the square of the total differential, the element $G$ must vanish.

As the next step we use the complex $CC_{123}$ to eliminate dependence on $\tilde{X}_{1,--}$ from the other terms in the tensor product (\ref{eq: long MF}). For that
we use the first type of row transformation and we do not need to worry about the admissibility of our eliminations.
The term $\bar{\pi}^*_{con}(\bar{\pi}_{12}^*(\calF))$ has no dependence on $\tilde{X}_{1,--}$ so it is not affected by this elimination. On the other hand the second
term becomes:
$$ \calG|_{Y=Y_3,X_+=\Ad_{g_{12}}^{-1}(\tilde{X}_{1})_+, \tilde{X}_{1,--}=0}= \calG|_{Y=Y_3,X_+=\Ad_{g_{12}}^{-1}(\tilde{X}_{1,+})_+}=\bar{\pi}^*(\bar{\pi}_{23}^*(\calG)).$$

After this round of elimination the only dependence on $\tilde{X}_{1,--}$ in (\ref{eq: long MF}) comes from the Koszul complex $[\tilde{X}_{1,--},G']$ hence $G'$ must be the derivative of $\pi_{13}^*(W)$ by
$\tilde{X}_{1,--}$ that is $G'=Y_1-\Ad^{-1}_{g_{13}}(Y_3).$ That is  after the eliminations $Y_2,\tilde{X}_{1,--}$ we obtain
the complex
\begin{equation}\label{eq: last step red}
 \bar{\pi}^*\left(\bar{\pi}_{12}^*(\calF)\otimes\bar{\pi}_{23}^*(\calG)\right)\otimes [Y_2 -\Ad_{g_{23}}(Y_3)_{++},0]\otimes \pi_{13}^*(CC)
\end{equation}
as we claimed in the proposition.



\def\frp{\mathfrak{p}}
\section{Induction functors}\label{sec: inc}

The main goal of this section is to construct induction functors:
$$ \SymbolPrint{\ind_k}: \MF_{B_k^2}(\calX_2^\circ(G_k),W)\times \MF_{B_{n-k}^2}(\calX^\circ_2(G_{n-k}),W)\to \MF_{B_n^2}(\calX_2^\circ(G_n),W),$$
$$ \SymbolPrint{\indb_k}: \MF_{B_k^2}(\calXr_2(G_k),\Wr)\times \MF_{B_{n-k}^2}(\calXr_2(G_{n-k}),\Wr)\to \MF_{B_n^2}(\calXr_2 (G_n),\Wr).$$
 and to show that these functors are intertwined by the Knorrer functors.
\begin{proposition}\label{prop: ind phi} We have
$$ \ind_k\circ (\Phi_k\times \Phi_{n-k} ) =\Phi_n\circ \ind_k.$$
\end{proposition}
We also show that the induction functor is the convolution algebra homomorphism.
\begin{proposition}\label{prop: ind homo} For any
  \[\calF_1,\calG_1\in \MF_{B_k^2}(\calX^\circ_2(G_k),W), \quad \calF_2,\calG_2\in \MF_{B_{n-k}^2}(\calX^\circ_2(G_{n-k}),W).\]
  We have
$$ \ind_k(\calF_1,\calF_2)\star\ind_{k}(\calG_1,\calG_2)=\ind_k(\calF_1\star\calG_1,\calG_1\star\calG_2).$$
\end{proposition}

Finally, we define the generators of the braid group on two strands and using the induction functor explain how we construct the generators for the braid group on
arbitrary number of strands.

\subsection{Construction of the induction functor}
Let us set notations for the parabolic subgroups of $G_n$. We denote by $P_k\subset G_n$  the subgroup with Lie algebra generated by
$E_{ij}$, $i\le j$ and $E_{i+1,i}$, $i\ne k$. We denote the corresponding Lie algebra as $\SymbolPrint{\frp_k}:=\Lie(P_k)$.
There is a natural homomorphism $\SymbolPrint{i_k}: P_k\to G_n$.

We define $\calX^\circ_2(P_k)\subset\calX^\circ_2(G_n)$ as a locus of $(X,g_{12},Y_1,Y_2)$ such that $g_{12}\in P_k$ and $X\in \frp_k$.
Let $i_k$ be the natural  map $\calX^\circ_2(P_k)\to \calX_2^\circ(G_n)$. The image of $i_k$ is defined by the equations
$\SymbolPrint{\Pi_{--}^k}(g_{12})=0$, $\Pi_{--}^k(X)=0$  where $\Pi^k_{--}$ is the projection from the space of $n\times n$ matrices to its $ij$ entries $i\in [k+1,n]$, $j\in [1,k]$.

The ideal generated by $\Pi_{--}^k(g_{12})$, $\Pi_{--}^k(X)$  is $B_n^2$-invariant and these elements form a regular sequence. Hence there is a well-defined functor
$$ i_{k*}:\MF_{B^2}(\calX^\circ_2(P_k),W)\to \MF_{B^2}(\calX_2^\circ(G_n),W)$$
induced by the embedding.

Let us denote by $\SymbolPrint{\Pi_{++}^k}:\frg\to \frg$ the linear projection from the space of $n\times n$ matrices to its  $ij$ entries $i\in [1,k]$, $j\in [k+1,n]$. In particular
the image of $(\mathrm{Id}-\Pi_{++}^k)(\frp_k)$ is naturally identified with $\frg_k\oplus\frg_{n-k}$, we denote this map by $p_k:\frp_k\to \frg_k\oplus \frg_{n-k}$.
Similarly, we have a map $\SymbolPrint{p_k}:P_k\to G_k\times G_{n-k}$.

By applying $p_k$ to all factors of the space $\calX_2^\circ(P_k)$ we obtain the map
$p_k:\calX^\circ_2(P_k)\to \calX^\circ_2(G_k)\times\calX^\circ_2(G_{n-k})$.
Using homomorphism $p_k\colon B_n\to B_{k,n-k}$ we can extend $\frn_k^2\times \frn_{n-k}^2$-equivariant structure of $\calX_2(G_k)\times \calX_2(G_{n-k})$ to the
$\frn_n^2$ equivariant structure and the map $p_k$ is $\frn_n^2$-equivariant in this setting.
Moreover, since $T_n=T_k\times T_{n-k}\times \mu$ where $\mu$ is an finite subgroup of the group of roots of unity, we have the following functor:
$$ p_k^*: \MF_{B_k^2}(\calX^\circ_2(G_k),W)\times \MF_{B_{n-k}^2}(\calX^\circ_2(G_{n-k},)W)\to \MF_{B_n^2}(\calX^\circ_2(P_n),W).$$


The induction functor is defined as composition of maps:
$$  \ind_k:=i_{k*}\circ p_k^*$$

\subsection{Computational aspects of push-forward}\label{ssec: comp ins}
In our general construction of the push forward $i_{k*}(\calF)$, $\calF=(M,D,d_l,d_r)$  we make a choice of the extension $\tilde{D}$ of the differential $D$
from sub-variety $\calX^\circ_2(P_k)$ to $\calX^\circ_2(G_n)$. However in our case we have a canonical choice of
the extension because we could use the projection map $\Lie(G_n)\to \Lie(P_{k})$ to construct such an extension and use the remark~\ref{rem: section}, the details are provided below.
Morally, we would like to interpret $i_{k}$ as a section of some map and it is possible if we enlarge our space slightly.

In this subsection we use notation $\adj(A)$ for the adjoint of matrix $A$.
Let us introduce space $\SymbolPrint{\calX^\circ_2(G_n)^\flat}:=\frg_n\times \Lie(G_n) \times \frn_n^{2}$. There is
a natural embedding $$j_n: \calX_2^\circ(G_n)\to \calX_2^\circ(G_n)^\flat$$ induced by the embedding $G_n\to \mathfrak{gl}_n$. If we define
\[W^\flat(X,g,Y_1,Y_2)=
\Tr(X( Y_1\det- gY_2\adj(g))),\quad \det=\det(g)\] then we have
$j^*_n(W\cdot\det)=W^\flat$ and the map \[j^*_n: \MF_{B^2_n}(\calX_2^\circ(G_n)^\flat,W^\flat)\to \MF_{B^2_n}(\calX_2^\circ(G_n)^\flat,W'),\]
where \(W'=W\cdot\det\), is well defined.

Similarly, we define  $\calX^\circ_2(P_k)^\flat$ and $\calXr^\circ_2(P_k)^\flat$.
  The natural  map $$\SymbolPrint{j^k_n}: \calX^\circ_2(P_k)\to \calX^\circ_2(P_k)^\flat$$
induces the pull back map $$j^{k*}_{n}: \MF_{B^2}(\calX^\circ_2(P_k)^\flat,W^\flat)\to \MF_{B^2}(\calX^\circ_2(P_k),W').$$

Since \(W'\) and \(W\) differ by the non-vanishing factor \(\det\), there is an embedding of categories \(\det: \MF_{B^2}(\calX^{\circ}_2,W)\rightarrow\MF_{B^2}(\calX^{\circ}_2,W')\):
\[\det: [M_0\xrightarrow{D_0} M_1\xrightarrow{D_1} M_1]\mapsto [M_0\xrightarrow{D_0\cdot \det} M_1\xrightarrow{D_1}M_0]. \]
We the inverse functor we denote by \(\det^{-1}\) and we use the same notation for the functors between the categories related to the parabolic versions of
the categories.

There is no natural map between $\calX^\circ_2(G_n)$ and $\calX^\circ_2(P_k)$ but there is a natural linear projection $\SymbolPrint{\rho_k}: \calX^\circ_2(G_n)^\flat\to
\calX_2^\circ(P_k)^\flat$ and
there is a pull back map $$\rho_k^*: \MF(\calX_2^\circ(P_k)^\flat,W^\flat)\to \MF(\calX_2^\circ(G_n)^\flat,\rho_k^*(W^\flat)).$$

There is also a natural inclusion map $i_k: \calX^\circ_2(P_k)^\flat\to \calX^\circ_2(G_n)^\flat$ which is a section of the projection
$\rho_k$: $\rho_k\circ i_{k}=id$.
Hence there is $\calF'\in \MF_{B^2_n}(\calX^\circ_2(P_k)^\flat,W^\flat)$ such that $j^{k*}_n(\calF')=\det(\calF)$
and we define \[\tilde{\calF}=(\tilde{M},\tilde{D},\tilde{\partial}_l,\tilde{\partial}):=\mathrm{det}^{-1}\circ j^{*}_n\circ \rho_k^*(\calF')\in \MF(\calX^\circ_2(G_n),j_n^{*}\circ{\rho}_{k}^*(W^\flat)\cdot\mathrm{det}^{-1}).\]
By construction we have $\tilde{\calF}|_{\calX^\circ_2(P_k)}=\calF$. Thus the matrix factorization  $\tilde{\calF}$ provides the first step of the construction of the push forward
$i_{k*}(\calF)$ from section~\ref{ssec: eq push f}. By completing the second step of the construction from the section~\ref{ssec: eq push f} we obtain the matrix factorization:
$$ i_{k*}(\calF)=(\tilde{M},d_K+\tilde{D}+d^-,\partial_l,\partial_r),$$
where $d_K: \tilde{M}\otimes \Lambda^{\bullet} V\to \tilde{M}\otimes \Lambda^{\bullet-1} V$ is the Koszul differential for the
defining ideal of $\calX_2^\circ(P_k)\subset \calX_2^\circ(G_n)$ and $d^-=\sum_{i<j} d^-_{ij}$,
$d^-_{ij}:\tilde{M}\otimes\Lambda^i V\to \tilde{M}\otimes \Lambda^{j} V$ is the complementary differential: $(d_k+\tilde{D}+d^-)^2=\Wr$

Note that because of our choice of extension $\tilde{\calF}$ the differentials $d^-_{ij}$ vanish unless $j=i+1$, we give more details below.
Define a linear operator operator $\Pi^k_{+}:\frg_n\to \mathfrak{p}_k$ which is a natural projection. Respectively, we have $\Pi_{--}^k(A)=A-\Pi_+^{k}(A)$.
The difference $\delta_kW:=W-j_n^*\circ\rho_n^{k*}(W^\flat)\cdot \det^{-1}$ could be written as sum of two terms:
$$\delta_kW_1:=\Tr(\Pi_{--}^k(X)\cdot \Pi_{++}^k(Y_1-\Ad_{g_{12}}(Y_2))),$$
$$\delta_kW_2:=\Tr(\Pi_+^k(X)\cdot
\Pi^k_{-}(\Ad_{g_{12}}(Y_2)-\rho_k(g_{12})(Y_2)\adj(\rho_k(g_{12}))\det(g_{12})^{-1}  ),$$
where we use convention $\mathrm{Id}=\Pi^k_{-}+\Pi^k_{++}$.

Since $\delta_k W_1\in (\Pi^k_{--}(X))$ and $\delta_kW_2\in (\Pi^k_{--}(g_{12}))$ we have
$$ i_{k*}(\calF)^\sharp=\tilde{\calF}\otimes [\Pi_{--}^k(X), \Pi_{++}^k(Y_1-\Ad_{g_{12}}(Y_2))]\otimes  \mathrm{K}^{\delta_k W_2} (\Pi_{--}^{k}(g_{12})).$$
That is the differentials \(d_{ij}^-\)   are the negative differentials of the Koszul matrix factorization and hence have above mentioned property.

\subsection{Induction functor in reduced case}

Let us
define $\calXr_2(P_k):=\frb\times P_k\times \frn$ and $\calXr_2(P_k)^\flat:=\frb\times\mathrm{Lie}(P_k)\times\frn$.
We also use notation $\Wr$ for the restriction of the potential $\Wr$ on the subspace $\calXr_2(P_{k})$.

There is a natural embedding $\bar{i}_k: \calXr_2(P_k)\to \calXr_2(G_n)$.
The map of the rings $\bar{i}_{k}: \CC[\calXr_2(G_n)]\to \CC[\calXr_2(P_k)]$ is the projection map with the kernel $(\Pi^k_{--}(g))$ and it is easy to see that these generators
form a regular sequence. Thus we have a well-defined functor:
$$ \bar{i}_{k*}: \MF_{B^2}(\calXr_2(P_k),\Wr)\to \MF_{B^2}(\calXr_2(G_n),\Wr).$$
We also have the projection map $\bar{p}_k: \calXr_2(P_k)\to \calXr_2(G_k)\times\calXr_2(G_{n-k})$.

The map $\bar{p}_k$ is $B_{k}^2\times B_{n-k}^2$-equivariant and we extend $B_{k}^2\times B_{n-k}^2$-equivariant structure to
$B_n^2$-equivariant structure to obtain  an induced map $$\bar{p}^*_k:
\MF_{B_k^2}(\calXr_2(G_k),\Wr)\times\MF_{B_{n-k}^2}(\calXr_2(G_{n-k}),\Wr) \to \MF_{B_n^2}(\calXr_2(P_k),\Wr).$$  Thus we define
$$\indb_k:=\bar{i}_{k*}\circ \bar{p}_k^*.$$

Now let us discuss an analogs of the map $\rho_k$ for reduced spaces and the corresponding properties of the push forward $\bar{i}_{k*}$.
Consider a space $\calXr_2^{\flat}(G_n):=\frb\times \mathfrak{gl}_n\times \frn$. There is
a natural embedding $\bar{j}_n: \calXr_2(G_n)\to \calXr_2(G_n)^\flat$ induced by the embedding $G_n\to \mathfrak{g}_n$. If we define $\Wr^\flat(X,g,Y)=\Tr(X( gY\adj(g)))$ then we have
$\bar{j}^*_n(\Wr')=\Wr^\flat$, \(\Wr'=\Wr\cdot\det(g)\) and map
$$\bar{j}^*_n: \MF_{B^2_n}(\calXr_2(G_n)^\flat,\Wr^\flat)\to
\MF_{B^2_n}(\calXr_2(G_n),\Wr')$$ is well-defined.

Similarly we define  $\calXr_2(P_{k})$ and $\calXr_2(P_{k})^\flat$.
  The natural open embedding $\bar{j}^k_n: \calXr_2(P_{k})\to \calXr_2(P_{k})^\flat$
  induces the pull back map $$\bar{j}^{k*}_n: \MF_{B^2}(\calXr_2(P_k),\Wr')\to \MF_{B^2}(\calXr_2(P_{k})^\flat,\Wr^\flat).$$
  
 There is a natural linear projection $$\bar{\rho}_k: \calXr_2(G_n)^\flat\to \calXr_2(
P_{k})^\flat$$ and
 a pull back map $$\bar{\rho}_k^*: \MF(\calXr_2(P_{k})^\flat,\Wr^\flat)\to \MF(\calXr(G_n)^\flat,\rho_{k}^*(\Wr^\flat)).$$


 We can choose $\calF'\in \MF_{B^2}(\calX_2(G_n)^\flat,\Wr)$ such that $\bar{j}^{k*}_n(\calF')=\det(\calF)$
 and we define
 \[\tilde{\calF}:=\mathrm{det}^{-1}\circ\bar{j}^{k*}_n\circ \bar{\rho}_{k}^*(\calF')\in \MF(\calX_2(G_n),\bar{j}_n^{k*}\circ\bar{\rho}_{k}^*(\Wr^\flat)\cdot\mathrm{det}^{-1}).\]
Let $\delta_k\Wr=\Wr-\bar{j}_n^{k*}\circ\bar{\rho}_k^*(\Wr^\flat)/\det$, then analogously to the non-reduced case we have:
%
$$ \bar{i}_{k*}(\calF)^\sharp=\tilde{\calF}\otimes \mathrm{K}^{\delta_k\Wr} (\Pi_{--}^{k}(g_{12})).$$

\subsection{Proof of proposition~\ref{prop: ind phi}}
Note that the statement of the proposition refers to the Knorrer functor
$$\Phi_n:\quad \MF_{B_n^2}(\calXr_2(G_n),\Wr)\to \MF_{B_n^2}(\calX^\circ_2(G_n),W),$$
which is constructed the in same way as our usual functor, the only adjustment being
the replacement of the
space $\widetilde{\calX}_2$ by the space $\widetilde{\calX}_2^\circ$ and the appropriate adjustment of the maps.

The statement of the proposition is equivalent to the isomorphism between two functors relating the following categories:
$$\MF_{B_k^2}(\calXr(G_k),\Wr)\times \MF_{B_{n-k}^2}(\calXr_2(G_{n-k}),\Wr)\longrightarrow \MF_{B_n^2}(\calX^\circ_2(G_n),W).$$
The first functor is $ i_{k*}\circ p_k^*\circ (j^x\times j^x)_*\circ (\pi_y\times\pi_y)^*$ and the second is $j^x_*\circ \pi_y^*\circ\bar{i}_{k*}\circ \bar{p}_k^*$.

We begin with the first functor. Using maps $p_k:\calX^\circ_2(P_k)\to \calX^\circ_2(G_k)\times \calX^\circ_2(G_{n-k})$ and
$j^x\times j^x: \widetilde{\calX}_2^\circ(G_k)\times \widetilde{\calX}_2^\circ(G_{n-k})\to \calX^\circ_2(G_k)\times \calX^\circ_2(G_{n-k})$ we can form
 a product of $\calX^\circ_2(P_k)$ and $ \widetilde{\calX}_2^\circ(G_k)\times \widetilde{\calX}_2^\circ(G_{n-k})$ over $\calX^\circ_2(G_k)\times \calX^\circ_2(G_{n-k})$ which is
 equal $\frb_n\times P_k\times \frn_n^2$. Thus we obtain a commuting diagram
 $$ \begin{tikzcd}
 \frb_n\times P_k\times \frn_n^2 
 \arrow[dotted]{r}{j^x} 
 \arrow[dotted]{d}{p_k}   & \calX^\circ_2(P_k)\arrow{d}{p_k}\\
\widetilde{\calX}_2^\circ(G_k)\times \widetilde{\calX}_2^\circ(G_{n-k})\arrow{r}{j^x\times j^x} &                                  \calX^\circ_2(G_k)\times \calX^\circ_2(G_{n-k})
 \end{tikzcd}$$
where the horizontal dotted arrow map is a natural embedding $j^x: \frb_n\times P_k\times \frn_n^2\to \frg_n\times P_k\times \frn_n^2$ while the vertical one is
$$p_k(X,g,\vec{Y})=(p_k(X)',p_k(g)',p_k(Y_1)',p_k(Y_2)')\times(p_k(X)'',p_k(g)'',p_k(Y_1)'',p_k(Y_2)''),$$
where \(\vec{Y}=(Y_1,Y_2).\)

The maps in the diagram are either projections or regular embeddings thus we have  the base change formula:
$ p_k^*\circ (j^x\times j^x)_*=j^x_*\circ p_k^*.$
Thus we have $$ i_{k*}\circ p_k^*\circ (j^x\times j^x)_*\circ (\pi_y\times\pi_y)^*=(i_k\circ j^x)_*\circ (\pi_y\times \pi_y\circ p_k)^*.$$

Analysis of the second functor is similar. As in the previous case we have the the commuting diagram:

$$ \begin{tikzcd}
 \frb_n\times P_k\times \frn_n^2 \arrow[dotted]{r}{i_k}\arrow[dotted]{d}{\pi_y}   & \widetilde{\calX}^\circ_2(G_n)\arrow{d}{\pi_y}\\
\calXr_2^\circ(P_k)\arrow{r}{\bar{i}_k} &                                  \calXr_2(G_n)
 \end{tikzcd}
 $$
where the dotted arrow maps are: $i_k: b_n\times P_k\times \frn_n^2\to \widetilde{\calX}_2^\circ(G_n)$ is induced my $i_k:P_k\to G_n$;
$\pi_y(X,g,Y_1,Y_2)=(X,g,Y_2)$.
In this case  we have base change equation: $\pi_y^*\circ \bar{i}_{k*}=i_{k*}\circ \pi_y^*$ and
$$j^x_*\circ \pi_y^*\circ\bar{i}_{k*}\circ \bar{p}_k^*=(j^x\circ i_k)_*\circ (\bar{p}_k\circ \pi_y)^*.$$

Finally note that the map $i_k\circ j^x$ is equal the map to $j^x\circ i_k$ and the map $\pi_y\times \pi_y\circ p_k$ is equal to the map
$\bar{p}_k\circ \pi_y$.

\subsection{Proof of proposition~\ref{prop: ind homo}}
We have to show that
$$ i_{k*}\circ p_k^*(\calF)\star i_{k*}\circ p_{k}^*(\calG)= i_{k*}\circ p_{k}^*(\calF\star\calG),$$
where $\calF=\calF_1\boxtimes\calF_2$, $\calG=\calG_1\boxtimes\calG_2\in \MF_{B_k^2}(\calX^\circ_2(G_k),W)\times\MF_{B_{n-k}^2}(\calX_2^\circ(G_{n-k}),W).$
The LHS of the later formula is the result of the application of the functor $\CE_{\frn^{(2)}_n}$ combined with extraction of  $T^{(2)}$ invariant part to the matrix factorization
$C$ such that:
\begin{multline}\label{eq: long mf prod}
C^\sharp=\tilde{\calF}\otimes [\Pi_{--}^k(X), \Pi_{++}^k(Y_1-\Ad_{g_{12}}(Y_2))]\otimes  \mathrm{K}^{\pi_{12}^*(\delta_k W_2)} (\Pi_{--}^{k}(g_{12}))\otimes\\
\tilde{\calG}\otimes [\Pi_{--}^k(\Ad_{g_{12}}^{-1}(X)), \Pi_{++}^k(Y_2-\Ad_{g_{23}}(Y_3))]\otimes  \mathrm{K}^{\pi_{23}^{\circ*}(\delta_k W_2)} (\Pi_{--}^{k}(g_{23})).
\end{multline}
%
with  $\tilde{\calF}= j_n^{k*}\circ \rho_{k}^*(\calF')$ and $\tilde{\calG}=j_n^{k*}\circ\rho_{k}^*(\calG')$ in conventions of subsection~\ref{ssec: comp ins}:
$j_n^{k*}(\calF')=p^*_{k}(\calF)$, $j_{n}^{k*}(\calG')=p^*_{k}(\calG)$.

Let us fix coordinates $X,g_{23},g_{13},Y_1,Y_2,Y_3$ on $\calX^\circ_3(G_n)$.
We use  the factor $[\Pi_{--}^k(X), \Pi_{++}^k(Y_1-\Ad_{g_{12}}(Y_2))]$ to exclude the variables $\Pi_{--}^k(X)$ from the other factors.
Since $\Pi_{--}^k(X)$  is $B_n^2$-equivariant, the exclusion process uses only the row transformations of the first
kind and we do not need to worry about the admissibility conditions.
The elimination process changes the factor in  (\ref{eq: long mf prod}) in the following manner:
$$\tilde{\calF}\to\tilde{\calF}_1:= \tilde{\calF}|_{\Pi_{--}^k(X)=0},\quad
     \tilde{\calG}\to\tilde{\calF}_1:= \tilde{\calF}|_{\Pi_{--}^k(\Ad_{g_{12}}^{-1}(X))=0}$$
$$  [\Pi_{--}^k(X), \Pi_{++}^k(Y_1-\Ad_{g_{12}}(Y_2))]\to  [\Pi_{--}^k(X), R_1],$$
$$ [\Pi_{--}^k(\Ad_{g_{12}}^{-1}(X)), \Pi_{++}^k(Y_2-\Ad_{g_{23}}(Y_3))]\to [0,\Pi_{++}^k(Y_2-\Ad_{g_{23}}(Y_3))],$$
$$ \mathrm{K}^{\pi_{12}^{\circ*}(\delta_k W_2)} (\Pi_{--}^{k}(g_{12}))
\to  \mathrm{K}^{\pi_{12}^*(\delta_k W_2)} (\Pi_{--}^{k}(g_{12}))|_{\Pi_{--}^k(X)=0},$$
$$  \mathrm{K}^{\pi_{12}^{\circ*}(\delta_k W_2)} (\Pi_{--}^{k}(g_{23}))\to  \mathrm{K}^{\pi_{23}^{\circ*}(\delta_k W_2)} (\Pi_{--}^{k}(g_{23}))|_{\Pi_{--}^k(X)=0}$$
We denote the resulting equivariant matrix factorization as $C_1$. The non-equivariant limit $C_1^\sharp$ is the product of the above listed terms. Moreover, since the only term in
$C_1^\sharp$ that depends on $\Pi_{--}^k(X)$ is the Koszul matrix factorization $[\Pi_{--}^k(X),R_1]$, we conclude that $R_1$ is the derivative of the potential $\pi_{13}^{\circ*}(W)$ by
variables $\Pi_{--}^k(X)$ which equals $\Pi_{++}^k(Y_1-\Ad_{g_{13}}(Y_3))$.

Next let us observe that using the row transformations of the first kind we can show that the following equivalence of the matrix factorization
 \begin{multline*}
 \mathrm{K}^{\pi_{12}^{\circ,*}(\delta_k W_2)} (\Pi_{--}^{k}(g_{12}))\otimes \mathrm{K}^{\pi_{23}^{\circ*}(\delta_k W_2)} (\Pi_{--}^{k}(g_{23}))|_{\Pi_{--}^k(X)=0}\\
 = \mathrm{K}^{\pi_{13}^{\circ*}(\delta_k W_2)} (\Pi_{--}^{k}(g_{13}))\otimes \mathrm{K}^{0} (\Pi_{--}^{k}(g_{23}))|_{\Pi_{--}^k(X)=0}
 \end{multline*}

 Finally, we use the Koszul complex $\mathrm{K}^{0} (\Pi_{--}^{k}(g_{23}))|_{\Pi_{--}^k(X)=0}$ to exclude the variables $\Pi_{--}^k(g_{23})$ from the product; for that
 we use the first type row transformations so we do not need to worry about the admissibility. The final result of the sequence of row transformations is the
 equivariant matrix factorizations $C_2$ such that its non-equivariant limit $C_2^\sharp$ is the product the following factors:
 $$\tilde{\calF}_2:= \tilde{\calF}|_{\Pi_{--}^k(X)=0,\Pi_{--}^k(g_{23})=0},\quad
     \tilde{\calG}_2:= \tilde{\calF}|_{\Pi_{--}^k(\Ad_{g_{12}}^{-1}(X))=0, \Pi_{--}^k(g_{23})=0}$$
 $$[\Pi^k_{--}(X),Y_1-\Ad_{g_{13}}(Y_3)],\quad     [0,\Pi_{++}^k(Y_2-\Ad_{g_{23}}(Y_3))],$$
$$ \mathrm{K}^{\pi_{13}^{\circ*}(\delta_k W_2)} (\Pi_{--}^{k}(g_{13}))\otimes \mathrm{K}^{0} (\Pi_{--}^{k}(g_{23}))|_{\Pi_{--}^k(X)=0}.$$
Since the potential $\pi_{13}^{\circ*}(\delta_k W_2)$ does not depend on $\Pi_{--}^{k}(X)$ the last product is actually just
 $$ \mathrm{K}^{\pi_{13}^{\circ*}(\delta_k W_2)} (\Pi_{--}^{k}(g_{13}))\otimes \mathrm{K}^{0} (\Pi_{--}^{k}(g_{23})).$$

 Let us introduce auxiliary space $\SymbolPrint{\calZ_{con}(G_n)}:=\frg_n\times P_k\times G_n\times \frn_n\times (\frn_k\times \frn_{n-k})\times \frn_n$
 and $\SymbolPrint{\calZ_{con}(P_k)}:=\frp_n\times P_k\times P_k\times \frn_n\times (\frn_k\times \frn_{n-k})\times \frn_n$. The variety $\calZ_{con}(G_n)$ is a subspace of
 $\calX^\circ_3$ defined by equations $\Pi_{++}^k(Y_2-\Ad_{g_{23}}(Y_3))=0$, $\Pi_{--}^k(g_{23})=0$, thus we can define the map $\pi_{13}^\circ:\calZ_{con}(G_n)\to
 \calX_2^\circ(G_n)$ by the restriction.

   We can contract  differentials in complexes $ [0,\Pi_{++}^k(Y_2-\Ad_{g_{23}}(Y_3))]$,
 $\Pi_{--}^{k}(g_{23})$ to obtain homotopy equivalence of matrix factorizations over $\pi_{13}^{\circ,*}(\CC[\calX^\circ_2(G_n)])$:
 $$ C_2\sim C_3,$$
 where $C_3\in \MF_{B_n^2}(\calZ_{con}(G_n),\pi_{13}^\circ(W))$. Moreover, since we only used admissible row transformations the matrix factorization $C_3$ has a
 structure of quasi-Koszul matrix factorization:
 $$ C_3=(\tilde{M}\otimes \Lambda^\bullet V, d^++D+d^-,\partial+\partial^-),$$
 where $V$ is space dual to span of equations $\Pi_{--}^k(X),\Pi_{--}^k(g_{13})$, $d^+$ is the corresponding Koszul differential
$d^+:\tilde{M}\otimes\Lambda^\bullet V\to \tilde{M}\otimes \Lambda^{\bullet-1} V,$ and $d^-: \tilde{M}\otimes\Lambda^\bullet V\to \tilde{M}\otimes \Lambda^{\bullet+1} V$
the completing differentials and \(\partial^-=\sum_{i>j} \partial_{ij}^-\), \[\partial_{ij}: \Lambda^i (\frn_n^2)\otimes \tilde{M}\otimes \Lambda^\bullet V\to
\Lambda^j (\frn_n^2)\otimes \tilde{M}\otimes \Lambda^{\ge\bullet} V, \quad (\tilde{M},D,\partial)\in \Mod_{per}^{B_n^3}(\calZ(G_n)).\]

The locus inside $\calZ_{con}(G_n)$ defined by the equations $\Pi_{--}^k(X)=0,\Pi_{--}^k(g_{13})=0$ is isomorphic to $\calZ_{con}(P_n)$, by construction
$$(\tilde{M},D,\partial)|_{\Pi_{--}^k(X)=0,\Pi_{--}^k(g_{13})=0}=p_k^*(\pi_{12}^{\circ*}(\calF)\otimes\pi_{23}^{\circ*}(\calG)),$$
where $p_k: \calZ_{con}(P_k)\to \calX^\circ_3(G_k)\times\calX^\circ_3(G_{n-k})$ is the natural projection.
Hence we can use uniqueness lemma~\ref{lem: uniq MF+eq} to imply that we have an isomorphism of $B_n^3$ equivariant matrix factorizations
$$ C_3=i_*(p_k^*(\pi_{12}^{\circ*}(\calF)\otimes\pi_{23}^{\circ*}(\calG)),$$
where $i$ is map $\calZ_{con}(P_k)\to \calZ_{con}(G_n)$.

The Lie algebra $\frn^{(2)}$ does not act on variables $\Pi^k_{--}(X)$ and $\Pi_{--}^k(g_{13})$ hence we have
$$ \CE_{\frn^{(2)}}(C_3)=i_*(\CE_{\frn^{(2)}}(p_k^*(\pi_{12}^{\circ*}(\calF)\otimes \pi_{23}^{\circ*}(\calG)))).$$

Let us denote by $I_k$ the kernel of the projection homomorphism $P_k\to B_k\times B_{n-k}$ and let $I'_k$ be its unipotent part. Hence we have the short exact sequence of
Lie algebras:
$$0\to \Lie{(I'_k)}\to \frn_n \to \frn_k\oplus \frn_{n-k}\to 0$$
and  we can use Hochschild-Serre spectral sequence to compute functor $\CE_{\frn_n^{(2)}}$:
$$ \CE_{\frn_{k}^{(2)}\oplus \frn_{n-k}^{(2)}}(\CE_{\Lie(I'_{k})}(C_3))\Rightarrow \CE_{\frn_n^{(2)}}(C_1). $$

Since the Lie algebra $\frn^{(2)}$ only acts on on the second copy of $P_k$ in the product
$\calZ_{con}(P_n)=\frp_n\times P_k\times P_k\times \frn_n\times (\frn_k\times \frn_{n-k})\times \frn_n$
and the differentials in $p_k^*(\pi_{12}^{\circ*}(\calF)\otimes \pi_{23}^{\circ*}(\calG))$ are also $I'_k$ invariant,
the computation of $\CE_{\Lie(I'_k)}(p_k^*(\pi_{12}^{\circ*}(\calF)\otimes \pi_{23}^{\circ*}(\calG)))$ reduces to the computation
of
$$\textup{H}_{\Lie}^*(\Lie(I'_k),\CC[\calZ_{con}(P_k)])=\CC[\frn_n\times P_k\times \frn_n\times (\frn_k\times \frn_{n-k})\times \frn_n]\otimes
\textup{H}_{\Lie}^*(\Lie(I'_k),\CC[P_k]).$$

The last computation is equivalent to the following

\begin{proposition}
$$\CE_{\Lie(I'_k)}(\CC[P_k])=\CC[G_k\times G_{n-k}],$$
where $G_k\times G_{n-k}$ is the subgroup of $P_k$ defined by the equation $\Pi_{++}^k(g)=0$.
\end{proposition}
\begin{proof}
The group $I_k$ acts on $P_k$ by right multiplication. The Lie algebra $\Lie(I'_k)$ is generated by the elements $E_{ij}$, $i\in [1,n]$, $j\in [k+1,n]$ and direct computation
shows that it is a commutative algebra.

Let us denote by $E^\vee_{kl}$ the generators of the algebra $\CC[P_k]$:
$E^\vee_{kl}(X)=X_{kl}.$ Let us denote by $\partial_{ij}$ the differential operator corresponding to the action $E_{ij}\in \Lie(I_k')$.
The action of these generators of $\Lie(I'_k)$ is given by:
$$ \partial_{ij}=\sum_{k+1\le l \le n} E_{jl}^\vee\frac{\partial}{\partial E_{il}^\vee}.$$

Let us observe that vector of differential operators $\partial_{i\bullet}$ is related to the vector of differential operators $\frac{\partial}{\partial E_{i\bullet}^\vee}$, $\bullet\in [k+1,n]$ by the matrix
$M=(E^\vee_{lm})_{l,m\in [k+1,n]}$ and this matrix is invertible. Thus the condition $\partial_{ij}(f)=0$ for all $i,j$ is equivalent to the condition
$\frac{\partial f}{\partial E_{ij}^\vee}=0$ for all $i,j\in [k+1,n]$.

Since the operator $\frac{\partial}{\partial E_{ij}^\vee}$ is surjective on $\CC[P_k]$ the statement of the proposition follows.
\end{proof}

The proposition implies that $\CE_{\Lie(I'_k)}(\CC[\calZ_{con}(P_k)])=\CC[\calZ_{con}(G_k\times G_n)]$ where $\calZ_{con}(G_k\times G_n):=\frp_n\times P_k\times G_k\times G_{n-k}\times \frn_n\times (\frn_k\times \frn_{n-k})\times \frn_n$.
Thus the proposition implies that
$$ \CE_{\frn^{(2)}}(C_3)=i'_*(\CE_{\frn_k\oplus \frn_{n-k}}((p'_k)^*(\pi_{12}^{\circ*}(\calF)\otimes\pi_{23}^{\circ*}(\calG)))),$$
where $i': \calZ_{con}(G_k\times G_{n-k})\to \calZ_{con}(G_n)$ and $p'_k:\calZ_{con}(G_k\times G_{n-k})\to \calX^{\circ}_3(G_k)\times \calX^{\circ}_3(G_{n-k}).$



Hence after extracting the torus invariant part of the last expression we obtain the formula from the statement of the proposition:
 $$ \pi_{13*}^{\circ}(\CE_{\frn^{(2)}}(C_3))^{T^{(2)}}=i^*_k\circ p_k^*( \calF\bar{\star}\calG).$$

\subsection{Transitivity of induction}
In this section we prove the following
\begin{proposition}\label{prop: ind trans} The following functors:
$$\MFs_k\times \MFs_{m-k}\times \MFs_{n-m}\to \MFs_{n},$$
$\MFs_r:=\MF_{G_r\times B_r^2}(\calX_2(G_r),W)$
 are isomorphic
$$ \ind_k\circ \mathrm{Id}\times \ind_{m-k}=\ind_m\circ \ind_k\times \mathrm{Id}.$$
\end{proposition}

Before we start the proof let us fix some notations. We denote by $\SymbolPrint{P_{k,m}}$ a subgroup of $G_n$ with the Lie algebra generated by
$E_{ij}$, $i\le j$ and $E_{i,i+1}$, $i\ne k,m$. Given an element $g\in P_{k,m}$ we denote by $g'$ its $k\times k$ block, by $g''$ its $(m-k)\times (m-k)$ block and
by $g'''$ its $(n-m)\times (n-m)$ block.
 We also use $\SymbolPrint{\frp_{k,m}}$ for $\Lie(P_{k,m})$

 We have natural homomorphisms:
$\SymbolPrint{i_{k,m}}: P_{k,m}\to G_n$ and $p_{k,m}: P_{k,m}\to G_k\times G_{m-k}\times G_{n-m}$.
Respectively, we define $\calX_2^\circ(P_{k,m}):=\frp_{k,m}\times P_{k,m}\times \frn_n^2$ and using previously defined maps we construct
maps $\SymbolPrint{p_{k,m}}: \calX_2^\circ(P_{k,m})\to \calX_2^\circ(G_k)\times \calX_2^\circ(G_{m-k})\times \calX_2^\circ(G_{n-m})$ and
$i_{k,m}:  \calX_2^\circ(P_{k,m})\to \calX_2^\circ(G_n)$.
Using these maps we can define the functor
 $$\SymbolPrint{\ind_{k,m}}: \MFs^\circ_{k}\times \MFs^\circ_{m-k}\times \MFs^\circ_{n-m}\to \MFs^\circ_{n},$$
$$
\ind_{k,m}:=i_{k,m,*}\circ p_{k,m}^*,$$
$\MFs^\circ_r:=\MF_{B^2_r}(\calX^\circ_2(G_r),W).$

Also there are also natural maps $p_k: P_{k,m}\to G_k\times P_{m-k}$, $p_m: P_{k,m}\to P_k\times G_{n-m}$ such that $p_{m-k}\circ p_k=p_{k,m}$ and $p_k\circ p_m=p_{k,m}$.

\begin{proof}[Proof of proposition~\ref{prop: ind trans}]
We  will show that
 $$\ind_k\circ \mathrm{Id}\times \ind_{m-k}=\ind_{k,m}=\ind_m\circ \ind_k\times \mathrm{Id}$$
Let us indicate how we imply the first equality, the second equality has a similar proof. Let us notice that LHS is the composition of several functors:
$i_{k*}\circ p_k^*\circ(\mathrm{Id}\times i_{m-k})_*\circ (\mathrm{Id}\times p_{m-k})^*$.  We have  the
following commuting diagram:
$$ \begin{tikzcd}
 \frp_{k,m}\times P_{k,m}\times \frn_n^2 \arrow[dotted]{r}{i_{m-k}}\arrow[dotted]{d}{p_k}   & \calX^\circ_2(P_k)\arrow{d}{p_k}\\
\calX_2^\circ(G_k)\times \calX_2^\circ(P_{m-k})\arrow{r}{1\times i_{m-k}} &                                  \calX^\circ_2(G_k)\times \calX^\circ_2(G_{n-k})
 \end{tikzcd}$$
where the dotted arrow maps are: $p_k(X,g,Y_1,Y_2)=(p_k(X),p_k(g),Y_1,Y_2)$ where $p_k:P_{k,m}\to G_k\times P_{m-k}$ is the natural homomorphism and
$i_{m-k}=i_{m-k}\times i_{m-k}\times \mathrm{Id} \times \mathrm{Id}$.

Finally let us observe that $p_{m-k}\circ p_k= p_{k,m}$ and $i_k\circ i_{m-k}=i_{k,m}$. Now we use the base change to show:
\begin{multline*}
\ind_k\circ \mathrm{Id}\times \ind_{m-k}= i_{k*}\circ p_k^*\circ (\mathrm{Id}\times i_{m-k})_*\circ (\mathrm{Id}\times p_{m-k})^*\\= i_{k*}\circ i_{m-k*}\circ p_k^*\circ p_{m-k}^*=
(i_{k}\circ i_{m-k})_*\circ (p_{m-k}\circ p_k)^*
=\ind_{k,m}.\end{multline*}
\end{proof}

Since the functor $\Phi_n$ provides an embedding of categories we get an immediate corollary:

\begin{corollary} The following functors from 
\[\MF_{B_k^2}(\calXr_2(G_k),\Wr)\times \MF_{B_{m-k}^2}(\calXr_2(G_{m-k}),\Wr)\times \MF_{B_{n-m}^2}(\calXr_2(G_{n-m}),\Wr)\] to \(\MF_{B_n^2}(\calXr_2(G_n),\Wr)\)
 are isomorphic
$$ \indb_k\circ \mathrm{Id}\times \indb_{m-k}=\indb_m\circ \indb_k\times \mathrm{Id}.$$
\end{corollary}


%% file: part4_gl.tex
\section{Inclusion functor and the generators for the braid group}\label{sec: inc fun gens}

This section collects various facts about the matrix factorization convolution algebra, mostly property that do not require intense computations but these properties will be used later.
In particular, the relation (\ref{eq: commuting}) implies that matrix factorizations corresponding commuting elementary braids commute inside the convolution algebra.

\subsection{Unit in convolution algebra}\label{sec: unit} In this section we describe the unit in the convolution algebra.
 Let us denote by $I_{--}\subset\CC[\calXr_2]$ the ideal generated by matrix coefficients of $(g)_{--}$ where $(X,g,Y)$ are coordinates on $\calXr_2$.
Since the generators of $I_{--}$ form a regular sequence
we can define $\bcalC_{\parallel}:=\mathrm{K}^{\Wr}(I_{--}).$ A direct computation shows that
$$\SymbolPrint{\calC_{\parallel}}:=\Phi_n(\bcalC_{\parallel})=\mathrm{K}^{W_1}(X_{--})\otimes \mathrm{K}^{W_2}(g_{--})\in \MF_{B_n^2}(\calX^{\circ}_2,W),$$
where $W_1=\Tr(X_{--}(Y_1-\Ad_g(Y_2)_{++})$, $W_2=\Tr(X(\Ad_g(Y_2)_{-})$.

\begin{proposition}\label{prop: unit} For any $\calF\in \MF_{B^2_n}(\calX^{\circ},W)$ we have
$$ \calC_{\parallel}\star \calF=\calF=\calF\star\calC_{\parallel}.$$
\end{proposition}
\begin{proof}
Let us prove the second equation, first is analogous.
Let us denote by $C$ the $B_n^3$-equivariant matrix factorization $\pi_{12}^{\circ*}(\calF)
\otimes \pi_{23}^{\circ *}(\calC_{\parallel})$. Let us fix the  coordinates $(X,g_{13},g_{23}, Y_1,Y_2,Y_3)$ on $\calX^{\circ}_3$. Since
$$C^\sharp=\mathrm{K}^{\pi_{23}^{\circ*}(W_1)}((g_{23})_{--})\otimes [\Ad_{g_{12}}(X)_{--},Y_2-\Ad_{g_{23}}(Y_3)]\otimes \pi_{12}^{\circ*}(\calF)$$
we can use the differential of the Koszul complex to eliminate variables $(g_{23})_{--}$. We use the row transformation of the first type so we do not need to worry about
admissibility.

 As the result of elimination we obtain matrix factorization $C_1$ which is the tensor product of two factors $C'_1\otimes C''_1$  where $C''_1=\pi_{12}^{\circ *}(\calF)|_{(g_{23})_{--}=0}$ and
 $C'_1$ has the following equivariant limit
 $$ (C'_1)^{\sharp}=[(g_{23})_{--},R_1]\otimes [Y_2-\Ad_{(g_{23})_+}(Y_3),R_2].$$
 Now the us notice that $\pi_{13}^{\circ *}(W)$ does not depends on $g_{23}$ and the only term in $C_1^\sharp$ that contains dependence on $(g_{23})_{--}$ is
 $[(g_{23})_{--},R_1]$ hence $R_1=0$.

 Let us introduce $\calX_{3}(B_n)=\frg_n\times G_n\times B_n\times \frn_n^3$.
 The variety $\calX_3(B_n)$ is isomorphic to the subvariety of $\calX_3$ defined by the equations $(g_{23})_{--}=0$.
 Now we can contract the complex $C_1$ along the differentials of the Koszul complex $[(g_{23})_{--},0]$ to show the homotopy equivalence of
 the $B_n^3$-equivariant  matrix factorizations over the ring $\pi_{13}^{\circ*}(\CC[\calX_2^{\circ}])$:
 $$C_1\sim C_2= [Y_2+\Ad_{g_{23}}(Y_3),R_2]\otimes \pi_{12}^{\circ*}(\calF)\in \MF_{B_n^3}(\calX_3(B_n),\pi_{13}^{\circ*}(W)),$$
 where the maps $\pi_{ij}^{\circ}$ are the restriction of the corresponding maps from $\calX_3$.

 As the next step we use the first factor in the matrix factorization $C_2$ to eliminate $Y_2$ from the rest of the factors. Since $Y_2+\Ad_{g_{23}}(Y_3)$ is $B_n^3$ equivariant we again only
 use the first type of row transformations. Similarly to the previous step, after elimination we obtain the matrix factorization
 $$ C_3=[Y_2+\Ad_{g_{23}}(Y_3),0]\otimes  \pi_{12}^{\circ*}(\calF)|_{Y_2=\Ad_{g_{23}}(Y_3)}.$$

 Now we can contract the complex along the differentials of the Koszul complex $[Y_2+\Ad_{g_{23}}(Y_3),0]$ to obtain the homotopy equivalence of
 of
 the$B_n^3$-equivariant  matrix factorizations over the ring $\pi_{13}^{\circ*}(\CC[\calX_2^{\circ}])$:
 \begin{equation*}
  C_3\sim C_4=\tilde{\pi}_{12}^{\circ^*}(\calF)\in \MF_{B_n^3}(\calX_2\times B_n,W).
  \end{equation*}
 where $\tilde{\pi}_{12}^{\circ}:\calX_2\times B_n\to \calX_2$ is given by $(X,g,Y_1,Y_2,b)\mapsto (X,Y_1,\Ad_{b}(Y_2))$.

 Thus we see that the differentials of $\tilde{\pi}_{12}^{\circ*}(\calF)$ are $B_n^{(2)}$ invariant and the computation of
 $\CE_{\frn_n^{(2)}}(C_4)$ reduces to the computation of
 $$\CE_{\frn_n^{(2)}}(\CC[\calX_2\times B_n])=\CC[\calX_2]\otimes \CE_{\frn_n^{(2)}}(\CC[B_n]).$$
 
Since \(Q_{k\dots n}=B_n\) for \(k=n\), the proposition~\ref{prop: Lie cohs} implies
 that $\CE_{\frn_n^{(2)}}(\CC[B_n])=\CC[T_n]$. Thus after taking $T_n^{(2)}$-invariant part of
 $\CE_{\frn_n^{(2)}}(C_4)$ we obtain the matrix factorization $\calF$.
\end{proof}



\begin{corollary} The matrix factorization $\bcalC_{\parallel}$ is the unit in the convolution algebra $$(\MF_{B_n^2}(\calXr_2,\Wr),\bar{\star}).$$
\end{corollary}

Let us also provide an alternative proof of the proposition. The alternative argument is more geometric and uses  change of base. Let us first notice that
$$ \SymbolPrint{\bcalC_{\parallel}}=i_{*}^{\parallel}(\calO),$$
where $\SymbolPrint{i^{\parallel}}: \frg_n\times B\times \frn_n\to \calX_2$ is the embedding
$(X,b,Y)\to (X,b,Y,\Ad_b(Y))$ and $\calO\in \MF_{B_n^3}(\frb\times B\times \frn_n^2,0)$ is the trivial matrix
factorization.
  Then we following diagram of maps:
  $$ \begin{tikzcd}
 \calX_2^{\circ}\times B \arrow[dotted]{r}{i^{\parallel}}\arrow[dotted]{d}{\pi_{12}^\circ\times\pi_{23}^\circ}   & \calX_3\arrow{d}{\pi_{12}^\circ\times\pi_{23}^\circ}\arrow{r}{\pi_{13}^{\circ}}&\calX^\circ_2\\
\calX_2^\circ\times \frg_n\times B\times \frn_n\arrow{r}{1\times i^{\parallel}} &                                  \calX^\circ_2\times\calX^\circ_2&
 \end{tikzcd}$$
 where the dotted maps are: $i^{\parallel}$ is the  inclusion  \[ i^{\parallel}(X,g,Y_1,Y_2)=(X,gb^{-1},b,Y_1,Y_2,\Ad_b(Y_2))\] and
 the restriction of the  map \(\pi_{12}^{\circ}\)
to the subspace $\calX_3(B)$: \[\pi_{12}^\circ\times \pi_{23}^\circ(X,g,Y_1,Y_2,b)=(X,g,Y_1,Y_2,X,b,Y_2).\]

Now let use base change property to simplify \( \calF\star\calC_\parallel\):
\begin{multline*}
\CE_{\frn_n^{(2)}}(\pi_{13*}^{\circ}((\pi_{12}^{\circ}\times\pi_{23})^*(\calF\boxtimes \calC_\parallel)))^{T^{(2)}}=
\CE_{\frn_n^{(2)}}(\pi_{13*}^{\circ}((\pi_{12}^{\circ}\times\pi_{23})^*(i_*^\parallel(\calF\boxtimes \calO))))^{T^{(2)}}\\=
\CE_{\frn_n^{(2)}}((\pi_{13}^{\circ}\circ i^\parallel)_*(\pi_{12}^{\circ}\times\pi_{23}^\circ)^*(\calF\boxtimes \calO))^{T^{(2)}}=\CE_{\frn^{(2)}}(\tilde{\pi}_{12}^{\circ*}(\calF))^{T^{(2)}}
\end{multline*}
The last step of this proof the same as in the previous proof.

\subsection{Inclusion functor}
In the proof of proposition~\ref{prop: ind trans} we constructed the functor $\ind_{k,m}$ and we can use this functor to define the functor:
$$\SymbolPrint{\Ind_{k\dots m}}: \MF_{ B_{m-k+1}^2}(\calX_2(G_{m-k+1}),W)\to \MF_{ B_n^2}(\calX_2(G_n),W),$$
by $\SymbolPrint{\Ind_{k\dots m}}(\calF):=\ind_{k,m}(\calC_{\parallel}\boxtimes\calF\boxtimes\calC_{\parallel})$.
Using Knorrer functor we also define the reduced version of the inclusion functor:
$$\SymbolPrint{\Indb_{k\dots m}}: \MF_{B_{m-k+1}^2}(\calXr_2(G_{m-k+1}),\Wr)\to \MF_{B_n^2}(\calXr_2(G_n),\Wr),$$

The results of the previous section  the properties of the unit immediately imply the following corollaries

\begin{corollary}\label{cor: ind} The functors $\Indb_{k\dots m}$ and $\Ind_{k\dots m}$ are homomorphism of the convolution algebras.
\end{corollary}

\begin{corollary}\label{cor: commute non red}For any $k,m,k',m'$ such that $k<m<k'<m'$ and
 \[\calF\in \MF_{B_{m-k+1}^2}(\calX_2(G_{m-k+1}),W),\]
\[\calG\in \MF_{ B_{m'-k'+1}^2}(\calX_2(G_{m'-k'+1}),W)\] we have
$$ \Ind_{k\dots m} (\calF)\star\Ind_{k'\dots m'} (\calG)=\Ind_{k'\dots m'} (\calG)\star\Ind_{k\dots m} (\calF).$$
\end{corollary}

\begin{corollary}\label{cor: commute}For any $k,m,k',m'$ such that $k<m<k'<m'$ and
 $$\calF\in \MF_{ B_{m-k+1}^2}(\calXr_2(G_{m-k+1}),\Wr),\quad
\calG\in \MF_{ B_{m'-k'+1}^2}(\calXr_2(G_{m'-k'+1}),\Wr)$$ we have
$$ \Indb_{k\dots m} (\calF)\bar{\star}\Indb_{k'\dots m'} (\calG)=\Indb_{k'\dots m'} (\calG)\bar{\star}\Indb_{k\dots m} (\calF).$$
\end{corollary}

\subsection{Generators of the braid group}\label{ssec:gens}

When $n=2$ the potential is $\Wr(X,g,Y)=\Tr(X\Ad_g(Y))$ 
where
$$ g=\begin{bmatrix} g_{11}&g_{12}\\ g_{21}& g_{22}\end{bmatrix},\quad X=\begin{bmatrix} x_{11}&x_{12}\\ 0& x_{22}\end{bmatrix},\quad Y=\begin{bmatrix} 0& y_{12}\\0&0\end{bmatrix}$$
A direct computation shows that it factors:
$$\Wr(X,g,Y)=y_{12}(g_{11}(x_{11}-x_{22})+g_{21}x_{12})g_{21}/\Delta,\quad \Delta=\det(g).$$

If \(n=2\) then the unit matrix factorization from section~\ref{sec: unit} is:
\[ \crX_{\parallel}=\Phi(\begin{bmatrix}g_{21}\Delta^{-1}& y_{12}(g_{11}(x_{11}-x_{22})+g_{21}x_{12})\end{bmatrix}),\]

Let also introduce the matrix factorizations $\crX_{+}$,  $\crX_{\bullet}$:
 \[\SymbolPrint{\crX_{\bullet}}=\Phi(\begin{bmatrix}y_{12}\Delta^{-1}& g_{21}(g_{11}(X_{11}-X_{22})+g_{21}X_{12})\end{bmatrix}),\]
\begin{equation}\label{eq: C+}
\SymbolPrint{\crX_{+}}=\Phi(\begin{bmatrix}g_{21}y_{12}\Delta^{-1}& (g_{11}(x_{11}-x_{22})+g_{21}x_{12})\end{bmatrix}).
\end{equation}
The complex $\crX_+$, as we will see later, correspond to the positive crossing of two strands in the braid.
To define the complex corresponding to the negative crossing we  need to shift the weights of $B^2$ action on $\crX_+$.

Let $\chi_1,\chi_2$ be the generators of the group of characters of $B$:
$$\SymbolPrint{\chi_1}\left(\begin{bmatrix}a&b\\0&c\end{bmatrix}\right)=a,\quad \SymbolPrint{\chi_2}\left(\begin{bmatrix}a&b\\0&c\end{bmatrix}\right)=c.$$
Bellow we use convention that $\calF\brwsh{\chi'}{\chi''}\in \MF_{B^2}(\calX_2(G_2),W)$  is obtained
from the complex $\calF$ by twisting the action of the first factor in $B^2$ by the character
$\chi'$ and the action of the second copy is twisted by the character $\chi''$. With these conventions in mind we have the
following formula for the complex of negative crossing.
\begin{equation}\label{eq: C-}
 \SymbolPrint{\crX_-}:=\crX_+\brwsh{-\chi_1}{\chi_2}.
\end{equation}

We postpone the description of the $T_{sc}$-equivariant structure of matrix factorizations from above till the
section~\ref{sec: two str}.
In the section~\ref{sec: two str} we show that the convolution of  $\crX_{+}$  and $\crX_-$  is homotopic to \(\crX_\parallel\).

To define the generators of $\Brgr_n$ for general $n$ we use our induction functor:
$$\SymbolPrint{\crX_{\pm}^{(i)}}:=\Ind_{ii+1}(\crX_{\pm}),\quad \crX_{\parallel}^{(i)}:=\Ind_{ii+1}(\crX_\parallel).$$

 It is elementary to check that $\crX_{\parallel}^{(i)}$ is independent of $i$ and is a unit in the convolution algebra (see next section for the proof).
The main result of the first part of the paper is the following theorem

\begin{theorem}\label{thm: braids} The complexes $\crX_{\pm}^{(i)}$ satisfy the braid relations
\begin{gather}
\crX_{+}^{(i)}\star\crX_{-}^{(i)}\sim\crX_\parallel,\label{eq: inverse}\\
\crX_{+}^{(i)}\star \crX_{+}^{(i+1)} \star \crX_{+}^{(i)}\sim \crX_{+}^{(i+1)}\star \crX_{+}^{(i)}\star \crX_{+}^{(i+1)},\label{eq: cubic}\\
\crX_{+}^{(i)}\star \crX_{+}^{(j)}\sim\crX_{+}^{(j)}\star \crX_{+}^{(i)},\quad |i-j|>1.\label{eq: commuting}
\end{gather}
where \(\sim\) is the homotopy equivalence.
\end{theorem}
\begin{proof}
The relation (\ref{eq: inverse}) is proven for two strand braids in lemma~\ref{thm: two strands}. Hence by corollary~\ref{cor: ind} the relation (\ref{eq: inverse}) holds in general.
The equation (\ref{eq: cubic}) is proven for three strand braids in corollary~\ref{cor: cubic}. Finally, the commuting relation (\ref{eq: commuting}) follows from
the corollary~\ref{cor: commute}.
\end{proof}

For a given element $\beta=\sigma_{i_1}^{\epsilon_1}\cdots\sigma_{i_\ell}^{\epsilon_\ell}$ we introduce the following notations:
$$ \calC_\beta=\calC_{\epsilon_1}^{(i_1)}\star\dots \calC_{\epsilon_\ell}^{(i_\ell)},\quad \bcalC_\beta=\bcalC_{\epsilon_1}^{(i_1)}\bar{\star}\dots \bar{\star}\bcalC_{\epsilon_\ell}^{(i_\ell)}.$$

\section{General properties of convolution algebra}\label{sec: aux sec}

\subsection{Auxiliary Lemma} In the proof of the braid relations and Markov moves we will use some generalization of the previous corollary which we explain below.
The statement that we prove in this section allows us to use smaller intermediate space for computation of the convolution.

Let us introduce notation $\SymbolPrint{Q_{k\dots m}}\subset G_n$ for subgroup with Lie algebra $\SymbolPrint{\frq_{k\dots m}}$ spanned by $E_{ij}$, $i\le j$, $i,j\in [1,n]$ and by $E_{ij}$, $i>j$,
$i,j\in [k,m]$. Respectively, we denote by $\SymbolPrint{\Pi_+^{\overline{k\dots m}}}$ the projection map  $\Lie(G_n)\to \frq_{k\dots m}$ and $\SymbolPrint{\Pi_{--}^{\overline{k\dots m}}}$ is defined by
$\mathrm{Id}=\Pi_+^{\overline{k\dots m}}+\Pi_{--}^{\overline{k\dots m}}$. Similarly we define $\SymbolPrint{\Pi_{++}^{\overline{k\dots m}}}(X)=\Pi_{--}^{\overline{k\dots m}}(X^t)^t$, where $t$ stands for transpose.  Let us also denote by $\SymbolPrint{\Pi^{\overline{k\dots m}}}$ the projection $\frg_n\to \frg_{m-k+1}$ that extracts the block of the matrix elements
with entries $ij$, $i,j\in [k,m]$. We use the same notation for the maps of the corresponding groups.

Let us introduce the space
$\SymbolPrint{\calX^\circ_3(G_n,G_{k\dots m})}=\frg_n\times G_n\times G_{k-m+1}\times \frn_n\times \frn_{k-m+1}\times \frn_n$. Let us also fix an embedding
$\SymbolPrint{i_{k\dots m}}$ of this space inside $\calX_3$:

\[ i_{k\dots m}(X,g_{12},g_{23},\vec{Y})=(X,g_{12},i_{k\dots m}(g_{23}),Y_1,i_{k\dots m}(Y_2)+\Pi_{++}^{\overline{k\dots m}}(\Ad_{g_{23}}(Y_3)),Y_3),\]
where \(\vec{Y}=(Y_1,Y_2,Y_3)\), \(i_{k\dots m}:G_{k-m+1}\to G_n\) and \(i_{k\dots m}: \frn_{k-m+1}\to \frn_n\) are the natural embeddings.

Respectively we define maps \[\pi_{12}^{\circ}:\calX^\circ_3(G_n,G_{k\dots m})\to \calX_2(G_{n}),\]
\[\pi_{23}^{\circ}:\calX^\circ_3(G_n,G_{k\dots m})\to \calX_2(G_{m-k+1}),\] \[\pi_{13}^{\circ}:\calX^\circ_3(G_n,G_{k\dots m})\to \calX_2(G_n),\]
\begin{equation}\label{eq: pi12}
\pi_{12}^{\circ}(X,g_{12},g_{23},Y_1,Y_2,Y_3)=(X,g_{12},Y_1,i_{k\dots m}(Y_2)+\Pi_{++}^{\overline{k\dots m}}(\Ad_{g_{23}}(Y_3)).
\end{equation}
\begin{equation}\label{eq: pi23}
\pi_{23}^{\circ}(X,g_{12},g_{23},Y_1,Y_2,Y_3)=(\Pi^{\overline{k\dots m}}( \Ad_{g_{12}}^{-1}(X)),g_{23},Y_2,\Pi^{\overline{k\dots m}}(Y_3)),
\end{equation}
$$\pi_{13}^{\circ}(X,g_{12},g_{23},Y_1,Y_2,Y_3)=(X,g_{12}g_{23},Y_1,Y_3).$$
These maps are the restrictions of the corresponding maps from $\calX_3$ on the embedded subvariety $\calX^\circ_3(G_n,G_{k\dots m})$.

Similarly, we define $\calX^\circ_3(G_{k\dots m},G_n)=\frg_n\times G_{k\dots m}\times G_n\times \frn_n\times \frn_{k-m+1}\times \frn_n$. Let us also fix an embedding $i_{k\dots m}$ of this space inside $\calX_3$:
$$ i_{k\dots m}(X,g_{12},g_{23},\vec{Y})=(X,i_{k\dots m}(g_{12}),g_{23},Y_1,i_{k\dots m}(Y_2)+\Pi_{++}^{\overline{k\dots m}}(\Ad^{-1}_{g_{12}}(Y_1)),Y_3),$$
\(\vec{Y}=(Y_1,Y_2,Y_3)\).
The maps $\pi_{ij}^{\circ}$ are defined analogously to the previous case: by the restriction from ambient space.

The spaces $\calX^\circ_3(G_n,G_{k\dots m})$ and $\SymbolPrint{\calX^\circ_3(G_{k\dots m},G_n)}$ have unique $B_n^3$ equivariant structure that makes the maps $\pi_{ij}^\circ$ $B_n^3$-equivariant.

\begin{lemma}\label{prop: small flag}  For any \[\calF\in \MF_{B_{m-k+1}^2}(\calX_2(G_{m-k+1}),W),\calG\in \MF_{B_n^2}(\calX_2(G_n),W)\] we have
$$ \calG\star\Ind_{k\dots m}(\calF)=\CE_{\frn_{m-k+1}^{(2)}}\pi_{13*}^\circ(\pi_{12}^{\circ*}(\calG)\otimes \pi_{23}^{\circ*}(\calF))^{T^{(2)}_{m-k+1}},$$
$$ \Ind_{k\dots m}(\calF)\star \calG=\CE_{\frn_{m-k+1}^{(2)}}\pi_{13*}^\circ(\pi_{12}^{\circ*}(\calF)\otimes \pi_{23}^{\circ*}(\calG))^{T^{(2)}_{m-k+1}},$$
\end{lemma}
The key observation in the proof is an alternative description of the induction functor $\Ind_{k\dots m}$. The space $\SymbolPrint{\calX^\bullet_2(Q_{k\dots m})}:=
\frg_n\times Q_{k\dots m}\times \frn_{m-k+1}\times \frn_n$ can be mapped into
$\calX_2(G_n)$ with the map:
\[ \SymbolPrint{j_{k\dots m}}(X,g,Y_1,Y_2)=(X,j_{k\dots m}(g),i_{k\dots m}(Y_1)+\Pi_{++}^{\overline{k\dots m}}(\Ad_g(Y_2)),Y_2)\]
where \(j_{k\dots m}: Q_{k\dots m}\to G_n\) is the natural embedding.

There is also a natural projection map $\Pi^{\overline{k\dots m}}:\calX^\bullet_2(Q_{k\dots m})\to \calX_2^\circ(G_{m-k+1})$ defined by:
$$ \Pi^{\overline{k\dots m}}(X,g,Y_1,Y_2)=(\Pi^{\overline{k\dots m}}(X),\Pi^{\overline{k\dots m}}(g),\Pi^{\overline{k\dots m}}(Y_1),\Pi^{\overline{k\dots m}}(Y_2)).$$
Both of these maps are $B_n^3$-equivariant. In particular, we have a commuting diagram:
  \[ \begin{tikzcd}
 \calX_2^{\bullet}(Q_{k\dots m}) \arrow[dotted]{r}{i'}\arrow[dotted]{d}{\Pi'_1\times\Pi^{\overline{k\dots m}}\times\Pi'_3}   & \calX^\circ_2(P_{k,m})\arrow{d}{p_{k,m}}\arrow{r}{i_{k,m}}&\calX''_{2,n}\\
 \calX'_{2,k}\times \calX''_{2,m-k+1}\times \calX'_{2,n-m} \arrow{r}{i^{\parallel}\times 1\times i^{\parallel}}\arrow{d}{\pi_2} &
 \calX''_{2,k}\times\calX''_{2,m-k+1}\times \calX''_{2,n-m}&\\
\calX^\circ_2(G_{m-k+1})&&
\end{tikzcd}\]
where \(\calX'_{2,k}:=\frg_k\times B_k\times \frb_k\times \frb_k=\calX^{\circ}_2(B_k)\), \(\calX''_{2,k}:=\frg_k\times G_k\times \frb_k\times\frb_k=\calX^{\circ}_2(G_k)\)
and the dotted arrow maps are:
\[i'(X,g,Y_1,Y_2)=(\Pi^{\overline{k\dots m}}_+(X),j_{k\dots m}(g),i_{k\dots m}(Y_1)+
  \Pi^{\overline{k\dots m}}_{++}(\Ad_g(Y_2)),Y_2),\]
\[\Pi'_i(X,g,Y_1,Y_2)=(p_{k,m}^{(i)}(X),p_{k,m}^{(i)}(g),p_{k,m}^{(i))}(\Ad_g(Y_2)),p_{k,m}^{(i)}(Y_2)),\]
here \(p_{k,m}^{(1)}\), \(p_{k,m}^{(3)}\) are the compositions of \(p_{k,m}\) with the projections on \(k\times k\) and \((n-m)\times(n-m)\) blocks.

Using the base change in our previous diagram we obtain
\begin{multline*}
\Ind_{k\dots m}=i_{k,m,*}\circ p_{k,m}^*\circ (i^{\parallel}\times 1\times i^{\parallel})\circ \pi_2^*=i_{k,m,*}\circ i'_*\circ (\Pi'_1\times\Pi^{\overline{k\dots m}}\times\Pi'_3)^*\circ\pi_2^*\\=
j_{k\dots m,*}\circ \Pi^{\overline{k\dots m},*}.
\end{multline*}

\begin{proof}[Proof of proposition~\ref{prop: small flag}]
We prove the first formula the second is analogous.
We have the following commuting diagram of the maps:
  $$ \begin{tikzcd}
 \calX_3^{\circ}(G_n,Q_{k\dots m}) \arrow[dotted]{r}{i_{k\dots m}}\arrow[dotted]{d}{\pi_{12}^\circ\times\pi_{23}^\circ}   & \calX^\circ_3\arrow{d}{\pi_{12}^\circ\times\pi_{23}^\circ}\arrow{r}{\pi_{13}^{\circ}}&\calX^\circ_2\\
\calX_2^\circ\times \calX_2^\bullet(Q_{k\dots m}) \arrow{r}{1\times j_{k\dots m}}\arrow{d}{1\times\Pi^{\overline{k\dots m}}} &                                  \calX^\circ_2\times\calX^\circ_2&\\
\calX^\circ_2(G_n)\times\calX^\circ_2(G_{m-k+1})&&
 \end{tikzcd}$$
where the dotted arrow map $\pi_{12}^\circ\times\pi_{23}^\circ$ is given by
$$ \pi_{12}^{\circ}(X,g_{12},g_{23},Y_1,Y_2,Y_3)=(X,g_{12},Y_1,i_{k\dots m}(Y_2)+\Pi_{++}^{\overline{k\dots m}}(\Ad_{g_{23}}(Y_3)),$$
$$\pi_{23}^{\circ}(X,g_{12},g_{23},Y_1,Y_2,Y_3)=(\Ad_{g_{12}}^{-1}(X),g_{23},Y_2,\Pi_{-}^{\overline{k\dots m}}(Y_3)).$$

Now we can use the base change formula to simplify \( \calG\star\Ind_{k\dots m}(\calF)\) :
\begin{multline*}
\CE_{\frn_n^{(2)}}(\pi^\circ_{13*}\circ (\pi^\circ_{12}\times \pi_{23}^\circ)^*\circ (\mathrm{Id}\times j_{k\dots m})_*\circ (\mathrm{Id}\times\Pi^{\overline{k\dots m}})^*(\calG\boxtimes \calF)))^{T^{(2)}_n}=\\
 \CE_{\frn_n^{(2)}}(\pi^\circ_{13*} \circ  (i_{k\dots m})_*\circ  (\pi^\circ_{12}\times \pi_{23}^\circ)^* \circ(\mathrm{Id}\times\Pi^{\overline{k\dots m}})^*(\calG\boxtimes \calF)))^{T^{(2)}_n}=\\
  \CE_{\frn_n^{(2)}}(\pi^\circ_{13*}\circ (\pi^\circ_{12}\times \pi_{23}^\circ)^*(\calG\boxtimes \calF)))^{T^{(2)}_n},
\end{multline*}
where in the last equation we use $ \pi^\circ_{13}\circ i_{k\dots m}=\pi_{13}^\circ$ and
 $\mathrm{Id}\times\Pi^{\overline{k\dots m}}\circ \pi^\circ_{12}\times \pi_{23}^\circ=\pi^\circ_{12}\times \pi_{23}^\circ$ where the maps in the product
$\pi^\circ_{12}\times \pi_{23}^\circ: \calX_3^\circ(G_n,Q_{k\dots m})\to \calX^\circ_2(G_n)\times\calX^\circ_2(G_{m-k+1})$ are given by the formulas (\ref{eq: pi12}), (\ref{eq: pi23}).

Let us denote by $C$ the matrix factorization $\pi^\circ_{13*}\circ (\pi^\circ_{12}\times \pi_{23}^\circ)^*(\calG\boxtimes \calF)$
For last step  we  use Hochschild-Serre spectral sequence to compute functor $\CE_{\frn_n^{(2)}}$:
$$ \CE_{\frn_{m-k+1}^{(2)}}(\CE_{\Lie(J'_{k\dots m})}(C))\Rightarrow \CE_{\frn_n^{(2)}}(C) $$
where $J_{k\dots m}$ is the kernel of the projection homomorphism $Q_{k\dots m}\to G_{m-k+1}$ and $J'_{k\dots m}$ is its unipotent part.  Let us abbreviate the quotient $J_{k\dots m}/J'_{k\dots m}$ by
$J/J'$.
We have
\begin{equation}\label{eq: CE}
 \CE_{\Lie(J'_{k\dots m})}(C)^{J/J'}=\pi^\circ_{13*}(\CE_{\Lie(J'_{k\dots m})}(\pi^\circ_{12}\times \pi_{23}^\circ)^*(\calG\boxtimes \calF)).
\end{equation}
On the other hand the formulas for the maps $\pi^\circ_{12}$ and $\pi^\circ_{23}$ imply that the differentials of the pull back $(\pi^\circ_{12}\times \pi_{23}^\circ)^*(\calG\boxtimes \calF)$
are invariant with respect to the action of $J'_{k\dots m}$. Hence to complete the computation of (\ref{eq: CE}) we need to  compute
\[ \mathrm{H}_{\Lie}^*(J'_{k\dots m},\CC[\calX^\circ_3(G_n,Q_{k\dots m})])^{J/J'}\] which is a tensor product of \(\CC[\frb\times G_n\times \frn_{m-k+1}\times \frn_n]\) and \(\mathrm{H}_{\Lie}^*(J'_{k\dots m},\CC[Q_{k\dots m}])^{J/J'},\)
and in proposition below we show that  $\CE_{J'_{k\dots m}}(\CC[Q_{k\dots m}])^{J'/J}=\CC[G_{m-k+1}].$
Hence the statement follows since $T^{(2)}_n=(J/J')\times T_{m-k+1}$ where $T_{m-k+1}$ is the maximal torus of $G_{m-k+1}\subset G_n$.

\end{proof}

\begin{proposition}\label{prop: Lie cohs} For any $k,m$, $1\le k<m\le n$ we have
$$\CE_{\Lie(J'_{k\dots m})}(\CC[Q_{k\dots m}])=\CC[G_{m-k+1}]\otimes \CC[J/J'].$$
\end{proposition}
\begin{proof}
We show slightly more general statement
$$\CE_{\Lie(J'_{k\dots m})}(\CC[Q_{k'\dots m'}])=
\CC[Q_{k-k'+1\dots m-k'+1}\subset G_{m-k+1}]\otimes \CC[(J/J')],$$
where $k\ge k'$, $m'\le m$.

We use induction by $m-k$ and the Hochschild-Serre Spectral sequence (\cite{Wei}, sect 7.5). Since we have short exact sequence of groups
$$ 1\to J'_{k+1,\dots, m}\longrightarrow J'_{k\dots n}\longrightarrow \CC^{m-k+1}\to 1,$$
the spectra sequence tells us that
$$\CE_{\CC^{m-k+1}}(\CE_{\Lie\left(J'_{k+1\dots m}\right)}(\CC[Q_{\k'\dots m'}]))\Rightarrow \CE_{J_{k\dots m}}(\CC[Q_{k'\dots m'}]).$$
By induction we have
$\CE_{\Lie(J'_{k+1\dots m})}(\CC[Q_{k'\dots m'}])=\CC[Q_{k-k'+2\dots m-k'+1}\subset G_{m-k+2}]$.
After relabeling matrix indices our statement is equivalent to the computation of the  following cohomology
$\mathrm{H}^*_{\Lie}(\CC^{m-k+1},\CC[Q_{k-k'+2\dots m-k'+1}])$. The ring
of polynomial functions on $Q_{k-k'+2\dots m-k'+1}$ is generated by the
elements $E_{kl}^\vee$    
dual to the matrix units $E_{kl}$ of $E$. The generators
$\partial_i$ of the commutative algebra $\CC^{m-k+1}$ act on the ring generators by the formula:
$$\partial_i E^\vee_{kl}=\delta_{k1} E_{il}^\vee.$$
We can apply invertible linear transformation to our system of differential operators to obtain new system of generators:
$$\tilde{\partial}_k=\sum_{i=2}^{m-k+2} (-1)^{i+k}\Delta_{ik}\partial_k,$$
where $\Delta_{ik}\in \CC[Q_{k-k'+2\dots m-k'+1}]$ is the determinant of $(m-k)\times (m-k)$ obtained from $E$ by removing columns and rows containing $11$ and $ik$ entries.
Then  we have:
$$\tilde{\partial}_i E^\vee_{kl}=\delta_{k1}\delta_{il} E_{11}^\vee.$$
and by Poincare lemma we obtain
$$\mathrm{H}^*_{\Lie}(\CC^{m-k+1},\CC[Q_{k-k'+2\dots m-k'+1}])=\CC[Q_{k-k'+2\dots m-k'+1}]/(E_{12}^\vee,\dots,E_{1,m-k+2}^\vee)$$ and the
statement follows.
\end{proof}

\subsection{Reduced case version of the auxiliary lemma }


Let us introduce \[\SymbolPrint{\calXr^{\circ}_2(G_{k\dots m})}=\frb\times G_{k\dots m} \times \frn^2,\]
\[\SymbolPrint{\calXr_3(G_n,G_{k\dots m})}:=\frb_n\times G_n\times G_{m-k+1}\times \frn_n.\] Let us also define an embedding
$\SymbolPrint{\bar{i}_{k\dots m}}:\calXr_3(G_n,G_{k\dots m})\to \calXr_3$ by
$$ \bar{i}_{k\dots m}(X,g_{12},g_{23},Y)=(X,g_{12},i_{k\dots m}(g_{23}),Y).$$  Using the last embedding we define the maps $\SymbolPrint{\bar{\pi}_{ij}}$ by restricting from the ambient $\calX_3$:
$$ \bar{\pi}_{12},\bar{\pi}_{13}: \calXr_3(G_n,G_{k\dots m})\to \calXr_2(G_n),\quad \bar{\pi}_{23}: \calXr_3(G_n,G_{k\dots m})\to \calXr_2(G_{k\dots m}).$$

Similarly, we define $\calXr_3(G_{k\dots m},G_n)$ together with the embedding $\bar{i}_{k\dots m}: \calXr_3(G_{k\dots m}, G_n)\to \calX_3$:
$$ \bar{i}_{k\dots m}(X,g_{12},g_{23},Y)=(X,i_{k\dots m}(g_{12}),g_{23},Y).$$
The same way as before we define the corresponding maps $\bar{\pi}_{ij}$.

The spaces $\calXr_3(G_{k\dots m},G_n)$ and $\calXr_3(G_n,G_{k\dots m})$ have natural $B_n\times B_{m-k+1}\times B_n$-equivariant structure. Moreover, the maps $\bar{\pi}_{ij}$ are $B_n^2$-equivariant but not $B_{m-k+1}$-equivariant.
However, we can introduce the $B_n\times B_{m-k+1}\times B_n$-equivariant structure on the pull-back $(\bar{\pi}_{12}\times \bar{\pi}_{23})(\calF\otimes\calG)$ where $\calF,\calG$ are equivariant matrix factorizations
from the appropriate spaces. More precisely, repeating the argument of the proposition~\ref{prop: Knorrer factor} we obtain

\begin{proposition}\label{prop: small prod 1}  Let
  \[\calF\in \MF_{B_n^2}(\calXr_2(G_n),\Wr), \quad\calG\in \MF_{B_{m-k+1}^2}(\calXr_2(G_{m-k+1}),\Wr)\] then there is a unique matrix factorization
$$(\bar{\pi}_{12}\otimes_{B_{m-k+1}}\bar{\pi}_{23})^*(\calF\boxtimes \calG)\in \MF_{B_n\times B_{m-k+1}\times B_n}(\calXr_3(G_n,G_{m\dots k}),\bar{\pi}_{13}^*(\Wr)),$$
such that its $B_n^2$-equivariant specialization  has properties:
$$(\bar{\pi}_{12}\otimes_{B_{m-k+1}}\bar{\pi}_{23})^*(\calF\boxtimes \calG)^\sharp=(\bar{\pi}_{12}\times\bar{\pi}_{23})^*(\calF\boxtimes \calG),$$
and
$$\CE_{\frn_{m-k+1}}(\pi_{13*}((\bar{\pi}_{12}\otimes_{B_{m-k+1}}\bar{\pi}_{23})^*(\calF\boxtimes \calG)))^{T_{m-k+1}^{(2)}}=\calF\bar{\star}\Indb_{k\dots m}(\calG).$$
\end{proposition}

\begin{lemma}\label{prop: small prod 2}  Let \[\calF\in \MF_{B_{m-k+1}^2}(\calXr_2(G_{m-k+1}),\Wr),\quad\calG\in \MF_{B_{m-k+1}^2}(\calXr_2(G_n),\Wr)\] then there is a unique matrix factorization
$$(\bar{\pi}_{12}\otimes_{B_{m-k+1}}\bar{\pi}_{23})^*(\calF\boxtimes \calG)\in \MF_{B_n\times B_{m-k+1}\times B_n}(\calXr_3(G_{m\dots k},G_n),\bar{\pi}_{13}^*(\Wr)),$$
such that its $B_n^2$ limit has properties:
$$(\bar{\pi}_{12}\otimes_{B_{m-k+1}}\bar{\pi}_{23})^*(\calF\boxtimes \calG)^\sharp=(\bar{\pi}_{12}\times\bar{\pi}_{23})^*(\calF\boxtimes \calG),$$
and
$$\CE_{\frn_{m-k+1}}(\pi_{13*}((\bar{\pi}_{12}\otimes_{B_{m-k+1}}\bar{\pi}_{23})^*(\calF\boxtimes \calG)))^{T_{m-k+1}^{(2)}}=\Indb_{k\dots m}(\calF)\bar{\star}\calG.$$
\end{lemma}

\subsection{Reduced convolution with rank $1$ matrix factorizations} \label{ssec: conv rk 1}
The most important case of these theorems is the case $m=k+1$, that is when one of the matrix factorizations in
$\bar{\star}$ product is lifted from the space $\calXr_2(G_2)$, we call such matrix factorization rank $1$ matrix factorization. In this case the pull back functor $(\bar{\pi}_{12}\times_{B_{m-k+1}}\bar{\pi}_{23})^*$ could be made very explicit.
We do not provide the details of the computation, we essentially just follow the method of proposition~\ref{prop: Knorrer factor}.

Let us fix notations. Let $\calF=(M_1,D_1,\partial'_l,\partial'_r)\in \MF_{B_n\times B_2}(\calXr_2(G_n),\Wr)$, $\calG=(M_2,D_2,\partial''_l,\partial''_r)\in \MF_{B_2\times B_n}(\calXr_2(G_2),\Wr)$, $B_2$ in both cases
is the image of the embedding $i_{k,k+1}: B_2\to B_n$.  Since the Lie algebra $\frn_2$ is one-dimensional  $\frn_2=\langle e \rangle$ the corresponding $\frn_2$-equivariant structure is given by the
maps the even maps $$\partial'_r\in Hom_{\CC[\calXr_2(G_n)]}(M_1, M_1), \partial''_l\in Hom_{\CC[\calXr_2(G_2)]}(M_2, M_2).$$ Respectively
\begin{equation}\label{eq: pi12Bpi23 1}
 (\bar{\pi}_{12}\otimes_{B_{m-k+1}}\bar{\pi}_{23})^*(\calF\boxtimes \calG)=(\bar{\pi}_{12}^*(M_1)\otimes\bar{\pi}_{23}^*(M_2),D_1+D_2;\partial'_l,\partial'_r+\partial''_l+\partial_{rl},\partial'_r),
 \end{equation}
where the differential $\partial_{rl}$ is given by the formula
\begin{equation}\label{eq: pi12Bpi23 2}
\partial_{rl}=\bar{\pi}_{12}^*\left(\frac{\partial D_1}{\partial Y_{k,k+1}}\right)(\tilde{Y}^2_{k+1,k+1}-\tilde{Y}^2_{k,k})+ \bar{\pi}_{23}^*\left(\frac{\partial D_2}{\partial X_{kk}}-\frac{\partial D_2}{\partial X_{k+1,k+1}}\right) \tilde{X}^2_{k+1,k},
\end{equation}
where $\tilde{X}^2=\Ad_{g_{12}}^{-1}(X)$ and $\tilde{Y}^2=\Ad_{g_{12}}(Y)$.

In the other case  \[\calF=(M_1,D_1,\partial'_l,\partial'_r)\in \MF_{B_n\times B_2}(\calXr_2(G_2),\Wr),\]
\[ \calG=(M_2,D_2,\partial''_l,\partial''_r)\in \MF_{B_2\times B_n}(\calXr_2(G_n),\Wr)\] and the expressions for
$(\bar{\pi}_{12}\times_{B_{m-k+1}}\bar{\pi}_{23})^*(\calF\boxtimes \calG)$ is given by the same formulas (\ref{eq: pi12Bpi23 1}), (\ref{eq: pi12Bpi23 2}).



\begin{remark} As one can see, if we work with $B_2$-equivariant matrix factorizations on $Z$ with $B_2$-invariant potential then all the technicalities that are discussed in the first few sections are
unnecessary since $\MFs_{B_2}^{str}(Z,F)=\MF_{B_2}(Z,F)$.
\end{remark}

\subsection{Involution} Let $w_0\in G_n$ be the longest element of the Coxeter group $S_n$: $(w_0)_{i,j}=\delta_{i+j,n+1}$.   The involution $\SymbolPrint{\mathrm{D}}: G\to G$ defined by
$\mathrm{D}(g)=w_0 (g^{-1})^t w_0$ preserves $B$ and the Lie version of involution $\mathrm{D}(X)=-w_0 X^t w_0$, \(X\in \frg=\textup{Lie}(G)\) preserves $\frb\subset \frg$. Moreover, since $Tr(\mathrm{D}(X))=-Tr(X)$ the
have automorphism of category \[\mathrm{D}:\quad \MF_{B^2}(\calX_2(G),W)\to \MF_{B^2}(\calX_2(G),W),\]
\[ \mathrm{D}:\quad \MF_{B^2}(\calXr_2(G),\Wr)\to \MF_{B^2}(\calXr_2(G),\Wr).\]
The following proposition is almost tautological

\begin{proposition} Automorphism $\mathrm{D}$ respects convolution product:
$$\mathrm{D}(\calF)\star \mathrm{D}(\calG)=\mathrm{D}(\calF\star\calG),\quad \mathrm{D}(\calF)\bar{\star} \mathrm{D}(\calG)=\mathrm{D}(\calF\bar{\star}\calG),$$
and
$$ \mathrm{D}(\calC^{(i)}_{\pm})=\calC^{(n-i)}_{\pm},\quad \mathrm{D}(\bcalC^{(i)}_{\pm})=\bcalC^{(n-i)}_{\pm}.$$
\end{proposition}
\begin{proof}
  Let us show that \(\mathrm{D}\) induces an automorphism of the category of matrix factorizations. We define the action on the objects
  as
  \[\calF=(M,\tilde{D},\partial_l,\partial_r)\mapsto (M,\mathrm{D}\circ \tilde{D}\circ \mathrm{ D},
    \mathrm{D}\circ \partial_l\circ \mathrm{D},\mathrm{D}\circ\partial_r\circ \mathrm{D}).\]
  Since \(W\circ \mathrm{D}=W\) the last formula sends objects of \(\MF_{B^2}(\calX_2(G),W)\) to the objects of \(\MF_{B^2}(\calX_2(G),W\circ\mathrm{D})=
  \MF_{B^2}(\calX_2(G),W)\). The action on the space of morphisms  is also defined by conjugation with \(\mathrm{D}\):
  \[ Hom(\calF,\calF') \owns \Psi\mapsto \mathrm{D}\circ \Psi\circ \mathrm{D}.\]
  Thus is constructed an automorphism of the category.

  The induction functors \(\Ind_{k,m},\ind_{k,m}\) are defined with pull-backs and push-forwards along the maps \(p_{k,m},i_{k,m}\)
  Hence the rest of the proposition follows from the observation that \(\mathrm{D}\circ i_{k,m}= i_{k',m'}\circ \mathrm{D}\), \(
  \mathrm{D}\circ p_{k,m}= p_{k',m'}\circ \mathrm{D}\),
  \(k'=n+1-k,m'=n+1-m\).
\end{proof}

%% file: part5_gl.tex
\section{Two strands}\label{sec: two str}

In this section we study the properties of the matrix factorizations that generate the braid group on two strands. In particular we show that $\calC_+$ is the inverse to $\calC_-$.

\subsection{General method of the proof}
It is easier to discuss the reduced convolution and we  focus on this operation.
From construction of the reduced convolution $\bar{\star}$ we see that the key step of computation of  $\calG\bar{\star}\calF$ is the "computation" of
$\CE_{\frn^{2}}(\bar{\pi}_{12}^*\otimes_B \bar{\pi}_{23}^*)(\calG\boxtimes\calF)$. The latter complex is  a matrix factorization with non-trivial potential,
we can not compute its cohomology in conventional sense, since the differential does not square to $0$. However, we can get quite close to that by contracting
the part of the complex with non-trivial action of $\frn^{(2)}$. We give details below.

To simplify notations we consider problem of simplifying of the complex $\CE_{\frn}(\calF)^T$ where $\calF\in \MF_B(G\times Z, F)$, $F\in \CC[G\times Z]^B$ and $B$ acts on $G$ by the left multiplication
and acts trivially on $Z$.
Let $R=\CC[G]$ and $R\langle \chi\rangle$ is the ring with $B$-equivariant structure twisted by the character $\chi\in \Hom(T,\CC^*)$:
$b\cdot 1=\chi(b) 1$ for $b\in B$ and $1\in R\langle\chi\rangle$ is the generator of $R$-module
$R\langle \chi\rangle$.   We assume  that the matrix factorization that we study is the sum of free $R$-modules $\calF=(M,D)$
$M=\oplus_{i=1}^N M_i$ and $M_i=\CC[Z]\otimes R\langle \chi^i\rangle$. It is well known that the CE homology of such free modules has a geometric interpretation,
we were unable to pinpoint a proof of the interpretation and provide an outline of a  proof for it

\begin{proposition}\label{prop: homs G/B} For any \(\chi\in \Hom(T,\CC^*)\)
  $$ \textup{H}^*_{\Lie}(\frn,R\langle -\chi\rangle)^T=\textup{H}^*(G/B,L_\chi),$$
  where LHS is the coherent sheaf cohomology.
\end{proposition}
\begin{proof}
  Let \(S_n\) be the symmetric group  and \(B_-\subset G=GL_n\) the subgroup of lower triangular matrices.
  The Bruhat decomposition \(G=B S_n B\) implies   \(G=B_- S_n B\) because \(B_-=w_0 B w_0\) with \(w_0\in S_n\) being the longest element of \(S_n\).
  Each subset \(G_w=B_- w B\subset G\) is an open \(B\)-invariant subset of \(G\) and let us denote by \(R_w\) the ring of regular functions on \(G_w\).
  Respectively, we denote by \(G_{S}:=\cap_{w\in S} G_{w}\) where \(S\subset S_n\) is a finite subset and we use notation \(R_S\) for the ring of regular
  functions on \(G_S\). In particular, we have \(R=R_\emptyset\).
  
  There are natural inclusion maps \(R_S\to R_{S'}\) for \(S\subset S'\) and these inclusions maps with appropriate sings form a Cech  resolution \(\check{C}_\bullet(R)\) of
  \(R\) by free \(\frn\)-modules \(\check{C}_i(R)=\sum_{|S|=i} R_S\). On other hand let us observe that the decomposition above turns in a  Cech cover \(\check{C}_\bullet(G/U)\) of
  \(G/U\), \(G/U=\cup_{w\in S_n} B_-w\), we use same notation for the resolution of the sheaf of regular functions on \(G/U\) induced by the cover.
  
  Next we observe that \(R_w=\CC[B_-]\otimes \CC[U]\), \(\mathrm{Lie}(U)=\frn\) and \(\frn\) does
  not act on the first factor. Since \(H^*_{Lie}(\frn,\CC[U])=\CC\) we see that \(H^*_{Lie}(\frn,R)=\check{C}_*(G/U)\). The last isomorphism is \(T\)-equivariant hence the proposition
  follows.

\end{proof}

LHS of the last equation is the total space cohomology of the line bundle $L_\chi$. The fiber of $L_\chi$ at $V^\bullet=\{V^1\subset \dots\subset V^n \}\in G/B$ is equal to
$\bigotimes_i (V^{i+1}/V^i)^{d_i}$ where $\chi=\sum d_i \alpha_i$. The most important case of this theorem is the case $n=2$:
$$\textup{H}^*_{\Lie}(\frn, R\langle -k\alpha \rangle)=\textup{H}^*(\mathbb{P}^1,\calO(k)).$$

The other way to present previous proposition is to say that we use homotopy equivalence $\CE_{\frn}(R\langle\chi\rangle)^T\sim \oplus_i H^i(G/B,L_\chi)$ where the RHS is complex with zero differential.
The homotopy could be constructed explicitly by combining Chech resolution technique with the Gauss elimination method from below.

\begin{proposition} Let $\calF=(M,D)\in \MFs_B(G\times Z,F)$ as above then there is a homotopy equivalence
$$ \CE_{\frn}(\calF)^T\sim (\oplus_{i=1}^N H^*(G/B, L_{\rho_i}), \bar{D})\in \MFs(Z,F),$$
where $\bar{D}$ is the differential induced by the differential $D$.
\end{proposition}
\begin{proof}
  As in the proof of proposition~\ref{prop: homs G/B} we can construct the Cech resolution \(\check{C}_\bullet(M)\) of \(M\). Since the Chevalley-Eilenberg differential \(d_{CE}\)
  is acyclic on term of the \(\check{C}_\bullet(M)\) we apply Gauss elimination to construct homotopy
  \(\CE_{\frn}(\calF)\sim (H^0_{Lie}(\frn,\check{C}_\bullet(M),D')\) with \(D'\) induced by \(D\).
  As it is explained in proposition~\ref{prop: homs G/B} \(H^*(H^0(\frn,\check{C}_\bullet(M)),d_{\check{C}})=H^*(G/B,M)=\oplus_i H^*(B/G,L_{\rho_i})\)
  hence we again can apply the Gauss elimination and obtain the homotopy from the statement of the proposition.
  \end{proof}

  The Gauss elimination lemma is the standard tool of homological algebra, one can find more discussion (and applications) of this lemma in \cite{BN} in the case of complexes
  of vector spaces. Our proof below is very close to the argument of \cite{BN}.

\begin{lemma} Suppose \(M= W\oplus V\), \(V=V^0\oplus V^1\), \(W=W^0\oplus W^1\) and \[D^{i,i+1}=d_{VV}^{i,i+1}+d_{VW}^{i,i+1}+d_{Wv}^{i,i+1}+d_{WW}^{i,i+1},\]
  \(d_{AB}^{i,i+1}\in Hom(B^{i+1},A^i)\) form the matrix factorization \((M,D)\in \textup{MF}(Z,F)\), \(D=D^{01}+D^{10}\).
  Let us also assume \(d^{10}_{VV}=\phi\) is an isomorphism. 

 Then there is a homotopy equivalence in the category of matrix factorizations
 \[ (M,D)\sim (W,D_W),\quad D_W=D_W^{10}+D_W^{01},\]
 \[D_W^{10}=d_{WW}^{10}-d_{WV}^{10}\circ\phi^{-1}\circ d_{VW}^{10},\quad D_W^{01}=d_{WW}^{01}.\]
\end{lemma}
\begin{proof}
  Let us that \((W,D_W)\in \textup{MF}(Z,F)\). First, let us first compute \(D^2_W|_{W^0}\):
  \[d_{WW}^{01}\circ (d_{WW}^{10}-d_{WV}^{10}\circ\phi^{-1}\circ d_{VW}^{10})=d_{WW}^{01}\circ d_{WW}^{10}+d^{01}_{WV}\circ d^{10}_{VV}\circ \phi^{-1}\circ d_{VW}^{10}=F,\]
  where we used \(d_{WW}^{01}\circ d_{WV}^{10}+d_{WV}^{01}\circ d_{VV}^{10}=0\) for the first equality. The check that \(D^2_W|_{W^1}=F\) is similar.

  Next we define morphisms \(i:W\rightarrow M\) and \(\pi:M\rightarrow W\):
  \[ i=
    \begin{bmatrix}
      1\\
      \psi
    \end{bmatrix},\quad
   \pi= \begin{bmatrix}
      1&\rho
    \end{bmatrix},
  \]
  here \(\psi=\psi_{VW}^{00}+\psi_{VW}^{11},\rho=\rho_{WV}^{00}+\rho_{WV}^{11}\)  and \(\psi_{VW}^{00}=-\phi^{-1}\circ d_{VW}^{10}\), \(\psi_{VW}^{11}=0\),
  \(\rho_{WV}^{00}=0\), \(\rho_{WV}^{11}=-d_{WV}^{10}\circ \phi^{-1}\).

  It is immediate to see that \( D\circ i=i\circ D_W\) and \(D_W\circ \pi=\pi\circ D\). Also since \(\psi\circ \pi=0\) and \(\pi\circ\psi=0\) we get \(\pi\circ i=1\).
  On the other hand we can extend by \(0\) the map \(\phi^{-1}\) to obtain \(\phi'\in \oplus_i\Hom(V^{i},V^{i+1})\) and
  \[\chi=
    \begin{bmatrix}
    \phi'&0\\0&0  
    \end{bmatrix}.
  \]
  Since \(\phi^{-1}\circ d_{VV}^{10}=1\) and \(d_{VV}^{10}\circ \phi^{-1}=1\) we see that \(i\circ \pi-1=\chi\circ D+D\circ \chi\)
\end{proof}

Thus to obtain the reduction $\bar{D}$ from the lemma we apply previous lemma to our complex $\CE_{\frn}(\calF)^T$. To keep exposition clear let us explain one would apply
the elimination lemma to construct homotopy $(\oplus_{i=0}^N C_i, d_\bullet)\sim (H^\bullet(C,d),0)$: the case of matrix factorizations from proposition~\ref{prop: homs G/B} differs from this model case
by feature that on every step of our construction we need to compute correction differentials.

Indeed, in our model case we can construct a non-canonical splitting
$C_0=\textup{Im} d_1\oplus H^0(C,d)$, $C_1=\Ker d_1\oplus C_1/\Ker d_1$. The differential $d_1$ provides an isomorphism between $\textup{Im} d_1$ and $C_1/\Ker d_1$ and we can apply elimination lemma to
obtain homotopy equivalent complex $(C',d)\oplus (\textup{H}^0(C,d),0)$: $C_i'=C_i$, $i>1$, $C'_1=\Ker d_1$. Now we can apply the case method to the complex $(C',d)$ and so on.


To simplify notations, we assume for the rest of the section that \(G=G_2=SL_2\), \(B=B_2\subset G_2\), \(\frb=\textup{Lie}(B_2)\). 
Respectively, character group of \(B_2\) under this assumption  is \(\ZZ\), \(1\in\ZZ\) corresponds to the character \(\chi_1\) from the section~\ref{ssec:gens}

\subsection{Equivariant structure for the generators}
In this subsection we discuss the equivariant structure of the generators of the braid group. We also introduce the matrix factorization
$\calC_\bullet$ that corresponds to the blob or fat point in the MOY calculus \cite{MOY}. We use this connection later to relate our homology theory of
links to the HOMFLY-PT polynomial.

We use the following parameterization of the variety $\calX^\circ_2(G)=\frg\times G\times \frn^2$:
\begin{align*}
\xX &= \xmtr{ \SymbolPrint{\xz} & \SymbolPrint{\xo} \\ \SymbolPrint{\xmo} & -\xz },&
\xYi & = \xmtr{0 & \SymbolPrint{\yi} \\ 0 & 0 }\quad i=1,2,&
\xg &=\xmtr{ \SymbolPrint{\exaoo} & \SymbolPrint{\exaot} \\ \SymbolPrint{\exato} & \SymbolPrint{\exatt} }.
\end{align*}
The tori $T^{(1)}$, $T^{(2)}$ of the Borel groups in the product $B^2$ are $\CC^*$ and the weights of their action are
in the table:
\[
\begin{tabu}{| c ||r |r |r |r |r | r | r | r | r | r|}
\hline
\text{variable} &
\exaoo &\exaot & \exato &\exatt & \xmo & \xz & \xo &\txz & \yo & \yt
\\
\hline
T^{(1)} &
1 & 1 & -1 & -1 & -2 & 0 & 2 & 1 & 2 & 0
\\
\hline
T^{(2)} &
-1 & 1 & -1 & 1 & 0 & 0 & 0 & -1 & 0 & 2
\\
\hline
\end{tabu}.
\]

The generator of the nilpotent part of $\Lie(B)$, $\frn=\langle \delta\rangle$ acts on the coordinates by
\begin{align*}
\dlto\exaoo &= \exato,&\dlto\exaot & = \exatt,&\dlto\xz & = \xmo,&\dlto\xo & = -2\xz,
\\
\dltt\exaot & = -\exaoo,&\dltt\exatt&=-\exato, 
\end{align*}
In the table above and in the proofs in this section we use following functions of the coordinates:
\begin{align*}
\SymbolPrint{\txz} & = 2\xz \exaoo + \xo\exato, & \SymbolPrint{\ttxz} & = 2 \xz \exaot + \xo \exatt,
\end{align*}
so that
\begin{align*}
\dlto\txz & = 2\xmo\exaoo, & \dlto \ttxz &= 2\xmo \exaot, & \dltt\ttxz & = -\txz.
\end{align*}%
%
%

The matrix factorization was discussed previously but in this section we need an explicit description of this matrix factorization as a strongly
$B^2$-equivariant matrix factorization. We use  Koszul matrix factorization conventions of the subsection~\ref{ssec: str eq Kosz}, so
the matrix factorization of the identity braid has the form
\begin{equation*}
\crXpr =
\begin{bmatrix}
\xmo & \yo - \yt\exaoo^2 & \theta_1
\\
\yt\txz
& \exato & \theta_2
\end{bmatrix},\quad
\dlto\theta_1 = -2\yt\exaoo\theta_2,
\end{equation*}
with $\dgTsc\theta_1=(2,1)$ and $\dgTsc\theta_2=(0,-1)$.
The blob matrix factorization has the form
\begin{equation*}
\crXbl=
\begin{bmatrix}
\xmo & \yo - \yt\exaoo^2 & \theta_1
\\
\exato \txz
& \yt & \theta_2'
\end{bmatrix},\quad
\dlto\theta_1 = -2\exato\exaoo\theta_2'
\end{equation*}
or equivalently
\begin{equation*}
\crXbl=\begin{bmatrix}
\xmo & \yo & \theta'_1
\\
-\exaoo^2\xmo+
\exato \txz
& \yt & \theta_2'
\end{bmatrix},\quad \theta'_1=\theta_1 +\exatt^2\theta_2',\quad
\dlto\theta'_1 = 0
\end{equation*}

The matrix factorization of the positive intersection is 
\begin{equation*}
\crXp= (\sha\shq\sht)^{-1}
\begin{bmatrix}
\xmo & \yo - \yt\exaoo^2 & \theta_1
\\
\txz & \exato\yt & \theta_2
\end{bmatrix}
,\quad
\dlto\theta_1 = -2\exaoo\theta_2,
\end{equation*}
while
\begin{equation*}
\crXn = \sha\,\crXp\brwsh{-1}{-1}.
\end{equation*}


\subsection{Properties of the basic matrix factorizations}
In this subsection we collect various facts about the matrix factorizations from the previous subsection. In particular we explain how one can interpret
the positive intersection matrix factorization as a cone of the identity matrix factorization and the blob matrix factorization.

\begin{proposition} We have the following homotopy
\[
\crXpr\sim \crXpr\brwsh{1}{-1}.
\]
\end{proposition}
\begin{proof} Let \(R=\CC[\calXr_2]\) then we have
 the following diagram


\begin{equation*}
\begin{tikzcd}[column sep=3cm, row sep=1.75cm]
&
\rgR
\arrow[xshift=-0.35ex]{r}{\yt\txz}
\arrow[xshift=-0.35ex]{d}{\exatt}
\arrow[dashed, bend left=40]{r}{\exaot}
&
\rgR
\arrow[xshift=-0.35ex]{l}{\exato}
\arrow[xshift=-0.35ex]{d}{\exatt}
&
\\
\vphantom{1mm}
&
\rgR
\arrow[phantom,pos=0.02]{l}{\Bigl(}
\arrow[xshift=-0.35ex]{r}{\yt\txz}
\arrow[xshift=-0.35ex]{u}{\exaoo}
\arrow[bend right=40,dashed,']{r}{\exaot}
&
\rgR
\arrow[phantom,pos=0.22]{r}{\Bigr)\brwsh{1}{-1}}
\arrow [xshift=-0.35ex]{l}{\exato}
\arrow[xshift=-0.35ex]{u}{\exaoo}
&
\vphantom{1mm}
\end{tikzcd}
\end{equation*}

Here dashed arrows denote homotopy morphisms and we used the relation
$\exaoo\exatt - \exaot\exato = 1$.
\end{proof}
\begin{proposition}\label{prop: cones} We have following homotopies:
\begin{align*}
\crXp & \sim(\sha\shq\sht)^{-1}\Bigl(\crXpr \xrightarrow{\;\xchp\;}\crXbl\brwsh{-1}{-1} \Bigr), &
\crXn & \sim
\sha(\shq\sht)^{-1}
\Bigl(\crXbl\brwsh{-1}{-1} \xrightarrow{\;\xchm\;}(\shq\sht)^2\,\crXpr\brwsh{-1}{1} \Bigr),
\end{align*}
where


\begin{equation*}
\begin{tikzcd}[column sep=3cm, row sep=1.75cm]
\crXpr
\arrow{d}{\xchp}
& \sht^{-1}\rgRmomo
\arrow[xshift=-0.35ex]{r}{-\exato}
\arrow{d}{1}
&
\rgR
\arrow[xshift=-0.35ex]{l}{-\yt\txz}
\arrow{d}{-\txz}
\\
\crXbl\arrow[phantom,pos=0.22]{r}{\brwsh{-1}{-1}}
\arrow{d}{\xchm}
& \rgR \arrow[xshift=-0.35ex]{r}{\exato\txz}
\arrow{d}{-\txz}
&
\shq^2\sht\rgRmomo
\arrow[xshift=-0.35ex]{l}{\yt}
\arrow{d}{1}
\\
(\shq\sht)^2\,\crXpr \arrow[phantom,pos=0.22]{r}{\brwsh{-1}{1}}
& \shq^2\sht\,\rgR \arrow[xshift=-0.35ex]{r}{-\exato}
&
\shq^2\sht^2\,\rgRmomo
\arrow[xshift=-0.35ex]{l}{-\yt\txz}
\end{tikzcd}
\end{equation*}

\end{proposition}


\subsection{Inversion of the elementary positive braid}

Main result of this section is the following

\begin{theorem}\label{thm: two strands} Inside convolution algebra \((\MF_{G\times B^2}(G_2),\star)\) we have
\[
\crXp \star \crXn \sim \crXpr
\]
\end{theorem}


\subsection{Proof of theorem \ref{thm: two strands}}

\def\yh{ y_3 }

\def\exav#1{a_{#1}}
\def\exaoo{\exav{11}}
\def\exaot{\exav{12}}
\def\exato{\exav{21}}
\def\exatt{\exav{22}}
\def\exbv#1{b_{#1}}
\def\exboo{\exbv{11}}
\def\exbot{\exbv{12}}
\def\exbto{\exbv{21}}
\def\exbtt{\exbv{22}}
\def\excv#1{c_{#1}}
\def\excoo{\excv{11}}
\def\excot{\excv{12}}
\def\excto{\excv{21}}
\def\exctt{\excv{22}}

\def\xXtl{ \tilde{\xX} }
\def\dgBs{ \deg_{\mflg^2}}
\def\dltv#1{ \delta_{#1}}
\def\dlto{ \dltv{1} }
\def\dltt{ \dltv{2}}

\def\brwsh#1#2{ \langle #1,#2 \rangle }


In order to compute a convolution of two objects, we introduce the matrices $\gxot=||a_{ij}||$, $\gxth=||b_{ij}||$ and $g_{13} = ||c_{ij}||$ so that
\begin{align}\label{eqn: diffs}
\dlto\exaoo&=\exato, &\dlto\exaot &= \exatt, &\dltt\exaot&=-\exaoo,&\dltt\exatt &= -\exato,
\\
\dltt\exboo&=\exbto, &\dltt\exbot &=\exbtt,&\dlth\exbot&=-\exboo,&\dlth\exbtt&=-\exbto,
\\
\dlto\excoo &= \excto, &\dlto\excot &= \exctt, &\dlth\excot&=-\excoo,&\dlth\exctt &=-\excto.
\end{align}
We also introduce $\xX'=\Adv{\gxot^{-1}}\xX$:
\begin{align*}
\xmo' &=
-2\exaoo \exato \xz - \exato^2\xo + \exaoo^2\xmo
\\
\xz' & = (\exaot\exato + \exaoo\exatt)\xz + \exato\exatt\xo - \exaoo\exaot\xmo, \\
\xo' & = 2\exaot\exatt\xz + \exatt^2 \xo - \exaot^2 \xmo.
\end{align*}
For computation of convolution $\calC_+\star\calC_-$ we need to analyze the pull-back  matrix factorization:
\begin{gather*}
\calC_{+-}:=\pi_{12}^{\circ*}(\crXp) \otimes \pi_{23}^{\circ*}( \crXn) =
\begin{bmatrix}
\xmo & \yo - \yt\exaoo^2 & \theta_1
\\
2\xz \exaoo + \xo\exato & \exato\yt & \theta_2
\\
\xmo' & \yt - \yh\exboo^2 & \theta_3
\\
2\xz' \exboo + \xo'\exbto & \exbto\yh & \theta_4
\end{bmatrix}'
\\
\dlto\theta_1 = -2\exaoo\theta_2,\qquad
\dltt\theta_3 =-2\exboo\theta_4,
\end{gather*}
 we used a shortcut notation $[-]' = [-]\brhsh{0}{-1}{-1}$, and $ [-]\brhsh{0}{-1}{-1}$ stands for the shift of the action of the torus $(\CC^*)^3=T^{(1)}\times T^{(2)}\times T^{(3)}$.
%


Next we get rid of $\yt$, that is, we need to find
a matrix factorization on $\frg\times G\times \frn$ that is homotopy equivalent to $\calC_{+-}$ as
the matrix factorization over $\pi_{13}^{\circ *}(\CC[\calX_2^\circ])$. We use the row transformations together with the contraction of the injective differential to do that. The result of the row transformation is the LHS of the equation below:
\begin{equation*}
\begin{bmatrix}
\xmo & \yo - \yh\exaoo^2\exboo^2 & \theta'_1
\\
2\xz \exaoo + \xo\exato & \exato\exboo^2\yh & \theta'_2
\\
0 & \yt - \yh\exboo^2 & \theta_3
\\
2\xz' \exboo + \xo'\exbto & \exbto\yh & \theta_4
\end{bmatrix}' \sim
\begin{bmatrix}
\xmo & \yo - \yh\exaoo^2\exboo^2 & \theta'_1
\\
2\xz \exaoo + \xo\exato & \exato\exboo^2\yh & \theta'_2
\\
2\xz' \exboo + \xo'\exbto & \exbto\yh & \theta_4
\end{bmatrix}',
\end{equation*}
where
\begin{align*}
\theta'_1 & = \theta_1 +\exaoo^2\theta_3, &
\theta'_2 & = \theta_2 - \exato\theta_3,
\end{align*}
so that
\begin{align*}
\dlto \theta'_1 &= -2\exaoo\theta'_2,
&
\dltt\theta'_1 & = -2\exboo\exaoo^2\theta_4,&
\dltt\theta'_2 =  2\exboo\exato\theta_4.
\end{align*}
After performing two row operations with the help of the relation
\begin{equation*}
2\xz' \exboo + \xo'\exbto =
-\xmo(\exboo\exaoo\exaot+\excoo\exaot)
+\excto(2\xz\exaot + \xo\exatt) + \exboo\exatt(2\xz\exaoo + \xo\exato).
\end{equation*}
we find
\begin{equation*}
\calC_{+-} \sim
\begin{bmatrix}
\xmo & \yo - \yh\excoo^2 & \theta''_1
\\
2\xz \exaoo + \xo\exato & \excto\exboo\yh & \theta''_2
\\
(2\xz\exaot + \xo\exatt)\excto & \exbto\yh & \theta_4
\end{bmatrix}'
=
\begin{bmatrix}
\xmo & \yo - \yh\excoo^2 & \theta''_1
\\
\txz & \excto\exboo\yh & \theta''_2
\\
\ttxz\,\excto & \exbto\yh & \theta_4
\end{bmatrix}'
\end{equation*}
where
\begin{align*}
\theta''_1&=\theta'_1-(\exboo\exaoo\exaot+\excoo\exaot)\theta_4
,&
\theta''_2&=\theta'_2 +\exatt \exboo\theta_4,
\end{align*}
so that
\begin{align}\label{eq: diff CE twoline}
\dlto\theta_1'' & = -2\exaoo\theta_2'' - 2 \exaot\excto\theta_4, &
\dltt\theta_2'' & = \excto \theta_4.
\end{align}

The Borel group $B^{(2)}$ acts trivially on the first line of the last matrix factorization thus to complete our computation of $\calC_+\star\calC_-$ we
  need to analyze the complex composed of the last two lines of the above Koszul complex:

$$
 C=\begin{bmatrix}
\txz & \excto\exboo\yh & \theta''_2
\\
\ttxz\,\excto & \exbto\yh & \theta_4
\end{bmatrix}
$$

More precisely, we need to study the $\CC^*$-invariant part of the complex $\CE_\frn^{(2)}(C)$. First let us notice that the ring $R:=\CC[\calXr_3]$ factors as
$R'\times \bar{\pi}_{13}^*(\CC[\calXr_2])$ where $R'=\CC[G]$ and the second factor is $B^{(2)}$-invariant and thus  $\CE_\frn^{(2)}(C)=\CE_\frn^{(2)}(C')\otimes \bar{\pi}_{13}(\calXr_2)$
where $C'$ is the matrix factorization $C$ consider as the matrix factorization over $R'$.
Let us depict the complex  $\CE_\frn^{(2)}(C')^{T^{(2)}}$ below

$$
\begin{tikzcd}[row sep=scriptsize, column sep=scriptsize]
&\left[ \theta''_2;0\right]\arrow[dl] \arrow[rr,shift left=0.5ex] \arrow[dd, dashed, two heads] & & \left[1;1\right] \arrow[dl] \arrow[dd, dashed,blue]\arrow[ll,shift left=0.5ex,blue] \\
\left[\theta''_2\theta_4;1\right]\arrow[ur, shift left=1ex,red]  \arrow[dd, dashed, red] & & \left[\theta_4;2\right] \arrow[ur,shift right=1ex]\\
& \left[\theta''_2;-2\right] e^* \arrow[dl] \arrow[rr, shift left=0.5ex] &  & \left[1;-1\right] e^* \arrow[dl]\arrow[ll, shift left=0.5ex] \\
\left[\theta''_2\theta_4;-1\right]e^* \arrow[ur,shift right=1ex] \arrow[rr, shift left=0.5ex] & & \left[\theta_4,0\right]e^*\arrow[ur, shift right=1ex, blue]
\arrow[ll, shift left=0.5ex,red]\\
 \arrow[from=2-3, to=4-3, crossing over, dashed, hook]
\arrow[to=4-3, from=1-2, crossing over,dashed,green]
\arrow[from=2-1, to=2-3, shift left=0.5ex,crossing over]
\arrow[to=2-1, from=2-3, shift left=0.5ex,crossing over]
\end{tikzcd}
$$


where $e^*$ is the basis of $\frn^*=\Hom(\frn,\CC)$ and  $[\alpha;k]$ stands for $R[k]\alpha$ which  is the part of the ring of $\CC^*$ weight $k$ . In the picture the dashed arrows are the Chevalley-Eilenberg  differentials for $\cohog{*}{\frn,\bullet}$.

Note that $\cohog{*}{\frb,R\langle 1\rangle}=0$ hence the blue and red dashed arrows in the diagram are isomorphisms and we can contract them but some new morphisms would appear after we perform the contraction.

Indeed, if we contract red  dashed arrow we obtain four correction morphisms. The first correction morphism $d_1$ is obtained by following blue arrows (dashed arrow in reverse direction) starting at $R[0]\theta_4 e^*$ and ending at $R[0]\theta''_2$:
$$d_1=b_{21}y_3\delta^{-1}_2(2x_0a_{11}+x_1a_{21})=b_{21}y_3(2x_0a_{12}+x_1a_{22}).$$
Thus we obtained the differential that goes in the direction opposite of the diagonal green dashed  arrow.
The another corrections differentials have the following origins and destinations:
$$d_2: R[0]\theta_4e^*\to R[2]\theta_4,\quad d_3: R[-2]\theta''_2e^*\to R[0]\theta''_2,\quad d_4: R[-2]\theta''_2e^*\to R[2]\theta_4.$$

Similarly, contraction of the blue dashed arrow results yields four correction morphisms. The first correction morphism $d_5$ is obtained by following red arrows (dashed arrow in reverse direction) starting at at $R[0]\theta_4 e^*$ and ending at $R[0]\theta''_2$:
$$d_5=(2x_0a_{11}+x_1a_{21})\delta_2^{-1}(b_{21}y_3)=b_{11}y_3(2x_0a_{11}+x_1a_{21}).$$
Thus we obtained another differential that goes in the direction opposite of the direction of diagonal dashed green arrow.
The another corrections differentials have the following origins and destinations:
$$d_6: R[0]\theta_4e^*\to R[2]\theta_4,\quad d_7: R[-2]\theta''_2e^*\to R[0]\theta''_2,\quad d_8: R[-2]\theta''_2e^*\to R[2]\theta_4.$$

Finally, let us notice that since $\cohog{0}{\frb,R\langle 2\rangle}=\cohog{1}{\frb,R}=0$,   $\cohog{1}{\frb,R\langle 2\rangle}=\cohog{0}{\frb,R}=\CC[c,x,y]$ we see that Chevalley-Eilenberg differential: $d_{10}:R[2]\theta_4\to R[0]\theta_4 e^*$ is
injective and Chevalley-Eilenberg differential: $d_{11}:R[0]\theta''_2\to R[-2]\theta''_2 e^*$ is surjective. Thus we can perform contraction along these differentials and obtain two term complex consisting of   $\cohog{1}{\frb,R\langle 2\rangle}=\CC[c,x,y]$,
$\cohog{0}{\frb,R}=\CC[c,x,y]$ and differentials between them. Let us point out that the differentials $d_2,d_3,d_4,d_6,d_7,d_9$ do not result in any correction differential after the contraction since none of these differentials ends on the
target of the arrows $d_{10},d_{11}$.

Thus the only differentials in our final complex are the following. The differential that starts at $\cohog{1}{\frb,R\langle 2\rangle}$ and ends at $\cohog{0}{\frb,R}$ is the sum of $d_2$ and $d_5$:
$$d_2+d_5=y_3(2x_0c_{11}+x_1c_{21}).$$
The differential going in the reverse direction is the green arrow differential caused by (\ref{eq: diff CE twoline}) that is it is just a multiplication by $c_{21}$. Thus
we have shown that
$$\CE_{\frn^{(2)}}(\calC_{+-})^{T^{(2)}}\sim \calC_\parallel.$$

\section{Three strands and two crossings}\label{sec: three str 2}
In this section we compute explicitly the matrix factorizations $\calC_1\bar{\star}\calC_2$ and $\calC_2\bar{\star}\calC_1$. Both of these matrix factorizations
turn out to be Koszul matrix factorizations and we start with explanations for the equations defining these Koszul factorizations. Then we proceed with
computations of the convolution.

\subsection{A geometric interpretation for the matrix factorization}
Let us fix notations:
$$X=\begin{bmatrix} \SymbolPrint{x_{11}}& \SymbolPrint{x_{12}}& \SymbolPrint{x_{13}}\\ 0&
  \SymbolPrint{x_{22}}& \SymbolPrint{x_{23}}\\ 0& 0& \SymbolPrint{x_{33}}\end{bmatrix},\quad
a=\begin{bmatrix} a_{11}&a_{12}&\SymbolPrint{a_{13}}\\a_{21}&a_{22}&\SymbolPrint{a_{23}}\\ \SymbolPrint{a_{31}}&\SymbolPrint{a_{32}}&\SymbolPrint{a_{33}}\end{bmatrix}$$
We assume in this section that all group elements have determinant $1$.

\begin{proposition}\label{prop: geom 3 cross MF} If the the  following equations hold
\begin{gather}
\exaho = 0
\label{eq: eq.rl21}
\\
\label{eq: eq.rl22}
 (\xoo - \xtt) \exaoo + \xot\exato = 0
 \\
 \label{eq: eq.rl23}
 \def\aiv#1#2{ (a^{-1})_{#1#2}}
\def\aioo{\aiv{1}{1}}
\def\aiot{\aiv{1}{2}}
\def\aioh{\aiv{1}{3}}
\def\aito{\aiv{2}{1}}
\def\aitt{\aiv{2}{2}}
\def\aith{\aiv{2}{3}}
\def\aiho{\aiv{3}{1}}
\def\aiht{\aiv{3}{2}}
\def\aihh{\aiv{3}{3}}
\aihh(\xoo-\xhh) -\aiht\xth -\aiho\xoh=0.
 \end{gather}
then
$$ \Ad^{-1}_a(X)\in \frb.$$

\end{proposition}
\begin{proof}
The first two equations (\ref{eq: eq.rl21},\ref{eq: eq.rl22}) imply
$((X-x_{22}\mathrm{I}_3)a)_{\bullet 1}=0$ and hence
\begin{equation*}
(a^{-1}(X-x_{22}\mathrm{I}_3)a)_{\bullet 1}=0.
\end{equation*}

The equation $(a^{-1}(X-x_{11}\mathrm{I}_3))_{3\bullet}=0$ is equivalent to the
system consisting of the equation (\ref{eq: eq.rl23}) and  equation:
\begin{equation}\label{eq: eq aux}
 (a^{-1})_{31}x_{12}+(a^{-1})_{32}(x_{22}-x_{11})=0.
\end{equation}
Now let us observe that if $a_{31}=0$ then $(a^{-1})_{31}=a_{12}a_{32}$ and $(a^{-1})_{32}=-a_{11}a_{32}$ thus
the equation (\ref{eq: eq aux}) is equal to $a_{32}$ times (\ref{eq: eq.rl22}) modulo $a_{31}$ and the statement follows.
\end{proof}

Let us denote the LHS of (\ref{eq: eq.rl22}) by $f$ and LHS of (\ref{eq: eq.rl23}) by $g$.

\begin{proposition} The elements $a_{31},f,g$ form  a regular sequence in $\CC[\frb\times G]$.
\end{proposition}
\begin{proof}
We use the Bruhat cover of the group to show it.
Indeed, if $a_{31}=0$ then either $a_{11}\ne 0$ or $a_{21}\ne 0$. Let's treat the case $a_{11}\ne 0$ first.
then either $(a^{-1})_{32}\ne 0$ and
$a_{31}, f/a_{11}, g/(a^{-1})_{32}$ is regular
or $(a^{-1})_{33}\ne 0$ and $a_{31}, f/a_{11}, g/(a^{-1})_{33}$ is regular.
The second case $a_{21}\ne 0$ is analogous since then either $(a^{-1})_{31}\ne 0$ and
$a_{31}, f/a_{21}, g/(a^{-1})_{31}$ is regular or $(a^{-1})_{33}\ne 0$ and
$a_{31}, f/a_{21}, g/(a^{-1})_{33}$ is regular. Since we show regularity in each affine chart the statement follows.
\end{proof}

Let us point out that the proof of the proposition~\ref{prop: geom 3 cross MF}  implies that the ideal $I'_{12}=(a_{31},f,g)\subset \CC[\frg\times G]$  defines the subvariety
that consists of pairs $(X,g)$ such that
$$ \Ad_g(X)-\textup{diag}(X_{22},X_{33},X_{11})\in \frn.$$
The later matrix equation if $B^2$-invariant hence the ideal $I'_{12}$ is $B^2$-invariant too.

The ideal $I_{12}=(a_{31},f,g)\subset \CC[\calXr_2]$ contains the potential $\Wr$  and because of regularity of the sequence  and of $B^2$-invariance we have, according to the Lemma~\ref{lem: ext MF eq}
the following well-defined matrix factorization
$$\bcalC_{12}=\mathrm{K}^{\Wr}(I_{12})\in \MF_{B^2}(\calXr_2,\Wr).$$

Indeed, the regularity of \((a_{31},f,g)\) implies that the Koszul complex \(\mathrm{K}(a_{31},f,g)\) is homotopic to \(\CC[\calXr_2]/I_{12}\) supported in the zeroth homological degree. Thus the conditions of
the Lemma~\ref{lem: ext MF eq} are satisfied and there is a unique matrix factorization \(\mathrm{K}^{\Wr}(I_{12})\) that extends \(\mathrm{K}(a_{31},f,g)\).

Let us define the following elements of $\CC[\frb\times G]$:
$$ h=(a^{-1})_{32} x_{23}+(a^{-1})_{33}(x_{33}-x_{22}),\quad k=a_{11}(x_{11}-x_{33})+a_{21}x_{12}+a_{31}x_{13}.$$

Proof of the proposition below is parallel to the propositions above and we leave them for reader to check
\begin{proposition}We have
\begin{enumerate}
\item The conditions $(a^{-1})_{31}=0, h=0, k=0$ imply $\Ad^{-1}_{a}(X)\in \frb$.
\item The sequence $(a^{-1})_{31}, h, k$ is regular in $\CC[\frb\times G]$.
\end{enumerate}
\end{proposition}

Let $I_{21}=((a^{-1})_{31},h,k)\subset \CC[\calXr_2]$ the ideal.
Then analogously to the previous case we can define complex
$$ \bcalC_{21}:=\mathrm{K}^{\Wr}(I_{21})\in \MF_{B^2}(\calXr_2,\Wr).$$

The main result of this section is the formula for the matrix factorization for the braid on three strands with two crossings:
\begin{lemma} \label{lem: two cross}We have
$$ \bcalC^{(1)}_+\bar{\star}\bcalC^{(2)}_+=\bcalC_{12},\quad \bcalC^{(2)}_+\bar{\star}\bcalC^{(1)}_-=\bcalC_{21}.$$
\end{lemma}

\subsection{Proof of Lemma~\ref{lem: two cross}}
We present only a proof for the first equation. The second equation is obtained from the first by applying the involution $D$.

Since $\bcalC_+^{(2)}=\Indb_{1,2}(\bcalC_+)$ we can use the proposition~\ref{prop: small prod 2} and the corrections for the $B_2$ equivariant structure from
the subsection~\ref{ssec: conv rk 1} to proceed with our computation of  $\bcalC_+^{(1)}\bar{\star}\bcalC_+^{(2)}$.

First we introduce the matrices
\begin{align*}
\xgot & = \xbmtr{\exaoo & \exaot & \exaoh\\ \exato & \exatt & \exath \\ \exaho & \exaht & \exahh}, &
\xgth &= \xbmtr{ 1& 0 & 0 \\ 0 & \SymbolPrint{\exbtt} & \SymbolPrint{\exbth} \\ 0 & \SymbolPrint{\exbth} & \SymbolPrint{\exbhh}} &
                                                                                                                         \xgoh & = \xbmtr{\SymbolPrint{\excoo} & \SymbolPrint{\excot} & \SymbolPrint{\excoh}\\
  \SymbolPrint{\excto} & \SymbolPrint{\exctt} & \SymbolPrint{\excth} \\ \SymbolPrint{ \excho} & \SymbolPrint{\excht} & \SymbolPrint{\exchh}}.
\end{align*}
They satisfy the relation
\[
\xgoh = \xgot \xgth.
\]
Polynomials of the entries of the matrices  $a$, $c$, $X$ and $Y$ span the space of regular functions on the space $\calX_3(G_3,G_{2,3})$.

By the proposition~\ref{prop: small prod 1} $\bcalC^{(1)}_+\bar{\star}\bcalC^{(2)}_+=
\brpi_{13*}(\CE_{\frn^{(2)}_2}((\bar{\pi}^*_{12}\otimes_{B_2} \bar{\pi}^*_{23})( (\bar{\calC}_+^{(1)})\boxtimes\bcalC_+^{(2)})))$ where
$$\brpi_{12}:\calXr_3(G_3,G_{2,3})\to \calXr_2(G_3),\quad \brpi_{23}: \calXr_3(G_3,G_{2,3})\to \calXr_2(G_2), $$
$$\brpi_{13}:\calXr_3(G_3,G_{2,3})\to \calXr_2(G_3).$$

We introduce the matrices
\def\xtX{ \tilde{\xX}}
\def\tXt{\xtX_2}
\def\tXtp{\tXt^+}
\def\tXtm{\tXt^-}
\def\xtY{ \tilde{\xY} }
\def\tYt{ \xtY_2 }
\def\ty{ \tilde{y} }
\begin{align*}
\tXt & = \Adv{\xigot}\xX = ||\tx_{ij} ||, &
\tYt& =  \Adv{\xgth} \xY = || \ty_{ij} ||.
\end{align*}

The tensor product $(\bar{\pi}_{12}\otimes_{B_2} \bar{\pi}_{23})^* (\bar{\calC}_+^{(1)}\boxtimes\bcalC_+^{(2)})$ in the above tensor product is the Koszul complex:
\[ C_{12}=
\xbmtr{
a_{31}&*&\theta_0\\
(\xoo-\xtt)\exaoo + \xot\exato & * &\theta_1
\\
\exaht & * &\theta_2
\\
(\txtt-\txhh)\exbtt + \txth\exbht & * &\theta_3
}
\]

Since
\[
\xigth =
\xbmtr{ 1 & 0 & 0 \\ 0 & \exbhh & -\exbth \\ 0 & -\exbth & \exbtt},
\]
it turns out that $\exaho = \excho$. Thus we can use the row operations to remove $ \excho$ from the other rows of $C_{12}$. After completing the process we obtain
the matrix factorization $C'_{12}$ with last last three rows equal to the rows of $C_{12}$ with $c_{31}=a_{31}=0$ and the first row  matching with
the row of
$\calC_{12}$ that contains the relation~\eqref{eq: eq.rl21}. The last row is $B^{(2)}_2$ invariant  hence for our computation of $\CE_{\frn^{(2)}_2}(C_{12})$ we only need to analyze the part of the complex $C'_{12}$ that consists of
the last three lines:
\[
C''_{12}=\xbmtr{
(\xoo-\xtt)\exaoo + \xot\exato & * &\theta_1
\\
\exaht & * &\theta_2
\\
(\txtt-\txhh)\exbtt + \txth\exbht & * &\theta_3
}
=
\xbmtr{
(\xoo-\xtt)\excoo + \xot\excto & * &\theta_1
\\
\exbhh \excht - \exbht \exchh & * &\theta_2
\\
(\txtt-\txhh)\exbtt + \txth\exbht & * &\theta_3
}
\]
The first two rows come from the (1,2)-crossing and the second row comes from the (2,3)-crossing.
It turns out that the first row of $C''_{12}$  matches the  row of $\calC_{12}$ that contains the relation~\eqref{eq: eq.rl22} and it is $B_2^{(2)}$-invariant.
So complete our computation of $\CE_{\frn_2^{(2)}}(C_{12})^{T^{(2)}}$ we have to simplify only the remaining two rows:
\[
C=\xbmtr{
\exbhh \excht - \exbht \exchh & * &\theta_2
\\
(\txtt-\txhh)\exbtt + \txth\exbht & * &\theta_3
} =
\xbmtr{
\kappa_2 & * &\theta_2
\\
\kappa_3 & * &\theta_3
}
\]
where
\[
\kappa_2 = \exaht = \exbhh \excht - \exbht \exchh,\qquad
\kappa_3 = (\txtt-\txhh)\exbtt + \txth\exbht.
\]

To write out our Chevalley-Eilenberg complex we need to explain the $B_2$-equivariant structure on $C$ and it is done below.
Denote the elements of Borel group $B^{(2)}_2$  as
\[
\brht = \xbmtr{
1 & 0 & 0
\\
0 & \ybettt & \ybetth
\\
0 & 0 & \ybettt^{-1}
}
\]
Its action is
\[
\xgot\mapsto \xgot\briht,\qquad
\xgth\mapsto \brht\xgth,\qquad
\tXt\mapsto \Adv{\brht} \tXt.
\]
The weights of the action of $diag(0,1,-1)\in \Lie(B_2^{(2)})$  on  the \rhs elements in the Koszul rows of $C$ are -1 and 1 respectively, hence the weights of $\theta_2$ and $\theta_3$ are 1 and -1.

It remains to establish the action of $\ybetth$ denoted by $\dltt$. One way to figure out this action is to use formulas (\ref{eq: pi12Bpi23 1}) and (\ref{eq: pi12Bpi23 2}) to fix the
equivariant structure on $C_{12}$ and then follow up the changes of this equivariant structure as do our reduction from $C_{12}$ to $C$. However, we choose a different method:
we just show that there is a unique strongly $B_2$-equivariant structure on $C$ and compute it explicitly.

The element $\kappa_2$ is $\dltt$-invariant. The variation $\dltt\kappa_3$ can be computed by brute force, but since we are interested in it being proportional to the second element, we use an indirect computation. Present $\tXt$ as a sum of the upper-triangular and strictly lower-triangular parts:
\[
\tXt = \tilde{X}_{2,+} + \tilde{X}_{2,--}.
\]
Note that
\[
\kappa_3\exbht = -(\Adv{\xigth}\tilde{X}_{2,+})_{32}.
\]
while $\dltt\exbht=0$, so
\[
\dltt\kappa_3 = -\dltt(\Adv{\xigth}\tilde{X}_{2,+})_{32}/\exbht.
\]
The matrix
\[
\Adv{\xigth}\tXt = \Adv{\xigth}\Adv{\xigot}\tilde{X}_1 = \Adv{\xigoh}\tilde{X}_1
\]
is Borel-invariant, hence
\[
\dltt\kappa_3 = \dltt(\Adv{\xigth}\tilde{X}_{2,--})_{32}/\exbht.
\]
A direct computation shows that
\[
(\Adv{\xigth}\tilde{X}_{2,--})_{32} = \exbtt^2 (\tilde{X}_{2,--})_{32} = \exbtt^2\tx_{32}.
\]
The formulas for action of \(B^{(2)}\) imply  that $\dltt\tx_{32}=0$, while $\dltt\exbtt=\exbht$, so
\[
\dltt\kappa_3 = 2\tx_{32}\exbtt,
\]
while
\begin{multline*}
\tx_{32} = \exaht\left(
  (\xoo - \xtt)\exaot\exato + \exato\exatt\xot\right.\\
\left.+ (\xhh-\xtt)(\exaoo\exatt-\exaot\exato) + \exaht
(\exato\xoh -\exaoo\xth)\right)
\end{multline*}
Since the Koszul differential $\kappa_2\theta_2 + \kappa_3\theta_3$ must remain Borel-invariant, it follows that
\begin{multline}
\label{eq:varxta}
\dltt\theta_2 = -2\exbtt\left(
(\xoo - \xtt)\exaot\exato + \exato\exatt\xot\right.\\
\left. + (\xhh-\xtt)(\exaoo\exatt-\exaot\exato) + \exaht
(\exato\xoh -\exaoo\xth)\right)\theta_3
\end{multline}
\def\xith{\xi_{23}}

Similarly to the case of two strands we have tensor product decomposition $R=\calXr_3=R'\otimes \bar{\pi}^*_{13}(\CC[\calXr_2])$ where $R'=\CC[G_{2,3}]=\CC[G_2]$ is the factor with non-trivial $B^{(2)}_2$-action.
Thus we can use the equality  $\CE_{\frn^{(2)}_2}(C)^{T^{(2)}}=\CE_{\frn^{(2)}_2}(C')^{T^{(2)}}\otimes \bar{\pi}_{13}^*(\CC[\calX_3])$ were $C'$ is the complex $C$ considered as the matrix factorization of  $R'$-modules.

The complex  of $T^{(2)}$-invariant part of the complex $\CE_{\frn_2^{(2)}}(C')$ has the following shape:

$$
\begin{tikzcd}[row sep=scriptsize, column sep=scriptsize]
&\left[ \theta_2;0\right]\arrow[dl] \arrow[rr,shift left=0.5ex] \arrow[dd, dashed, two heads] & & \left[1;1\right] \arrow[dl] \arrow[dd, dashed,blue]\arrow[ll,shift left=0.5ex,blue] \\
\left[\theta_2\theta_3;1\right]\arrow[ur, shift left=1ex,red]  \arrow[dd, dashed, red] & & \left[\theta_3;2\right] \arrow[ur,shift right=1ex]\\
& \left[\theta_2;-2\right] e^* \arrow[dl] \arrow[rr, shift left=0.5ex] &  & \left[1;-1\right] e^* \arrow[dl]\arrow[ll, shift left=0.5ex] \\
\left[\theta_2\theta_3;-1\right]e^* \arrow[ur,shift right=1ex] \arrow[rr, shift left=0.5ex] & & \left[\theta_3,0\right]e^*\arrow[ur, shift right=1ex, blue]
\arrow[ll, shift left=0.5ex,red]\\
\arrow[to=4-3, from=2-3, crossing over, dashed, hook]\\
\arrow[to=4-3, from=1-2, dashed, crossing over,green]
\arrow[from=2-1, to=2-3, shift left=0.5ex,crossing over]
\arrow[to=2-1, from=2-3, shift left=0.5ex,crossing over]
\end{tikzcd}
$$


The shape of the complex is identical to the shape of the complex from the proof of the theorem~\ref{thm: two strands}. Thus we can perform the same sequence of contractions as in the proof of
the theorem~\ref{thm: two strands}. As we have seen in the proof of theorem~\ref{thm: two strands},
at the end of our contraction process, all vertices of the complex except two connected by the green dashed arrow can be contracted to $0$ and the vertices $R[0]\theta_3 e_3^*$ to $H^1(\mathbb{P}^1,\calO(-2))\otimes \brpi_{13}^*(\CC[\calX_2])$,
$R[0]\theta_2$ to $H^0(\mathbb{P}^1,\calO)\otimes \brpi_{13}^*(\CC[\calX_2]).$

Thus at the end we obtain a two term Koszul complex and as we have seen previously all of the new correction differentials
in the final complex will be running in the direction opposite to the direction of the green arrow. Let's study the differential that induced by the green arrow in our final complex.

 The target of our differential
 is \(\textup{H}^1(\mathbb{P}^1,\calO(-2))\otimes \brpi_{13}^*(\CC[\calX_2])=\textup{H}^1_{\Lie}(\frn,R\langle -2\rangle)\otimes \brpi_{13}^*(\CC[\calX_2]),\) hence we can
  replace the coefficients of the differential by the expressions that are homologous with respect to the differential $\dltt$. Below we take advantage of this observation.
 Indeed,  note that
\[
\dltt\exbtt = \exbht,\qquad \dltt\exbth = \exbhh,
\]
so, first, $\dltt(\exbtt^2) = 2\exbtt\exbht$, hence $\exbtt\exbht$ is exact and, second, $\dltt(\exbtt\exbth) =
\exbht\exbth + \exbtt\exbhh$, hence in view of $\exbtt\exbhh-\exbth\exbht=1$ we find $\exbtt\exbhh\sim \frac{1}{2}$. As a result,
\def\shlf{\tfrac{1}{2}}
\begin{equation}\label{eq:homoac}
\exbtt \xbmtr{\exaoo & \exaot & \exaoh\\ \exato & \exatt & \exath \\ \exaho & \exaht & \exahh} =
\xgoh\exbtt\xigth \sim\xgoh
\xbmtr{1 & 0 & 0\\ 0 & \shlf & * \\ 0  & 0 & *} =
\xbmtr{\excoo & \shlf\excot & * \\
\excto & \shlf\exctt  & *\\
\excho & \shlf\excht & *}
\end{equation}
Note that in the expression~\eqref{eq:varxta} each term has a product of an element from the first column and the second column of the matrix $\xgot$, also the first column of \(g_{13}=c\)
is equal to the first column of \(g_{12}=a\). Thus we can combine the last observation with (\ref{eq:homoac}) to see that 
 the differential induced by the green arrow
is homologous to
\[
-\Bigl(
(\xoo - \xtt)\excot\excto + \excto\exctt\xot + (\xhh-\xtt)(\excoo\exctt-\excot\excto) + \excht
(\excto\xoh -\excoo\xth)\Bigr)
\]

Since this is the third condition of~\eqref{eq: eq.rl23} we have shown that the matrix factorization $\bcalC_+^{(1)}\bar{\star}\bcalC_+^{(2)}$ is actually an element of
$\MF_{B^2}(\calXr_2,\Wr)$ and its positive part in the sense of Lemma~\ref{lem: ext MF eq}  is a Koszul complex for the regular sequence $(c_{31},f,g)$. Hence by the lemma~\ref{lem: uniq MF+eq}
there is an isomorphism between $\bcalC_+^{(1)}\bar{\star}\bcalC_+^{(2)}$ and $\bcalC_{12}$.


%% file: part6_gl.tex
\section{Three strands three crossings}\label{sec:  three str 3}

\def\xoo{x_{11}}
\def\xot{x_{12}}
\def\xoh{x_{13}}
\def\xtt{x_{22}}
\def\xth{x_{23}}

\def\yot{y_{12}}
\def\yoh{y_{13}}
\def\yth{y_{23}}

\def\zto{z_{21}}
\def\zho{z_{31}}
\def\zht{z_{32}}

\def\brh{ h }
\def\brht{ \brh_2 }
\def\briht{ \brh^{-1}_2 }
\def\ybet{ \beta }
\def\ybetth{ \ybet_{23} }
\def\ybettt{ \ybet_{22} }
\def\ybethh{ \ybet_{33} }
\def\ybetoo{ \ybet_{11}}
\def\ybetot{ \ybet_{12}}
\def\ybetto{ \ybet_{21}}
\def\ybetooi{ \ybet_{11}^{-1}}

\def\aiv#1#2{ (a^{-1})_{#1#2}}
\def\aioo{\aiv{1}{1}}
\def\aiot{\aiv{1}{2}}
\def\aioh{\aiv{1}{3}}
\def\aito{\aiv{2}{1}}
\def\aitt{\aiv{2}{2}}
\def\aith{\aiv{2}{3}}
\def\aiho{\aiv{3}{1}}
\def\aiht{\aiv{3}{2}}
\def\aihh{\aiv{3}{3}}
\def\civ#1#2{ (c^{-1})_{#1#2}}
\def\cioo{\civ{1}{1}}
\def\ciot{\civ{1}{2}}
\def\cioh{\civ{1}{3}}
\def\cito{\civ{2}{1}}
\def\citt{\civ{2}{2}}
\def\cith{\civ{2}{3}}
\def\ciho{\civ{3}{1}}
\def\ciht{\civ{3}{2}}
\def\cihh{\civ{3}{3}}
\def\biv#1#2{ (b^{-1})_{#1#2} }

\def\tho{ \theta_1 }
\def\thw{ \theta_2 }
\def\thh{ \theta_3 }

\def\zXv#1{ X - x_{#1#1}\xId}
\def\zXuv#1#2{ \bigl( \zXv{#1}\bigr)_{#2}}
\def\zXtv#1{ \zXuv{2}{#1}}

\def\krdlv#1{ \delta_{#1} }

In this section we complete our proof theorem~\ref{thm: braids} by proving the cubic relation for the braid group generators.
We show that two sides of the relations realize the same matrix factorization that has a geometric interpretation: the left hand side
is realized by the extension of the complex of the minimal resolution of the ideal $I_{121}$ and the other side for the ideal $I_{212}$.
First we introduce the ideals and discuss their properties.

\subsection{Involution $D$ and the ideals $I_{121}$ and $I_{212}$}
We use special notations  for some elements of $\CC[\frb\times G]$:
$$f_i:=\Ad_{c^{-1}}(X-x_{22}\xId)_{2i},\quad h_i=(c^{-1}(X-x_{11}\xId))_{3i}.$$
$$f'_i:=\Ad_{c^{-1}}(X-x_{22}\xId)_{i2},\quad h'_i=((X-x_{33}\xId)c)_{i1},$$
$$ k_{ij}=(\Ad_{c^{-1}}(X))_{ij}-\delta_{ij} x_{4-i,4-i}.$$
where $X=|| x_{ij}||$, $g\in G$ provide local coordinates on $\frb\times G$. Let us introduce the ideals
$$ I_{121}=(f_1,f_2,h_2,h_3),\quad I_{212}=(f'_3,f'_2,h'_2,h'_1) \quad I=(k_{31},k_{32},k_{33},k_{21},k_{22},k_{11}).$$

\begin{proposition}\label{prop: ideals aux} We have
\begin{enumerate}
\item $\mathrm{D}(I_{121})=I_{212}$,
\item $\mathrm{D}(I)=I$,
\item $I_{121}=I$.
\end{enumerate}
\end{proposition}
\begin{proof}
The first two statements are immediate from our definition of the involution $\mathrm{D}$. We now show the last statement. First we show that $I\subset I_{121}$.

Indeed, the conditions $c^{-1}(X-x_{11}\xId)_{31}$ because $X$ is upper triangular. Hence $h_2=0,h_3=0$ implies that the last row of $c^{-1}(X-x_{11}\xId)_{31}$ is zero and
thus the last row of $\Ad_{c^{-1}}(X-x_{11}\xId)$ has vanishing last row. Finally, $f_1=h_2=0$ implies $\Ad_{c^{-1}}(X)_{21}=0$ and $\Ad_{c^{-1}}(X)_{22}=x_{22}$.
Thus $\Ad_{c^{-1}}(X)\in \frb $ and combining with $\Tr(X)=\Tr(\Ad_{c^{-1}}(X))$ we obtain that $\Ad_{c^{-1}}(X)_{ii}=x_{4-i,4-i}$.

Now let us show that $I_{121}\subset I$. The fact that $h_i\in I$ is immediate. Let us denote by $\vec{w},\vec{u},\vec{v}$ the first, second and third rows of $c^{-1}$.
The relation from $I$ imply that  $c^{-1}(X)=\tilde{X} c^{-1}$ where $\tilde{X}$ is upper triangular with elements $x_{33},x_{22},x_{11}$ on the diagonal. Hence
$\vec{v}X=x_{11}\vec{v}$ and these relations are equivalent to $h_2=h_3=0$.
\end{proof}

\subsection{Resolution for $I_{121}$}
In this subsection we construct a free resolution of the module $\CC[\frb\times G]/I_{121}$. Let us fix notations:
\begin{align*}
A&=\Bigl[-\left(\Adv{c^{-1}}\xX\right)_{21},\;\;
  \Bigl(\Adv{c^{-1}}(\xX-\xtt\xId)\Bigr)_{22},\;\;
 c^{-1}(\xX - \xoo\xId)_{32}
\Bigr].
\end{align*}

The direct computation shows that
$$A=-\Bigl[\det B_{23},\;\det B_{13},\;\det B_{12}\Bigr],$$
where
\begin{align*}
B&=
\xbmtr{
(\xoo-\xtt)\excot +\xot\exctt
&
-\excht
\\
-((\xoo-\xtt)\excoo +\xot\excto)
&
\excho
\\
-\civ{2}{i}\zXtv{i3}
&
-\civ{2}{i}\krdlv{i1}
}.
\end{align*}

The Laplace formula for the $3\times 3$ matrix made of $B$ and $A$ implies that the following is a complex:
\begin{equation}\label{eq: reso}
K=[R^2\overset{B}{\longrightarrow} R^3\overset{A}{\longrightarrow} R]
\end{equation}
where $R=\CC[\frb\times G]$.

Let $\tilde{I}_{121}=(f_1,f_2,h_2)$.
Below we show that the complex $K$  is actually a resolution of the $R$-module $R/\tilde{I}_{121}$:
\begin{proposition}\label{prop: reso K} We have:
$$\textup{H}^0(K)=R/\tilde{I}_{121},\quad \textup{H}^{>0}(K)=0.$$
\end{proposition}

The variety $Z$ defined by the ideal generated by $\tilde{I}_{121}$ is the key geometric ingredient of the proof of theorem~\ref{thm: braids}.
The proposition below describes the fibers of the projection map $\pi_G: Z\to G$:

\begin{proposition} The fibers of the map $\pi_G$ are linear subspaces of $\frb$ and
if $c\notin B$ then $\mathrm{codim}_{\frb}(\pi^{-1}_G(c))=2$, otherwise if $g\in B$, $\pi_G^{-1}(c)=\frb$.
\end{proposition}
\begin{proof}
Since the second column of the matrix $B$ is a vector valued function on $G$ which is non-zero at the generic point of $G$ the linear conditions
imposed by $f_1,f_2,h_2$ are not linearly independent at a generic point of $G$. Hence the $\mathrm{codim}_{\frb}(\pi_G^{-1}(c))\le 2$ for any $c\in G$.

Let us determine the locus of $g$ such that $\mathrm{codim}_{\frb}(\pi_G^{-1}(c))\le 1$. Let us assume that $(c^{-1}_{31},c^{-1}_{32})\ne (0,0)$ and hence  $(c_{31},c_{32})\ne (0,0)$.
Since $h_2$ is a non-zero linear function on $\frb$, the
linear system $f_1=0,f_2=0,h_2=0$ has rank $1$ if $f_1,f_2$ are proportional to  $h_2$. If $c_{31}\ne 0$ then the condition that $f_1$ is proportional to $h_1$ implies that
$c^{-1}_{2i}=0$ for all $i$ hence contradiction to $det(c)=1$. Similarly, if $c_{32}\ne 0$, the condition $f_2$ is proportional to $h_1$  also implies  $c^{-1}_{2i}=0$ for all $i$.
Thus if $\mathrm{codim}_{\frb}(\pi_G^{-1}(c))\le 1$ then $c_{31}=c_{32}=c^{-1}_{31}=c^{-1}_{32}=0$.

The last condition on $c$ implies $h_1=0$ and the nontrivial entries of the linear equations $f_1$ and $f_2$ are the coefficients in front of  $x_{12},x_{11},x_{22}$. These are the  vectors
$(c^{-1}_{21} c_{21}, c_{21}^{-1} c_{11}, c_{22}^{-1}c_{21})$ and $(c^{-1}_{21} c_{22}, c_{21}^{-1} c_{12}, c_{22}^{-1}c_{12})$. If $c^{-1}_{12}\ne 0$ then these vectors are proportional
to $(c_{21},c_{11},*)$ and $(c_{22},c_{12},*)$ and these two vectors can not be proportional because of $det(c)=1$.

Thus we have shown that $\mathrm{codim}_{\frb}(\pi^{-1}_G(c))\le 1$ implies $c\in B$. Finally, it is immediate to see that if $c\in B$ then $f_1,f_2,h_2$ are zero linear functions on $\frb$.
\end{proof}

\begin{proof}[Proof of proposition~\ref{prop: reso K}]
The previous proposition implies that $\mathrm{codim}_{\frb\times G}(Z)=2$. Hence we can apply the Hilbert-Burch theorem \cite{E2} which implies that $K$ is indeed a resolution of
the module $R/\tilde{I}_{121}$.
\end{proof}

In the next subsection we need a slight refinement of the previous proposition. We define a complex
$$ C_{121}=[R\overset{h_2}{\longrightarrow} R]\otimes K.$$

\begin{proposition}We have
$$\textup{H}^0(C_{121})=R/I_{121},\quad \textup{H}^{>0}(C_{121})=0.$$
\end{proposition}
\begin{proof}
By the spectral sequence for the tensor product of complexes and the previous proposition it is enough to
show that the complex $R/\tilde{I}_{121}\xrightarrow{h_2} R/\tilde{I}_{121}$ has no homology in non-zero degree.
That is we need to show that $h_2$ is not a zero divisor in $R/\tilde{I}_{121}$.

The variety $Z$ could have irreducible components of codimension at most $3$.
The proposition~\ref{thm: braids} implies that $Z$ has at most two irreducible components. The first component $Z_1$ is the closure
of $\pi_G^{-1}(G\setminus B)$ and the second component $Z_2$ is $\pi^{-1}_G(B)$. The fist component has codimension $2$ and
the second codimension $3$. Thus we need to show that $h_2$ is no vanishing identically on any of the components $Z_i$.

Since $c=1\in B$ and $h_2|_{c=1}=x_{33}-x_{11}$ and $f_1|_{c=1}=f_2|_{c=1}=h_1|_{c=1}=0$, the function $h_2$ does not vanish on $Z_2$.
Respectively $s=E_{12}-E_{21}+E_{33}\notin B$ and $h_2|_{c=s}=x_{33}-x_{11}$, $f_1|_{c=s}=x_{12}$, $f_2|_{c=s}=x_{11}-x_{22}$, $h_1|_{c=s}=0$ we see that
$h_2$ is not vanishing on $\pi_G^{-1}(s)$ and hence $h_2$ does not vanish on $Z_1$ and it is a zero-divisor.
\end{proof}


\subsection{Convolution} Let us denote by $\overline{C}_{121}$  a complex of $R$-modules (here $R=\CC[\calXr_2]$) that is obtained by a pull back of the complex  $C_{121}$ along the
projection map $\calXr_2\to \frb\times G$. Since the projection map is flat,
the results of the previous subsection imply that  the complex $\overline{C}_{121}$ satisfies the conditions of the Lemmas~\ref{lem: ext MF},\ref{lem: ext uniq} with $F=\Wr$. Hence there is
a unique (up to homotopy) matrix factorization $\bcalC^{(121)}\in \MF_{B^2}(\calXr_2,\Wr)$ that extends the complex $\overline{C}_{121}$. The main result of this section is the proof
of the following

\begin{lemma} We have
$$\bcalC^{(1)}\bar{\star}\bcalC^{(2)}\bar{\star}\bcalC^{(1)}=\bcalC^{(121)}$$
\end{lemma}
\begin{proof}
To prove the lemma we use the results of the previous section and compute the convolution
$ \bcalC_{12}\bar{\star}\bcalC^{(1)}_+.$
Since $\bcalC^{(1)}_+=\Indb_{12}(\bcalC_+)$ we can use the proposition~\ref{prop: small prod 1}  and subsection~\ref{ssec: conv rk 1}.
That is for our proof,  we  analyze $T$-invariant part of $\CE_{\frn^{(2)}_2}(\calC)$ the matrix matrix factorization $\calC\in \MF_{B_3\times B_2\times B_3}(\calXr_3(G_3,G_{1,2}),\bar{\pi}_{13}^*(\Wr))$,
$$\calC=(\bar{\pi}_{12}\otimes_{B_2}\bar{\pi}_{23})^*(\bcalC_{12}\boxtimes\bcalC^{(1)}_+).$$

\def\hA{\hat{A}}
\def\adhA{\mathop{\mathrm{adj}} \hA }
\def\vcw{ \vec{w}}

First let us discuss the coordinates on the space $\calX_3(G_3,G_{1,2})$, for that we introduce the following matrices:

\begin{align*}
\xgot & = \xbmtr{\exaoo & \exaot & \exaoh\\ \exato & \exatt & \exath \\ \exaho & \exaht & \exahh}, &
\xgth &= \xbmtr{ \exboo & \exbot & 0 \\ \exbto & \exbtt & 0 \\ 0 & 0 & 1},   
\end{align*}

\begin{align*}
\xgth^{-1}&=  \xbmtr{ \exbtt  & -\exbot & 0 \\ -\exbto & \exboo & 0 \\ 0 & 0 & 1 }, &
\xgoh & = \xbmtr{\excoo & \excot & \excoh\\ \excto & \exctt & \excth \\ \excho & \excht & \exchh}
\end{align*}
and
\begin{align*}
\xX &=
\xbmtr{ \xoo & \xot & \xoh\\ 0 & \xtt & \xth \\ 0 & 0 & \xhh}, &
\xY & = \xbmtr{ 0 & \yot & \yoh \\ 0 & 0 & \yth \\ 0 & 0 & 0}, &
\xtX & = \Adv{\xigot}\xX = ||\tx_{ij} ||.
\end{align*}
The Borel group $B_2^{(2)}$ corresponding to a middle factor in $B_3\times B_2\times B_3$ could be parameterized as follows:
\[
\brht =
\xbmtr{ \ybetoo & \ybetot & 0 \\ 0 & \ybetooi & 0 \\ 0 & 0 & 1}
\]
The weights of the action of  Lie algebra element  $diag(1,-1,0)\in \Lie(B_2^{(2)})$ are
\[
\begin{array}{|c||c|c|c||c|c|c||c|c|c|}
\hline
\text{element} &\exav{i1} & \exav{i2} &\exav {i3} & \aiv{1}{i} & \aiv{2}{i} & \aiv{3}{i}
&\exbv{1i} & \exbv{2i} & \exbv{3i}
\\
\hline
\text{weight} & -1 & 1 & 0 & 1 &-1 & 0 & 1 & -1 & 0
\\
\hline
\end{array}
\]
and the action of $\delta_{12}:=E_{12}\in \Lie(B_2^{(2)})$ is
\def\ydltot{ \delta_{12}}
\def\ydltoti{ \ydltot^{-1}}
\begin{align*}
\ydltot \exav{i2}& = -\exav{i1}, &
\ydltot \aiv{1}{i} & = \aiv{2}{i}, &
\ydltot \exbv{1i} & = \exbv{2i}.
\end{align*}

\def\ydv#1{ d_{#1}}
\def\ydo{ \ydv{1}}
\def\ydt{ \ydv{2}}
\def\ydh{ \ydv{3}}
\def\ydf{ \ydv{4}}
The matrix factorization $\calC$ is generated by four Koszul differentials:
\begin{align*}
\ydo & = \exaho
\\
\ydt &= (\xoo-\xtt)\exaoo + \xot\exato
\\
\ydh &=(\txoo-\txtt)\exboo + \txot\exbto
\\
\ydf &=(\xhh-\xoo) (\exaoo\exatt-\exaot\exato) - \exaht(\exaoo\xth-\exato\xoh)
\\
&=(\xhh-\xoo) \aihh + \xth\aiht + \xho\aiho
\\
&=\aiv{3}{i}(\xX-\xoo\xId)_{i3}.
\end{align*}
The differentials $\ydo$, $\ydt$ and $\ydf$ come from the first two crossings and together they constitute $\bar{\pi}_{12}(\calC_{12})$; $\ydh$ comes from the last crossing, that is
from the matrix factorization $\bar{\pi}_{23}^*(\calC_+)$.

The weights of the differentials with respect to the maximal torus the Borel group $B_2^{(2)}$ are
\[
\begin{array}{|c||c|c|c|c|}
\hline
\text{differential} &\ydo &\ydt &\ydh & \ydf
\\
\hline
\text{weight} & -1 &-1 & 1 & 0
\\
\hline
\end{array}
\]
The differential $\ydf$ is completely Borel-invariant and thus for our analysis of $\CE_{\frn_2^{(2)}}(\calC)$ we need to concentrate on  the first three rows of $\calC$ that compose the matrix factorization
$\calC'$. Moreover,
\[
\ydf = (\xhh-\xoo) \cihh + \xth\ciht + \xho\ciho =
\civ{3}{i}(\xX-\xoo\xId)_{i3}.
\]
thus it matches with $h_3$.

The differentials $\ydo$ and $\ydt$ are invariant  under $\ydltot$.

Let us compute $\ydltot\ydh$. We observe that $\xtX_+ = \xtX -\xtX_{--}$, hence
\begin{equation}
\label{eq:xtthf}
\ydh = - \left(\Adv{\xigth}\xtX_+\right)_{21}/\exbto
=-\left(\Adv{\xigoh}\xX\right)_{21}/\exbto +\left(\Adv{\xigth}\xtX_{--}\right)_{21}/\exbto
\end{equation}
The first term here is $\ydltot$-invariant and $\exbto$ is also $\ydltot$-invariant, hence
\[
\ydltot\ydh = \ydltot\left(\Adv{\xigth}\xtX_{--}\right)_{21}/\exbto
\]
Now
\begin{align*}
\left(\Adv{\xigth}\xtX_{--}\right)_{21} &= \tx_{21}\exboo^2,&
\ydltot\tx_{21}&=0, &
\ydltot\exboo = \exbto,
\end{align*}
hence
\[
\ydltot\ydh = 2\tx_{21}\exboo.
\]
It remains to express $\tx_{21}$ in terms of $\ydo$ and $\ydt$. By definition,
\[
\tx_{21} = \aiv{2}{i} \xX_{ij} \exav{j1} = \aiv{2}{i}\zXtv{ij}\exav{j1} .
\]
Note that

\[
\zXtv{ij}\exav{j1} = \ydo\zXtv{i3} + \ydt \krdlv{i1}
\]
hence
\def\zbetv#1{ \beta_{#1} }
\def\zbeto{ \zbetv{1}}
\def\zbett{ \zbetv{2}}
\def\xcfo{\alpha_1}
\def\xcft{\alpha_2}
\begin{equation*}
\label{eq:txot}
\tx_{21} = \ydo \zbeto + \ydt \zbett,
\end{equation*}
where
\begin{align}
\label{eq.bttw}
\zbeto & = \aiv{2}{i}\zXtv{i3}, &
\zbett & = \aiv{2}{i}\krdlv{i1}
\end{align}
so
\[
\ydltot\ydh = 2\exboo(\zbeto \ydo + \zbett\ydt),
\]

%

The first three differentials $d_1,d_2,d_3$ thus form the following Koszul complex $\calC'_+$ which form a positive part of the
matrix factorization $\calC$.

Below we simplify the complex $\CE_{\frn_2^{(2)}}(\calC'_+)^{T^{(2)}}$ and match the simplified complex with the complex of the minimal resolution
of $I_{121}.$

The complex $\CE_{\frn_2^{(2)}}(\calC'_+)$ has the form
\def\xvrt#1#2{ [#1;#2] }

\[
\begin{tikzcd}[column sep=1.75cm, row sep=1.75cm]
&
\xvrt{\tho\thw\thh}{0}
\\
\xvrt{\tho\thh}{-1}
\ar{ur}{-\ydt}
&
\xvrt{\tho\thw}{1}
\ar{u}{\ydh}
\ar[dashed]{r}{-\xcfo}
\ar[dashed,']{l}{\xcft}
&
\xvrt{\thw\thh}{-1}
\ar[']{ul}{\ydo}
\\
\xvrt{\tho}{0}
\ar{u}{-\ydh}
\ar[near start]{ur}{-\ydt}
\ar[dashed,']{r}{-\xcfo}
&
\xvrt{\thh}{-2}
\ar[from=3-2,to=2-2,near start]{u}{\ydo}
\ar[crossing over, near start]{ul}{\ydo}
\ar[crossing over, near start,']{ur}{\ydt}
\ar[from=3-3,to=2-2,crossing over,near start,"-\ydt"]
&
\xvrt{\thw}{0}
\ar[dashed]{l}{-\xcft}
\ar{u}{\ydh}
\\
& \xvrt{1}{-1}
\ar{ul}{\ydo}
\ar{u}{\ydh}
\ar[']{ur}{\ydt}
\end{tikzcd}
\]

In the picture $[\Theta,k]$ stands for the tensor  product  of the  monomial $\Theta$ of the odd variables from the Koszul complex and of Chevalley-Eilenberg complex for the $R\langle k\rangle:=\CC[\calXr_3(G_3,G_{1,2})]=R'\otimes \bar{\pi}_{13}^*(\CC[\calXr_2])$, $R'=\CC[G_{1,2}]=\CC[G_2]$:
$\CE_{\frn_2^{(2)}}(R'\langle -k \rangle)\otimes\bar{\pi}_{13}^*(\calXr_2)\otimes \Theta.$
We used  the dashed lines additional  Chevalley-Eilenberg differentials and we used a shortcut
$$
\xcfo = 2 \exboo \zbeto, \quad \xcft  = 2 \exboo\zbett.
$$

\def\ghzot{\zeta_{12}}
Now we perform contractions along the differentials inside complexes $[\Theta,k]$. After the contraction the corresponding vertices in the diagram get replaced with
$\textup{H}^*(\mathbb{P}^1,\calO(-k))\otimes \bar{\pi}_{13}^*(\CC[\calX_2]):$

\[
\begin{tikzcd}[column sep=1.75cm, row sep=1.75cm]
&
1
\\
0
\ar{ur}{-\ydt}
&
\exbto\oplus\exbtt
\ar{u}{\ydh}
\ar[dashed]{r}{-\xcfo}
\ar[',dashed]{l}{\xcft}
&
0
\ar[']{ul}{\ydo}
\\
1_1
\ar{u}{-\ydh}
\ar[near start]{ur}{-\ydt}
\ar[',dashed]{r}{-\xcfo}
&
\xvrt{\thh}{-2}
\ghzot
\ar[from=3-3,to=2-2,near start]{u}{\ydo}
\ar[crossing over, near start]{ul}{\ydo}
\ar[crossing over, near start,']{ur}{\ydt}
\ar[from=3-2,to=2-2,crossing over,near start,"-\ydt"]
&
1_2
\ar[dashed]{l}{-\xcft}
\ar{u}{-\ydh}
\\
&
0
\ar{ul}{\ydo}
\ar{u}{\ydh}
\ar[']{ur}{\ydt}
\end{tikzcd}
\]

Thus, we obtained the complex of the form
\[
\begin{tikzcd}
1_1 \oplus 1_2
\ar{r}{B}
&
\exbto\oplus\exbtt\oplus \ghzot
\ar{r}{A}
&
1
\end{tikzcd}
\]
and the rest of the proof is devoted to demonstrating that these matrices are
\begin{multline*}
A=\Bigl[-\left(\Adv{\xigoh}\xX\right)_{21},\;\;
  \Bigl(\Adv{\xigoh}(\xX-\xtt\xId)\Bigr)_{22},\;\;
 \civ{3}{i}(\xX - \xoo\xId)_{i2}
\Bigr]
\\=
-\Bigl[\det B_{23},\;\det B_{13},\;\det B_{12}\Bigr].
\end{multline*}

\begin{align*}
B&=
\xbmtr{
(\xoo-\xtt)\excot +\xot\exctt
&
-\excht
\\
-((\xoo-\xtt)\excoo +\xot\excto)
&
\excho
\\
-\civ{2}{i}\zXtv{i3}
&
-\civ{2}{i}\krdlv{i1}
}.
\end{align*}

We begin with $B$. Its lower row comes from a pair $-\xcfo$, $-\xcft$. Note that both have the same form
\[
\alpha_* = 2\exboo \aiv{2}{i}(\text{something})_i,\qquad \aiv{2}{i} = \exbv{2j}\civ{ji}.
\]
Now we used the fact that
\begin{align*}
\ydltot \exboo &= \exbto,  & \ydltot\exbot & =\exbtt, & \exboo\exbtt - \exbot\exbto = 1,
\end{align*}
from which we conclude that
\begin{align*}
\exboo\exbto\ghzot &\sim 0,&
\exboo\exbtt\ghzot &\sim \shlf \ghzot.
\end{align*}
and as a result
\[
2\exboo\aiv{2}{i}\ghzot\sim \civ{2}{i}\ghzot
\]
which means that we should simply replace $\aiv{2}{i}$ with $\civ{2}{i}$ in the formulas~\eqref{eq.bttw}.

The first two rows of $B$ come from $\ydo$ and $\ydt$ in which we substitute
\[
\exav{i1} = \excv{ij}\biv{j}{1}
\]
and then extract the coefficients at $\exbto$ and at $\exbtt$.

Now we turn to the row-matrix $A$. Its third element $A_3$ is the simplest: it comes from two secondary differentials going from $\ghzot$ to $1$, so it is
\[
A_3 =  \ydo\,\ydltoti(\ydt)-\ydt\, \ydltoti (\ydo).
\]
Both $\ydo$ and $\ydt$ are bilinear combinations of $\ydltot$-invariants and matrix elements $\exav{i1}$. Note that $\ydltoti(\exav{i1}) = \exav{i2}$, hence if $i\neq j$, then
\[
\exav{i1}\, \ydltoti(\exav{j1}) - \exav{j1}\,\ydltoti(\exav{i1}) =- \exav{i1}\exav{j2}+\exav{j1}\exav{i2} =
\pm \aiv{3k},
\]
where $k\neq i,j$. As a result,
\[
A_3 =
 (\xtt-\xoo)\aiht + \xot\aiho
\]
The third row of $\xgot^{-1}$ coincides with the third row of $\xgoh^{-1}$, hence ultimately
\[
A_3 =
(\xtt-\xoo)\ciht + \xot\ciho = \civ{3}{i}(\xX - \xoo\xId)_{i2}
\]
\noindent

The first and second entries of $A$ are both sums of three terms: the first one being $\ydh$ and the last two being secondary differentials
\[
\ydv{i} \ydltoti (\alpha_{i}\exbv{2j}) = \ydv{i}\ydltoti(2\exboo\exbv{2j}\beta_i) ,
\qquad i=1,2.
\]
Let us being with $A_1$ originating from $\exbto$. According to~\eqref{eq:xtthf},
\[
\ydh\exbto =
-\left(\Adv{\xigoh}\xX\right)_{21} +\left(\Adv{\xigth}\xtX_-\right)_{21} =
-\left(\Adv{\xigoh}\xX\right)_{21} + \tx_{21}\exboo^2.
\]
At the same time, $\ydltot\beta_i=0$ and $\ydltot(\exboo^2)=\exboo\exbto$, hence
\[
\ydv{i}\ydltoti(2\exboo\exbv{21}\beta_i) = \exboo^2 \ydv{i}\beta_i,
\]
and
\[
\sum_{i=1,2} \ydv{i}\ydltoti(2\exboo\exbv{21}\beta_i) = \exboo^2 (\ydo\beta_1 + \ydt\beta_2) =
\exboo^2 \tx_{21}
\]
in accordance with~\eqref{eq:txot}. Now
\[
A_1 = \ydh\exbto - \sum_{i=1,2}\ydv{i} \ydltoti (\alpha_{i}\exbv{21}) =
-\left(\Adv{\xigoh}\xX\right)_{21}.
\]

In order to computer $A_2$, we use the formula
\[
\ydh =(\txoo-\txtt)\exboo + \txot\exbto
\]
In order to compute the secondary differentials,  use the relations
\[
\ydltot(\exboo\exbot) = \exboo\exbtt + \exbot\exbto,\qquad \exboo\exbtt-\exbot\exbto = 1,
\]
from which it follows that
\[
2\exboo\exbtt = \ydltot(\exboo\exbot) + 1.
\]
As a result,
\[
\ydv{i}\ydltoti(2\exboo\exbv{22}\beta_i)
=
\ydv{i}\exboo\exbot\beta_i + \ydv{i}\ydltoti(\beta_i).
\]
The first term yields
\[
\sum_{i=1,2}\exboo\exbot\ydv{i}\beta_i = \tx_{21}\exboo\exbot.
\]
As for the second term, from the relation~\eqref{eq:txot} and $\ydltot$-invariance of $\ydo$ and $\ydt$ of we conclude:
\[
\ydo \ydltoti(\zbeto) + \ydt \ydltoti(\zbett) =
\ydltoti \Bigl( (\Adv{\xigot}\xX- \xtt\xId)_{21} \Bigr) =
 (\Adv{\xigot}\xX- \xtt\xId)_{11} = \tx_{11} - \xtt.
\]
Adding all three terms together, we find
\[
A_2 = \ydh\exbtt - \tx_{21}\exboo\exbot - (\tx_{11}-\xtt) =(\Adv{\xigth}\xtX)_{22} - \xtt
=(\Adv{\xigoh}\xX)_{22} - \xtt.
\]

Thus we have shown that the matrix factorization $\calC_{12}\bar{\star}\bcalC_+^{(1)}$ is a $2$-periodic folding of the matrix
factorization that is obtain from $B_3^2$ equivariant complex $\bar{C}_{121}$ by adding the negative differentials and the differentials that
correct equivariant structure. Moreover, since
the complex $\bar{C}_{121}$ satisfies the conditions of the lemma~\ref{lem: ext MF eq} such matrix factorization exists and unique and the statement
of our theorem follows.

\end{proof}

\subsection{End of the proof of the theorem~\ref{thm: braids}}
\begin{corollary} \label{cor: cubic} There is a homotopy realizing relation:
$$ \bcalC^{(1)}_+\bar{\star}\bcalC^{(2)}_+\bar{\star}\bcalC^{(1)}_+\sim\bcalC_+^{(2)}\bar{\star}\bcalC_+^{(1)}\bar{\star}\bcalC^{(2)}.$$
\end{corollary}
\begin{proof}
  From the previous discussion we see that the complexes $C_{121}$ and $\mathrm{D}(C_{121})$ are minimal resolutions of the modules $R/I_{121}$ and $R/I_{212}$.
  The ideals \(I_{212}\) and \(I_{121}\) are isomorphic hence there are morphism \(\psi,\theta: R^4\rightarrow R^4\) such that
  \(\Psi_{<2}=(1,\psi_1)\) and \(\Theta_{<2}=(1,\theta_1)\) are morphisms between the truncated complexes \([C_{121}]_{<2}=[R\xleftarrow{d} R^4]\),
  \([\mathrm{D}(C_{121})]_{<2}=[R\xleftarrow{\mathrm{D}(d)} R^4]\) and \(\Psi_{<2}\circ \Theta_{<2}=1\), \(\Theta_{<2}\circ \Psi_{<2}=1\).
The morphisms could be extended to some morphism \(\Psi_{<k},\Theta_{<k}\) of \([C_{121}]_{<k}\) and \([C_{212}]_{<k}\).
Moreover, any such extensions satisfy \(\Psi_{<k}\circ \Theta_{<k}\sim 1\), \(\Theta_{<k}\circ \Psi_{<k}\sim 1\). Let us explain
an inductive construction for the first homotopy.

Suppose constructed \(\Psi_{<k+1}=\Psi_{<k}+\psi_k\),  \(\Theta_{<k+1}=\Theta_{<k}+\theta_k\) and homotopy \(\chi_{<k}\) such that \(\Psi_{<k}\circ \Theta_{<k}-1=D\circ \chi_{<k}-
\chi_{<k}\circ D\). Then \(D\circ \Psi_{<k+1}\circ\Theta_{<k+1}-D=-D\circ\chi_{<k}\circ D+D\circ \psi_k\circ \theta_k\) hence \(D\circ\psi_k\circ \theta_k=D\) as
map \([C_{121}]_k\to [C_{212}]_{k+1}\). Thus there is \(\chi_k\) such that \(D\circ\chi_k=\psi_k\circ\theta_k\) and we can set \(\chi_{<k+1}=\chi_{<k}+\chi_k\).

Thus we have a homotopy equivalence \(C_{121}\sim C_{212}\). The matrix factorizations
\(\bcalC_+^{(1)}\bar{\star}\bcalC_+^{(2)}\bar{\star}\bcalC^{(1)}_+\) and
\(\bcalC_+^{(2)}\bar{\star}\bcalC_+^{(1)}\bar{\star}\bcalC^{(2)}_+\)
are lifts of the complexes \(C_{121}\) and \(C_{212}\). Since the complexes \(C_{121},C_{212}\) satisfy the conditions of the lemma~\ref{lem: ext mor} we lift the homotopy
between \(C_{121}\) and \(C_{212}\) to the homotopy in the statement of corollary.
  
\end{proof}

The main conclusion of the computations in the last two section is a construction of the braid group  action on the category of matrix factorizations:
\begin{corollary} The assignment
  \[\sigma_i\mapsto \bcalC^{(i)}_+\bar{\star},\quad \sigma_{i}^{-1}\mapsto \bcalC^{(i)}_-\bar{\star} \]
  extends to the homomorphism \(\Brgr_n\to End(\MF_{B^2}(\calXr_2,\Wr)).\)
\end{corollary}

\section{Sheaves on the Hilbert scheme}\label{sec: link inv}
In this section we explain how for a word $\beta$ of generators of the braid group $ \Brgr_n$ one can construct a two-periodic complex $C_\beta$ of quasi-coherent sheaves on the space $\calX_\ell$ with the property that
its homology is supported on
a version of the Steinberg-like
variety $\St_\ell$ that covers  the nested Hilbert Scheme $\Hilb_{1,n}$ of $n$-points on $\CC^2$.
We show later  that $B^\ell\times G$-invariant part $\cohog{*}{C_\beta}$ is related to the isotopy   invariant of the link $L(\beta)$. First let us remind a definition of the nested Hilbert scheme.

\subsection{Nested Hilbert scheme: definitions and notations}\label{ssec: NHilb defs} Direct definition of the scheme $\Hilb_{1,n}$ is the moduli space of the sequences of the ideals $I_1\subset\dots\subset I_n\subset \CC[x,y]$ such that $\dim \CC[x,y]/I_k=k$.
We are also interested in the subvariety $\Hilb_{1,n}^L$ which consists of the sequence like above with extra condition $\supp(I_{k-1}/I_k)$ is a point on the line $L=\{ y=0\}$.

We can also think of a point on $\Hilb_{1,n}$ as  quotient sequence \[\CC[x,y]/I_n\to\CC[x,y]/I_{n-1}\to\dots\to \CC[x,y]/I_1.\]
Let us use notation $V_k=\CC[x,y]/I_k$.
If we choose a basis $\{v_i\}_{i=1,\dots,n}$ of $V_n$ such that $\langle v_1,\dots,v_k\rangle=V_k$ then the operators of multiplication by $x$ and $y$ become upper triangular in this basis.
Let us denote these operators by $X$, $Y$. Let us also notice that the fact that $V_k$ are obtained as a quotient of $\CC[x,y]$ implies that there is a special vector in $u\in V_n$ which is the image of
$1\in \CC[x,y]$. Thus on the level of points we obtained identification of $\Hilb_{1,n}$ with the quotient of the variety
$$\SymbolPrint{ \mathcal{CV}^{st}}_{\frb\times\frb}=\{(X,Y)\in \frb| [X,Y]=0 \mbox{ and exists } u \mbox{ such that } \CC[X,Y] u=V\}$$
by the group $B$. Respectively, to obtain a model for $\Hilb_{1,n}^L$ one need to impose the condition $Y\in \frn$

The tricky part of the previous description is the fact that we take quotient by the group $B$ which is not a reductive group.
Thus to reveal the scheme theoretic structure of the quotient $\mathcal{CV}^{st}_{\frb\times \frb}/B$ one could consider slightly different model for $\Hilb_{1,n}$ which we describe below.
Let us define $\tilde{C}\subset \Fl\times \gl_n\times \gl_n$ which consists of the triples $(V_{\bullet},X,Y)$ satisfying the following conditions:
$$ X V_i\subset V_i,\quad YV_i\subset V_i, \quad [X,Y]=0.$$

There is a natural action of $GL(n)$ on $\tilde{C}$ and
analogously to the argument in \cite{Nak} one can show that the semi-stable locus $\tilde{C}^{ss}$ of the action is exactly the locus where there is a vector $u$ such that $\CC[X,Y]u=V_n$.
Thus on the set theoretic level we established the isomorphism between $\mathcal{CV}^{st}_{\frb\times\frb}/B$ and $\tilde{C}^{ss}/GL_n$. Lastly, let us notice that the quotient $\tilde{C}^{ss}/GL_n$ has a natural
scheme structure due to GIT theory \cite{N}.

\subsection{Braid complex}\label{ssec: Cb} Now let us recall our construction of the braid group action and relate the construction of the braid group action to the knot invariants. Recall that in the previous sections we discussed
various properties of the particular matrix factorization
\( \calC_\epsilon^i\in \MF_{G\times B^2}(\frg\times \left( G\times \frn\right)^2,W), \):
\[ W(X,g_1,Y_1,g_2,Y_2)=\Tr(\Ad^{-1}_{g_1}(X)Y_1-\Ad^{-1}_{g_2}(X)Y_2).\]

Suppose we are given a braid $\beta=\sigma_{i_1}^{\epsilon_1}\dots \sigma_{i_\ell}^{\epsilon_\ell},$  where $\epsilon_j=\pm 1$. Then we construct the following two-periodic complex of coherent sheaves on
$\calX_\ell:=\frg\times \left( G\times \frn \right)^N$:
$$\calC_{\sigma_{i_1}^{\epsilon_1},\dots,\sigma_{i_\ell}^{\epsilon_\ell}}:=\pi_{12}^*(\calC^{(i_1)}_{\epsilon_1})\otimes\dots\otimes\pi_{n1}^*(\calC^{(i_\ell)}_{\epsilon_\ell})\in \MF_{B^\ell\times G}(\calX_\ell,\pi_{1,n}^*(W)),$$
where $\pi_{i,i+1}:\calX_\ell\to \calX_2$ is the natural projection.





Let us define the ideal $\Icr_\ell\subset\CC[\calX_\ell]$  be generated by the partial derivatives of the functions $W_{i,i+1}:=\pi_{i,i+1}^*(W)$, $i=1,\dots,n-1$. Elements of the ideal $\Icr_\ell$  provide a morphisms of the complex $\calC_{\sigma_{i_1}^{\epsilon_1},\dots,\sigma_{i_\ell}^{\epsilon_\ell}}$  that are homotopic to $0$.
In the next subsections of this section we construct from  $\calC_{\sigma_{i_1}^{\epsilon_1},\dots,\sigma_{i_\ell}^{\epsilon_\ell}}$
a complex of sheaves \(\calS_\beta\), \(\beta=\sigma_{i_1}^{\epsilon_1}\dots\sigma_{i_\ell}^{\epsilon_\ell}\in \Brgr_n\) on the usual Hilbert scheme \(\Hilb_n\) that
depends only on the conjugacy class of \(\beta\). Other words we construct a trace map \(\Brgr_n\to D^{per}_{T_{sc}}(\Hilb_n)\).

\subsection{Geometry of the critical locus} Before we describe the last step of our construction of the knot invariant let us discuss the structure of the scheme $\SymbolPrint{\calX^{crit}_\ell}:=\Spec(R_\ell)$.  First, we describe a particular set of  generators of the ideal $\Icr_\ell$.  Let us notice that we can rewrite
the intermediate potential $W_{i,i+1}$ as follows
\begin{equation}\label{eq: two forms of crit}
\SymbolPrint{W_{i,i+1}}= \Tr(\Ad^{-1}_{g_i}(X)Y_i-\Ad^{-1}_{g_{i+1}}(X)Y_{i+1})
\end{equation}


As we explain in the section~\ref{sec: Knor} the potential $W_{i,i+1}$ could be rewritten as follows:
$$
\Tr(\tilde{X}_{i,+}\Ad_{g_{i,i+1}}(Y_{i+1}))-\Tr(\tilde{X}_{i,--}(Y_{i}-\Ad_{g_{i,i+1}}(Y_{i+1})_{++})).
$$

Here and everywhere below the notations are consistent with the notations of section~\ref{sec: Knor}, that is  $M_+$ ($M_-$) is the upper-triangular (lower-triangular) part of a matrix $M$ and $M_{++}$ ($M_{--}$) is the strictly upper-triangular (lower-triangular) part of $M$ and
$$ g_{i,i+1}=g_{i}^{-1}g_{i+1},\quad \tilde{X}_{i}=\Ad_{g_{i}}^{-1}(X).$$
Taking partial derivatives of the last function with respect to the coordinates $Y_{i+1}$, $\tilde{X}_{i,+}$, $g$, $Y_{i}-\Ad_{g_{i,i+1}}(Y_{i+1})_{++},\tilde{X}_{i,--}$ we obtain equations:
\begin{gather}\label{eq: part ders}
\Ad^{-1}_{g_{i,i+1}}(\tilde{X}_{i})_{--}=0,\\
\Ad_{g_{i,i+1}}(Y_{i+1})_{-}=0,\label{eq: Y neg}\\
 [\Ad_{g_{i,i+1}}(Y_{i+1}),\tilde{X}_{i,+}]=0,\\
 \tilde{X}_{i,--}=0,\\
  Y_{i}=\Ad_{g_{i,i+1}}(Y_{i+1})_{++}.\label{eq: Yi Yi+1}
 \end{gather}
 Let us notice that the last equations the first and  the second lines  actually imply:
 \begin{equation}\label{eq: flag comm}
  [Y_{i},\tilde{X}_{i,+}]=0,
 \end{equation}
 since $[\Ad_{g_{i,i+1}}(Y_{i+1}),\tilde{X}_{i,+}]=[\Ad_{g_{i,i+1}}(Y_{i+1})_{++},\tilde{X}_{i,+}]+[\Ad_{g_{i,i+1}}(Y_{i+1})_{-},\tilde{X}_{i,+}]$.

 On the other hand if we compute the partial derivatives by $X$ of the LHS of (\ref{eq: two forms of crit}) we get
 \begin{equation}\label{eq: AdAd}
 \Ad_{g_i}(Y_i)=\Ad_{g_{i+1}}( Y_{i+1})
 \end{equation}

 Finally let us notice that the (\ref{eq: part ders}) and (\ref{eq: flag comm}) imply that $[Y_{i+1},\tilde{X}_{i+1}]=0$ and hence
 \begin{equation}\label{eq: comm}
 [\Ad_{g_i}Y_i,X]=0.
 \end{equation}
 Let us notice that because of the previous observation (\ref{eq: AdAd}), the last equations (\ref{eq: comm}) are equivalent for different values of $i$.

 Thus we see  that the $G\times B^{\ell}$ equivariant scheme $\calX_\ell^{crit}$ comes equipped with the following maps:
 $$pr_{i}: \calX_\ell^{crit}\to \mathcal{CV}_{\frb\times\frn}, \quad pr_i(g_\bullet,X,Y_\bullet)=(Y_i,\tilde{X}_{i,+}),$$ this map is $B$-equivariant;
 $$pr:\calX_\ell^{crit}\to \mathcal{CV}_{\frg\times\frg}, \quad pr(g_\bullet,X,Y_\bullet)=(\Ad_{g_i}(Y_i),X),$$ this map is $G$-equivariant.
 The varieties $\mathcal{CV}_B$ and $\mathcal{CV}_G$ are the commuting sub varieties of  $\frn\times\frb$ and $\frg\times\frg$ respectively.

 \begin{proposition} The critical locus $\calX^{crit}_\ell$ consists of elements
 $(X,g_1,Y_1,\dots,g_\ell,Y_\ell)$ satisfying
 \begin{gather*}
 \Ad_{g_i}(Y_i)=\Ad_{g_{i+1}}(Y_{i+1})\\
 \Ad_{g_i}^{-1}(X)_{--}=0\\
 [X,\Ad_{g_1}(Y_1)]=0.
 \end{gather*}
 \end{proposition}





\subsection{Stability conditions on critical locus} The locally closed locus $\calX^{crit}_{\ell,st}$ defined by the condition that there is vector $u\in V$ such that
\begin{equation}\label{eq: cyclic}
 \CC[\Ad_{g_i}(Y_i),X]u=\CC^n.
 \end{equation}
From above discussion we see that the image $\pi_i(\calX^{crit}_{\ell,st})$ is inside $\mathcal{CV}_B^{st}$ because $u_i=g_i^{-1}(u)$ is cyclic
$$\CC[Y_i,\tilde{X}_{i,+}]u_i=\CC^n.$$
We denote intersection of $\calX^{crit}_\ell/U^\ell$ with the stable locus by $\calX^{crit}_{\ell,st}/U^\ell$.
Let us denote $j_{st}$ the inclusion of the stable locus.


Let us introduce slightly bigger space $\calX^{crit}_{\ell,fr}$ that is a subvariety of $\calX^{crit}_\ell\times V$ consisting of collections of elements \((X,g_1,Y_1,\dots,g_\ell,Y_\ell,u)\)
 satisfying conditions (\ref{eq: cyclic}). 
The space $\calX^{crit}_{\ell,fr}$ has a natural action of the group $\hat{G}\times B^{\ell}$, $\hat{G}:=G\times \CC^*$, where $\CC^*$-factor acts on the space $V$ by rescaling.
Let us denote by \(\St_\ell\) the quotient \(\hat{G}\backslash \calX^{crit}_{\ell,fr}/B^\ell\).



 \subsection{Stability conditions for the matrix factorizations.}

 We denote by \[\SymbolPrint{V^0}\subset V=\CC^n \] the subset of vectors with a non-zero last coordinate.
 Let us define $\SymbolPrint{\calX_{\ell,fr}}$ as the subvariety  of  $\calX_\ell\times V$ consisting of collections
 \( (X,g_1,Y_1,,\dots,g_\ell,Y_\ell,u)\) such that
\[\CC \langle \Ad_{g_i}^{-1}(X),Y_i\rangle g^{-1}_i(u)=V,\quad g^{-1}(u)\in V^0,\quad i=1,\dots,\ell-1,\]

The map that forgets framing  $\SymbolPrint{\forg }:\calX_{\ell,fr}\to \calX_\ell$  is defined as restriction of the projection \(\calX_\ell\times V\rightarrow \calX_\ell\) on the
open subset \(\calX_{\ell,fr}\subset\calX_\ell\times V\). It
is dominant over the open that we call
the stable locus $\calX_{\ell,st}$. There a unique $G\times B^\ell$-equivariant structure on $\calX_{\ell,fr,i}$ that makes map $\forg$ $G\times B^\ell$-equivariant.

The framed space comes equipped with the natural convolution algebra structure:
$$ \calF\SymbolPrint{\star} \calG:=\pi_{13*}(\CE_{\frn^{(2)}}(\pi_{12}^*(\calF)\otimes\pi_{23}^*(\calG))^{T^{(2)}}),$$
where $\pi_{ij}:\calX_{3,fr}\to \calX_{2,fr}$ are the natural projection maps.
There is a natural pull back map \[\SymbolPrint{\forg}: \MF_{B^2}(\calX_2,W)\to \MF_{B^2}(\calX_{2,fr},W)\] and the key to what follows in the next sections is the Lemma below:

\begin{lemma}\label{lem: for homom}
  The pull-back map $\forg^*$ is the homomorpism of the convolution algebras.
 \end{lemma}
 \begin{proof}
   Let us denote by $\calC$ the matrix factorization \[\pi_{12}^*(\calF)\otimes\pi_{23}^*(\calG)\in\MF_{B^2}(\calX_{3,fr},\pi_{13}^*(W)).\] Then what 
   we need to show is that for any Zarisky open set $U\subset \calX_{2,fr}$  the matrix factorizations of
   $\CC[U]$-modules $\forg^*(\CE_{\frn^{(2)}}(\calC)^{T^{(2)}})$ and $\CE_{\frn^{(2)}}(\forg^*(\calC))^{T^{(2)}})$ are homotopic.

   It is more convenient to restate the above homotopic relation in a slightly more geometric terms.
   Let us denote by $\calX_{3,fr,1}$ the subvariety of $\calX_3\times V$ that is defined by the stability condition:
   $$\CC\langle \Ad_{g_1}^{-1}(X), Y_1\rangle g_1^{-1}(u)=V,\quad g_1^{-1}(u)\in V^0$$
   Both spaces $\calX_{3,fr,1}$ and $\calX_{3,fr}$ project onto $\calX_{2,fr}$ and to distinguish these projections we denote the
   projection from the first space by $\pi_{13}$ and from the second by $\pi_{13}^{fr}$. Thus the statement of the lemma is equivalent
   to the homotopy equivalence of the $\CC[U]$ matrix factorizations
\begin{equation}\label{eq: framed homotopy}
     \CE_{\frn^{(2)}}(\calC|_{\pi_{13}^{-1}(U)})^{T^{(2)}}\sim \CE_{\frn^{(2)}}(\calC|_{\pi^{fr,-1}_{13}(U)})^{T^{(2)}}.
\end{equation}
   
   From the previous discussion of the equations for the support of $\calC_{\sigma_{i_1}^{\epsilon_1},\dots,\sigma_{i_l}^{\epsilon_l}}$ we see that
   $$ \calX_3^{crit}\times V \cap\calX_{3,fr,1}=\calX_3^{crit}\times V \cap \calX_{3,fr}. $$
   Finally since $Z=\calX_{3,fr,1}\setminus\calX_{3,fr}$ is the closed subvariety the equation (\ref{eq: framed homotopy}) follows from the general lemma that
   follows the proof: we use $X=\calX_{3,fr,1}$ and $Z'=\calX_3^{crit}\times V$ in the lemma.
 \end{proof}

 Typically the stable locus is a open subset of some more natural bigger space and we are interesting in comparing the matrix factorizations on the big
 space with the ones on the stable locus. The lemma below shows that the category does not change as we pass to the stable locus as long the unstable
 locus does not intersect the critical locus of the potential. The lemma is shown in more general case and we do not refer to a stability a condition. 

 \begin{lemma} Suppose $X$ is a quasi-affine variety and $\calF=(M,D)\in \MFs(X,W)$, $W\in \CC[X]$. The elements of \(\CC[X]\) act on
   \(\MFs(X,W)\) by multiplication. Let us assume that
   the elements of the ideal $I=(f_1,\dots,f_m)\subset \CC[X]$ act by zero homotopies on $\calF$ and $Z'\subset X$ is the
   zero locus of $I$. Let $Z\subset X$ be a subvariety
   defined by $J=(g_1,\dots,g_n)$ such that $Z\cap Z'=\emptyset$. Then $\calF$ is homotopic to $\calF|_{X\setminus Z}$ as
   matrix factorization over $\CC[X]$.
\end{lemma}
\begin{proof}
  Let us the convention $X_f$ for the open Zarisky subset of $X$ which is a complement of the zero locus of $f$.
  By the conditions of the lemma the open sets $X_{f_i}$ and $X_{g_i}$ form a Cech cover of $X$ and we denote by the
  $\check{C}_{f,g}(\calF)$ the Cech resolution of $\calF$. The terms of the complex $\check{C}_{f,g}(\calF)$ are the matrix factorizations
  $\calF|_{X_{S,T}}$ where $S\subset \{ 1,\dots, m\}$ and $T\subset \{ 1,\dots, n\}$  and $X_{S,T}$ is the complement to the zero locus
    of $\prod_{s\in S}f_s\prod_{t\in T} g_t$.

    By the conditions of the lemma the matrix factorization $\calF|_{X_{S,T}}$ is contractible if $S\ne \emptyset$. Thus after performing all the
    contractions we  are left with $\check{C}_{g}(\calF)$ which is exactly $\calF|_{X\setminus Z}$ and the lemma follows.
\end{proof}

\subsection{Taking the $B^\ell$ quotient}

Let us define the closure map: $\cl: \calX_{\ell-1}\rightarrow \calX_\ell$ by
$$ \SymbolPrint{\cl}(X,g_1,Y_1,\dots,g_{\ell-1},Y_{\ell-1}):=(X,g_1,Y_1,\dots,g_{\ell-1},Y_{\ell-1},g_1,Y_1).$$
This map is $G\times B^{\ell-1}$-equivariant with $B^{\ell-1}$-action on $\calX_{\ell-1}$ induced by
the group homomorphism $\mathrm{id}_{2,\dots,\ell-1}\times \Delta_{1,\ell}$.

We use the same notation
for the maps of the framed and stable spaces. Let us notice that  the preimage  of the map $\cl$ applied to $\calX_{\ell+1,fr}$ is the cyclicly framed space \(\SymbolPrint{\calX_{\ell,fc}}\) that consists of points  \((X,g_1,Y_1,\dots,g_\ell,Y_\ell,v)\) satisfying 
:
$$ \CC\langle\Ad_{g_i}^{-1}(X),Y_i\rangle g^{-1}(u)=V,\quad  g^{-1}(u)\in V^0, i=1,\dots,\ell.$$
The intersection with the zero-locus of $\mathcal{I}^{crit}_\ell$ we denote by $\calX_{\ell,fc}^{crit}$.

Respectively, we have the pull back map:
$$\cl^*:\MF_{B^{\ell-1}}(\calX_{\ell,fr},W)\to \MF_{B^{\ell-1}}(\calX_{\ell-1,fc},0),$$
and let us denote by $C_{\sigma_{i_1}^{\epsilon_1},\dots,\sigma_\ell^{\epsilon_\ell}}$ the pull-back $\cl^*(\calC_{\sigma_{i_1}^{\epsilon_1},\dots,\sigma_\ell^{\epsilon_\ell}})$.
We use the same nation for $\cl^*(\calC_{\sigma_{i_1}^{\epsilon_1},\dots,\sigma_\ell^{\epsilon_\ell}})\in \MF_{B^{\ell-1}}(\calX_{\ell-1},0)$, usually it is clear from the
context which category of matrix factorizations we work with.



Let us denote by
$\SymbolPrint{\MFcrit_{B^\ell}}(\calX_{\ell,fc},0)\subset  \MF_{B^\ell}(\calX_{\ell,fc},0)$, $\MFcrit_{B^\ell}(\calX_{\ell},0)\subset  \MF_{B^\ell}(\calX_{\ell},0)$
the subcategory consisting  of  matrix factorizations
with the property that the ideal $\calI^{crit}_\ell$ acts by zero homotopy. The element $C_{\sigma_{i_1}^{\epsilon_1},\dots,\sigma_{\ell+1}^{\epsilon_{\ell+1}}}$ is an element of
the subcategory $\MFcrit_{B^\ell}(\calX_{\ell,fc},0)$.

Given an element $C\in \MFcrit_{B^\ell}(\calX_{\ell},0)$ the homology $$\cohog{*}{\CE_{\frn^\ell}(C)}:=\cohog{even}{\CE_{\frn^\ell}(C)}
\oplus\cohog{odd}{\CE_{\frn^\ell}(C)}$$ is
a module over the ring of $\frn^\ell$ invariants $R_\ell^{\frn^\ell}$ where $R_\ell:=\CC[\calX_{\ell}^{crit}]$.  Let us fix a notation for  the affine scheme $\calX^{crit}_{\ell}/U^\ell:=\Spec(R_{\ell}^{\frn^\ell})$.  We use notation $\mathcal{H}_U(\calX_{\ell},C)$ for $\cohog{*}{\CE_{\frn^\ell}(C)}$ considered as
$R_\ell^{\frn^\ell}$-module.

There is an natural map $\forg: \calX^{crit}_{\ell,fc}\rightarrow\calX^{crit}_{\ell}$ and by applying the respective pull-back map we obtain
$$\SymbolPrint{\mathcal{H}_U}(\calX_{\ell,fc},C):=\forg^*(\mathcal{H}_U(\calX_\ell, C))\in \MFs_{T^{\ell}}(\calX_{\ell,fc}/U^{\ell},0).$$

The complex of sheaves $\calH_U(\calX_{\ell,fc},C)$ is $T^\ell\times \hat{G}$-equivariant complex of sheaves on $\calX^{crit}_{\ell,fr}/U^\ell$ and we define
$\SymbolPrint{\calH_B}(\calX_{\ell,fc},C)$ as a complex of sheaves of $T^\ell$-invariant sections of $\calH_U(\calX_{\ell,fc},C)$. Thus $\calH_B(\calX_{\ell,fc},C)$ is a
complex of quasi-coherent sheaves on the quotient $\calX_{\ell,fc}/B^\ell$.
The complex of sheaves $\calH_{B}(\calX_{\ell,fc},C)$ is $\hat{G}$-equivariant hence the complex sheaves $\calH_{B}(\calX_{\ell,fc},C)^{\hat{G}}$ of $\hat{G}$-invariant sections of $\calH_{B}(\calX_{\ell,fc},C)$ is a complex of quasi coherent sheaves on
$\St_\ell$.






\subsection{Conjugacy class invariant} In section~\ref{ssec: Cb} we constructed for every word $\beta=\sigma_{i_1}^{\epsilon_1}\cdot\dots\cdot\sigma_{i_\ell}^{\epsilon_\ell}$
of length $\ell$ of generators of the braid group $\Brgr_n$  complexes $C_{\sigma_{i_1}^{\epsilon_1},\dots,\sigma_{i_\ell}^{\epsilon_\ell}}\in
\MF_{B^\ell\times G}(\calX_{\ell,fc},0)$. The group \(\hat{G}\) acts only on the first and last factors in \(\calX_{\ell,fc}\) there is a \(\hat{G}\)-equivariant projection 
\(pr^{i}:\calX_{\ell,fc}\rightarrow \frg^2\times V\):
\[pr^i(X,g_1,Y_1,\dots,g_\ell,Y_\ell,u)=(X,\Ad_{g_i}(Y_i),u).\]
The image of the map \(pr^1\) restricted to \(\calX^{crit}_{\ell,fc}\) consists of stable triples \((X,Y,u)\), \([X,Y]=0,\CC[X,Y]u=V\). Thus we have a well-defined map
\(pr^1:\St_\ell\to \Hilb_n\).

Since $C_{\sigma_{i_1}^{\epsilon_1},\dots,\sigma_{i_\ell}^{\epsilon_\ell}}\in \MFcrit_{B^\ell\times G}(\calX_{\ell,fc},0)$
we can apply the above constructions
to obtain  a complex of  sheaves on \(\Hilb_n\) \[\SymbolPrint{\calS_{\sigma_{i_1}^{\epsilon_1},\dots,\sigma_{i_\ell}^{\epsilon_\ell}}}:=pr^1_{*}(\calH_{B}(C_{\sigma_{i_1}^{\epsilon_1},\dots,\sigma_{i_\ell}^{\epsilon_\ell}}))^{\hat{G}}.\]

\begin{theorem} \label{thm: conj inv} The  sheaf  $\calS_{\sigma_{i_1}^{\epsilon_1},\dots,\sigma_{i_\ell}^{\epsilon_\ell}}$ depends only on conjugacy class of $\beta=\sigma_{i_1}^{\epsilon_1}\cdot\dots\cdot\sigma_{i_\ell}^{\epsilon_\ell}$:
\begin{equation}\label{eq: 1 in mid}
  \calS_{\sigma_{i_1}^{\epsilon_1},\dots,1,\dots,\sigma_{i_\ell}^{\epsilon_\ell}}
  \sim\calS_{\sigma_{i_1}^{\epsilon_1},\dots,\sigma_{i_\ell}^{\epsilon_\ell}}
\end{equation}
\begin{equation}\label{eq: conj invce}
  \calS_{\sigma_{i_1}^{\epsilon_1},\dots,\sigma_{i_\ell}^{\epsilon_\ell}}=\calS_{\sigma_{i_2}^{\epsilon_2},
    \dots,\sigma_{i_\ell}^{\epsilon_\ell},\sigma_{i_1}^{\epsilon_1}}
\end{equation}
\begin{equation}
\calS_{\sigma_k, \sigma_k^{-1},\sigma_{i_1}^{\epsilon_1}\dots}\sim\calS_{\sigma_{i_1}^{\epsilon_1},\dots},\quad \calS_{\sigma_k, \sigma_{k+1}, \sigma_k,\dots}=\calS_{\sigma_{k+1},\sigma_k,\sigma_{k+1},\dots}
\end{equation}
where \(\sim\) stands for the homotopy equivalence.
\end{theorem}
\begin{proof}
  The conjugacy invariance (\ref{eq: conj invce}) follows almost immediately from the construction. Indeed, the definition of \(pr^i\) implies that
\[  \calS_{\sigma_{i_2}^{\epsilon_2},
  \dots,\sigma_{i_\ell}^{\epsilon_\ell},\sigma_{i_1}^{\epsilon_1}}=pr^2_*(\calH_{B}(C_{\sigma_{i_1}^{\epsilon_1},\dots,\sigma_{i_\ell}^{\epsilon_\ell}})).\] On the other hand
\(pr^1|_{\calX_{\ell,fc}^{crit}}=pr^2|_{\calX^{crit}_{\ell,fc}}\) and hence the equality.

  To prove the rest of the theorem we introduce slightly more general definition:
  \[\calS_{\beta_1,\dots,\beta_m}:=pr^1_*(\mathcal{H}_B(\calX_{m,fc},\forg^*(\cl^*(\pi_{12}^*(\calC_{\beta_1})\otimes
    \dots\otimes  \pi_{m,m+1}^*(\calC_{\beta_m})))))^{\hat{G}}\]

  Now let us observe that $pr^1=pr^1\circ \hat{\pi}_2$ where $\hat{\pi}_2:\calX_{m,fc}\rightarrow\calX_{m-1,fc}$ is the
  projection along the second factor. Hence we have
  \[ \calS_{\beta_1,\beta_2,\dots,\beta_m}=\calS_{\beta_1\cdot\beta_2,\dots,\beta_m}\]
  It implies the desired equations.


\end{proof}

For an element $\beta\in \Brgr_n$ we denote by $[\beta]$ its conjugacy class.
Thus for any element of the braid group $\beta\in \Brgr_n$ we have a quasi-coherent sheaf $\calS_{[\beta]}$
on $\Hilb_n$.

It is not obvious to us how one could attach sheaf on the flag Hilbert scheme to a conjugacy class in the braid group. There are well-defined maps
\(pr_i:\St_\ell\rightarrow \Hilb_{1,n}^L\) that send \((X,g_1,Y_1,\dots,g_\ell,Y_\ell,u)\) to \((\Ad_{g_i}^{-1}(X),Y_i),g^{-1}(u))\).
Naively one could hope that
 the push-forwards $pr_{i*}(\calH^*_B(\calX_{2,fc},C_{\sigma_{i_1}^{\epsilon_1},\dots,\sigma_{i_\ell}^{\epsilon_\ell}}))$  for some $i$ is a such sheaf
   but it is unclear why it would satisfy the properties from the
theorem.  However, as we will see later, the push-forwards on the flag Hilbert schemes are useful for understanding Markov moves, since there are natural maps between the
flag Hilbert schemes of different ranks. It is also clear from the argument of the previous theorem
that for any $\beta=\sigma_{i_1}^{\epsilon_1}\cdot\dots\cdot\sigma_{i_\ell}^{\epsilon_\ell}$ there is a well-defined complex:
$$\SymbolPrint{\mathcal{S}_\beta}:=pr_{1,*}(\calH^*_B(\calX_{2,fc},C_{\sigma_{i_1}^{\epsilon_1},\dots,\sigma_{i_\ell}^{\epsilon_\ell}}))\in D^{per}_{T_{sc}}(\Hilb_{1,n}^L).$$